\documentclass[12pt]{article}
\usepackage[a4paper]{geometry}
\usepackage[latin1]{inputenc}

\usepackage{amsmath,amsfonts,amssymb,amsbsy,verbatim}  %,graphicx,pstricks
\usepackage{amsthm}
\usepackage{textcmds}

\newtheoremstyle{tplain}{3pt}{3pt}{\rmfamily}{}{\bfseries}{.}{0.5em}{}
\theoremstyle{tplain}

\usepackage{pgf,tikz}
\usetikzlibrary{patterns}
\usetikzlibrary{calc,decorations.pathmorphing,decorations.pathreplacing,fit}
\usetikzlibrary{matrix}
\usetikzlibrary{arrows}

\newtheorem{Theorem}{Theorem}[section]
\newtheorem{Corollary}[Theorem]{Corollary}
\newtheorem{Proposition}[Theorem]{Proposition}
\newtheorem{Definition}[Theorem]{Definition}
\newtheorem{Lemma}[Theorem]{Lemma}
\newtheorem{Remark}[Theorem]{Remark} %\newtheorem{Remark}{Remark}[section]
\newtheorem{Example}[Theorem]{Example} %\newtheorem{Example}{Example}[section]

\numberwithin{equation}{section}

\usepackage{hyperref} %this seems to produce errors with hyperlinks

\newcommand{\N}{{\mathbb N}}
\newcommand{\R}{{\mathbb R}}
\newcommand{\C}{{\mathbb C}}
\newcommand{\Z}{{\mathbb Z}}
\newcommand{\red}{\color{red}}
\newcommand{\blue}{\color{blue}}

\newcommand{\green}{\color{newgreen}}
\newcommand{\magenta}{\color{magenta}}

\definecolor{newgreen}{rgb}{0.05625,0.5625,0}

\definecolor{brown}{cmyk}{0,0.3,1,0.3}
\definecolor{darkgreen}{cmyk}{0.5,0,0.5,0}
\newcommand{\black}{\bf\textcolor{black}{black}}
\newcommand{\brown}{\textcolor{brown}{brown}}
\newcommand{\darkgreen}{\textcolor{darkgreen}{green}}

\definecolor{newpurple}{rgb}{0.6,0,1.2}   
\newcommand{\purple}{\color{newpurple}}%%%%%%%%%%%%%

\newcommand{\rootthree}{1.7320508}

\newcommand{\vYTd}[4]{
\vcenter{\hbox{
\begin{tikzpicture}[x={(0in,-#1)},y={(#1,0in)}] % matrix coordinate
\foreach \rowi [count=\i] in {#3} {
 \foreach \e [count=\j] in \rowi {
  \draw (\i,\j) rectangle +(-1,-1);
  \draw (\i-0.5,\j-0.5) node {$#2\e$};
 }
}
{#4}
\end{tikzpicture}
}}
}

\newcommand{\YT}[3]{
\vcenter{\hbox{
\begin{tikzpicture}[x={(0in,-#1)},y={(#1,0in)}] % matrix coordinate
\foreach \rowi [count=\i] in {#3} {
 \foreach \e [count=\j] in \rowi {
  \draw (\i,\j) rectangle +(-1,-1);
  \draw (\i-0.5,\j-0.5) node {$#2\e$};
 }
}
\end{tikzpicture}
}}
}

\def\ov{\overline}
\newcommand{\tov}[1]{$\ov{\text{#1}}$}

\def\ds{\displaystyle}
\def\sc{\scriptstyle}

\def\rev{{\rm rev}}
\def\H{{\hat{H}}}   %\def\H{{H'}}
\def\K{{\hat{K}}}   %\def\K{{K'}}
\newcommand{\hiveH}[2]{
\foreach\i in{0,...,#1}\draw(\i+#2,\i+#2)--(\i+#2,#1+#2); % alpha edges
\foreach\i in{0,...,#1}\draw(0+#2,\i+#2)--(\i+#2,\i+#2); % beta edges
\draw(0+#2,0+#2)--(#1+#2,#1+#2) %gamma edges
}
\newcommand{\hiveK}[3]{
\foreach\i in{0,...,#1}\draw(\i+#3,#1+#3)--(\i+#3,#2+#3); % alpha edges
\foreach\i in{#1,...,#2}\draw(\i+#3,\i+#3)--(\i+#3,#2+#3); % alpha edges
\foreach\i in{#1,...,#2}\draw(0+#3,\i+#3)--(\i+#3,\i+#3); % beta edges
\draw(#1+#3,#1+#3)--(#2+#3,#2+#3); %gamma edges
}
\def\qed{\hfill$\Box$\vskip3pt}
\def\leqq{\leq}

\newcommand\cell[2]{\vcenter{\hbox{\tikz[x={(0in,-0.25in)},y={(0.25in,0in)}]\draw[#1,thick] (0,0) rectangle (1,1) (0.5,0.5) node[black]{$\mathstrut#2$};}}}

\newcommand{\numero}[1]{$\vcenter{\hbox{\tikz\node[circle,draw,inner sep=1pt]{\small#1};}}$}

\begin{document}

\title{The involutive nature of the Littlewood--Richardson commutativity bijection}

\author{O. Azenhas,\thanks{e-mail: oazenhas@mat.uc.pt} \\
CMUC, Department of Mathematics, University of Coimbra\\
3000-501 Coimbra, Portugal\\
\\
R.C. King\thanks{e-mail: r.c.king@soton.ac.uk} \\
Mathematical Sciences, University of Southampton, \\
Southampton, SO17 1BJ, England \\
and \\
\\
I. Terada\thanks{e-mail: terada@ms.u-tokyo.ac.jp} \\
Graduate School of Mathematical Sciences, University of Tokyo, \\
Komaba 3-8-1, 
Meguro-ku, Tokyo 153-8914,
Japan \\
}

\maketitle

\date

\begin{abstract}

Littlewood-Richardson (LR) coefficients $c_{\mu\nu}^\lambda$ may be evaluated by means of several combinatorial models, including the original LR tableaux of skew shape $\lambda/\mu$ and weight $\nu$ and the LR hives with boundary edge labels $\lambda$, $\mu$ and $\nu$. 
Unfortunately, neither of these reveal in any obvious way the well-known symmetry property $c_{\mu\nu}^\lambda=c_{\nu\mu}^\lambda$. 
Here we introduce two maps, $\rho^{(n)}$ on LR tableaux and $\sigma^{(n)}$ on LR hives, that each interchange contributions to $c_{\mu\nu}^\lambda$ and $c_{\nu\mu}^\lambda$ for any partitions $\lambda$, $\mu$, $\nu$ of lengths no greater than $n$, and then prove not only that each of them is a bijection, thereby making manifest the required symmetry property, but also that both maps are involutions.
The map $\rho^{(n)}$ involves the iterative action of deletion operators on a given LR tableau of skew shape $\lambda/\mu$ and weight $\nu$, that produce a sequence of successively smaller tableaux whose consecutive inner shapes define a  certain Gelfand-Tsetlin pattern and determine a partner LR tableau of skew shape $\lambda/\nu$ and weight $\mu$. 
Similarly, the map $\sigma^{(n)}$ involves repeated path removals from a given LR hive with boundary edge labels $(\lambda,\mu,\nu)$ that give rise to a sequence of hives whose left-hand boundary edge labels define the same Gelfand-Tsetlin pattern as before, which is sufficient to determine a partner LR hive with boundary edge labels $(\lambda,\nu,\mu)$. 
The deletions in tableaux are organised so as to preserve the semistandard and lattice permutation properties of LR tableaux, while the path removals in hives
are designed to preserve both the triangle condition on edge labels and the hive rhombus gradient positivity conditions. 
At all stages illustrative examples are provided.

\end{abstract}

\section{Introduction and statement of results}\label{Sec-intro}
Let $n$ be a fixed positive integer and let $x=(x_1,x_2,\ldots,x_n)$ be a sequence of
indeterminates. Then, for each partition $\lambda$ of length $\ell(\lambda)\leq n$ and weight $|\lambda|$,
there exists a Schur function $s_\lambda(x)$ which is a homogeneous symmetric polynomial in the $x_k$
of total degree $|\lambda|$. These Schur functions $s_\lambda(x)$ for all such $\lambda$ form a linear basis
of the ring $\Lambda_n$ of symmetric polynomials in the components of $x$. It follows
that
\begin{equation}\label{Eq-LRcoeffs}
    s_\mu(x)\ s_\nu(x)  = \sum_\lambda\ c_{\mu\nu}^\lambda\ s_\lambda(x)\,,
\end{equation}
where the coefficients $c_{\mu\nu}^\lambda$ are known as Littlewood--Richardson (LR) coefficients.
These coefficients are independent of $n$. They are non-negative integers that may be evaluated by means of the
Littlewood--Richardson rule~\cite{LR} as the number of Littlewood--Richardson tableaux
of skew shape $\lambda/\mu$ and of weight $\nu$, where $\lambda$ and $\mu$ specify the outer and the inner shape, respectively,
of the tableau, and the parts of $\nu$ specify the numbers of its entries $k$ for $k=1,2,\ldots,n$.
Equivalently, $c_{\mu\nu}^\lambda$ is the number of Littlewood--Richardson $n$-hives with boundary edge labels 
specified by the ordered triple $(\lambda,\mu,\nu)$~\cite{KT1,Bu},
where each of the three partitions has $n$ parts through the inclusion if necessary of trailing zeros.
These are the two combinatorial models to be considered in this paper; see Sections \ref{Subsec-tableaux} and \ref{Subsec-hive-model} 
for details.
Although logically distinct, they may be thought of as being equivalent thanks to the existence of a bijection between 
LR tableaux and LR hives described by Fulton in the Appendix to~\cite{Bu}. 
The bijective map $f^{(n)}$ from the set ${\cal LR}^{(n)}$ of all LR tableaux with no more than 
$n$ rows to the set ${\cal H}^{(n)}$ of all hives of side length $n$ is
explained in Section~\ref{Subsec-hives-tableaux}. (See also \cite{PV1}.)  

Within these two models, the set of LR tableaux of shape $\lambda/\mu$ and  
weight $\nu$ is denoted by $\mathcal{LR}(\lambda/\mu,\nu)$,
and the corresponding set of $n$-hives by ${\cal H}^{(n)}(\lambda,\mu,\nu)$
for any fixed $n\geq\ell(\lambda)$. We then have
\begin{equation}\label{Eq-LRcoeff-tab-hive}
   c_{\mu\nu}^\lambda = \#\{T\in{\cal LR}(\lambda/\mu,\nu)\} = \#\{H\in{\cal H}^{(n)}(\lambda,\mu,\nu)\}\,.
\end{equation}
Unfortunately, although the definition~(\ref{Eq-LRcoeffs}) makes it immediately clear that $c_{\mu\nu}^\lambda=c_{\nu\mu}^\lambda$,
the same cannot be said of either of the combinatorial formulae in (\ref{Eq-LRcoeff-tab-hive}). The aim here is to provide 
bijective maps $\rho^{(n)}:{\cal LR}^{(n)}\rightarrow{\cal LR}^{(n)}$ and $\sigma^{(n)}:{\cal H}^{(n)}\rightarrow{\cal H}^{(n)}$,
such that $\rho^{(n)}:T\in {\cal LR}(\lambda/\mu,\nu)\mapsto S\in{\cal LR}(\lambda/\nu,\mu)$ and
$\sigma^{(n)}:H\in{\cal H}^{(n)}(\lambda,\mu,\nu)\mapsto K\in{\cal H}^{(n)}(\lambda,\nu,\mu)$,
thereby proving the symmetry of the LR-coefficients. It is further shown that these maps are involutive.	

Each LR $n$-hive is an equilateral, triangular graph of side length $n$, 
which may be specified by means of vertex labels, edge labels or rhombus gradients, as described in Section~\ref{Sec-notation}.
It should be noted that in the edge labelling scheme an LR hive consists of three interlocking, symmetrically oriented 
Gelfand--Tsetlin (GT) patterns \cite{GT1}, with their types specified by the inner shape, weight and outer shape of the corresponding LR tableau,
as explained in Section~\ref{Subsec-hive-model}. 
We would remark, in particular, that as a result of the bijection between LR hives and LR tableaux, 
an LR tableau can be completely specified by means of any one of the three GT patterns associated with 
an LR hive and one additional sequence of boundary edge labels~\cite{BZphy,GZpolyed,KiBe,HenKam,PV2}.

The nature of our two maps $\rho^{(n)}$ on $T\in\mathcal{LR}(\lambda/\mu,\nu)$ and $\sigma^{(n)}$ on $H\in{\cal H}^{(n)}(\lambda,\mu,\nu)$ 
is that they proceed by providing sequences consisting of successively smaller LR tableaux and of successively smaller LR hives in which the sequences of inner shapes
and left-hand boundary edges, respectively, determine the corresponding images $S\in\mathcal{LR}(\lambda/\nu,\mu)$ and $K\in{\cal H}^{(n)}(\lambda,\nu,\mu)$.
They do so by specifying completely a GT pattern of type $\mu$ associated with $K$. 
Both maps have corresponding inverses $\rho^{(n)}{}^{-1}$ and $\sigma^{(n)}{}^{-1}$ acting in a reverse manner.
Their action on $T\in\mathcal{LR}(\lambda/\mu,\nu)$ and $H\in{\cal H}^{(n)}(\lambda,\mu,\nu)$, build up sequences 
consisting of successively larger LR tableaux and of successively larger LR hives, repectively, with their inner shapes and left-hand boundary edges
determined by the GT pattern of type $\nu$ associated with $H$, culminating, as will be shown, in the same images as before, namely 
$S\in\mathcal{LR}(\lambda/\nu,\mu)$ and $K\in {\cal H}^{(n)}(\lambda,\nu,\mu)$. Thus the maps $\rho^{(n)}$ and $\sigma^{(n)}$
are not only bijective, but also involutive.

\subsection{Littlewood--Richardson commutativity involutions and the $\mathfrak{gl}_n$-crystal commuter}
\label{Subsec-intro-LRcrystal}
Schur functions in $n$ variables may be interpreted as irreducible characters of the general linear group, $GL_n(\mathbb{C})$, 
and the Littlewood--Richardson coefficient $c_{\mu\nu}^\lambda$ % %rule 
then has a representation-theoretic significance %interpretation 
as the number of times the irreducible representation $V_\lambda$ appears in the decomposition of 
the tensor product $V_\mu\otimes V_\nu$ of irreducible representations of $GL_n(\mathbb{C})$. %and to an induced tensor product of Specht modules. 
Additionally, the Schur functions have an intersection-theoretic interpretation as representatives of Schubert classes and, 
in the context of multiplication in the cohomology of the Grassmannian, the Littlewood--Richardson coefficient $c_{\mu\nu}^\lambda$ is the multiplicity 
of the Schubert class $\sigma_\lambda$ in the product $\sigma_\mu \sigma_\nu$ of Schubert classes $\sigma_\mu$ and $\sigma_\nu$.
In both these contexts the products are manifestly commutative so that $c_{\mu\nu}^\lambda=c_{\nu\mu}^\lambda$.

So far as the equality of the numbers $c_{\mu\nu}^{\lambda}$ and $c_{\nu\mu}^{\lambda}$ is concerned just the existence of any bijective map $\rho$
between the sets they serve to enumerate may conclude the story. On the other hand, there are situations where each element 
of the two sets represents some distinguishable object, and then the action of the map $\rho$ taking an element of one set 
to an element of the other may carry a quite explicit piece of information that deserves further investigation.
The parametrisation of the connected components of the tensor product of crystal bases $B_\mu\otimes B_\nu$ is one such situation, 
particularly so in view of the fact that the flip map $a\otimes b\mapsto b\otimes a$ does not give an isomorphism of 
crystals $B_\mu\otimes B_\nu\to B_\nu\otimes B_\mu$. 

The $\mathfrak{gl}_n$-crystal version of the Littlewood--Richardson rule \cite{T1,T2,Nak,hoon} amounts to 
the decomposition of the tensor product of crystals $B_\mu\otimes B_\nu$ into a disjoint union of crystals $B_\lambda$ with 
multiplicities $c_{\mu\nu}^\lambda$ counted by the LR tableaux in $\mathcal{LR}(\lambda/\mu,\nu)$.  
In~\cite{HenKam}, Henriques and Kamnitzer realise 
their crystal commutor, which had 
been defined in \cite{HK2} in a more general setting, using 
the language of tableaux and hives in the case of $\mathfrak{gl}_n$,
where the key roles are played by the highest and the lowest weight elements of the crystal $B_\lambda$, and the Sch\"utzenberger 
involution $\xi$ which \qq{reverses} 
the crystal. These highest and lowest weight elements are characterised in part by two of 
the three interlocking GT patterns associated with a hive, as will be made clear in Section~\ref{2subSec-intro}.
The $\mathfrak{gl}_n$-crystal commutor $B_\mu\otimes B_\nu\to B_\nu\otimes B_\mu$, linked in [HK06b] with the Sch\"{u}t\-zen\-ber\-ger involution, 
induces a bijection from the set of connected components isomorphic to $B_\lambda$ in $B_\mu\otimes B_\nu$ to the analogous set in $B_\nu\otimes B_\mu$.
Through the connection between such connected components and hives, they obtain a bijective map 
$\mathcal{H}^{(n)}(\lambda,\mu,\nu)\to\mathcal{H}^{(n)}(\lambda,\nu,\mu)$, inheriting its
involutive nature from the crystal commutor.
Thus they prove not only that crystal tensor product multiplicities are symmetric, ensuring that $c_{\mu\nu}^\lambda=c_{\nu\mu}^\lambda$,
but also that this map is involutive. This map is also shown to coincide with another map constructed using a special case of octahedron recurrence.
In \cite{DK05}, instead of using Young tableaux, Danilov and Koshevoy introduce other combinatorial objects,  called arrays, and  endow them with a crystal structure to define a similar commutor.

The Littlewood--Richardson commutativity bijections in \cite{HenKam,PV2,DK05,DK08}, with exception of the bijection $\rho_3$ described in \cite{PV2} and originally defined 
by the first author in \cite{Az1,Az2}, have their involution property based on that of the above crystal commutor \cite{HK2, DK05}, which in turn is based on properties of the Sch\"utzenberger involution. 
In particular, the map $\rho_1$ defined by means of tableau switching operations~~\cite{bss}, is written as a composition of three Sch\"utzenberger involutions~(\cite{PV2}, Section 3.1), in order to conclude its involution property. In fact, Danilov and Koshevoy show in \cite{DK08} %\cite{DK051} 
that their commutor, the Henriques--Kamnitzer hive and crystal commutors~\cite{HenKam}, and the maps $\rho_2$ and its inverse $\rho'_2$ in~\cite{PV2}, as well as the tableau switching map $\rho_1$~\cite{bss}, all coincide, and are therefore all involutions. The map $\rho_3$, was not considered by Danilov and Koshevoy. Our task here is to obtain an independent proof of the involutive nature of $\rho_3$, realised here
through the bijective action of $\rho^{(n)}$ on LR tableau and $\sigma^{(n)}$ on LR hives.

\subsection{Our Littlewood--Richardson commutativity involution} 
\label{Subsec-ourLR}

In Sections~\ref{Sec-tab-del-ops}--\ref{Sec-prf-delaying}, we discuss the map originally defined by the first author, which we denote by $\rho^{(n)}$, taking any LR tableau $T$ of shape $\lambda/\mu$ and weight $\nu$ to another such tableau $S$ of shape $\lambda/\nu$ and weight $\mu$, where $\lambda$, $\mu$ and $\nu$ are partitions of length $\le n$, and prove that the map $\rho^{(n)}$ has the involutive nature.

The map $\rho^{(n)}$ is introduced in Sections~\ref{Sec-tab-del-ops} and~\ref{Sec-tab-bijection}. %3 and 4 in a bottom-up manner, but here we take the top-down approach.
Here we begin by outlining the content of Section~\ref{Sec-tab-bijection}, only referring back to the content of Section~\ref{Sec-tab-del-ops} as and when necessary.
The map is defined by way of transforming the tableau $T=T^{(n)}$ into gradually smaller tableaux $T^{(n-1)},T^{(n-2)},\dots,T^{(1)}$ and finally into the empty tableau $T^{(0)}$, in such a way that, upon each transition from $T^{(r)}$ to $T^{(r-1)}$, the outer shape simply loses its $r$th row, the inner shape shrinks by losing a horizontal strip, the entries $r$ are all removed but the entries $1,2,\dots,r-1$ are shifted in the tableau in a certain way without being lost; and while doing so the tableau $\rho^{(n)}(T)=S$ is constructed, gradually from bottom to top, in such a way that the $r$th row of $S$ is determined by the information on the change of the inner shape from $T^{(r)}$ to $T^{(r-1)}$ along with the number of entries $r$ removed from $T^{(r)}$ (See Definition~\ref{Def-TtoS}).
The transformation from $T^{(r)}$ to $T^{(r-1)}$ is achieved by applying a sequence of \textit{deletion operators} (Definitions \ref{Def-dp} and \ref{Def-full-r-del}).
The action of each such operator is a gentle modification of the inverse of Sagan--Stanley's \textit{internal insertion}.
To be more precise, the action starts from the rightmost corner of the bottommost row of the tableau, and that corner is deleted in terms of the tableau's shape.
As for the entry of that corner, it is simply removed if that entry matches the row number $r$, and in such a case that is the only change made in the tableau by the operator.
Otherwise a certain path is generated starting from that corner, and the entries get shifted upwards along the path.
The path extends until reaching one of the corners of the inner shape, and an entry bumped out of the path is placed in that corner.
Moreover, this operation preserves the LR property.
Such an operation is repeated as many times as there are cells in the $r$th row of $T^{(r)}$ to produce $T^{(r-1)}$.
It easily follows from the definitions that the constructed tableau $S$ has shape $\lambda/\nu$ and weight $\mu$, but in showing its LR property (Theorem~\ref{The-TtoS}), a key role is played by the comparison results (Lemmas \ref{Lem-ep}, \ref{Lem-hp} and \ref{Lem-vp}) of the paths taken by the deletion procedures when applied successively.
An illustration of the complete procedure of constructing $S$ from $T$ can be found in Example~\ref{Ex-tab}.

The map $\rho^{(n)}$ is injective, since one can recover $T^{(r)}$ from $T^{(r-1)}$ along with %and the knowledge of 
the inner and outer shapes of $T^{(r)}$, by applying internal insertions and augmenting by a suitable number of entries $r$ in row $r$.
This gives the inequality of the cardinalities $|{\cal LR}(\lambda/\mu,\nu)|\leq|{\cal LR}(\lambda/\nu,\mu)|$.
By applying the same argument with $\mu$ and $\nu$ switched, we see that the inequality is actually an equality, and the injection is forced to be bijective (see Theorem~\ref{The-bij}).

Then we prove that the map $\rho^{(n)}$ is an involution (Theorem~\ref{The-invol-tab-main}).
This proof is somewhat long and occupies both Sections~\ref{Sec-tab-inv} and~\ref{Sec-prf-delaying}. 
In Section~\ref{Sec-tab-inv} 
the proof is reduced to the analysis, in the case where the tableau $T$ has a single cell containing $r$ in its bottom $r$th row, of how the deletion from the bottom cell containing $r$ \qq{commutes} in a certain way with the deletions starting from the cells in the ($r-1$)th row, and how this commutation relation is reflected in the constructed tableau $S$.
This is spelt out in Proposition~\ref{Prop-invol-tab-erase-one-ind} and the whole of Section~\ref{Sec-prf-delaying} 
is devoted to its proof.

Let us add here that 
internal insertions, together with deletions, were originally used
in~\cite{Az2} by the first author in showing the bijectivity of $\rho^{(n)}$ and also in the statement to which the proof of the involutive nature was reduced. 
This statement is equivalent to Lemma~\ref{Lem-invol-tab-erase-one-gen}, and
Proposition~\ref{Prop-invol-tab-erase-one-ind} mentioned above is a refinement of Lemma~\ref{Lem-invol-tab-erase-one-gen}.

We turn now to our hive model and the problem of identifying a bijective map $\sigma^{(n)}: {\cal H}^{(n)} \rightarrow {\cal H}^{(n)}$ such that
for all $H\in{\cal H}^{(n)}(\lambda,\mu,\nu)$ we have $\sigma^{(n)}: H\mapsto K$ with $K\in{\cal H}^{(n)}(\lambda,\nu,\mu)$. 
By virtue of the bijection $f^{(n)}$ from LR $n$-tableaux to LR $n$-hives introduced in Section~\ref{Subsec-hives-tableaux},
deletions on tableaux correspond to path removals from a hive. These are defined in
Section~\ref{Sec-hive-path-removal} in terms of the action of three types of path removal operators, $\chi_r$, $\phi_r$ and $\omega_r$.
Each of these reduces or increases by $1$ the values of certain edge labels in a given hive. These
edges form a connected path, starting from the base of the hive and proceeding to its left or right hand
boundary. The route taken by each path is determined by the occurrences of upright rhombi of gradient $0$.
It is first shown that this action on an LR hive in ${\cal H}^{(r)}$ always yields another LR hive in ${\cal H}^{(r)}$.

The next step, described in Section~\ref{Sec-hive-bijection}, is to show that the systematic application of our path removal 
operators to $H\in{\cal H}^{(n)}(\lambda,\mu,\nu)$ reduces it to an $n$-hive involving an empty rightmost diagonal which may be
removed to give an $(n-1)$-hive. Iterating this procedure reduces $H$ to an an empty hive, and recording
the succession of left-hand boundary edge labels of the intervening hives enables us to build up a partner hive $K$. 
This procedure which involves a succession of pairs, consisting at each stage of a hive and what we call a truncated hive,
is illustrated in Example~\ref{Ex-hive}. The main result in this section is the proof of Theorem~\ref{The-HtoK} in which it is stated, as required, that 
$K\in {\cal H}^{(n)}(\lambda,\nu,\mu)$. This allows us to define our map $\sigma^{(n)}$ through the identification of $K$ with 
$\sigma^{(n)}H$. 

In the following Section~\ref{Sec-hive-path-addition} we introduce a hive path addition procedure that is used
to define a map $\ov{\sigma}^{(n)}:{\cal H}^{(n)} \rightarrow {\cal H}^{(n)}$ such that
for any $K\in{\cal H}^{(n)}(\lambda,\nu,\mu)$ we have $\ov{\sigma}^{(n)}: K\mapsto H\in{\cal H}^{(n)}(\lambda,\mu,\nu)$. 
By virtue of the bijection $f^{(n)}$ from LR $n$-tableaux to LR $n$-hives, it may be noted that the insertion path 
operation on LR tableaux introduced in~\cite{Az1,Az2} corresponds to our hive path addition procedure.
The hive $K$ determines the precise sequence of path additions, and we successfully show in the proof of Theorem~\ref{The-KtoH} 
that the result is as claimed. This involves the use of three Lemmas concerning detailed considerations of the relationship 
between successive path additions.

It is then not difficult to show in Theorem~\ref{The-hive-bijection} that the maps $\sigma^{(n)}$ amd 
$\ov{\sigma}^{(n)}$ are mutually inverse bijections on ${\cal H}^{(n)}$. However, it is considerably more
demanding to prove that these maps are involutions. This is done in Section~\ref{Sec-hive-inv} where the approach
makes use of two hives $H\in{\cal H}^{(n)}(\lambda,\mu,\nu)$ and $\hat{H}\in{\cal H}^{(n)}(\lambda-\epsilon_n,\mu,\nu-\epsilon_k)$
where $\hat{H}=\psi_n\,H$ is obtained from $H$ through the removal of single path reducing the uppermost non-zero 
upright rhombus gradient, say $U_{kn}$, in the rightmost diagonal of $H$ by $1$. With this notation
a crucial step is the proof of Lemma~\ref{Lem-KthetaK} which states that $\sigma^{(n)}\psi_n\,H = \phi_n\,\sigma^{(n)} H$.
This is conveniently illustrated in the commutative diagram of Lemma~\ref{Lem-KthetaK}.
Once again the proof involves an iterative path removal procedure in which $H$ and $\hat{H}$ are first reduced in size from
$n$-hives to $(n-1)$-hives. They differ by edge label differences of $\pm1$ along what we call a line
of difference, and the terminating level of this line on the left-hand boundary defines a difference in 
the upright rhombus gradients of their partner truncated hives that consist at this stage of a single diagonal.
Continuing this process results in succession of consecutive lines of difference in the
hives generated from $H$ and $\hat{H}$ whose terminating levels lead to differences of upright rhombus gradients 
in the consecutive diagonals, taken from right to left, of their final partner hives $K$ and $\hat{K}$.
This procedure is not, however, without its problems. The removal paths in $H$ and $\hat{H}$ may intersect,
bifurcate and rejoin, and the heart of the proof is taken up with proving that these possibilities, associated
with what we call critical rhombi, are no impediment to proving Lemma~\ref{Lem-KthetaK}. Armed with this Lemma
it is then straightforward to complete the proof of Theorem~\ref{The-inv} stating that $\sigma^{(n)}$
is an involution.

This completes the main part of our exercise, but in Section~\ref{Sec-obs-U} it is pointed out that 
all our hive path removal and addition procedures are driven and controlled by upright rhombus gradients.
This allows us to conclude that if $K=\sigma^{(n)}H$ and any other hive $H'$ shares the same set of upright 
rhombus gradients as $H$, then its partner $K'=\sigma^{(n)}H'$ shares the same set of
upright rhombus gradients as $K$, regardless of the boundary edge labels. If follows that one can define 
what we call a $U$-system of triangular arrays of non-negative upright rhombus gradients, on which
$\sigma^{(n)}$ acts in a uniform way regardless of the specification of the boundary labels necessary to turn 
such an array into a hive. It is shown in particular that for every set of non-negative upright rhombus gradients
not only does there always exists a corresponding dressing by means of boundary labels that is an $n$-hive, 
but also that there is an infinite set of such dressings.

Section~\ref{Sec-coincidence} consists of the discussion of the question of coincidence of all Littlewood--Richardson commutativity bijections,
$\rho_1$, $\rho_2$, $\rho_2'$ and $\rho_3$, as raised by Pak and Vallejo in~\cite{PV2}. These include the Benkart {\it et al.} tableaux switching map~\cite{bss},
referred to as $\rho_1$, and the Henriques-Kamnitzer crystal~\cite{HK2} and hive~\cite{HenKam} commutor maps, and those based on arrays due to Danilov
and Koshevoy~\cite{DK05,DK08} that constitute $\rho_2$ and $\rho_2'$. The maps $\rho^{(n)}$ and $\sigma^{(n)}$ discussed here 
both correspond to the map introduced by~\cite{Az1,Az2} 
in an attempt to provide a shortcut alternative to
the tableau switching map $\rho_1$ of~\cite{bss} and referred to as the fundamental symmetry map $\rho_3$ in~\cite{PV2}.
This section closes with an illustration in support of the coincidence of $\rho_1$ and $\rho_3$, an approach to whose proof was proposed in~\cite{fpsac08}.

\section{Notation and definitions}\label{Sec-notation}

\subsection{Partitions, Young diagrams and tableaux} 
\label{Subsec-tableaux}

A {\it partition\/} is a finite weakly decreasing sequence $\lambda=(\lambda_1,\lambda_2,\dots,\lambda_l)$ of strictly positive integers, each term of which is called a {\it part}. %%%%%its part. 
We use lower case Greek letters such as $\lambda$, $\mu$, $\nu$, \dots to represent partitions, and for a partition $\lambda$ we use $\lambda_i$ to denote its $i$th part. 
The number of parts of a partition $\lambda$ is called its {\it length\/} and is denoted by $\ell(\lambda)$.
The sum of its parts is called its {\it weight\/} and is denoted by $|\lambda|$.
Sometimes the sequence $\lambda$ is identified with a longer sequence obtained by adjoining an arbitrary number of  trailing zeros. 
In this situation the length $\ell(\lambda)$ still refers to the number of its strictly positive parts and we set $\lambda_i=0$ for all $i>\ell(\lambda)$. 
We regard an empty sequence, or a sequence whose terms are all zero, as a special case of a partition, namely the zero partition of length $0$ and weight $0$.

The {\it Young diagram\/} $D(\lambda)$ of a partition $\lambda$ is defined to be the subset $\{\,(i,j)\in\N\times\N\mid1\leqq i\leqq \ell(\lambda),\ 1\leqq j\leqq\lambda_i\,\}$ of $\N\times\N$.
We represent it as an aggregate of $|\lambda|$ {\it boxes\/} or {\it cells\/} arranged in $\ell(\lambda)$ rows left-adjusted to a vertical line
with its row $i$ (counted from the top) consisting of $\lambda_i$ cells for each $i$. % (see Fig.\ 1).
We refer to the cell at the intersection of row $i$ and column $j$ (counted from the left) as the cell $(i,j)$.
A {\it corner\/} of $D(\lambda)$ is a cell $(i,j)$ of $D(\lambda)$ such that its removal from $D(\lambda)$ leaves the Young diagram of some partition of weight $|\lambda|-1$. 

If $\lambda$ and $\mu$ are any two partitions such that $\mu_i\leq \lambda_i$ for all $i$ then we write $\mu\subseteq\lambda$. 
In such a case $\ell(\mu)\leq\ell(\lambda)$, $|\mu|\leq|\lambda|$ and $D(\mu)\subseteq D(\lambda)$ and the pair $(\lambda,\mu)$
specifies a {\it skew} Young diagram $D(\lambda/\mu)$ of shape $\lambda/\mu$. 
This is defined to be $D(\lambda)\setminus D(\mu)=\{\,(i,j)\mid1\leqq i\leqq \ell(\lambda),\ \mu_i<j\leqq\lambda_i\,\}$,
so that $\lambda$ and $\mu$ specify the {\it outer} and {\it inner} shape respectively, and
the number of cells in $D(\lambda/\mu)$ is $|\lambda|\!-\!|\mu|$.
\medskip

\begin{Example}\label{Ex-D-skD}
In the case $\lambda=(6,5,5,1)$ and $\mu=(3,3,2)$ we have 
\begin{equation}\label{Eq-D-skD}
D(\lambda)\ = \ 
\vcenter{\hbox{
\begin{tikzpicture}[x={(0in,-0.2in)},y={(0.2in,0in)}] % matrix coordinate
\foreach \i/\j in {1/1,1/2,1/3,1/4,1/5,1/6,  2/1,2/2,2/3,2/4,2/5,  3/1,3/2,3/3,3/4,3/5,  4/1}
  {\draw (\i,\j) rectangle +(-1,-1);}
\end{tikzpicture}
}}
\qquad\hbox{and}\qquad
D(\lambda/\mu)\ = \
\vcenter{\hbox{
\begin{tikzpicture}[x={(0in,-0.15in)},y={(0.15in,0in)}] % matrix coordinate
\foreach \i/\j in {1/4,1/5,1/6,  2/4,2/5,  3/3,3/4,3/5,  4/1}
  {\draw (\i,\j) rectangle +(-1,-1);}
\end{tikzpicture}
}}
\end{equation}
\end{Example}

A {\it tableau} $T$  shape $\lambda$ (resp.\ $\lambda/\mu$) is any map $T\!\!:\!\!D(\lambda)\to\N$ 
(resp.\ $D(\lambda/\mu)\to\N$). 
Such a tableau $T$ is represented by a Young diagram with each cell $(i,j)$ containing 
the  number $s$ as an entry %as an entry the letter or number $s$ 
if $T(i,j)=s$. 
In this context, we sometimes call an entry $s\in\N$ a {\it letter}.
If $T$ is a tableau of shape $\lambda/\mu$, and $(i,j)\in D(\mu)$, we may
extend the definition of $T$ by saying
that the cell $(i,j)$ is empty in $T$ (even though strictly speaking
$(i,j)$ is not a cell of $T$), 
or that it contains $0$, and accordingly we may write $T(i,j)=0$ in this case.
The {\it weight\/} of a tableau $T$ is the sequence $\alpha=(\alpha_1,\alpha_2,\dots,\alpha_s, \ldots)$ where $\alpha_s$ is the number of entries $s$ occurring in $T$ (i.e.\ $\alpha_s=\#T^{-1}(s)$), and, just as in the case of partitions, trailing zeros may be omitted.

\begin{Definition}\label{Def-ss-lp}
Let $\lambda$ and $\mu$ %, $\nu$ be three 
be partitions such that $\mu\subseteq\lambda$. % and $|\lambda|=|\mu|+|\nu|$. 
A tableau $T$ of shape $\lambda/\mu$ % and weight $\nu$ 
is said to be {\it semistandard\/} if
\begin{enumerate}
\item $T(i,j)\leqq T(i,j+1)$ if $(i,j)$ and $(i,j+1)$ both lie in $D(\lambda/\mu)$, that is to say the entries in each row are weakly increasing from left to right, and
\item $T(i,j)<T(i+1,j)$ if $(i,j)$ and $(i+1,j)$ both lie in $D(\lambda/\mu)$, that is to say the entries in each column are strictly increasing from top to bottom, 
\end{enumerate}
and is said to satisfy the {\it lattice permutation property\/} if
\begin{enumerate}
\item[3.] the number of entries $s$ occurring in the first $r$ rows of $T$ does not exceed the number of entries 
$s-1$ occurring in the first $r-1$ rows of $T$, for all $r\geq 1$ and $s\geq 2$.
\end{enumerate}
\end{Definition}
\medskip

It follows from the lattice permutation condition that the weight $\nu$ of $T$
must be a partition of length $\ell(\nu)\leq\ell(\lambda)$ such that $\nu\subseteq\lambda$ and $|\nu|=|\lambda|-|\mu|$.
\medskip

\begin{Definition}\label{Def-LRtab}
A tableau $T$ is said to be a {\it Littlewood--Richardson (LR) tableau\/} if it is semistandard and 
if it satisfies the lattice permutation property. 
We denote the set of all LR 
tableaux of shape $\lambda/\mu$ and weight $\nu$
by $\mathcal{LR}(\lambda/\mu,\nu)$. 
This set will be empty unless $\mu\subseteq\lambda$, in which case $\ell(\mu)\leq\ell(\lambda)$. 
For each positive integer $n$ we denote the union of those sets $\mathcal{LR}(\lambda/\mu,\nu)$ for which $\ell(\lambda)\leq n$ by %,\ell(\mu),\ell(\nu)\leq n$ by
$\mathcal{LR}^{(n)}$. We refer to elements of this set as LR $n$-tableaux. 
\end{Definition}
\medskip

\begin{Example} \label{Ex-LR-T}
In the case $\lambda=(9,9,6,4,1)$, $\mu=(7,5,3)$
and $\nu=(7,5,2)$ a typical tableau $T\in{\cal LR}(\lambda/\mu,\nu)\subset{\cal LR}^{(5)}\subset{\cal LR}^{(6)}\subset\cdots$
takes the form:
\begin{equation}\label{Ex-1}
\YT{0.2in}{}{
 {{},{},{},{},{},{},{},{1},{1}},
 {{},{},{},{},{},{1},{1},{2},{2}},
 {{},{},{},{1},{1},{2}},
 {{1},{2},{2},{3}},
 {{3}}
}
\end{equation}
where the empty cells occupy the shape specified by $\mu$, and the
cells with entries occupy the shape specified by $\lambda/\mu$, and the numbers
of these entries for each integer $k$ give the parts $\nu_k$ of $\nu$.
\end{Example}

%%%%%%%%%%%%%%%%%%
\subsection{Hives, edge labels, gradients and interlocking Gelfand--Tsetlin patterns}
\label{Subsec-hive-model}
%%%%%%%%%%%%%%%%%%%%

It is convenient to introduce a second combinatorial construct, namely, that of a {\it hive\/} 
first introduced by Knutson and Tao \cite{KT1} with properties described in more detail by Buch \cite{Bu}.

A hive expresses itself in three different forms.
In its {\it vertex representation}, as given in \cite{KT1} and \cite{Bu}, a hive is a labelling of the vertices of 
a planar, equilateral triangular graph, 
as shown below on the left of (\ref{Eq-hive-vertices}), satisfying the {\it rhombus inequalities\/}
indicated on the right of (\ref{Eq-hive-vertices}).
If the side length of the hive is $n$, then the hive is called an $n$-hive, and 
the edges connecting all neighbouring vertices divide the entire triangle into $n^2$ 
{\it elementary triangles\/} of side length $1$.  
An {\it elementary rhombus\/} is the union of any two elementary triangles sharing a common edge.
The rhombus inequalities require that for each such rhombus the sum of the labels of the vertices at the ends of 
the common edge must be greater than or equal to the sum of the labels of the other two vertices of the rhombus.
The hive is said to be an {\it integer hive\/} if all its vertex labels are integers.

For example in the case $n=4$ we have:

\begin{equation}\label{Eq-hive-vertices}
\vcenter{\hbox{\begin{tikzpicture}[x={(1cm*0.5,-\rootthree cm*0.5)},y={(1cm*0.5,\rootthree cm*0.5)}]
\foreach\i in{0,...,4}
 \foreach\j in{\i,...,4} \draw(\i,\j)node(a\i\j){$a_{\i\j}$};
\foreach\i/\ii in{0/-1,1/0,2/1,3/2,4/3}
 \foreach\j/\jj in{0/-1,1/0,2/1,3/2,4/3}{
  \ifnum\i<\j \draw(a\i\jj)--(a\i\j); \fi
  \ifnum\i>\j\else\ifnum\i>0 \draw(a\ii\j)--(a\i\j) (a\ii\jj)--(a\i\j); \fi\fi
 }
\end{tikzpicture}}}\qquad
\vcenter{\hbox{\tikz[x={(1cm*0.5,-\rootthree cm*0.5)},y={(1cm*0.5,\rootthree cm*0.5)},every node/.style={inner sep=0em}]
\draw(-1,-2)node[below left]{$-$}--(-1,-1)node[above left]{$+$}--
     (0,0)node[above right]{$-$} (-1,-2)--(0,-1)node[below right]{$+$}--(0,0)
     (-1,-1)--(0,-1);
}}\ge0
%%\tag{3}
\end{equation}
The display on the left shows our indexing convention of the vertex labels, 
and the diagram on the right indicates the nature of the rhombus inequality, 
to be applied to all elementary rhombi of the three different dispositions via rotation.

It is convenient to anchor each hive by setting $a_{00}=0$, as we shall see below.

\begin{Example} \label{Ex-H-vertices}
In the case $n=5$, a typical hive with $a_{00}=0$ takes the form:
\begin{equation}\label{Eq-H-vertices-ex}
H\ = \ \vcenter{\hbox{
}}
\end{equation}
\end{Example}

Its second form, the {\it edge representation\/} 
as introduced by~\cite{KTT1} and used in~\cite{KTT2,CJM}, is one in which a hive is a labelling of all edges of 
its graph that satisfies the {\it triangle conditions\/} and the {\it betweenness conditions} 
specified below in (\ref{Eq-triangle-condition}) and (\ref{Eq-betweenness-conditions}), respectively.

If $n=5$ such an edge labelling takes the form: 
\begin{equation}\label{Eq-hive-edges}
\vcenter{\hbox{%
}}
\end{equation}

A naming convention for the edge labels has been chosen so that $\alpha_{ij}$, $\beta_{ij}$ and $\gamma_{ij}$ 
label the edges parallel to the left, right and lower boundaries, respectively, of the hive,  
with the vertex labelled $a_{ij}$ in (\ref{Eq-hive-vertices}) at their rightmost end.
As a matter of convention we sometimes refer to any edge parallel to the 
left, right or lower boundary as being an $\alpha$-edge, $\beta$-edge or $\gamma$-edge,
respectively.
The triangle conditions require that the edge labels $\alpha$, $\beta$ and $\gamma$ 
satisfy
\begin{equation}\label{Eq-triangle-condition}
\alpha+\beta=\gamma 
\end{equation}
in each elementary triangle of the hive of the two types   
\begin{equation}\label{Eq-elementary-triangles}
\vcenter{\hbox{\tikz[x={(1cm*0.5,-1.7320508cm*0.5)},y={(1cm*0.5,1.7320508cm*0.5)}]
 \foreach \p/\q/\anchor/\label in{{(0,0)}/{(0,1)}/left/\alpha,{(0,1)}/{(1,1)}/right/\beta,{(0,0)}/{(1,1)}/below/\gamma} \draw[-]\p--node[\anchor]{$\label$}\q;
}}
\quad\hbox{and}\quad
\vcenter{\hbox{\tikz[x={(1cm*0.5,-1.7320508cm*0.5)},y={(1cm*0.5,1.7320508cm*0.5)}]
 \foreach \p/\q/\anchor/\label in{{(0,0)}/{(1,0)}/left/\beta,{(1,0)}/{(1,1)}/right/\alpha,{(0,0)}/{(1,1)}/above/\gamma} \draw[-]\p--node[\anchor]{$\label$}\q;
}}
\end{equation}
In addition, the betweenness conditions require that 
\begin{equation}\label{Eq-betweenness-conditions}
\alpha_{i-1,j-1}\geq\alpha_{ij}\geq\alpha_{i-1,j};
\qquad
\beta_{ij}\geq\beta_{i,j-1}\geq\beta_{i+1,j};
\qquad
\gamma_{i,j-1}\geq\gamma_{ij}\geq\gamma_{i+1,j}
\end{equation}
hold for $1\leq i<j\leq n$, where as an {\it aide-m\'emoire} the arrangements of
these edge labels are illustrated by
\begin{equation}\label{Eq-trapezia}
\vcenter{\hbox{%
\begin{tikzpicture}
 [x={(1cm*0.4,-1.7320508cm*0.45)},y={(1cm*0.4,1.7320508cm*0.45)}]
 \foreach \p/\q/\anchor/\label/\pos/\tint/\width in {
  {(0,0)}/{(0,1)}/left/\alpha_{i-1,j-1}/0.57/blue/4pt,
  {(0,1)}/{(0,2)}/left/\alpha_{i-1,j}/0.6/red/1.5pt,
  {(1,1)}/{(1,2)}/right/\alpha_{ij}/0.35/magenta/2.75pt
 } \draw[color=\tint,line width=\width]\p--node[pos=\pos,\anchor,black]{$\label$}\q;
 \foreach \p/\q in {
  {(0,1)}/{(1,1)}, %\beta_{i,j-1},
  {(0,2)}/{(1,2)}, %\beta_{ij},
  {(0,0)}/{(1,1)}, %\gamma_{i,j-1},
  {(0,1)}/{(1,2)}%  %\gamma_{ij}
 } \draw[-] \p--\q;
\end{tikzpicture}}}
\qquad 
\vcenter{\hbox{%
\begin{tikzpicture}
 [x={(1cm*0.45,-1.7320508cm*0.45)},y={(1cm*0.45,1.7320508cm*0.45)}]
 \foreach \p/\q/\anchor/\label/\pos/\tint/\width in{
  {(0,1)}/{(1,1)}/left/\beta_{i,j-1} /0.75/magenta/2.75pt,
  {(0,2)}/{(1,2)}/right/\beta_{ij}   /0.3/blue/4pt,
  {(1,2)}/{(2,2)}/right/\beta_{i+1,j}/0.4/red/1.5pt
 } \draw[color=\tint,line width=\width]\p--node[pos=\pos,\anchor,black]{$\label$}\q;
 \foreach \p/\q in {
  {(0,1)}/{(0,2)}, %\alpha_{i-1,j},
  {(1,1)}/{(1,2)}, %\alpha_{ij},
  {(1,1)}/{(2,2)}, %\gamma_{i+1,j},
  {(0,1)}/{(1,2)}% %\gamma_{ij},
 } \draw[-]\p--\q;
\end{tikzpicture}}}
\qquad
\vcenter{\hbox{%
\begin{tikzpicture}
 [x={(1cm*0.45,-1.7320508cm*0.45)},y={(1cm*0.45,1.7320508cm*0.45)}]
 \foreach \p/\q/\anchor/\label/\pos/\tint/\width in {
  {(0,0)}/{(1,1)}/below/\gamma_{i,j-1}/0.3/blue/4pt,
  {(1,1)}/{(2,2)}/below/\gamma_{i+1,j}/0.7/red/1.5pt,
  {(0,1)}/{(1,2)}/above/\gamma_{ij}   /0.5/magenta/2.75pt
 } \draw[color=\tint,line width=\width]\p--node[pos=\pos,\anchor,black]{$\label$}\q;
 \foreach \p/\q in {
  {(0,0)}/{(0,1)}, %\alpha_{i-1,j-1},
  {(1,1)}/{(1,2)}, %\alpha_{ij},
  {(0,1)}/{(1,1)}, %\beta_{i,j-1}
  {(1,2)}/{(2,2)}% %\beta_{i+1,j}
 } \draw[-]\p--\q;
\end{tikzpicture}}}
\end{equation}

With this convention, the transformation from the vertex labelling to the edge labelling is done by
setting each edge label equal to the label of the edge's right hand vertex minus that of its left hand 
vertex, that is by setting:
\begin{equation}\label{Eq-edge-vertex-labels}
\alpha_{ij}=a_{ij}-a_{i,j-1}\,;
\qquad
\beta_{ij}=a_{ij}-a_{i-1,j}\,;
\qquad
\gamma_{ij}=a_{ij}-a_{i-1,j-1}\,.
\end{equation}
This ensures that the edge labels of (\ref{Eq-elementary-triangles}) automatically satisfy (\ref{Eq-triangle-condition}). 
In the other direction, (\ref{Eq-triangle-condition}) ensures 
that inductively finding the vertex labels from the edge labels starting with $a_{00}=0$ 
gives a consistent result.

\begin{Example} \label{Ex-H-edges}
Applying (\ref{Eq-edge-vertex-labels}) to the hive specified in Example~\ref{Ex-H-vertices} by means of its vertex representation (\ref{Eq-H-vertices-ex})
yields the edge representation:
\begin{equation}\label{Eq-H-edges-ex}
\vcenter{\hbox{
}}
\end{equation}
\end{Example}

Under this correspondence between vertex and edge representations of a hive, the rhombus inequalities of (\ref{Eq-hive-vertices}) 
are translated into the betweenness conditions (\ref{Eq-betweenness-conditions}) 
as follows. Within a hive there are three types of elementary rhombi: right-leaning, upright 
and left-leaning, as illustrated by
\begin{equation}\label{Eq-rhombi}
\vcenter{\hbox{
\tikz[x={(1cm*0.5,-1.7320508cm*0.5)},y={(1cm*0.5,1.7320508cm*0.5)},
      every node/.style={inner sep=0em}]{
 \draw[very thick,-]
  (-1,-2)node[below left]{$-$}
       --node[left,outer sep=5pt]{$\alpha_{i-1,j-1}$}
  (-1,-1)node[above left]{$+$};
 \draw[very thick,-]
  (-1,-2)                     
       --node[below,outer sep=5pt]{$\gamma_{i,j-1}$}
  (0,-1) node[below right]{$+$};
 \draw[-]
  (-1,-1)
       --node[above,outer sep=3pt]{$\gamma_{ij}$}
  (0,0)  node[above right]{$-$};
 \draw[-]
  (0,-1)
       --node[pos=0.4,right,outer sep=5pt]{$\alpha_{ij}$}
  (0,0);
 \path(-1,-1)--node{$R_{ij}$}(0,-1);
}
}}
\quad
\vcenter{\hbox{
\tikz[x={(1cm*0.5,-1.7320508cm*0.5)},y={(1cm*0.5,1.7320508cm*0.5)},
      every node/.style={inner sep=0.1em}]{
 \draw[-]
  (-1,-1)node[left]{$+$}
       --node[outer sep=5pt,left]{$\alpha_{i-1,j}$}
  (-1,0) node[above]{$-$};
 \draw[very thick,-]
  (-1,0)
       --node[pos=0.4,outer sep=5pt,right]{$\beta_{ij}$}
  (0,0)  node[right]{$+$};
 \draw[-]
  (-1,-1)
       --node[pos=0.6,outer sep=4pt,left]{$\beta_{i,j-1}$}
  (0,-1) node[below]{$-$};
 \draw[very thick,-]
  (0,-1)
       --node[pos=0.3,outer sep=5pt,right]{$\alpha_{ij}$}
  (0,0);
 \path(-1,-1)--node{$U_{ij}$}(0,0);
}
}}
\quad
\vcenter{\hbox{
\tikz[x={(1cm*0.5,-1.7320508cm*0.5)},y={(1cm*0.5,1.7320508cm*0.5)},
      every node/.style={inner sep=0em}]{
 \draw[very thick,-]
  (-1,-1)node[above left]{$-$}
       --node[above,outer sep=3pt]{$\gamma_{ij}$}
  (0,0)  node[above right]{$+$};
 \draw[-]
  (0,0)
       --node[right,outer sep=5pt]{$\beta_{i+1,j}$}
  (1,0)  node[below right]{$-$};
 \draw[very thick,-]
  (-1,-1)
       --node[left,outer sep=5pt]{$\beta_{i,j-1}$}
  (0,-1) node[below left]{$+$};
 \draw[-]
  (0,-1)
      --node[below,outer sep=5pt]{$\gamma_{i+1,j}$}
  (1,0);
 \path(0,-1)--node{$L_{ij}$}(0,0);
}
}}
\end{equation}
The parameters $R_{ij}$, $U_{ij}$ and $L_{ij}$ introduced in these diagrams are not edge labels.
They are referred to as the {\it gradients\/} of the corresponding right-leaning, upright and left-leaning rhombi,
respectively.
Each gradient is defined to be the difference between parallel edge labels in the relevant rhombus
or, equivalently, the four-term alternating sign sum of the vertex labels of that rhombus.
To be precise, in terms of the edge labels the gradients are given by:
\begin{equation}\label{Eq-rhombus-gradients}
\begin{alignedat}{3}
& R_{ij} & {}={} & \alpha_{i-1,j-1}-\alpha_{ij}  & {}={} & \gamma_{i,j-1}-\gamma_{ij}\,; \\
& U_{ij} & {}={} & \alpha_{ij}-\alpha_{i-1,j}    & {}={} & \beta_{ij}-\beta_{i,j-1}\,; \\
& L_{ij} & {}={} & \beta_{i,j-1}-\beta_{i+1,j} & {}={} & \gamma_{ij}-\gamma_{i+1,j}\,.
\end{alignedat}
\end{equation}
From the edge representation, the equalities on the right are a consequence of the  triangle conditions.
The rhombus inequalities of (\ref{Eq-hive-vertices}) then take the form:
\begin{equation}\label{Eq-rhombus-inequalities}
R_{ij}\geq0,\quad U_{ij}\geq0,\quad L_{ij}\geq0\quad\text{for all $i,j$ such that $1\leq i<j\leq n$.}
\end{equation}
They are equivalent to the requirement that in each of the above diagrams (\ref{Eq-rhombi})  
the label on a thick edge is greater than or equal to the label on its parallel thin edge.
Combining the rhombi of (\ref{Eq-rhombi}) in pairs produces the trapezia of (\ref{Eq-trapezia}) 
and leads immediately to the betweenness conditions (\ref{Eq-betweenness-conditions}). 

All this gives rise to a third way of specifying hives, namely {\it the gradient representation\/}, 
which involves labelling its boundary edges and giving the gradients of one or other of its three sets of right-leaning, upright or left-leaning elementary rhombi. This is illustrated in the case $n=5$ by
\begin{equation}\label{Eq-hive-gradients}
\vcenter{\hbox{
\begin{tikzpicture}[x={(1cm*0.4,-\rootthree cm*0.4)},
                    y={(1cm*0.4,\rootthree cm*0.4)}]
\draw(0,0)--(0,5)--(5,5)--cycle;
\foreach\i in{1,...,5}\path(0,\i-1)--node[pos=0.6,left]{$\alpha_{\i}$}(0,\i);
\foreach\i in{1,...,5}\path(\i-1,5)--node[pos=0.4,right]{$\beta_{\i}$}(\i,5);
\foreach\i in{1,...,5}\path(\i-1,\i-1)--node[below]{$\gamma_{\i}$}(\i,\i);
\foreach\i in{1,...,4}\draw(0,\i)--(5-\i,5); % horizontal edges
\foreach\i in{1,...,4}\draw(\i,\i)--(\i,5); % ascending edges
\foreach\i in{1,...,4}
 \foreach\j in{\i,...,5} {
  \ifnum\i<\j\path(\i-1,\j-1)--node[pos=0.475]{$R_{\i\j}$}(\i,\j-1);\fi
 }
\end{tikzpicture}
}}
\!\!\!
\vcenter{\hbox{
\begin{tikzpicture}[x={(1cm*0.4,-\rootthree cm*0.4)},
                    y={(1cm*0.4,\rootthree cm*0.4)}]
\draw(0,0)--(0,5)--(5,5)--cycle;
\foreach\i in{1,...,5}\path(0,\i-1)--node[pos=0.6,left]{$\alpha_{\i}$}(0,\i);
\foreach\i in{1,...,5}\path(\i-1,5)--node[pos=0.4,right]{$\beta_{\i}$}(\i,5);
\foreach\i in{1,...,5}\path(\i-1,\i-1)--node[below]{$\gamma_{\i}$}(\i,\i);
\foreach\i in{1,...,4}\draw(0,\i)--(\i,\i); % descending edges
\foreach\i in{1,...,4}\draw(\i,\i)--(\i,5); % ascending edges
\foreach\i in{1,...,4}
 \foreach\j in{\i,...,5} {
  \ifnum\i<\j\path(\i-1,\j-1)--node{$U_{\i\j}$}(\i,\j);\fi
 }
\end{tikzpicture}
}}
\!\!\!
\vcenter{\hbox{
\begin{tikzpicture}[x={(1cm*0.4,-\rootthree cm*0.4)},
                    y={(1cm*0.4,\rootthree cm*0.4)}]
\draw(0,0)--(0,5)--(5,5)--cycle;
\foreach\i in{1,...,5}\path(0,\i-1)--node[pos=0.6,left]{$\alpha_{\i}$}(0,\i);
\foreach\i in{1,...,5}\path(\i-1,5)--node[pos=0.4,right]{$\beta_{\i}$}(\i,5);
\foreach\i in{1,...,5}\path(\i-1,\i-1)--node[below]{$\gamma_{\i}$}(\i,\i);
\foreach\i in{1,...,4}\draw(0,\i)--(5-\i,5); % horizontal edges
\foreach\i in{1,...,4}\draw(0,\i)--(\i,\i); % descending edges
\foreach\i in{1,...,4}
 \foreach\j in{\i,...,5} {
  \ifnum\i<\j\path(\i,\j-1)--node[pos=0.545]{$L_{\i\j}$}(\i,\j);\fi
 }
\end{tikzpicture}
}}
\end{equation}
In each of these three labelling schemes there are constraints on one set
of boundary edge labels, namely the $\beta$, $\gamma$ and $\alpha$ labels 
in the case of right-leaning, upright and left-leaning rhombi. For example,
in the case of the upright rhombi which we mostly use in this paper,
we must have
\begin{equation}\label{Eq-gabU}
  \gamma_k = \left(\alpha_k+\sum_{i=1}^{k-1}U_{ik}\right)
            +\left(\beta_k-\sum_{j=k+1}^n U_{kj}\right)\,.
\end{equation}
When expressed in tems of the $\alpha$ and $\beta$ edge labels, and just the upright rhombus gradients $U_{ij}$,
the three sets of rhombus inequalities take the form:
\begin{equation}\label{Eq-aUb-inequalities}
\alpha_{j-1}+\textstyle\sum\limits_{k=1}^{i-1}U_{k,j-1}\geq\alpha_j+\textstyle\sum\limits_{k=1}^iU_{kj},
\quad
U_{ij}\geq0,
\quad
\beta_i-\textstyle\sum\limits_{k=j}^nU_{ik}\geq\beta_{i+1}-\textstyle\sum\limits_{k=j+1}^nU_{i+1,k}\,.
\end{equation}

%%%%%%%%%%%%%%%%%%%%

\begin{Definition}~\label{Def-LRhive}
An integer hive with $a_{00}=0$ is said to be a {\it Littlewood--Richardson (LR) hive\/} if 
all of its edge labels are non-negative integers. By virtue of the
betweenness conditions (\ref{Eq-betweenness-conditions}) this is equivalent
to the requirement that its boundary edges are labelled by the parts of partitions. 
For a fixed positive integer $n$ let 
${\cal H}^{(n)}$ denote the set of all LR $n$-hives and let
${\cal H}^{(n)}(\lambda,\mu,\nu)$
be the set of LR $n$-hives whose lower, left and right boundary edge 
labels are specified by the parts of the partitions $\lambda$, $\mu$ and $\nu$ of lengths
$\ell(\lambda),\ell(\mu),\ell(\nu)\leq n$, arranged as exemplified in the case $n=5$ by  
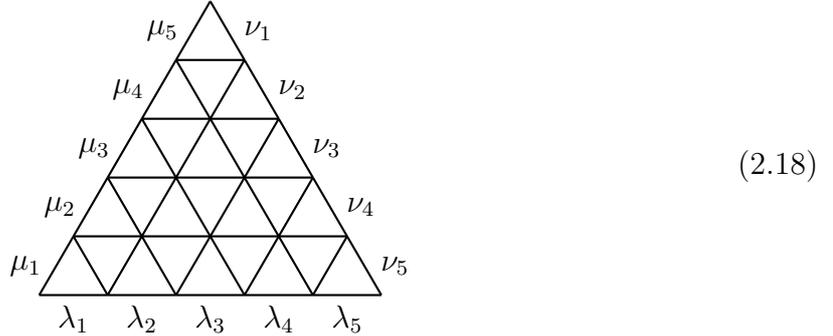
\begin{equation}\label{Eq-LRhives}
\vcenter{\hbox{
 }}
\end{equation}  
where the edge labels on the left, right and lower boundary
are enumerated from bottom to top, from top to bottom and from left to right, respectively. 
\end{Definition}

Since all edge labels are $\geq0$ it also follows that in every elementary triangle labelled as
in (\ref{Eq-triangle-condition}) we have both $\gamma\geq\alpha$ and $\gamma\geq\beta$. 
Moreover, repeated use of the triangle condition (\ref{Eq-triangle-condition}) ensures that
\begin{equation}
          |\lambda| = |\mu| +|\nu|\,.
\end{equation}

It should be noted that the betweenness conditions~(\ref{Eq-betweenness-conditions})
imply that within an LR hive $H\in{\cal H}^{(n)}(\lambda,\mu,\nu)$ the 
non-negative integer $\alpha$, $\beta$ and $\gamma$ edge labels automatically specify 
three differently oriented interlocking Gelfand--Tsetlin (GT) patterns~\cite{GT1}
of types specified by the  boundary edge labels $\mu$, $\nu$ and $\lambda$, respectively.
 
For the relevant definitions of GT patterns, their type and weight and their one--to--one correspondence 
with semistandard Young tableaux, see, for example,~\cite{Louck}.
Suffice to say at this point, that a GT pattern, $G$, is a triangular array of 
non-negative integers $(\lambda_j^{(i)})_{1\leq j\leq i\leq n}$ that may be abbreviated to $(\boldsymbol\lambda)$ and displayed
as below:
\begin{equation}
\begin{array}{cccccccccccccccccc}
\lambda_1^{(n)}&&\lambda_2^{(n)}&&\cdots&&\lambda_{n-1}^{(n)}&&\lambda_n^{(n)}&\cr
&\lambda_1^{(n-1)}&&\lambda_2^{(n-1)}&&\cdots&&\lambda_{n-1}^{(n-1)}\cr
&&\cdots&&\cdots&&\cdots\cr
&&&\lambda_1^{(2)}&&\lambda_2^{(2)}\cr
&&&&\lambda_1^{(1)}\cr
\end{array}
\end{equation}
The entries satisfy the betweenness conditions $\lambda_j^{(i+1)}\geq \lambda_j^{(i)}\geq \lambda_{j+1}^{(i+1)}$
for $1\leq j\leq i<n$. The $i$th row, enumerated from bottom to top, necessarily constitutes a partition $\lambda^{(i)}$
of length $\leq i$. Such a GT pattern is said to be of type $\lambda^{(n)}$ and of weight $\gamma=(\gamma_1,\gamma_2,\ldots,\gamma_n)$
where $\gamma_i=|\lambda^{(i)}|-|\lambda^{(i-1)}|$ for $i=1,2,\ldots,n$, with $|\lambda^{(0)}|=0$.
The GT patterns of type $\lambda^{(n)}$ are in bijective correspondence with semistandard Young tableaux of shape $\lambda^{(n)}$.
The partitions $\lambda^{(i)}$ specify the shape of that part of the semistandard tableaux consisting of entries $\leq i$,
that is to say having $\lambda_j^{(i)}-\lambda_j^{(i-1)}$ entries $i$ in row $j$ for $1\leq j\leq i\leq n$.

\begin{Example} \label{Ex-gt}
Our previous Examples~\ref{Ex-H-vertices} and \ref{Ex-H-edges} illustrate an LR hive
$H\in{\cal H}^{(5)}(\lambda,\mu,\nu)$ with $\lambda=(9,9,6,4,1)$, $\mu=(7,5,3,0,0)$ and $\nu=(7,5,2,0,0)$
in its vertex and edge representations, (\ref{Eq-H-vertices-ex}) and (\ref{Eq-H-edges-ex}), respectively. 
From the latter, one can see that the three differently oriented interlocking Gelfand--Tsetlin patterns 
are of types $\mu$, $\nu$ and $\lambda$ and weights $\rev(\lambda-\nu)$, $\lambda-\mu$ and 
$\nu+\rev(\mu)$, respectively, where $\rev(\kappa)=(\kappa_n,\ldots,\kappa_2,\kappa_1)$ for any $\kappa=(\kappa_1,\kappa_2,\ldots,\kappa_n)$.
They are displayed below along with their corresponding semistandard tableaux. 
\begin{equation}\label{Eq-GTx3}

\end{equation}
The superscripts $(\alpha)$, $(\beta)$ and $(\gamma)$ are intended to indicate 
that the three GT patterns specify the $\alpha$, $\beta$ and $\gamma$-edge labels of a hive, in this case
the hive $H$ of (\ref{Eq-H-edges-ex}), and that the three semistandard tableaux correspond to these same GT patterns. 
In the case of each GT pattern the specification of the original hive $H$ may be completed by adding either set of missing boundary edge labels,
and then evaluating all the remaining edge labels by means of the hive elementary triangle condition~(\ref{Eq-triangle-condition}).
\end{Example}

Alternatively, to avoid the necessity of giving any internal edge labels we may complete the specification of a 
typical LR hive by giving in addition to its non-negative boundary edge labels, 
$\lambda_i$, $\mu_i$ and $\nu_i$ for $1\leq i\leq n$, one set of rhombus 
non-negative integer gradients, namely $U_{ij}$ for $1\leq i<j\leq n$, 
as illustrated by
\begin{equation}\label{Eq-LRhives-Ugrads}
\qquad\qquad

\end{equation}
It is convenient to refer to the portion of such a hive consisting of a sequence of upright rhombi 
with gradients $U_{ij}$ for fixed $j$ and $i=1,2,\ldots,j-1$,
together with a corresponding elementary triangle with bottom edge 
labelled $\lambda_j$, as the $j$th {\it diagonal\/} of the hive.

With the above labelling the application of the condition $\gamma\geq\alpha$ to the
elementary triangle at the foot of the hive with edge label $\lambda_k$ ensures that
\begin{equation}
 \lambda_k\geq \mu_k+\sum_{i=1}^{k-1} U_{ik}\geq \mu_k\quad \hbox{for $k=1,2,\ldots,n$} \,,
\end{equation}
so that $\mu\subseteq\lambda$. Similarly, if one consults the
leftmost diagram in (\ref{Eq-hive-gradients}), the same argument applied to the 
elementary triangle with edge label $\beta_k$, one finds in the case of our LR hive that
\begin{equation}
 \nu_k\leq \lambda_k-\sum_{j=k+1}^{n} R_{kj}\leq \lambda_k\quad \hbox{for $k=1,2,\ldots,n$} \,,
\end{equation}
so that $\nu\subseteq\lambda$.

\subsection{Connection between hives and tableaux}
\label{Subsec-hives-tableaux}

Let $n$ be a fixed positive integer. 
Then there is a one-to-one correspondence between LR $n$-tableaux and LR $n$-hives.
Let $f^{(n)}$ denote the map from ${\mathcal{LR}}^{(n)}$ to ${\mathcal{H}}^{(n)}$.
taking $T\in{\cal LR}(\lambda/\mu,\nu)$ to $H\in{\cal H}^{(n)}(\lambda,\mu,\nu)$
by setting the upright rhombus gradients of $H$ to
\begin{equation}\label{Eq-TtoU}
   U_{ij} = \#\hbox{entries $i$ in row $j$ of $T$ for $1\leq i<j\leq n$}\,,
\end{equation}
and adding the boundary edge labels $\lambda$, $\mu$ and $\nu$.

\begin{Example} 
In the case $\lambda=(9,9,6,4,1)$, $\mu=(7,5,3)$, $\nu=(7,5,2)$, $n=5$ and $T$
as shown on the left, the corresponding hive $H$ is displayed on the right:
\begin{equation}\label{Ex-tab-to-hive}
T\ =\ 
\YT{0.2in}{}{
 {{},{},{},{},{},{},{},{1},{1}},
 {{},{},{},{},{},{1},{1},{2},{2}},
 {{},{},{},{1},{1},{2}},
 {{1},{2},{2},{3}},
 {{3}}
}
\qquad H\ = \ 
\vcenter{\hbox{

}}
\end{equation}
\end{Example}

Before discussing the hive conditions, let us mention that it is also
easy to describe the inverse map $f^{(n)}{}^{-1}$. The boundary edge labels
$\lambda_i$ and $\mu_i$ for $i=1,2,\ldots,n$ specify the skew shape $\lambda/\mu$.
Each rhombus gradient $U_{ij}$ for $1\leq i<j\leq n$ specifies the number of 
entries $i$ in row $j$ of $T$. For $i$ from $1$ to $j-1$ these are arranged in 
weakly increasing order from left to right across the $j$th row of $T$ and the remaining
$\lambda_j-\mu_j-\sum_{i=1}^{j-1}U_{ij}$ cells of this $j$th row are then filled 
with entries $j$. The number of these cells is necessarily non-negative since it is nothing
other than the edge label $\beta_{jj}$, and within any LR hive all
edge labels are $\geq0$. This also ensures that the total number of entries $j$ in $T$ is
$\beta_{jj}+\sum_{j=i+1}^n U_{ij}$, but from the definition of the gradients this is just 
$\beta_{jn}=\nu_j$ so that $T$ has weight $\nu=(\nu_1,\nu_2,\ldots,\nu_n)$. 

Now we show that the map $f^{(n)}$ from tableaux to hives and its inverse $f^{(n)}{}^{-1}$ preserve 
the Littlewood--Richardson nature of both tableaux and hives. To see this it should 
be noted that 
\begin{eqnarray}
  U_{ij}&=& \#\hbox{~of entries $i$ in row $j$ of $T$}\,;\\ \label{Eq-U}\cr
  R_{ij}&=&(\mu_{j-1}+\sum_{k=1}^{i-1} U_{k,j-1})
           -(\mu_j+\sum_{k=1}^i U_{k,j})\cr
        &=&(\mu_{j-1}+ \#\hbox{~of entries $<i$ in row $j-1$ of $T$}) \cr  
        &&-(\mu_j + \#\hbox{~of entries $\leq i$ in row $j$ of $T$})\,;\\  \label{Eq-R}
        \cr
  L_{ij}&=&(\nu_i-\sum_{k=j}^n U_{ik})-(\nu_{i+1}-\sum_{k=j+1}^n U_{i+1,k})
        =\sum_{k=1}^{j-1} U_{i,k}-\sum_{k=1}^{j} U_{i+1,k}\cr
        &=&\#\hbox{~of entries $i$ in the topmost $j-1$ rows of $T$}\cr
        && -\#\hbox{~of entries $i+1$ in the topmost $j$ rows of $T$}\,, \label{Eq-L}
\end{eqnarray} 
for all $(i,j)$ such that $1\leq i<j\leq n$.
It will be recognised that the conditions $U_{ij}\geq0$ correspond 
to the filling of the rows of $T$ with non-negative numbers of integer entries, 
while the conditions $R_{ij}\geq0$ correspond to the fact that entries
of $T$ are strictly increasing down columns, while the conditions
$L_{ij}\geq0$ correspond to the lattice permutation rule of Definition~\ref{Def-ss-lp} for $2\leq s\leq r\leq n$, with all other cases trivial.

This suffices to show that the above map $f^{(n)}$ that takes 
$T\in{\cal LR}(\lambda/\mu,\nu)$ to $H\in{\cal H}^{(n)}(\lambda,\mu,\nu)$,
for any fixed $n\geq\ell(\lambda)$, is a bijection. 
Since this is true for all $\lambda,\mu,\nu$ with $\ell(\lambda)\leq n$, $f^{(n)}$ provides a bijection from ${\mathcal{LR}}^{(n)}$ to ${\mathcal{H}}^{(n)}$.

In order to interpret hive edge labels and the constituent Gelfand-Tsetlin patterns in terms of tableau parameters, 
it might be noted that under this bijection, the 
partition $(\lambda_1,\lambda_2,\ldots,\lambda_i,\alpha_{i,i+1},\alpha_{i,i+2},\ldots,\alpha_{in})$ specifies the 
shape of the subtableau of $T$ consisting of all entries $\leq i$ including $0$ or empty boxes, for $i\geq0$, 
with this partition coinciding with $\mu$ in the special case $i=0$. 
On the other hand the partitions $(\beta_{1j},\beta_{2j},\ldots,\beta_{jj})$ for $j=1,2,\ldots,n$
specify the weights of the first $j$ rows of $T$, with this partition coinciding with $\nu$ in the special case $j=n$. 
More explicitly, $\alpha_{ij}$ is the number of entries $\leq i$ in row $j$, $\beta_{ij}$ is the total number of entries $i$ in rows $1,2,\ldots,j$
and $\gamma_{ij}=\alpha_{i-1,j}+\beta_{ij}$.
It follows that:
\begin{itemize}
\item $G_\mu^{(\alpha)}$ is obtained from $T\in{\mathcal{LR}}^{(n)}$ by recording the sequence of partitions 
giving the shapes occupied by the entries $<r$ in rows $r,r+1,\dots,n$ of $T$, for $r=1,2,\ldots,n$. 
\item $G_\nu^{(\beta)}$ is obtained from $T\in{\mathcal{LR}}^{(n)}$ by recording the sequence of partitions
giving the weights of the subtableau consisting of the first $r$ rows of $T$ for $r=n,\ldots,2,1$.
\end{itemize}

\subsection{Combinatorial models for LR coefficients}
\label{Subsec-models}

The Littlewood--Richardson rule states that the coefficients appearing in the 
expansion (\ref{Eq-LRcoeffs}) of a product of Schur functions are given by
\begin{equation}
  c_{\mu\nu}^\lambda = \#\{T\in{\cal LR}(\lambda/\mu,\nu)\}.
\end{equation}
It follows from the bijection between Littlewood--Richardson $n$-tableaux and Littlewood--Richardson $n$-hives
for any fixed $n\geq\ell(\lambda)$ that
\begin{equation}\label{Eq-LRcoeff-Bu}
  c_{\mu\nu}^\lambda = \#\{H\in{\cal H}^{(n)}(\lambda,\mu,\nu)\}.
\end{equation}
This result is due to Fulton, as reported in the Appendix to ~\cite{Bu}.

Alternative combinatorial models for the LR coefficients have been provided, based not on the full specification of hives,
but rather on one or two
of their three interlocking Gelfand--Tsetlin patterns as exemplified in \eqref{Eq-GTx3}.
Let
${\cal GT}_\nu^\kappa$ denote the set of GT patterns of type $\nu$ and weight $\kappa$.
Then, for example, Theorem~1 of~\cite{GZpolyed} may be expressed in the form
\begin{equation}\label{Eq-LRcoeff-GZ}
   c_{\mu\nu}^\lambda = \#\{G\in{\cal GT}_\nu^{\lambda-\mu}\,|\, d_j^{(i)}(\boldsymbol\nu)\leq \mu_i-\mu_{i+1}~~\hbox{for}~~1\leq j\leq i<n\}\,,
\end{equation}
where $G=(\boldsymbol\nu)=(\nu^{(i)}_j)_{1\leq i\leq j\leq n}$.
If $G$ is reoriented so that its entries label the $\beta$-edges of a potential hive, $H$, then by virtue of the GT pattern betweenness conditions
it automatically satisfies the two sets of rhombus conditions, $L_{ij}\geq0$ and $U_{ij}\geq0$ for $1\leq i<j\leq n$. 
Moreover, it may then be seen that $d_j^{(i)}(\boldsymbol\nu)=\sum_{k=1}^j U_{k,i+1}-\sum_{k=1}^{j-1} U_{ki}$ for $1\leq j\leq i<n$. These constraints are 
precisely those appearing on the left of~(\ref{Eq-aUb-inequalities}) with $\alpha_i=\mu_i$ for $i=1,2,\ldots,n$, which correspond
to the conditions $R_{j,i+1}\geq0$. It follows that $G$ restricted by these conditions is precisely a GT
pattern of the form $G^{(\beta)}_\nu$, whose completion by means of the specification of the remaining boundary edges and the use of the 
hive elementary triangle condition~(\ref{Eq-triangle-condition})
gives some $H\in{\cal H}^{(n)}(\lambda,\mu,\nu)$. 

Similarly, Theorem~4.3 of~\cite{BZphy} may be expressed in the form
\begin{equation}\label{Eq-LRcoeff-BZ}
   c_{\mu\nu}^\lambda = \#\{G\in{\cal GT}_\mu^{\rev(\lambda-\nu)}\,|\, n_i^{(j)}(\boldsymbol\mu)\leq \nu_i-\nu_{i+1}~\hbox{for}~1\leq i<j\leq n\}\, 
\end{equation}
where $G=(\boldsymbol\mu)=(\mu^{(i)}_j)_{1\leq i\leq j\leq n}$. 
This time if $G$ is reoriented so that its entries label the $\alpha$-edges of a potential hive, then by virtue of the GT pattern betweenness conditions
it automatically satisfies the two sets of rhombus conditions, $R_{ij}\geq0$ and $U_{ij}\geq0$ for $1\leq i<j\leq n$. 
It may then be seen that $n_i^{(j)}(\boldsymbol\mu)=\sum_{k=j}^n U_{ik}-\sum_{k=j+1}^{n} U_{i+1,k}$ for $1\leq i<j\leq n$. These constraints are 
precisely those appearing on the right of~(\ref{Eq-aUb-inequalities}) with $\beta_i=\nu_i$ for $i=1,2,\ldots,n$, which correspond
to the conditions $L_{ij}\geq0$. It follows that $G$ restricted by these conditions is precisely a GT pattern of the form $G^{(\alpha)}_\mu$,
whose completion as before gives some $H\in{\cal H}^{(n)}(\lambda,\mu,\nu)$.
Alternatively, a complete specification of the hive may be accomplished by using the GT pattern $G^{(\alpha)}_\mu$ augmented this time through the 
addition of the boundary edge labels $\lambda_i$, 
as is done in~\cite{Louck}, for example, where in equation (11.135) the LR coefficients are expressed as the number of such modified GT patterns
satisfying lattice permutation conditions derived from the corresponding LR tableau. The conditions involve parameters $l_{ij}$,
where $l_{i,j}=U_{ji}$ for $j<i$, and $l_{i,i}=\lambda_i-\mu_i-\sum_{j=1}^{i-1} U_{ji}$, and the conditions themselves
once more correspond to the non-negativity of $L_{ij}$ for $1\leq i<j\leq n$.

Finally, from a consideration of the highest and lowest weight elements of the connected components isomorphic 
to $B_\lambda$ in $B_\mu\otimes B_\nu$~\cite{HenKam} it may be seen that  
\begin{equation}\label{Eq-LRcoeff-KH}
\begin{split}
   c_{\mu\nu}^\lambda = \#\{(G_1,G_2)  & \in \mathcal{GT}_{\mu}^{\rev(\lambda-\nu)}\times\mathcal{GT}_{\nu}^{\lambda-\mu}\,|\, \\ 
													             & \mu^{(n-i)}_{j-i}-\mu^{(n-i+1)}_{j-i+1}=\nu^{(j)}_i-\nu^{(j-1)}_i~~\hbox{for}~~1\leq i<j\leq n\,\}\,.
\end{split}
\end{equation}
where $G_1=(\mu^{(i)}_j)_{1\leq i\leq j\leq n}$ and $G_2=(\nu^{(i)}_j)_{1\leq i\leq j\leq n}$. In this case
$G_1$ and $G_2$ may be reoriented so as to form two interlocking GT patterns in the shape of a hive $H$
with their entries labelling the $\alpha$ and $\beta$ edges, respectively. Then the result, 
obtained through the application of the hive triangle conditions to specify the remaining $\gamma$-edge labels,
is a hive with all rhombus gradients non-negative if and only if the gradients $U_{ij}$ of the upright rhombi arising from 
$G_1$ and $G_2$ coincide. In such a case the upright rhombus gradients are given by
$U_{ij}=\mu^{(n-i)}_{j-i}-\mu^{(n-i+1)}_{j-i+1}=\nu^{(j)}_i-\nu^{(j-1)}_i$, thereby explaining the necessity of
these conditions appearing in (\ref{Eq-LRcoeff-KH}). So the pair $(G_1,G_2)$ is to be identified with
the pair of GT patterns $(G^{(\alpha)}_\mu,G^{(\beta)}_\nu)$ associated with some $H\in{\cal H}^{(n)}(\lambda,\mu,\nu)$.

Thus all three of the above combinatorial expressions (\ref{Eq-LRcoeff-GZ})-(\ref{Eq-LRcoeff-KH}) for $c_{\mu\nu}^\lambda$
coincide with that given in (\ref{Eq-LRcoeff-Bu}). In what follows, rather than explicitly using GT patterns, we use the hive model based on (\ref{Eq-LRcoeff-Bu}).
In doing so it is worth commenting here on the case $\ell(\lambda)=r$ with $r<n$ for which we have $\lambda_n=0$. 
The hive conditions then imply that $\mu_n=\nu_n=0$ and the upright
rhombus gradients $U_{in}$ in the $n$th diagonal of any $n$-hive $H\in{\cal H}^{(n)}(\lambda,\mu,\nu)$
with $\lambda_n=0$ are all $0$. We say in such a case that the $n$th diagonal of $H$ is empty. Let $\kappa_n$ be an operator 
whose action is to restrict any $n$-hive $H$ with empty $n$th diagonal to 
an $(n-1)$-hive $\tilde{H}$ consisting of the leftmost $(n-1)$ diagonals of $H$. This map from
${\cal H}^{(n)}(\lambda,\mu,\nu)$ with $\lambda_n=0$ to ${\cal H}^{(n-1)}(\lambda,\mu,\nu)$
generated by the application of $\kappa_n$ is necessarily a bijection since both
${\cal H}^{(n)}(\lambda,\mu,\nu)$ and ${\cal H}^{(n-1)}(\lambda,\mu,\nu)$
are in bijective correspondence with ${\cal LR}(\lambda/\mu,\nu)$, with the two
maps from $T$ to $H$ and from $T$ to $\tilde{H}$, governed by (\ref{Eq-TtoU}) for appropriate
values $n$ and $n-1$, respectively,
ensuring consistency with the map $\kappa_n:H\mapsto \tilde{H}$.
\medskip

\begin{Example}~~Typically for $n=4$, $\lambda=(8,6,5,0)$, $\mu=(5,4,0)$, $\nu=(5,4,1,0)$ and 
$T$ as shown below, we have
\[
\begin{array}{c}
T\ =\ 
\YT{0.2in}{}{
 {{},{},{},{},{},{1},{1},{1}},
 {{},{},{},{},{1},{2}},
 {{1},{2},{2},{2},{3}},
} 
\in {\cal LR}(\lambda/\mu,\nu) \cr\cr
\end{array}
\]
\end{Example}
The operations $\kappa_k$ for $k=1,2,\ldots,n$ are required in Sections~\ref{Sec-hive-bijection} and \ref{Sec-hive-inv} 
where we construct a bijective map from ${\cal H}^{(n)}(\lambda,\mu,\nu)$ to ${\cal H}^{(n)}(\lambda,\nu,\mu)$ and prove its 
involutive property when extended to map from ${\cal H}^{(n)}$ to ${\cal H}^{(n)}$. 

%%%%%%%%%%%%%%%%%%%%

\section{The action of deletion operators on tableaux}\label{Sec-tab-del-ops}

Our task is to provide a bijective map from all $T\in{\cal LR}(\lambda/\mu,\nu)$
to all $S\in{\cal LR}(\lambda/\nu,\mu)$, and to show that this bijection is an involution, 
or, equivalently, to provide a map %do the same for a bijective map 
from all $H\in{\cal H}^{(n)}(\lambda,\mu,\nu)$ to all $K\in{\cal H}^{(n)}(\lambda,\nu,\mu)$, 
and to show it has the same properties. 
As a first step it is useful to define certain 
deletion operators whose action on $T$ or $H$ will enable us to build up $S$ or $K$. 

First we deal with tableaux, for which the deletion operators were first introduced
in~\cite{Az1,Az2}, and later described in~\cite{PV2}. Their action can be specified by 
modifying slightly a deletion procedure described by Sagan and Stanley~\cite{SS}.
\begin{Definition}~\label{Def-dp} 
For any given $T\in{\cal LR}(\lambda/\mu,\nu)$, augmented by the inclusion
of entries $0$ in each cell of $D(\mu)$, and any corner cell $(r,\lambda_r)$ of $D(\lambda)$
the {\it deletion operator\/} $\delta_{r,\lambda_r}$ is defined to be such that:\\
(i) if $T(r,\lambda_r)=r$ then the action of $\delta_{r,\lambda_r}$ on $T$ is just to delete the
cell $(r,\lambda_r)$. In this case a parameter called the {\it terminating row number\/} 
is defined to be $0$.\\
(ii) if $0<T(r,\lambda_r)<r$ then the action of $\delta_{r,\lambda_r}$ on $T$ is first to delete
the cell $(r,\lambda_r)$ and move its entry into row $r-1$ by replacing the rightmost entry 
strictly less than itself, with that replaced entry in turn being bumped into row $r-2$, again replacing 
the rightmost entry strictly less than itself, and 
continuing in this way until an entry $0$ is replaced by some non-zero entry $x$, say.
Note that this final replacement occurs at some corner cell $(k,\mu_k)$ of $D(\mu)$
since, if the original position of the entry $x$ is $(k+1,j)$, the cell $(k,j)$ must have been empty (otherwise there would have been an entry $T(k,j)<x$ in row $k$) so that we have $\mu_k\geq j>\mu_{k+1}$.
In this case the {\it terminating row number\/} is defined to be $k$. 
The set of the cells where the changes are made
defines a {\it deletion path\/} $P$ in the form of a vertical strip from the corner $(r,\lambda_r)$ 
of $D(\lambda)$ to a corner $(k,\mu_k)$ of $D(\mu)$ with one cell $P[s]$ in row $s$ for each
$s=r,r-1,\ldots,k$.
The reason why the path forms a vertical strip, in other words why the column numbers in the path weakly increase as the replacements go up, is reviewed in the proof of Lemma~\ref{Lem-ss} below.\\
(iii) if $T(r,\lambda_r)=0$ then the action of $\delta_{r,\lambda_r}$ on $T$ is just to remove
the corner cell $(r,\lambda_r)=(r,\mu_r)$ of $D(\lambda)$ and $D(\mu)$. This cell constitutes a 
deletion path consisting of a single cell with a corresponding terminating row number $r$.  
It follows that type (iii) could be included as a special case of type (ii), although this time the
removed entry $0$ is not replaced by any non-zero entry.  
\end{Definition}

While cases (i) and (iii) are rather trivial it is worth illustrating an example of case (ii):

\begin{Example} 
In the case $n=6$, $\lambda=(9,9,6,4,1)$ and $\mu=(7,5,3)$ the action of
$\delta_{4,4}$ on a typical tableau $T$ yields $\tilde{T}$ as shown:
\begin{equation}\label{Ex-2}
T = \vcenter{\hbox{

}}
\end{equation}
In this example the path $P$ has been marked in $T$ by blue entries, with of course the 
initial corner cell deleted in $\tilde{T}$.  
\end{Example}

It might be noted that it is the Littlewood--Richardson property together with the fact
that the contents of a deletion path are strictly decreasing as one works up the tableau that
leads in cases of type (ii) to the deletion path terminating in row $k$ with $k\geq1$. In contrast
to this in cases of type (i), if one performed the deletion as in cases of type (ii) by means of 
deletions in the sense of Sagan and Stanley \cite{SS} starting from a corner cell $(r,\lambda_r)$ of $T$ 
containing the row number $r$, then, due to the Littlewood--Richardson property, the process would 
continue until it bumped a non-zero entry, namely $1$,  out of row $1$, $i.e.$ it would be external. We avoid 
this as in~\cite{Az1,Az2,PV2} by restricting the deletion to be just that of the single 
corner cell, but it explains why, in this situation, the terminating row number is defined to be zero.

One consequence of this is that in cases of type (i), while the semistandard property is automatically 
preserved under deletion, the same is not always true of the lattice permutation property.
However, if $T$ has no entries greater than $r$, including the case $\ell(\lambda)=r$, then it is easy to see that the type (i) deletion of the corner cell $(r,\lambda_r)$ 
does automatically preserve both the semistandard and the lattice permutation properties. 
Similarly, it is easy to see rather trivially that both the semistandard
and the lattice permutation properties are preserved under deletions of type (iii).
Moreover the deletion procedure has been designed so that it also preserves both the semistandard
and the lattice permutation properties of tableaux under deletions of type (ii)
as shown in the following two lemmas.

\begin{Lemma}\label{Lem-ss}
Let the partition $\lambda$ have length $\ell(\lambda)\leq n$ and let $T$ be a semistandard tableau of shape $\lambda/\mu$ for some $\mu\subseteq\lambda$.
If $(r,\lambda_r)$ is a corner cell of $D(\lambda)$ for which the entry in $T$ is 
strictly less than its row number $r$, then the action of $\delta_{r,\lambda_r}$ on 
$T$ preserves the semistandard property.
\end{Lemma}

\noindent{\bf Proof}:
Let $P$ be the deletion path corresponding to the action of $\delta_{r,\lambda_r}$ on $T$.
Since $T(r,\lambda_r)<r$, then as we saw earlier $P$ terminates in some corner $(k,\mu_k)$
of $D(\mu)$. Let $c_r>c_{r-1}>\cdots>c_{k+1}> (c_k=) 0$ be the entries of $P$ in $T$. 
Removing the corner entry $c_r$ from row $r$ clearly does not affect the semistandard
property.  
For $i=r,r-1,\ldots,k+1$ the bumping process involves the rightmost entry $c_i$ from row $i$ 
and say column $j$ replacing the rightmost entry $c_{i-1}<c_i$ in row $i-1$ and say column $l$. This ensures that entries remain weakly increasing from left to right across the row $i-1$. Moreover the column strictness of entries in $T$ ensures that $l\geq j$
(this is why the deletion path forms a vertical strip). If $l=j$  
the column strictness between rows $i-1$ and $i$ is preserved since after deletion $c_i$ will lie immediately above
either an empty site (outside the outer shape) if $i=r$ or the larger entry $c_{i+1}$ that has replaced it in row $i$
if $i<r$. If $l>j$ then after deletion $c_i$ will lie in row $i-1$ immediately above either an empty site (outside the outer shape) or an entry that was originally to the right of the rightmost $c_i$ in row $i$ and is therefore greater than $c_i$. Thus 
the column strictness between rows $i-1$ and $i$ is preserved in column $l$.
That in column $j$ is easier since, while the entry of the cell $(i-1,j)$ remains unchanged, that of the cell $(i,j)$ is now larger than before (or empty outside the outer shape).
Thus semistandardness is preserved.
\qed

\begin{Lemma}\label{Lem-lp}
Let the partition $\lambda$ have length $\ell(\lambda)\leq n$ and let $T$ be a Littlewood--Richardson tableau of shape $\lambda/\mu$ for some $\mu\subseteq\lambda$.
If $(r,\lambda_r)$ is a corner cell of $D(\lambda)$ for which the entry in $T$ is 
strictly less than its row number $r$, then the action of $\delta_{r,\lambda_r}$ on 
$T$ preserves the lattice permutation property.
\end{Lemma}

\noindent{\bf Proof}:
As usual for each $(i,j)$ such that $1\leq i\leq j\leq n$, we let $U_{ij}$ be 
the number of entries $i$ in row $j$ of $T$ and let $\beta_{ij}$ 
be the number of entries $i$ in rows $1$ through $j$ of $T$. 
The lattice permutation inequalities for $T$ take the form $L_{ij}=\beta_{i,j-1}-\beta_{i+1,j}\geq0$ for $1\leq i<j\leq n$.

Again let $P$ be the deletion path corresponding to the action of $\delta_{r,\lambda_r}$ 
on $T$, and note that since $T(r,\lambda_r)<r$, the path $P$ terminates in a corner of $D(\mu)$.
Let $k$ be the terminating row number and let $c_r>c_{r-1}>\cdots>c_{k+1}>0$ be the entries of $P$ in $T$. Then as a result of the deletion operation $\beta_{c_r,r-1},\beta_{c_{r-1},r-2},\dots,\beta_{c_{k+1},k}$ each increase by one, and all other $\beta_{ij}$ stay the same.
It follows that the inequality $L_{ij}=\beta_{i,j-1}-\beta_{i+1,j}\geq0$ in $T$ may only be lost if $\beta_{i+1,j}$ increases by $1$ and $\beta_{i,j-1}$ stays the same upon deletion.
This implies that along $P$ an entry $i+1$ is bumped from row $j+1$ to row $j$ where it 
either replaces an entry $0$ or bumps an entry $l$ from row $j$ to row $j-1$ with $l<i$.
In both cases this can only occur if $j+1\leq r\leq n$ and $U_{i+1,j+1}=\beta_{i+1,j+1}-\beta_{i+1,j}\geq1$ and $U_{ij}=\beta_{ij}-\beta_{i,j-1}=0$. It follows that 
$L_{ij}=\beta_{i,j-1}-\beta_{i+1,j}=\beta_{ij}-\beta_{i+1,j}\geq1+\beta_{ij}-\beta_{i+1,j+1}\geq1$,
where in the final step use has been made of the fact that $L_{i,j+1}=\beta_{ij}-\beta_{i+1,j+1}\geq0$.
Since $L_{ij}$ is therefore $\geq1$ before the deletion along $P$ and the deletion can only diminish it by at most $1$, the condition $L_{ij}\geq0$ is maintained after deletion.
Thus the lattice permutation property is preserved. 
\qed

These Lemmas~\ref{Lem-ss} and~\ref{Lem-lp} enable us to prove 
in the following Lemma~\ref{Lem-lr} that the deletion procedure
preserves the Littlewood--Richardson nature of a tableau.  
In order to specify changes of outer and inner shapes and weights we define 
$\epsilon_i\in \mathbb{Z}^n$ with $i$th entry equal to $1$ and 
all other entries $0$, for $1\le i\le n$.

\begin{Lemma}\label{Lem-lr} Let $\ell(\lambda)=n$. The deletion operator $\delta_{n,\lambda_n}$ on $T\in{\cal LR}({\lambda/\mu,\nu})$ maps $T$ to $\tilde{T}$, with 
$\tilde{T}\in{\cal LR}((\lambda-\epsilon_n)/\mu,(\nu-\epsilon_n))$ if $\nu_n>0$, and  
$\tilde{T}\in{\cal LR}((\lambda-\epsilon_n)/(\mu-\epsilon_k),\nu)$ 
if $\nu_n=0$ and the terminating row number is $k$ with $1\le k\le n$.
\end{Lemma}

\noindent{\bf Proof:~~} If $\nu_n>0$ there is nothing to prove. The lattice permutation rule ensures that there is at least one entry $n$ in row $n$, and the operator $\delta_{n,\lambda_n}$ just removes the cell $(n,\lambda_n)$ with entry $n$ from $T$. The resulting tableau $\tilde{T}$ is clearly in ${\cal LR}((\lambda-\epsilon_n)/\mu,(\nu-\epsilon_n))$.

If $\nu_n=0$ and $\lambda_n>\mu_n$ so that $T(n,\lambda_n)$ is non-zero and less than $n$, then $\delta_{n,\lambda_n}$ 
performs the usual type (ii) deletion as described above along a deletion path $P$ in $T$ starting at the corner cell $(n,\lambda_n)$ of $D(\lambda)$ and terminating in a corner $(k,\mu_k)$ of $D(\mu)$ with $1\leq k<n$. It follows from Lemmas~\ref{Lem-ss} and~\ref{Lem-lp} in the case $r=n$ that the action of $\delta_{n,\lambda_n}$ 
preserves both the semistandard and the lattice permutation property. Having emptied the 
cell $(n,\lambda_n)$ and filled the cell $(k,\mu_k)$ with a non-zero entry we have
$\tilde{T}\in{\cal LR}((\lambda-\epsilon_n)/(\mu-\epsilon_k),\nu)$ as claimed. 

Finally, if $\nu_n=0$ and $\lambda_n=\mu_n$ so that $T(n,\lambda_n)=0$, then the cell $(n,\lambda_n)$ is removed both from the inner shape $\mu$ and the outer shape $\lambda$, so that $k=n$ and again there is nothing to prove. \qed

We also require
\begin{Definition}\label{Def-full-r-del}
For any given $T\in{\cal LR}(\lambda/\mu,\nu)$ with $\ell(\lambda)=r$ the {\it full $r$-deletion\/}
operator $\delta_r$ is defined by
\begin{equation}
\delta_r:=\delta_{r,1}\delta_{r,2}\cdots \delta_{r,\lambda_r}
\end{equation}
Its action on $T$ is realised by acting with $\delta_{r,\lambda_r}$ first and $\delta_{r,1}$ last. 
In general the action proceeds in three phases: 
in Phase 1 corner cells are deleted in turn as in case (i) 
until there remain no entries $r$ in row $r$, and then Phase 2 proceeds by means of deletion operations as in
case (ii) until there are no non-zero entries in row $r$. Finally, Phase 3 corresponds to the removal of the
empty cells in row $r$ as in case (iii). If $\ell(\lambda)<r$ 
the full $r$-deletion is understood to be a null-operator. 
\end{Definition}

\begin{Example} 
Given the tableau $T\in{\cal LR}(\lambda/\mu,\nu)$ as shown on the left below, the full $5$-deletion operator $\delta_{5}=\delta_{5,1}\delta_{5,2}$ acts on $T$ to yield $\delta_5\,T$ on the right. The cell to be deleted under 
the action of $\delta_{5,2}$ is outlined in green, and the deletion path defined by the action of $\delta_{5,1}$
is the set of cells outlined in red. The terminating row numbers of the paths associated with the action of $\delta_{5,2}$ and $\delta_{5,1}$ are ${\green0}$ and ${\red2}$, respectively.

%%%%%%%%%%%%%%%%%%%%%%%%%%%%%%%%%%%%%%%%%%%%%

\begin{equation}\label{Ex-tab-path5251}

\end{equation}
\end{Example}

In this example we took the liberty of indicating the first corner deletion and then its successor
path deletion on the same diagram. This is always possible since the cells involved are quite
distinct from one another. The same is true for any two successive path deletions brought about 
by the action of $\delta_{r,j}$ and $\delta_{r,j-1}$, 
as will be made clear in the Horizontal Path Comparison Lemma below.

Before stating it, we provide some preliminaries.
Let  $(r,\lambda_r)$ be a corner cell of $D(\lambda)$ and let $P$ be the deletion path of $\delta_{r,\lambda_r}$ applicable 
to $T\in\mathcal{LR}(\lambda/\mu,\nu)$ with $0<T(r,\lambda_r)<r$ and terminating row number $t$ ($<r$).
Let $\lambda^\flat$ and $\mu^\flat$ be defined by $D(\lambda^\flat)=D(\lambda)-\{(r,\lambda_r)\}$ and $D(\mu^\flat)=D(\mu)-\{(t,\mu_t)\}$ respectively, so that $\delta_{r,\lambda_r}T$ has shape $\lambda^\flat/\mu^\flat$, and set $P^\flat=P\cap D(\lambda^\flat/\mu^\flat)=P-\{(r,\lambda_r)\}$.
Then the bottom cell $P^\flat[r-1]$ (resp.\ the top cell $P^\flat[t]$) of $P^\flat$ 
is bounded at the bottom (resp.\ the top) by
a horizontal segment of the outer border (resp.\ the inner border) of $D(\lambda^\flat/\mu^\flat)$.
In general, for any skew LR tableau, a vertical strip $P^\flat$ ascending from a cell 
bounded at the bottom by
a horizontal segment of its outer border to a cell 
bounded at the top by 
a horizontal segment of its inner border, and containing
one cell of its shape from \textit{every} row in between, determines as illustrated below the following subsets of the shape of the tableau: that subset \textit{strictly to the left of $P^\flat$} which, by convention,
%but it is natural to do so (see the illustration below), 
also includes everything below the row where $P^\flat$ starts, and that subset \textit{weakly to the right of $P^\flat$},
that is the union of $P^\flat$ itself and the subset \textit{strictly to the right of $P^\flat$} which, again by convention, also includes everything above the row where $P^\flat$ terminates.
If $P^\flat$ is obtained from a deletion path $P$ by removing its starting cell and if there would be no confusion, these subsets are also called \textit{strictly to the left of $P$} and so on.

In the example below on the left, the cell marked by \tikz[x={(0in,-0.08in)},y={(0.08in,0in)}]{\draw[thick,red](1,1) rectangle +(-1,-1);} is the starting cell of $P$ and those marked by \tikz[x={(0in,-0.08in)},y={(0.08in,0in)}]{\fill[red](1,1) rectangle +(-1,-1);} form $P^\flat$.
The plain white region shows the subset strictly to the left of $P$, the dots fill the subset strictly to the right of $P$, which is joined by the cells \tikz[x={(0in,-0.08in)},y={(0.08in,0in)}]{\fill[red](1,1) rectangle +(-1,-1);} to form the subset weakly to the right of $P$.
Below on the right is an example of part of some
shape $\lambda^\#/\mu^\#$ containing the same $P^\flat$ and whose bottom and top cells are the bottom and top nonempty cells in their respective columns.

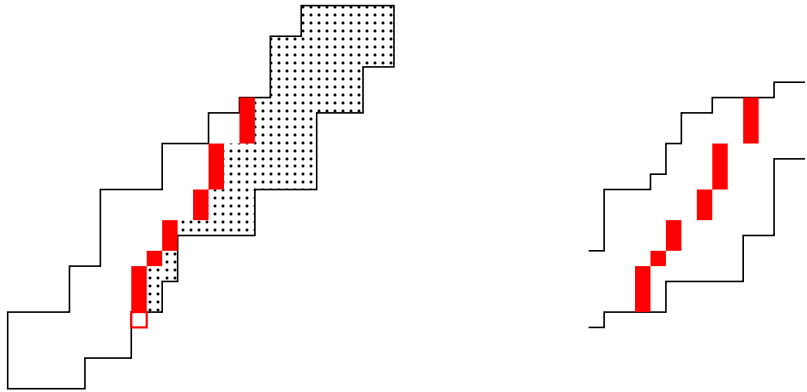
\[
\vcenter{\hbox{
\begin{tikzpicture}[x={(0in,-0.08in)},y={(0.08in,0in)}]
\matrix{
\draw[semithick]
     (25,0)--++(0,5)--++(-2,0)--++(0,3)--++(-3,0)--++(0,2)--++(-2,0)--++(0,1)--++(-3,0)--++(0,5)
 --++(-3,0)--++(0,4)--++(-5,0)--++(0,3)--++(-3,0)--++(0,2)--++(-4,0)
 --++(0,-6)--++(2,0)--++(0,-2)--++(4,0)--++(0,-2)--++(1,0)--++(0,-2)--++(2,0)
 --++(0,-3)--++(3,0)--++(0,-4)--++(5,0)--++(0,-2)--++(3,0)--++(0,-4)--cycle;
\draw[thick,red] (21,9) rectangle +(-1,-1);
\path[fill=red]
     (20,8)--++(0,1)--++(-3,0)--++(0,1)--++(-1,0)--++(0,1)--++(-2,0)--++(0,2)
 --++(-2,0)--++(0,1)--++(-3,0)--++(0,2)--++(-3,0)
 --++(0,-1)--++(3,0)--++(0,-2)--++(3,0)--++(0,-1)--++(2,0)--++(0,-2)--++(2,0)
 --++(0,-1)--++(1,0)--++(0,-1)--cycle;
\path[pattern=dots]
     (20,9)--++(0,1)--++(-2,0)--++(0,1)--++(-3,0)--++(0,5)
 --++(-3,0)--++(0,4)--++(-5,0)--++(0,3)--++(-3,0)--++(0,2)--++(-4,0)
 --++(0,-6)--++(2,0)--++(0,-2)--++(4,0)
 --++(0,-1)--++((3,0)--++(0,-2)--++(3,0)--++(0,-1)--++(2,0)--++(0,-2)--++(2,0)
 --++(0,-1)--++(1,0)--++(0,-1)--cycle;
&[1in]
\draw[semithick]
     (21,5)--++(0,1)--++(-1,0)--++(0,4)--++(-2,0)--++(0,5)--++(-3,0)--++(0,2)
 --++(-5,0)--++(0,2);
\draw[semithick]
     (5,19)--++(0,-2)--++(1,0)--++(0,-4)--++(1,0)--++(0,-2)--++(2,0)--++(0,-1)
 --++(2,0)--++(0,-1)--++(1,0)--++(0,-3)--++(4,0)--++(0,-1);
\path[fill=red]
     (20,8)--++(0,1)--++(-3,0)--++(0,1)--++(-1,0)--++(0,1)--++(-2,0)--++(0,2)
 --++(-2,0)--++(0,1)--++(-3,0)--++(0,2)--++(-3,0)
 --++(0,-1)--++(3,0)--++(0,-2)--++(3,0)--++(0,-1)--++(2,0)--++(0,-2)--++(2,0)
 --++(0,-1)--++(1,0)--++(0,-1)--cycle;
\\
};
\end{tikzpicture}
}}
\]

After a deletion along $P$ starting from $(r,\lambda_r)$, another deletion path $P'$ could start either from any corner strictly to the left of $P$ or from one weakly to the right of $P$, according to which the entire new deletion path $P'$
lies either strictly to the left of $P$ or weakly to the right of $P$, due to the type of argument 
used repeatedly since~\cite[proof of Theorem 1]{Kn}(Knuth) in dealing with insertions rather than deletions,
see also~\cite[Remarque 2]{Schu} (Sch\"{u}tzenberger).
We now include a proof of this claim not only for the sake of maintaining a self-contained presentation 
but also to take this opportunity to generalise it in a manner to be used later.

\begin{Lemma}~[Elementary Path Comparison]\label{Lem-ep} 
Let $T\in\mathcal{LR}(\lambda/\mu,\nu)$, $(r,\lambda_r)$ a corner of $D(\lambda)$ such that $T(r,\lambda_r)<r$ and $P$ the deletion path of $\delta_{r,\lambda_r}$ applied to $T$ with terminating row number $t$.
Set $T{\red'}=\delta_{r,\lambda{\red{_r}}}T\in\mathcal{LR}(\lambda^\flat/\mu^\flat,\nu)$ where $\lambda^\flat$ and $\mu^\flat$ are defined as above. Moreover, 
let $T^\#$ be an LR tableau of some shape $\lambda^\#/\mu^\#$ containing the same vertical strip $P^\flat=P-\{(r,\lambda_r)\}$ having the same entries as in $T'$, with its bottom cell (resp.\ top cell) bounded below (resp.\ above) by
a horizontal segment of the outer (resp.\ inner) border of $D(\lambda^\#/\mu^\#)$.
Let $(r',\lambda^\#_{r'})$ be a corner of $D(\lambda^\#)$ 
such that $T^\#(r',\lambda^\#_{r'})<r'$ and $P^\#$ the deletion path of $\delta_{r',\lambda^\#_{r'}}$ applied to $T^\#$ with terminating row number $t'$.
\begin{enumerate}
\item[(1)]
Suppose that $(r',\lambda^\#_{r'})$ lies strictly to the left of $P$, and that the part of $T^\#$ to the left of $P$ coincides that part of $T'$ (and hence that part of $T$ as well).
Then the entire $P^\#$ runs strictly to the left of $P$, and we have $t'\geq t$.
\item[(2)]
Suppose that $(r',\lambda^\#_{r'})$ lies weakly to the right of $P$ (in this case the part of $T^\#$ strictly to the left of $P$ can be different from that part of $T'$ so long as the assumptions made before part (1) are satisfied).
Then the entire $P^\#$ runs weakly to the right of $P$, and we have $t'<t$.
\end{enumerate}
\end{Lemma}

\begin{Remark}
Note that the assumptions for $T^\#$ in both cases (1) and (2)
are satisfied by the tableau $T'=\delta_{r,\lambda_r}T$, so that both statements (1) and (2) are applicable, 
in particular, to a deletion occurring immediately after $\delta_{r,\lambda_r}$.
In addition, we can (and will) apply (2) after performing several deletions from
$T'$ so long as they do not change the contents of $P$ or stack any cells above
the terminating cell of $P$ (see the proof of the Vertical Path Comparison Lemma~\ref{Lem-vp} below).
\end{Remark} 

\begin{proof}
As in Definition~\ref{Def-dp}, let $P[s]$ denote the cell of the path $P$ in row $s$ for any applicable $s$.
This notation will be used for other deletion paths as well.

The claim (1) is obviously true if $P^\#$ terminates in or below row $r$, that is if $t'\geq r$.
So we consider the case where $t'<r$.
Note that $P^\#$ lies strictly to the left of $P$ in rows $r$ and lower since everything in such rows is considered to be strictly to the left of $P^\flat$ by convention.
Now we proceed by decreasing induction on the row number $i$, with  
$t \leqq i<r$ and assume that $P^\#$, having run strictly to the left of $P$ in rows $>i$, enters row $i$.
Note that, since $i<r$, the cell $P[i+1]$ exists, at least in $T$.
Since $P^\#[i+1]$ lies strictly to the left of $P[i+1]$ by the induction hypothesis,
we have $T^\#(P^\#[i+1])=T(P^\#[i+1])\leq T(P[i+1])=T'(P[i])=T^\#(P[i])$.
This means that the candidate entries $T^\#(P^\#[i])$, which must be 
$< T^\#(P^\#[i+1])$, can only lie strictly to the 
left of $P[i]$ in row $i$ of $T^\#$, so that $P^\#$ proceeds strictly to the left of $P[i]$, including the possibility of terminating there with $t'=i\geq t$. If $i=t$, $P^\#$ must terminate there since $P[t]$ is the leftmost nonempty cell 
in row $t$ of $T^\#$.
This inductively shows that the entire path $P^\#$ runs strictly to the left of $P$, and also that $t'\geq t$.

On the other hand, the claim (2) is obviously true if $r'<t$, so we assume $r'\geq t$  
and again proceed by decreasing induction on the row number $i$, with $r'\geq i\geq t$.
The initial case $i=r'$ is covered by the assumption in (2) that $P^\#[r']$ is weakly to the right of $P[r']$.
If $r'>i\geq t$, we have $T^\#(P^\#[i+1])\geq T^\#(P[i+1])=T'(P[i+1])>T'(P[i])$ with the first inequality 
following from the induction hypothesis on the location of $P^\#$ in row $i+1$, 
and the next equality following from the equality of entries in $T^\#$ and $T'$ along $P$, while 
the second inequality is the preserved evidence of the deletion along $P$.
Thus the entry $T^\#(P^\#[i+1])$ goes to the cell $P[i]$ or somewhere to its right,
that is to say $P^\#[i]$ lies weakly to the right of $P[i]$.
Since this continues up to $i=t$, in particular the deletion $\delta_{r',\lambda^\#_{r'}}$ applied to $T^\#$
still bumps something from row $t$, we have $t'<t$.\qedhere
\end{proof}

In our setting where the starting cells of the deletions are taken first from right to left in row $n$, then from right to left in row $n-1$ and so on, these considerations can be usefully organized in the form of the following two lemmas, the first of which is already clear.

\begin{Lemma}~[Horizontal Path Comparison]\label{Lem-hp}
Let $T\in{\cal LR}(\lambda/\mu,\nu)$ with $\ell(\lambda)=r$.
Consider the full $r$-deletion operator $\delta_r=\delta_{r,1}\cdots\delta_{r,\lambda_r}$ applied to $T$.
Let $P$ and $P'$ denote the deletion paths of $\delta_{r,j}$ and $\delta_{r,j-1}$ in this sequence, 
where $2\leq j\leq\lambda_r-\nu_r$, and let $t$ and $t'$ be their terminating row numbers.
Then $P'$ runs strictly to the left of $P$, and we have $t'\geq t$.
\end{Lemma}

\noindent{\bf Proof:~~}
This is immediate from part (1) of Lemma~\ref{Lem-ep}.
\qed

The second one, which is helpful in considering the action of $\delta_r$ followed by that of $\delta_{r-1}$, is slightly less straightforward.

\begin{Lemma}~[Vertical Path Comparison]\label{Lem-vp}
Let $T\in{\cal LR}(\lambda/\mu,\nu)$ with $\ell(\lambda)=r$ and recall that
$U_{ir}$ is the number of entries $i$ in row $r$ of $T$. 
Phase 1 of the action of $\delta_r$ on $T$ removes all $U_{rr}$ cells containing entries $r$ in row $r$, and
we divide Phase 2 of $\delta_r$ into Phase 2A and Phase 2B:
Phase 2A consists of the $U_{r-1,r}$ type (ii) deletions starting from the cells containing $r-1$ in row $r$, 
and Phase 2B consists of the remaining type (ii) deletions starting from cells with non-zero entries $\leq r-2$.
We set $m=\lambda_r-\nu_r-U_{r-1,r}-\mu_r$, that is the number of deletions in Phase 2B of $\delta_r$.
Then it is to be noted that Phase 2 of $\delta_{r-1}$ applied to $T^\#:=\delta_rT$ has at least $m$ deletions, 
since $\delta_r$ has just delivered $m$ nonzero entries $\leq r-2$ to row $r-1$.
For each $j\in\mathbb{N}$ in the range $1\leq j\leq m$, let $P_j$ (resp.\ $P^\#_j$) denote the path of the $j$th deletion operation, in the order of occurrence, in Phase 2B of $\delta_r$ applied to $T$ (resp.\ Phase 2 of $\delta_{r-1}$ applied to $T^\#$).
Let $t_j$ (resp.\ $t^\#_j$) be the terminating row number of $P_j$ (resp.\ $P^\#_j$).
Then for each such $j$, $P^\#_j$ runs weakly to the right of $P_j$, and we have $t^\#_j<t_j$.
\end{Lemma}

\begin{Example} 
In the case $\lambda=(9,6,5,4)$, $\mu=(6,5,2)$ and the following LR tableau $T$, as shown on the left, 
we have $r=4$, $\lambda_r=4$, $\mu_r=0$, $\nu_r=0$ and $U_{r-1,r}=1$ so that $m=3$.

\[
T \ = \
\vcenter{\hbox{

}}
\]

The Phase 2A of $\delta_4=\delta_{41}\delta_{42}\delta_{43}\delta_{44}$ applied to $T$ consists only of the deletion $\delta_{44}$which produces the deletion blue path in $T$; the Phase 2B consists of deletions $\delta_{43}$, $\delta_{42}$ and $\delta_{41}$, in order of occurrence, producing the deletion  paths $P_1$ in red, $P_2$ in green, and $P_3$ in violet.  The Phase 2A of $\delta_3$ applied to $T^\#$ consists exactly of three deletions and produces the deletion paths   $P_1^\#$  in red,   $P_2^\#$  in green,
  and  $P_3^\#$  in violet.   $P_1^\#$ is weakly to the right of $P_1$, both sharing the cell $(2,5)$, $P_i^\#$ is strictly  to the right of $P_i$ for $i=2,3$; the corresponding terminating rows are $t_1=2>t_1^\#=1$, $t_2=3>t_2^\#=2$, $t_3=3>t_3^\#=2$.
\end{Example}
\medskip

\noindent{\bf Proof:~~}
Once again let $P_j[s]$ (resp.\ $P^\#_j[s]$) denote the unique cell where $P_j$ (resp.\ $P^\#_j$) intersects row $s$, for all rows $s$ over which $P_j$ (resp.\ $P^\#_j$) extends.  
With this notation $\delta_{P_j[r]}$ (resp.\ $\delta_{P^\#_j[r-1]}$) denotes the deletion along the path $P_j$ (resp.\ $P^\#_j$), which is part of $\delta_r$ (resp.\ $\delta_{r-1}$).

The deletions in Phase 2A of $\delta_r$ carry entries $r-1$ from row $r$ to row $r-1$, which form a contiguous segment immediately to the left of any original $r-1$ in row $r-1$.
The deletions in Phase 2B of $\delta_r$ then carry smaller entries into row $r-1$, which settle in cells further to the left.
Moreover, these cells $P_1[r-1],P_2[r-1],\ldots,P_m[r-1]$, which may or may not be contiguous, lie from right to left in this order, due to the Horizontal Path Comparison Lemma~\ref{Lem-hp}.
Hence, for each $1\leq j\leq m$, the starting cell $P^\#_j[r-1]$ of $P^\#_j$, being the $j$th rightmost cell to 
the left of the entries $r-1$ in row $r-1$ of $T^\#$, lies weakly to the right of $P_j[r-1]$.
Part (2) of Lemma~\ref{Lem-ep} then implies the required result, namely that 
$P^\#_j$ lies weakly to the right of $P_j$ and $t_j^\#<t_j$, provided only that immediately
before the deletion $\delta_{P^\#_j[r-1]}$ starts, the tableau, say $T^\#_j$, satisfies all
the assumptions regarding in this instance $P_j-\{P_j[r]\}$ implied by the hypotheses of Lemma~\ref{Lem-ep}.

To establish this we use induction on $j$. 
First assume that $j=1$.
By the Horizontal Path Comparison Lemma~\ref{Lem-hp}, 
no deletions in $\delta_r$ occurring after $\delta_{P_1[r]}$ alter the contents of $P_1$, 
nor do they put any entries above $P_1[t_1]$ since they terminate in rows $\geq t_1$.
Moreover, all operations between the end of $\delta_r$ and the beginning of $\delta_{P^\#_1[r-1]}$ are Phase 1 deletions 
which simply remove the cells in row $r-1$ lying strictly to the right of $P_1$, so after these operations the contents of $P_1-\{P_1[r]\}$ are still preserved, and the cell $P_1[r-1]$ (resp.\ $P_1[t_1]$) is still bounded below (resp.\ above) by
a horizontal segment of the outer (resp.\ inner) border of $T^\#_1$.

Now assume that $2\leq j\leq m$.
By the Horizontal Path Comparison Lemma~\ref{Lem-hp},
the deletions in $\delta_r$ occurring after $\delta_{P_j[r]}$ do not alter the contents of $P_j$, 
nor do they put any entries above $P_j[t_j]$ since they terminate in rows $\geq t_j$.
Now, by the Horizontal Path Comparison Lemma~\ref{Lem-hp}
applied to $P_{j-1}$ and $P_j$, $P_{j-1}$ lies strictly to the right of $P_j$, and $P^\#_{j-1}$ runs weakly to the right of $P_{j-1}$ by the induction hypothesis. 
Moreover, again by the Horizontal Path Comparison Lemma~\ref{Lem-hp},
this time applied to $P^\#_1,\dots,P^\#_{j-1}$, these paths all lie strictly to the right of $P_j$ (symbolically we have $\cdots<P_{j+1}<P_j<P_{j-1}\leq P^\#_{j-1}<\cdots<P^\#_1$).
Hence, including the Phase 1 deletions in $\delta_{r-1}$, all operations occurring between $\delta_{P_j[r]}$ and $\delta_{P^\#_j[r-1]}$ leave the contents of $P_j$ untouched, and the cell $P_j[r-1]$ (resp.\ $P_j[t_j]$)
is still bounded below (resp.\ above) by a horizontal segment of the outer (resp.\ inner) border of $T^\#_j$.

By the induction argument this shows, as required, that $T^\#_j$ does satisfy all the conditions of the 
preamble in Lemma~\ref{Lem-ep} for all $j=1,2\ldots,m$, thereby completing the proof of the current Lemma. \qed

\section{Bijective map from ${\cal LR}(\lambda/\mu,\nu)$ to ${\cal LR}(\lambda/\nu,\mu)$}\label{Sec-tab-bijection}

With our deletion operators we are able to construct
from any LR tableau $T\in{\cal LR}(\lambda/\mu,\nu)$ a partner LR tableau 
$S\in{\cal LR}(\lambda/\nu,\mu)$.
To do so we empty the initial LR tableau $T$ by performing
a sequence of deletion operations, from right to left across each row in turn from bottom to top.
This renders it empty and simultaneously allows us to build the final LR tableau 
$S$ from the data on the row number of the starting cell and the terminating row number of 
each deletion path.

To accomplish this, we use a 
formulation in which the initial $T\in\mathcal{LR}(\lambda/\mu,\nu)$ and the final $S\in\mathcal{LR}(\lambda/\nu,\mu)$ are associated with pairs $(T^{(n)},S^{(0)})$ and $(T^{(0)},S^{(n)})$, respectively, where $T^{(n)}=T$, $S^{(n)}=S$ and both $S^{(0)}$ and $T^{(0)}$ are empty tableaux, and we introduce operators that modify both components of pairs.

\begin{Definition}\label{Def-TtoS}
For an LR tableau $T$ with no more than $n$ rows, set $T^{(n)}=T$ and 
let $S^{(0)}$ be an empty tableau. Then  
let $\Delta^{(n)}$ denote the operation which transforms the pair $(T^{(n)},S^{(0)})$ into another pair 
$\Delta^{(n)}(T^{(n)},S^{(0)}):=(T^{(0)},S^{(n)})$ through the action of 
a succession of $n$ operators $\Delta_n$, $\Delta_{n-1}$, \dots, $\Delta_1$ that produce pairs 
$(T^{(r)},S^{(n-r)})$, with $r=n-1,n-2,\dots,0$, respectively, as indicated by
\begin{equation}
(T^{(n)},S^{(0)})\overset{\Delta_n}{\longmapsto}(T^{(n-1)},S^{(1)})\overset{\Delta_{n-1}}{\longmapsto}\cdots\overset{\Delta_2}{\longmapsto}(T^{(1)},S^{(n-1)})\overset{\Delta_1}{\longmapsto}(T^{(0)},S^{(n)}),
\end{equation}
where each $\Delta_r$ maps $(T^{(r)},S^{(n-r)})$ to 
$(T^{(r-1)},S^{(n-r+1)})$ by applying the operation $\delta_r$ defined in Definition~\ref{Def-dp} to the left-hand component $T^{(r)}$, so that $T^{(r-1)}=\delta_{r}T^{(r)}$, and by placing a new row on top of the right-hand component $S^{(n-r)}$ to produce $S^{(n-r+1)}$, with this row consisting of cells containing the terminating row numbers of this $\delta_{r}$ operation, including zeros, arranged from left to right across the new row in the order in which they are identified under the action of $\delta_r$.
We call the tableau $S^{(n)}$ the \textit{partner tableau} of $T^{(n)}$.
\end{Definition}

To see that the prescribed operations are applicable, note that, by 
repeated use of Lemma~\ref{Lem-lp}, $T^{(r)}$ is an LR tableau with outer shape $(\lambda_1,\dots,\lambda_{r})$.
$S^{(r)}$ will be a certain tableau of shape $(\lambda_{r+1},\dots,\lambda_n)$ with zero and positive entries,
or in fact a filling of the rows $r+1$ through $n$ of the Young diagram $D(\lambda)$ with such entries.
In particular, $T^{(0)}$ is an empty tableau, and $S^{(n)}$ is a filling of $D(\lambda)$ with zero and positive entries.
Furthermore, we have

\begin{Theorem}\label{The-TtoS}
Let $n$ be a positive integer and let $\lambda$, $\mu$ and $\nu$ be partitions such that $\ell(\lambda)\leq n$ and $\mu,\nu\subseteq\lambda$ with $|\lambda|=|\mu|+|\nu|$.
For each LR tableau $T\in\mathcal{LR}(\lambda/\mu,\nu)$ let $T^{(n)}=T$ and let $S^{(0)}$ be an empty tableau.
If we let $\Delta^{(n)}(T^{(n)},S^{(0)})=(T^{(0)},S^{(n)})$ as in Definition~\ref{Def-TtoS}, then $S=S^{(n)}$ is a LR tableau and $S\in\mathcal{LR}(\lambda/\nu,\mu)$.
\end{Theorem}

As a first step towards proving this theorem we establish with the notation of Definition~\ref{Def-TtoS}
the following:
\begin{Lemma}\label{Lem-TSshape}
Let $\mu^{(r)}=(\mu^{(r)}_1,\dots,\mu^{(r)}_r)$ denote the inner shape of $T^{(r)}$ for $r=n,n-1,\ldots,0$.
Then $T^{(r)}$ is an LR tableau of shape $(\lambda_1,\dots,\lambda_r)/\mu^{(r)}$ and weight $(\nu_1,\dots,\nu_r)$ while $S^{(n-r)}$ is a tableau of shape $(\lambda_{r+1},\dots,\lambda_n)/(\nu_{r+1},\dots,\nu_n)$ and weight $\mu-\mu^{(r)}$.
\end{Lemma}

\begin{proof}[Proof]
The result is obvious for $r=n$ in which case $\mu^{(n)}=\mu$, $T^{(n)}$ has 
shape $\lambda/\mu$ and $S^{(0)}$ is empty.
So let us assume that the result has been verified for the pair 
$(T^{(r)},S^{(n-r)})$ for some $1\leq r\leq n$, 
and show the result for the pair $(T^{(r-1)},S^{(n-r+1)})$.

By Definition~\ref{Def-TtoS}, we have $T^{(r-1)}=\delta_rT^{(r)}$.
Recall from Definition~\ref{Def-dp} the three phases of the action of $\delta_r$ 
which eventually shave off the whole row $r$ from the outer shape of $T^{(r)}$.
Since, by the induction hypothesis, $T^{(r)}$ has no more than $r$ rows and 
satisfies the lattice permutation property, the entries $r$ can only appear in row $r$.
Moreover, the number of such entries is $\nu_r$ since none can have been removed 
under the action of $\delta_s$ for $s>r$.
Hence, as described in Lemma~\ref{Lem-lr}, the operations in Phase 1 delete 
all $\nu_r$ entries $r$ from $T^{(r)}$, thereby reducing the weight from
$(\nu_1,\dots,\nu_{r})$ to $(\nu_1,\dots,\nu_{r-1})$.
The operations in Phases 2 and 3 do not remove any nonzero entry, so this weight remains unchanged. 
However, in general they remove some cells from the inner shape of $T^{(r)}$, 
including all inner cells in row $r$, making the new inner shape $\mu^{(r-1)}$ (by our choice of notation).
Since the LR property is maintained, again by Lemma~\ref{Lem-lr}, it follows that 
$T^{(r-1)}$ is an LR tableau of shape $(\lambda_1,\dots,\lambda_{r-1})/\mu^{(r-1)}$ and weight $(\nu_1,\dots,\nu_{r-1})$.

By Definition~\ref{Def-TtoS}, $S^{(n-r+1)}$ is constructed by placing on top of $S^{(n-r)}$ a 
new row containing the terminating row numbers of the deletions ocurring in the application of $\delta_r$ 
to $T^{(r)}$. This row has length $\lambda_r$ since $\delta_r$ consists of $\lambda_r$ deletion operations.
Phase 1 consists of $\nu_r$ type (i) deletions each producing a terminating row number $0$.  
Then Phases 2 and 3 produce $\lambda_r-\nu_r$ nonzero terminating row numbers 
of weight $\mu^{(r)}-\mu^{(r-1)}$. It follows by the induction argument that
the weight of $S^{(n-r+1)}$ is $(\mu-\mu^{(r)})+(\mu^{(r)}-\mu^{(r-1)})=\mu-\mu^{(r-1)}$,
as required.
\end{proof}

\begin{Example}\label{Ex-tab}
The situation is illustrated in the case $n=4$ and $r=4,3,2,1,0$ as follows.
On the left, at each stage, the red boxes represent the terminating cells of the deletion paths that 
occur during Phases 2 and 3 of the operation $\delta_r$, which form the skew shape $\mu^{(r)}/\mu^{(r-1)}$, 
while the green boxes represent the cells to be removed during Phase 1 of $\delta_r$.
The tableaux $T^{(r)}$ are LR tableaux by Lemma~\ref{Lem-lr}.
On the right, it may be noted that the tableaux $S^{(n-r)}$ are all 
semistandard, and the final one $S^{(n)}$ is an LR tableau. These properties will
be established in complete generality below.

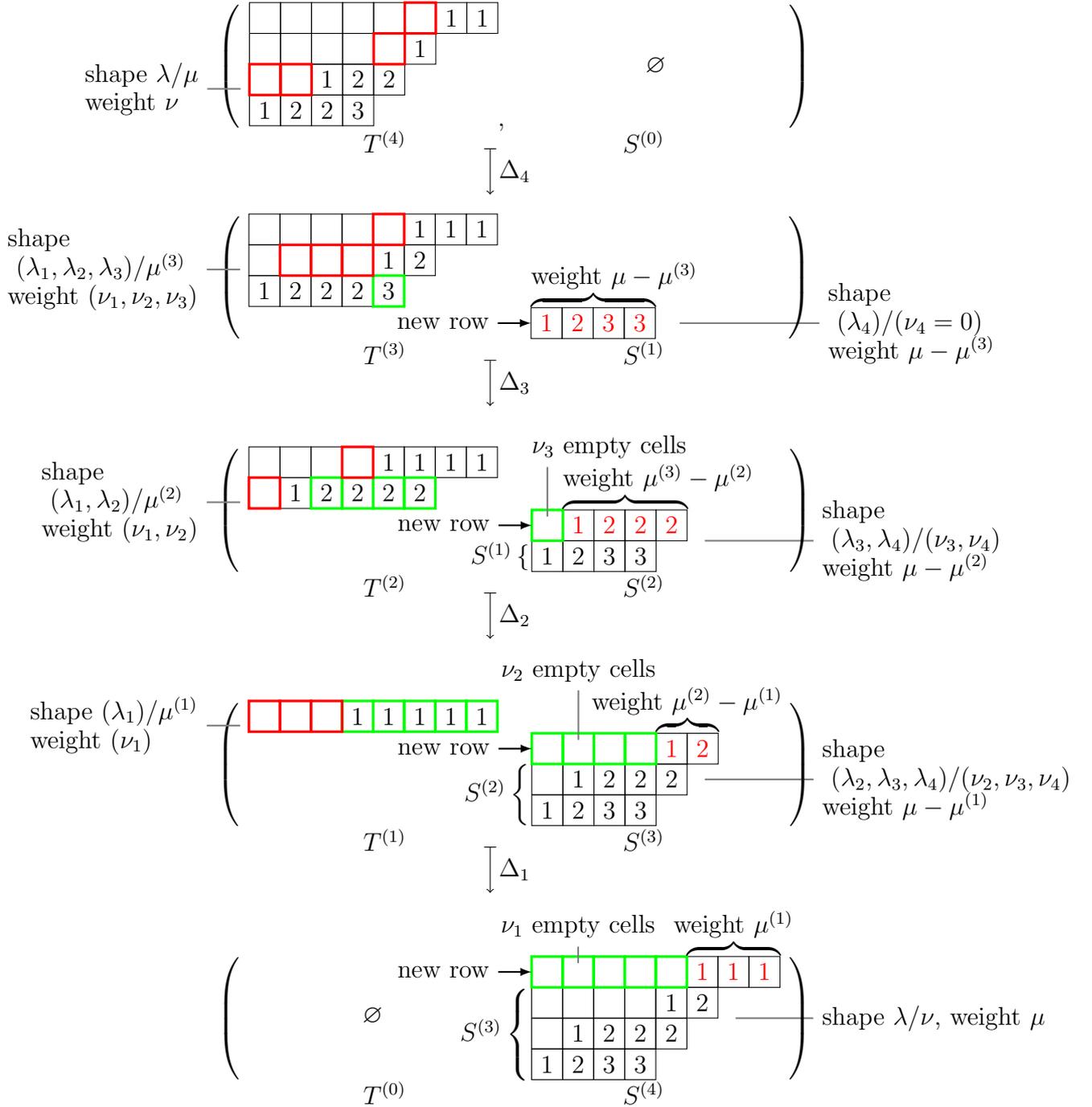
\begin{gather*}
\left(
\vYTd{0.2in}{}{% T^{(4)}
 {{},{},{},{},{},{},1,1},
 {{},{},{},{},{},1},
 {{},{},1,2,2},
 {1,2,2,3}}
{
 \foreach \i/\j in {1/6,2/5,3/1,3/2}
  \draw[red,very thick] (\i,\j) rectangle +(-1,-1);
 \node[inner sep=0pt,outer sep=0pt,anchor=base west] at (4,8) {, };
 \begin{scope}[overlay]
  \node
   [pin={[pin edge=thick,align=left]left:{shape $\lambda/\mu$\\[-4pt]weight $\nu$}}]
   at (2.8,0) {};
 \end{scope}
}
\right.\left.
\vYTd{0.2in}{}{% S^{(0)}
 {}}
{
 \draw[transparent](0,0)--(0,8) (4,0)--(4,8);
 \node at (2,4) {$\varnothing$};
}
\right)
\\[-8pt]
\left.\tikz[x={(0in,0.2in)},y={(-0.2in,0in)}]\draw[transparent](0,0)--node[opaque]{$T^{(4)}$}(0,8);\right.
\left.\tikz[x={(0in,0.2in)},y={(-0.2in,0in)}]\draw[transparent](0,0)--node[opaque]{$S^{(0)}$}(0,8);\right.
\displaybreak[0]\\[-0.2in]
\tikz\draw[|->](0in,0in)-- node[right]{$\Delta_4$} (0in,-0.3in);
\displaybreak[0]\\
\left(
\vYTd{0.2in}{}{% T^{(3)}
 {{},{},{},{},{},1,1,1},
 {{},{},{},{},1,2},
 {1,2,2,2,3}}
{
 \foreach \i/\j in {3/5}
  \draw[green,very thick] (\i,\j) rectangle +(-1,-1);
 \foreach \i/\j in {1/5,2/2,2/3,2/4}
  \draw[red,very thick] (\i,\j) rectangle +(-1,-1);
 \draw[transparent](4,0)--(4,8)
  node[inner sep=0pt,outer sep=0pt,anchor=base west] {, };
 \begin{scope}[overlay]
  \node
   [pin={[pin edge=thick,align=left]left:{shape\\[-4pt]\ $(\lambda_1,\lambda_2,\lambda_3)/\mu^{(3)}$\\[-4pt]weight $(\nu_1,\nu_2,\nu_3)$}}]
   at (1.8,0) {};
 \end{scope}
}
\right.\left.
\vYTd{0.2in}{}{% S^{(1)}
 {{\red 1},{\red 2},{\red 3},{\red 3}}}
{
 \draw[transparent](-3,0)--(-3,8);
 \node[above,inner sep=0pt,outer sep=0pt] at (0,2)
  {$\overbrace{\hphantom{\hbox to0.8in{\hss}}}
     ^{\hbox to0.8in{weight $\mu-\mu^{(3)}$\hss}}$};
 \begin{scope}[overlay]
  \node[coordinate,pin={[pin edge={thick,black,latex-}]left:new row}] at (0.5,0) {};
  \node[pin={[pin edge=thick,pin distance=0.9in,align=left]right:{shape\\[-4pt]\ $(\lambda_4)/(\nu_4=0)$\\[-4pt]weight $\mu-\mu^{(3)}$}}] at (0.5,4.5) {};
 \end{scope}
}
\right)
\\[-8pt]
\left.\tikz[x={(0in,0.2in)},y={(-0.2in,0in)}]\draw[transparent](0,0)--node[opaque]{$T^{(3)}$}(0,8);\right.
\left.\tikz[x={(0in,0.2in)},y={(-0.2in,0in)}]\draw[transparent](0,0)--node[opaque]{$S^{(1)}$}(0,8);\right.
\displaybreak[0]\\[-0.2in]
\tikz\draw[|->](0in,0in)-- node[right]{$\Delta_3$} (0in,-0.3in);
\displaybreak[0]\\[10pt]
\left(
\vYTd{0.2in}{}{% T^{(2)}
 {{},{},{},{},1,1,1,1},
 {{},1,2,2,2,2}}
{
 \foreach \i/\j in {2/3,2/4,2/5,2/6}
  \draw[green,very thick] (\i,\j) rectangle +(-1,-1);
 \foreach \i/\j in {1/4,2/1}
  \draw[red,very thick] (\i,\j) rectangle +(-1,-1);
 \draw[transparent](4,0)--(4,8)
  node[inner sep=0pt,outer sep=0pt,anchor=base west] {, };
 \begin{scope}[overlay]
  \node
   [pin={[pin edge=thick,align=left]left:{shape\\[-4pt]\ $(\lambda_1,\lambda_2)/\mu^{(2)}$\\[-4pt]weight $(\nu_1,\nu_2)$}}]
   at (1.8,0) {};
 \end{scope}
}
\right.\left.
\vYTd{0.2in}{}{% S^{(2)}
 {{},{\red 1},{\red 2},{\red 2},{\red 2}},
 {1,2,3,3}}
{
 \draw[transparent](-2,0)--(-2,8);
 \foreach \i/\j in {1/1}
  \draw[green,very thick] (\i,\j) rectangle +(-1,-1);  %\fill[green]
 \begin{scope}[overlay]
  \node[coordinate,pin={[pin edge={thick,black,latex-}]left:new row}] at (0.5,0) {};
  \node[left,inner sep=0pt,outer sep=0pt] at (1.5,0) {$S^{(1)}\left\{\vbox to8.5pt{}\right.$};
  \node[above,inner sep=0pt,outer sep=0pt] at (0,3)
   {$\overbrace{\hphantom{\hbox to0.8in{\hss}}}
      ^{\hbox to0.8in{weight $\mu^{(3)}-\mu^{(2)}$\hss}}$};
  \node[pin={[pin edge=thick,pin distance=4.5ex]north:
        \hbox to0.2in{$\nu_3$ empty cells\hss}}] at (0.5,0.5) {};
  \node[right,inner sep=0pt,outer sep=0pt,pin={[pin edge=thick,pin distance=0.7in,align=left]right:{shape\\[-4pt]\ $(\lambda_3,\lambda_4)/(\nu_3,\nu_4)$\\[-4pt]weight $\mu-\mu^{(2)}$}}] at (1,5.5) {};
 \end{scope}
}
\right)
\\[-8pt]
\left.\tikz[x={(0in,0.2in)},y={(-0.2in,0in)}]\draw[transparent](0,0)--node[opaque]{$T^{(2)}$}(0,8);\right.
\left.\tikz[x={(0in,0.2in)},y={(-0.2in,0in)}]\draw[transparent](0,0)--node[opaque]{$S^{(2)}$}(0,8);\right.
\displaybreak[0]\\[-0.2in]
\tikz\draw[|->](0in,0in)-- node[right]{$\Delta_2$} (0in,-0.3in);
\displaybreak[0]\\[20pt]
\left(
\vYTd{0.2in}{}{% T^{(1)}
 {{},{},{},1,1,1,1,1}}
{
 \foreach \i/\j in {1/4,1/5,1/6,1/7,1/8}
  \draw[green,very thick] (\i,\j) rectangle +(-1,-1);
 \foreach \i/\j in {1/1,1/2,1/3}
  \draw[red,very thick] (\i,\j) rectangle +(-1,-1);
 \draw[transparent](4,0)--(4,8)
  node[inner sep=0pt,outer sep=0pt,anchor=base west] {, };
 \begin{scope}[overlay]
  \node
   [pin={[pin edge=thick,align=left]left:{shape $(\lambda_1)/\mu^{(1)}$\\[-4pt]weight $(\nu_1)$}}]
   at (0.8,0) {};
 \end{scope}
}
\right.\left.
\vYTd{0.2in}{}{% S^{(3)}
 {{},{},{},{},{\red 1},{\red 2}},
 {{},1,2,2,2},
 {1,2,3,3}}
{
 \draw[transparent](-1,0)--(-1,8);
 \foreach \i/\j in {1/1,1/2,1/3,1/4}
  \draw[green,very thick] (\i,\j) rectangle +(-1,-1); %\fill[green]
 \begin{scope}[overlay]
  \node[coordinate,pin={[pin edge={thick,black,latex-}]left:new row}] at (0.5,0) {};
  \node[left,inner sep=0pt,outer sep=0pt] at (2,0) {$S^{(2)}\left\{\vbox to17pt{}\right.$};
  \node[above,inner sep=0pt,outer sep=0pt] at (0,5)
   {$\overbrace{\hphantom{\hbox to0.4in{\hss}}}
      ^{\textstyle\text{weight $\mu^{(2)}-\mu^{(1)}$}}$};
  \node[pin={[pin edge=thick,pin distance=4.5ex]north:
             $\nu_2$ empty cells}] at (0.5,1.5) {};
  \node[right,inner sep=0pt,outer sep=0pt,pin={[pin edge=thick,pin distance=0.7in,align=left]right:{shape\\[-4pt]\ $(\lambda_2,\lambda_3,\lambda_4)/(\nu_2,\nu_3,\nu_4)$\\[-4pt]weight $\mu-\mu^{(1)}$}}] at (1.5,5.5) {};
 \end{scope}
}
\right)
\\[-8pt]
\left.\tikz[x={(0in,0.2in)},y={(-0.2in,0in)}]\draw[transparent](0,0)--node[opaque]{$T^{(1)}$}(0,8);\right.
\left.\tikz[x={(0in,0.2in)},y={(-0.2in,0in)}]\draw[transparent](0,0)--node[opaque]{$S^{(3)}$}(0,8);\right.
\displaybreak[0]\\[-0.2in]
\tikz\draw[|->](0in,0in)-- node[right]{$\Delta_1$} (0in,-0.3in);
\displaybreak[0]\\[20pt]
\left(
\vYTd{0.2in}{}{% T^{(0)}
 {}}
{
 \node at (2,4) {$\varnothing$};
 \draw[transparent](0,0)--(0,8) (4,0)--(4,8)
  node[inner sep=0pt,outer sep=0pt,anchor=base west] {, };
}
\right.\left.
\vYTd{0.2in}{}{% S^{(4)}
 {{},{},{},{},{},{\red 1},{\red 1},{\red 1}},
 {{},{},{},{},1,2},
 {{},1,2,2,2},
 {1,2,3,3}}
{
 \foreach \i/\j in {1/1,1/2,1/3,1/4,1/5}
  \draw[green,very thick]  (\i,\j) rectangle +(-1,-1);  %\fill[green]
 \begin{scope}[overlay]
  \node[coordinate,pin={[pin edge={thick,black,latex-}]left:new row}] at (0.5,0) {};
  \node[left,inner sep=0pt,outer sep=0pt] at (2.5,0) {$S^{(3)}\left\{\vbox to25.5pt{}\right.$};
  \node[above,inner sep=0pt,outer sep=0pt] at (0,6.5)
   {$\overbrace{\hphantom{\hbox to0.6in{\hss}}}
      ^{\textstyle\text{weight $\mu^{(1)}$}}$};
  \node[pin={[pin edge=thick,pin distance=1.5ex]north:
             $\nu_1$ empty cells}] at (0.5,1.5) {};
  \node[right,inner sep=0pt,outer sep=0pt,pin={[pin edge=thick,pin distance=0.5in]right:{shape $\lambda/\nu$, weight $\mu$}}] at (2,6.5) {};
 \end{scope}
}
\right)
\\[-8pt]
\left.\tikz[x={(0in,0.2in)},y={(-0.2in,0in)}]\draw[transparent](0,0)--node[opaque]{$T^{(0)}$}(0,8);\right.
\left.\tikz[x={(0in,0.2in)},y={(-0.2in,0in)}]\draw[transparent](0,0)--node[opaque]{$S^{(4)}$}(0,8);\right.
\end{gather*}
\end{Example}

To complete the proof of Theorem~\ref{The-TtoS} we require two additional Lemmas. First, using
the same notation as before:
 
\begin{Lemma}\label{Lem-Sss}
For $r=0,1,\ldots,n$ each tableau $S^{(r)}$ is semistandard.
\end{Lemma}

\begin{proof}
Again we can proceed by induction. Since it contains no cells, the empty tableau $S^{(0)}$ is
semistandard. Then by the induction hypothesis $S^{(n-r)}$ is semistandard.
However, as we have indicated in the proof of Lemma~\ref{Lem-TSshape},
$S^{(n-r+1)}$ is constructed by placing on top of $S^{(n-r)}$ a new row of length
$\lambda_r$ that contains on its left $\nu_r$ terminating row numbers $0$, followed from left to right by $\lambda_r-\nu_r$ non-zero terminating row numbers. As required for semistandardness these 
automatically appear in weakly increasing order thanks to the Horizontal Path Comparison Lemma~\ref{Lem-hp}.
 
Since $S^{(n-r)}$ also satisfies the strict vertical inequalities 
by the induction hypothesis, the only remaining condition required to establish the semistandardness of $S^{(n-r+1)}$ is the vertical condition between its newly added row and the top row of $S^{(n-r)}$ (assuming that $r<n$).
We verify this by applying the Vertical Path Comparison Lemma~\ref{Lem-vp} with its $r$ replaced 
here by $r+1$.
Phase 2 of $\delta_{r+1}$ is divided into Phases 2A and 2B with $U=U_{r,r+1}$ being the number of Phase 2A deletions (i.e.\ the number of entries $r$ in row $r+1$ of $T^{(r+1)}$).
By construction, the top row of $S^{(n-r)}$ and the new row of $S^{(n-r+1)}$ are as in the following picture 
(see the statement of Lemma~\ref{Lem-vp} for the notation $m$, $t_j$ and $t^\#_j$).
Note that the lattice permutation property of $T^{(r+1)}$ implies that entries $r$ can lie in rows $r$ and $r+1$ only, so that the number of entries $r$ in row $r$ of $T^{(r+1)}$ is $\nu_r-U$, and implies furthermore that $\nu_r-U\geq\nu_{r+1}$, or equivalently that $\nu_{r+1}+U\leq\nu_r$.

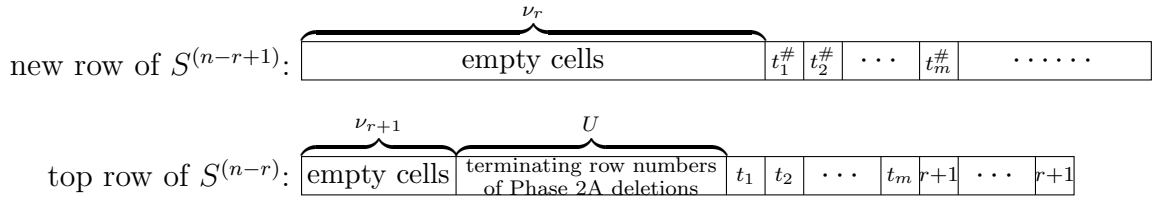
\[
\begin{tikzpicture}[x={(0in,-0.2in)},y={(0.2in,0in)}]% matrix coordinate
 \draw (0,0) rectangle (1,20);
 \node[left] at (0.5,0) {top row of $S^{(n-r)}$: };
 \node[above,inner sep=0pt,outer sep=0pt] at (0,2) {$\overbrace{\hbox to0.8in{}}^{\nu_{r+1}}$};
 \node at (0.5,2) {\small{empty cells}};
 \draw (0,4)--(1,4);
 \node[above,inner sep=0pt,outer sep=0pt] at (0,7.5) {$\overbrace{\hbox to1.4in{}}^{U}$};
 \node at (0.5,7.5) {$\begin{smallmatrix}\text{terminating row numbers} \\ \text{of Phase 2A deletions}\end{smallmatrix}$};
 \draw (0,11)--(1,11);
 \node at (0.5,11.5) {$\scriptstyle t_1$};
 \draw (0,12)--(1,12);
 \node at (0.5,12.5) {$\scriptstyle t_2$};
 \draw (0,13)--(1,13);
 \node at (0.5,14) {$\cdots$};
 \draw (0,15)--(1,15);
 \node at (0.5,15.5) {$\scriptstyle t_m$};
 \draw (0,16)--(1,16);
 \node at (0.5,16.5) {$\scriptstyle r+1$};
 \draw (0,17)--(1,17);
 \node at (0.5,18) {$\cdots$};
 \draw (0,19)--(1,19);
 \node at (0.5,19.5) {$\scriptstyle r+1$};
 \draw (-3,0) rectangle (-2,22);
 \node[left] at (-2.5,0) {new row of $S^{(n-r+1)}$: };
 \node[above,inner sep=0pt,outer sep=0pt] at (-3,6) {$\overbrace{\hbox to2.4in{}}^{\nu_{r}}$};
 \node at (-2.5,6) {\small{empty cells}};
 \draw (-3,12)--(-2,12);
 \node at (-2.5,12.5) {$\scriptstyle t^\#_1$};
 \draw (-3,13)--(-2,13);
 \node at (-2.5,13.5) {$\scriptstyle t^\#_2$};
 \draw (-3,14)--(-2,14);
 \node at (-2.5,15) {$\cdots$};
 \draw (-3,16)--(-2,16);
 \node at (-2.5,16.5) {$\scriptstyle t^\#_m$};
 \draw (-3,17)--(-2,17);
 \node at (-2.5,19.5) {$\cdots\cdots$};
\end{tikzpicture}
\]
By the Vertical Path Comparison Lemma~\ref{Lem-vp} we have $t_j>t^\#_j$ for $1\leq j\leq m$, which implies that the cell immediately above $t_j$ in $S^{(n-r+1)}$ is either empty or filled with something $\leq t^\#_j$, verifying the required strict vertical inequalities for these cells.
Moreover, since all entries in the new row are at most $r$, the inequalities for the cells containing $r+1$ and their respective immediate upper neighbours also hold.
It follows that $S^{(n-r+1)}$ is semistandard, and the induction argument then implies that
$S^{(r)}$ is semistandard for all $r=0,1,\ldots,n$, as required.
\qedhere
\end{proof}

Secondly, again using the same notation as above, we have:

\begin{Lemma}\label{Lem-Slp}
$S=S^{(n)}$ satisfies the lattice permutation property.
\end{Lemma}

\begin{proof}.
Note that the construction of $S$ implies that $\mu^{(r)}_k$, namely the length of row $k$ of the inner shape of $T^{(r)}$, equals the number of entries $k$ to be put in rows 1 through $r$ of $S$.
Hence the lattice permutation property of $S$ is equivalent to the inequalities $\mu^{(r-1)}_k\geq\mu^{(r)}_{k+1}$ for $1\leq r\leq n$ and $1\leq k\leq n-1$ and, moreover, it is sufficient to restrict to the range $1\leq k<r\leq n$ since otherwise the inequality trivially holds (since $\mu^{(r-1)}$ and $\mu^{(r)}$ have at most $r-1$ and $r$ parts respectively).
The required inequality again follows from Horizontal Path Comparison Lemma~\ref{Lem-hp}, namely from the fact that the terminating row numbers in Phases 2 and 3 in $\delta_r$ weakly increase in the order of occurrence, as follows.
Let $V_{kr}$ denote the number of the terminating row numbers $k$ generated during Phases 2 and 3 of $\delta_r$,
that is to say $V_{kr}=\mu^{(r)}_k-\mu^{(r-1)}_k$.
Then, among the $\lambda_r-\nu_r$ relevant deletions, the first $V_{1r}$ terminate in row 1, the next $V_{2r}$ terminate in row 2, and so on, continuing until the final $V_{rr}$ terminate in row $r$ (this last part is nothing but Phase 3).
Observe the change of inner shape, noting that after each deletion it always has to be the Young diagram of a partition.
This in particular implies that $V_{1r}$ cannot exceed $\mu^{(r)}_1-\mu^{(r)}_2$ since, if the inner shape $D(\mu^{(r)})$ loses this 
number of cells from row~1 with the length of row~2 intact, then it will not have a corner in row 1 any more; so we have $\mu^{(r-1)}_1=\mu^{(r)}_1-V_{1r}\geq\mu^{(r)}_2$.
Next, before the deletions terminating in row 2 start, the lengths of rows 2 and 3 of the inner shape are still $\mu^{(r)}_2$ and $\mu^{(r)}_3$ respectively, and by the same argument we have $\mu^{(r-1)}_2=\mu^{(r)}_2-V_{2r}\geq\mu^{(r)}_3$.
Proceeding in this manner, we obtain $\mu^{(r-1)}_k\geq\mu^{(r)}_{k+1}$ for all $1\leq k\leq r-1$,
and the lattice permutation property of $S$ is ensured.
\qedhere
\end{proof}

Finally, we have
\begin{proof}[Proof of Theorem~\ref{The-TtoS}]
Taken together Lemmas~\ref{Lem-TSshape},~\ref{Lem-Sss} and~\ref{Lem-Slp} ensure that the tableau $S=S^{(n)}$ 
is of shape $\lambda/\nu$ and weight $\mu$, and that it satisfies both the semistandard and lattice permutation
properties. It is therefore an LR tableau and, more precisely, $S\in{\cal LR}(\lambda/\nu,\mu)$, as required.
\qedhere
\end{proof}

%%%%%%%%%%%%%%%%%%%%%%%%%

Remarkably, the map we have constructed from $T=T^{(n)}$ to $S=S^{(n)}$ is bijective.
The following lemma clarifies the basis of the proof of injectivity.

\begin{Lemma}\label{Lem-TSrcvr}
Let $n$, $\lambda$, $\mu$, $\nu$, $T$ and $S$ be as in Theorem~\ref{The-TtoS}, and let $(T^{(n)},S^{(0)})$, $(T^{(n-1)},S^{(1)})$, \dots,\ $(T^{(0)},S^{(n)})$ be as in Definition~\ref{Def-TtoS}.
Then, for each $1\leq r\leq n$, the pair $(T^{(r)},S^{(n-r)})$ can be recovered from the sole knowledge of $(T^{(r-1)},S^{(n-r+1)})$ in the following manner.
\begin{itemize}
\item
$S^{(n-r)}$ consists of the lowermost $n-r$ rows of $S^{(n-r+1)}$.
\item
Let $V_{kr}$ denote the number of entries $k$ in the top row of $S^{(n-r+1)}$ for $0\leq k\leq r$, including entries $0$ (in fact $V_{0r}$ equals $\nu_r$, but we shall set aside this fact for the moment).
Then $T^{(r)}$ can be obtained from $T^{(r-1)}$ by applying Sagan and Stanley's internal insertion \cite{SS} from row $k$ precisely $V_{kr}$ times with $k$ running from 1 up to $r$ in this order, and then adjoining $V_{0r}$ entries $r$ to the right of the existing entries in row $r$.
\end{itemize}
\end{Lemma}

\begin{proof}[Proof]
The assertion for $S^{(n-r)}$ is immediate from the construction.
Now, again by construction, the composition of the top row of $S^{(n-r+1)}$ tells that the full $r$-deletion applied $T^{(r)}$, yielding $T^{(r-1)}$, consists of $V_{0r}$ type (i) deletions producing entries $0$ in the top row of $S^{(n-r+1)}$, then $V_{kr}$ type (ii) deletions terminating in row $k$ and producing that many entries $k$ in the top row of $S^{(n-r+1)}$, with $k$ running from $1$ up to $r-1$ in this order, and finally $V_{rr}$ type (iii) deletions terminating in row $r$ and producing that many entries $r$ in the top row of $S^{(n-r+1)}$.
This ensures that it is possible to apply Sagan and Stanley's internal insertion to $T^{(r-1)}$ first from each row $k$ precisely $V_{kr}$ times with $k$ running from $r$ down to $1$ in this order (note that the first $V_{rr}$ insertions from row $r$ simply create that many empty cells in row $r$), all reaching row $r$, and finally to adjoin $V_{0r}$ entries $r$ to the right of whatever already exists in row $r$, and that the result of these operations coincides with $T^{(r)}$.
\end{proof}

Now the bijectivity is made explicit in the following:
\begin{Theorem}\label{The-bij}
Let $n$, $\lambda$, $\mu$ and $\nu$ be as in Theorem~\ref{The-TtoS},
and let $\rho_{\mu,\nu}^\lambda$ be the map from $\mathcal{LR}(\lambda/\mu,\nu)$ to 
$\mathcal{LR}(\lambda/\nu,\mu)$ that takes $T\in\mathcal{LR}(\lambda/\mu,\nu)$ 
to $S\in\mathcal{LR}(\lambda/\nu,\mu)$ by way of
\begin{equation}
T\longmapsto(T,\varnothing)\overset{\Delta^{(n)}}{\longmapsto}(\varnothing,S)\longmapsto S \,.
\end{equation}
Then $\rho_{\mu,\nu}^\lambda$ is a bijection from $\mathcal{LR}(\lambda/\mu,\nu)$ to $\mathcal{LR}(\lambda/\nu,\mu)$.
\end{Theorem}

\begin{proof}
Let $T\in\mathcal{LR}(\lambda/\mu,\nu)$ and let $(T^{(r)},S^{(n-r)})$, $r=n$, $n-1$, \dots, $0$, 
be as in Definition~\ref{Def-TtoS}.
Now start with the pair $(\varnothing,S)=(T^{(0)},S^{(n)})$ and apply Lemma~\ref{Lem-TSrcvr} with $r=n$, $n-1$, \dots,\ $0$ in this order to successively recover $(T^{(1)},S^{(n-1)})$, \dots,\ $(T^{(n)},S^{(0)})=(T,\varnothing)$.
Hence one can tell what $T$ was by the sole knowledge of $\rho_{\mu,\nu}^{\lambda}T$.
This proves that the map $\rho_{\mu,\nu}^\lambda$ 
from $\mathcal{LR}(\lambda/\mu,\nu)$ to $\mathcal{LR}(\lambda/\nu,\mu)$ is injective, 
and hence 
we have the inequality $|\mathcal{LR}(\lambda/\mu,\nu)|\leq|\mathcal{LR}(\lambda/\nu,\mu)|$.
Now a special feature of our situation is that we also have a map $\rho_{\nu,\mu}^\lambda$ going from $\mathcal{LR}(\lambda/\nu,\mu)$ to $\mathcal{LR}(\lambda/\mu,\nu)$ defined in the same manner but with the roles of $\mu$ and $\nu$ interchanged, since the construction works for any shape and weight (the sizes have to match but that condition is symmetric in $\mu$ and $\nu$), and by applying the same argument to $\rho_{\nu,\mu}^\lambda$ 
we also deduce $|\mathcal{LR}(\lambda/\nu,\mu)|\leq|\mathcal{LR}(\lambda/\mu,\nu)|$.
So these two sets have the same cardinality, and it is finite.
So any injection from one to the other is also a surjection.
Hence the map $\rho_{\mu,\nu}^\lambda$ in the Theorem is a bijection.
\qedhere
\end{proof}

\begin{Remark}\label{Rem-insertion}
With more effort, it can be proved that the above recovering procedure based on the internal insertion and adjunction described in Lemma~\ref{Lem-TSrcvr} is applicable to any $S\in\mathcal{LR}(\lambda/\nu,\mu)$ and that it yields an element of $\mathcal{LR}(\lambda/\mu,\nu)$, without assuming that $S$ comes from a $T\in\mathcal{LR}(\lambda/\mu,\nu)$ via the map $\rho_{\mu,\nu}^\lambda$, for an illustrative example see~\cite{Az1,Az2}. 
This allows one to prove the bijectivity without resorting to the finiteness of cardinalities.
This strategy will be adopted in the language of LR hives in Section~\ref{Sec-hive-path-addition}.
\end{Remark}

In the next section it will be proved that the collection of maps $\rho_{\mu,\nu}^\lambda$ has an involutory nature
in the sense that the composition of maps $\rho_{\nu,\mu}^\lambda\,\rho_{\mu,\nu}^\lambda$ is the identity.
This means that the inverse map from $\mathcal{LR}(\lambda/\nu,\mu)$ to
$\mathcal{LR}(\lambda/\mu,\nu)$ mentioned in the previous paragraph, based on internal insertion and adjunction, actually coincides with our original map $\rho_{\nu,\mu}^\lambda$ based on deletions.
It follows that the insertion procedure of Lemma~\ref{Lem-TSrcvr} provides another method to obtain the map in the statement of Theorem~\ref{The-TtoS}.

\section{Tableau proof of involutive nature of bijection}\label{Sec-tab-inv}
The aim of this section is to give a tableau-based proof of the involutive property of the collection of bijections in the statement of Theorem~\ref{The-bij}. 
To be more precise, let $\rho^{(n)}$ denote the map
\[
      \rho^{(n)}: \mathcal{LR}^{(n)} \to \mathcal{LR}^{(n)}
\]
defined by way of
\[
     \rho^{(n)}: T\longmapsto(T,\varnothing)\overset{\Delta^{(n)}}{\longmapsto}(\varnothing,S)\longmapsto S.
\]
It follows that for $\lambda$, $\mu$ and $\nu$ satisfying \thetag{*} 
\[
   \rho^{(n)}(\mathcal{LR}(\lambda/\mu,\nu))=\mathcal{LR}(\lambda/\nu,\mu) \quad\hbox{with}\quad
		  \rho^{(n)}\big|_{\mathcal{LR}(\lambda/\mu,\nu)}  = \rho^\lambda_{\mu\nu}\,
\]
where $\rho^\lambda_{\mu\nu}$ is as defined above in Theorem~\ref{The-bij}.

\begin{Theorem}\label{The-invol-tab-main}
The bijection $\rho^{(n)}$ is an involution.
\end{Theorem}

The rest of this section constitutes a proof of Theorem~\ref{The-invol-tab-main} 
fulfilling in complete generality the strategy developed in~\cite{Az2}. 

The case $n=1$ is easy, since if $\ell(\lambda)\leq 1$, then both $\mathcal{LR}(\lambda/\mu,\nu)$ and $\mathcal{LR}(\lambda/\nu,\mu)$ consist of single elements.
Having this as the base case, we prove Theorem \ref{The-invol-tab-main} by induction on $n$.
Let $n\geq2$ and $T\in\mathcal{LR}(\lambda/\mu,\nu)\subseteq\mathcal{LR}^{(n)}$, and set $\Delta^{(n)}(T,\varnothing)=(\varnothing,S)$, namely $S=\rho^{(n)}T$. 
It is convenient to let $S^-$ denote the tableau obtained from $S$ by removing 
the $n$th row. The inductive nature of the construction of $\rho^{(n)}$ is such that  
$S^-$ can be regarded as the output of $\rho^{(n-1)}$ applied to $\delta_nT\in\mathcal{LR}^{(n-1)}$
since $S^-$ is just the rows $1,2,\ldots,n-1$ of $S$.
By induction, we can assume that $\rho^{(n-1)}(S^-)=\delta_nT$.
Our goal is to show that $\rho^{(n)}(S)=T$.
Since we do not do this yet, we write $T'$ for $\rho^{(n)}(S)$ for the moment.
Suppose we can prove the following

\begin{Proposition}\label{Prop-invol-tab-erase-row}
Let $n\geq2$ and $S\in\mathcal{LR}^{(n)}$, and set $T'=\rho^{(n)}(S)$.
Set $T''=\rho^{(n-1)}(S^-)$.
Then we have $\delta_{n}T'=T''$.
Namely the following diagram commutes:
\begin{equation}\label{Eq-comm-diagram}
\begin{tikzpicture}
\node(tl) at (0,3){$\mathstrut\mathcal{LR}^{(n)}$};
\node(tr) at (4,3){$\mathstrut\mathcal{LR}^{(n)}$};
\node(bl) at (0,0){$\mathstrut\mathcal{LR}^{(n-1)}$};
\node(br) at (4,0){$\mathstrut\mathcal{LR}^{(n-1)}$};
\draw[->](tl)--node[above]{$\rho^{(n)}$}(tr);
\draw[->](bl)--node[below]{$\rho^{(n-1)}$}(br);
\draw[->](tl)--node[sloped,below]{$\begin{matrix}\text{removing}\\\text{the $n$th row}\end{matrix}$}(bl);
\draw[->](tr)--node[right]{$\delta_n$}(br);
\node[rotate=-45] at (0.4,2.7){$\mathstrut\ni{}$};
\node(tl1) at (0.7,2.4){$\mathstrut S$};
\node[rotate=45] at (0.35,0.42){$\mathstrut\ni{}$};
\node(bl1) at (0.7,0.65){$\mathstrut S^-$};
\node[rotate=45] at (3.4,2.7){$\mathstrut\in{}$};
\node(tr1) at (3.1,2.4){$\mathstrut T'$};
\node[rotate=-45] at (2.95,0.55){$\mathstrut\in{}$};
\node(br1) at (3.1,1.4){$\mathstrut\delta_nT'$.};
\node(br2) at (2.0,0.65){$\mathstrut T''$};
\draw[|->](tl1)--(tr1); \draw[|->](tl1)--(bl1);
\draw[|->](bl1)--(br2); \draw[|->](tr1)--(br1);
\draw[double equal sign distance](br1)--(br2);
\end{tikzpicture}
\end{equation}
\end{Proposition}
Let $\lambda$, $\nu$ and $\mu$ be the outer shape, the inner shape and the weight of $S$ respectively, namely $S\in\mathcal{LR}(\lambda/\nu,\mu)$.
The validity of Proposition~\ref{Prop-invol-tab-erase-row} means that both $T$ and $T'$ are transformed into the same tableau $T''=\rho^{(n-1)}(S^-)$ by $\delta_n$. Moreover, $T$ and $T'$ both have shape $\lambda/\mu$ and weight $\nu$.
Hence, whether one starts with $(T,\varnothing)$ or $(T',\varnothing)$, the one-row semistandard tableau created by $\Delta_n$ on the right-hand side consists of $\nu_r$ empty cells created in Phase 1 and, if we let $\mu^{(n-1)}$ denote the inner shape of $T''$, exactly $\mu_k-\mu^{(n-1)}_k$ entries $k$ with $k$ running from $1$ up to $n-1$ in this order, all created in Phase 2, and finally $\mu_n$ entries $n$ created in Phase 3; namely, if one denotes this one-row tableau by $R$, we have $\Delta_n(T,\varnothing)=\Delta_n(T',\varnothing)=(T'',R)$. By Lemma~\ref{Lem-TSrcvr} we have $T=T'$.

So the proof of Theorem~\ref{The-invol-tab-main} is now reduced to that of Proposition~\ref{Prop-invol-tab-erase-row}.

We first deal with the special case where row $n$ of $S$ consists of exactly one cell, but afterwards we will argue that the general case can be reduced to this special case.
We start with some particularly easy special cases.

\begin{Lemma}\label{Lem-invol-tab-erase-one-emp}
If $\lambda_n=1$ and the cell $(n,1)$ is empty in $S$, then Proposition~\ref{Prop-invol-tab-erase-row} holds.
\end{Lemma}
\begin{proof}
Consider applying $\Delta^{(n)}$ to $S$ to produce $T'$, as in Definition~\ref{Def-TtoS}.
In applying $\Delta_n$, the only deletion $\delta_{n,1}$ is of type (iii), and produces a single cell containing $n$ on the right-hand component, where the word right (and left also, which will appear presently) 
is chosen in accordance with the computation of $\rho^{(n)}(S)$, i.e.\ the pairs $(S,\varnothing)$ and $(\varnothing,T')$.
Hence the left-hand component at this stage, $\delta_nS$, coincides with $S^-$, and applying $\delta_{n-1}$, \dots, $\delta_1$ to it produces the same rows on the right-hand component as does applying $\Delta^{(n-1)}$ to $(S^-,\varnothing)$.
This means that $T'$ consists of $T''=\rho^{(n-1)}(S^-)$ plus one cell $(n,1)$ containing $n$.
Applying $\delta_n$ to $T'$ incurs a Phase 1 deletion only, which removes the cell $(n,1)$ and $T''$ is what is left.
Hence Proposition~\ref{Prop-invol-tab-erase-row} holds in this case.
\end{proof}

\begin{Lemma}\label{Lem-invol-tab-erase-one-sp}

If $\lambda_n=1$ and $S(n,1)=n$, then Proposition~\ref{Prop-invol-tab-erase-row} holds.
\end{Lemma}
\begin{proof}
Consider applying $\Delta^{(n)}$ to $S$ to produce $T'$, as in Definition~\ref{Def-TtoS}.
In applying $\Delta_n$, only Phase 1 arises, removing the cell $(n,1)$ and producing a single empty cell on the right-hand component.
Aside from this difference, the situation is very much similar to the preceding case, and the validity of Proposition~\ref{Prop-invol-tab-erase-row} in this case follows by the same argument.
\end{proof}

We treat one more case separately, which not only is a special case but also serves as the terminating step of an iteration occurring in a more general case to be discussed next.

\begin{Lemma}\label{Lem-invol-tab-erase-one-fin}
If $\lambda_n=1$ and $S(n,1)=n-1$, then Proposition~\ref{Prop-invol-tab-erase-row} holds.
\end{Lemma}
\begin{proof}
In once again applying $\Delta^{(n)}$ to $S$ to produce $T'$ as in Definition~\ref{Def-TtoS},
the operation $\delta_n$ applied to $S$ consists of a single operation $\delta_{n,1}$.
It brings the bottom entry $n-1$ to row $n-1$, where it replaces the rightmost entry $<n-1$ in that row, which we call $y$ and, unless $y=0$, incurs further bumping.
Let $P_0$ denote this deletion path, and $l$ its terminating row number.
In applying $\delta_{n-1}$ to $\delta_nS$, after the removal of the entries $n-1$ in Phase 1, the Phase 2 and 3 deletions start from the cells in row $n-1$ strictly to the left of $P_0[n-1]$ (see Definition~\ref{Def-dp} (ii) for this notation).
Let $P_1$, $P_2$, \dots, $P_m$ ($m=\lambda_{n-1}-\mu_{n-1}$) denote these deletion paths, and $k_1$, $k_2$, \dots, $k_m$ their terminating row numbers.
The bottom two rows of $T'$ produced by these operations are as follows.

%%%%%%%%%%%%%%%%%%%%%%%%%%%%%%%%%%%%%%
\[%%%%%\label{Eq-invol-tab-erase-one-fin}
\vcenter{\hbox{

}}
\]

Now consider $S^-$ the tableau obtained fom $S$ be deleting its $n$th row, and apply $\Delta^{(n)}$ to produce $T''$, again as in Definition~\ref{Def-TtoS}. In applying $\delta_{n-1}$, after
the removal of entries $n-1$ in Phase 1, the first Phase 2 or 3 deletion starts from the cell $P_0[n-1]$ containing $y$, and since the rows $n-2$ and above, if any, are the same as those of $S$, this deletion proceeds along the path $P_0^-:=P_0-\{(n,1)\}$ and terminates in row $l$.
The remaining Phase 2 and 3 deletions start from the cells in row $n-1$ to the left of $P_0[n-1]$ containing entries $x_1$, $x_2$, \dots, $x_m$, and since the portion of the current tableau to the left of $P_0^-$ is the same as that portion of $\delta_nS$, these deletions proceed along the paths $P_1$, $P_2$, \dots, $P_m$, and terminate in rows $k_1$, $k_2$, \dots, $k_m$ respectively.
The Horizontal Path Comparison Lemma~\ref{Lem-hp} 
applied to the operation $\delta_{n-1}$ for $S^-$ assures that $l\leq k_1$ ($\leq\cdots\leq k_m$).
The bottom row of $T''$ produced by these operations is as follows.

%%%%%%%%%%%%%%%%%%%%%
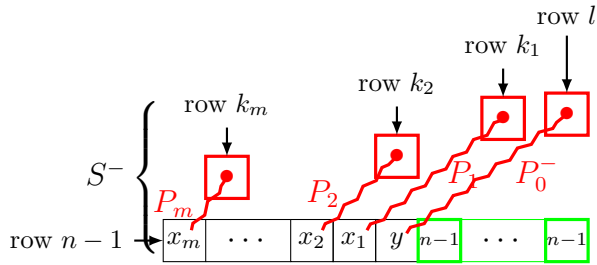
\[
\begin{tikzpicture}[x={(0in,-0.22in)},y={(0.22in,0in)}]
 \begin{scope}[shift={(0,0)}] % S^-
  \foreach \j in {1,4,5,6} \draw(1,\j)rectangle+(-1,-1);
  \draw(0,1)--(0,3) (1,1)--(1,3);
  \foreach \j in {7,10} \draw[green,very thick](1,\j)rectangle+(-1,-1);
  \draw[green](0,7)--(0,9) (1,7)--(1,9);
  \foreach \name/\j/\x in {xm/0.5/x_m,x2/3.5/x_2,x1/4.5/x_1,y/5.5/y}
   \node[inner sep=0.5pt,outer sep=0.5pt,font=\footnotesize] (\name) at (0.5,\j) {$\x$};
  \foreach \j in {2,8} \node at (0.5,\j) {$\cdots$};
  \foreach \j in {6.5,9.5} \node[font=\footnotesize] at (0.5,\j) {$\sc n-1$};
  \foreach \name/\k/\l/\row/\d/\o/\v/\p in {
   x1/-2.4/8/$k_1$/0.15in/right/0pt/$P_1$,
   x2/-1.5/5.5/$k_2$/0.15in/left/3pt/$P_2$,
   xm/-1/1.5/$k_m$/0.15in/left/1pt/$P_m$,
   y/-2.5/9.5/$l$/0.3in/right/6pt/$P_0^-$}
  {
   \draw[red,very thick](\name) decorate[decoration={zigzag,amplitude=1pt}]
    {--node[\o=\v]{\p} (\k,\l) node{$\bullet$}};
   \draw[red,very thick] (\k,\l) +(-0.5,-0.5)rectangle+(0.5,0.5)
    +(-0.5,0) node[coordinate,pin={[pin edge={thick,black,latex-},black,pin distance=\d,font=\footnotesize]
    above:row \row}] {};
  }
  \node[coordinate,pin={[pin edge={thick,black,latex-},pin distance=0.15in,font=\footnotesize]
   left:row $n-1$}] at (0.5,0) {};
  \node[left,inner sep=0pt,outer sep=0pt] at (-1,0)
   {$S^-\left\{\vbox to0.44in{}\right.$};
 \end{scope}
 \begin{scope}[shift={(2,11)}] % T''{}^{(1)}
  \foreach \j in {1,4}
   \draw(1,\j)rectangle+(-1,-1) node[green] at +(-0.5,-0.5){$\bullet$};
  \draw(0,1)--(0,3) (1,1)--(1,3) (0.5,2) node{$\cdots$};
  \foreach \j/\k in {5/$l$,6/$k_1$,7/$k_2$,10/$k_m$}
   \draw(1,\j)rectangle+(-1,-1) node[red,font=\footnotesize] at +(-0.5,-0.5){\k};
  \draw(0,7)--(0,9) (1,7)--(1,9) (0.5,8) node{$\cdots$};
  \node[left,inner sep=0pt,outer sep=0pt] at (0.5,0)
   {$T''{}^{(1)}\left\{\vbox to0.11in{}\right.$};
 \end{scope}
\end{tikzpicture}
\]

At this stage we have $\delta_{n-1}S^-=\delta_{n-1}\delta_nS$, so that the upper $n-2$ rows of $T''$ and $T'$, produced by applying $\delta_1\cdots\delta_{n-2}$ to them respectively, are the same.
The difference between $T'$ and $T''$ is exactly that will be incurred by applying $\delta_n$ to $T'$, which consists of a single operation $\delta_{n,1}$, pushing the bottom entry $l$ into the rightmost empty cell of row $n-1$ and terminating there.
Thus Proposition~\ref{Prop-invol-tab-erase-row} is verified in this case.\qedhere
\end{proof}

Using Lemma \ref{Lem-invol-tab-erase-one-fin} as a base case, we want to prove the following by induction on $n-S(n,1)$.

\begin{Lemma}\label{Lem-invol-tab-erase-one-gen}
If $\lambda_n=1$, the cell $(n,1)$ of $S$ is nonempty and $S(n,1)\leq n-1$, then Proposition~\ref{Prop-invol-tab-erase-row} holds.
\end{Lemma}

We take this opportunity to introduce some notation which helps compactify certain expressions in the proofs that follow.

\begin{Definition}\label{Def-compactifying-notation}
(1)
For any tableau $T$ and a natural number $k$, we let $^{k\{}T$ (resp.\ $_{k\{}T$) denote the tableau consisting of the topmost (resp. bottommost) $k$ rows of $T$.
Occasionally we apply the same notation for a deletion path $P$ as well.

(2)
If $T\in\mathcal{LR}(\lambda/\mu,\nu)$ with $\ell(\lambda)\leq r$, in which case we have defined $\delta_r=\delta_{r,1}\cdots\delta_{r,\lambda_r}$, we may also write $\delta_r^i$ (using $i$ as a superscript, not a subscript) for the $i$th piece of this sequence of operations in the order of application, namely \[
\delta_r^i=\delta_{r,\lambda_r+1-i},
\qquad\text{so that}\qquad\delta_r=\delta_r^{\lambda_r}\cdots\delta_r^1.
\]
It may be noted that, in the construction of the partner tableau, the terminating row number of the operation $\delta_r^i$ is to be placed in the cell $(r,i)$.

(3)
If $T$ is a tableau of shape $\lambda/\mu$, $s\in\mathbb{N}$ and if $(r,c)$ is not a cell of $D(\lambda/\mu)$ but if $D(\lambda/\mu)\cup\{(r,c)\}$ is another skew Young diagram, then we write $T\cup_{(r,c)}\boxed{s}$ for the tableau obtained from $T$ by adjoining a cell $(r,c)$ containing the entry $s$.

(4)
Suppose that a letter $e\in\mathbb{N}$ is given.
If $T$ is a LR tableau of shape $\lambda/\mu$ with $\ell(\lambda)\leq r-1$ such that $T\cup_{(r,1)}\boxed{e}$ is again an LR tableau, then we let $\hat\delta_{r,1}T=\hat\delta_{r,1}^{\tikz\node[draw,inner sep=2pt]{$\sc e$};}T$ denote the LR tableau 
$\delta_{r,1}(T\cup_{(r,1)}\boxed{e})$. 
The terminating row number of the operation $\hat\delta_{r,1}^{\tikz\node[draw,inner sep=2pt]{$\sc e$};}$ applied to $T$ refers to that of the deletion $\delta_{r,1}$ applied to $T\cup_{(r,1)}\boxed{e}$.
\end{Definition}
\medskip

The following provides an essential inductive step to prove Lemma~\ref{Lem-invol-tab-erase-one-gen} by way of showing that delaying, so to speak, the deletion from the bottom cell of $S$ (please refer to the statement of Proposition~\ref{Prop-invol-tab-erase-one-ind} for the exact meaning of this) retains the result of deletion on $S$ itself but corresponds to a single bumping step on the part of its partner tableau.

\begin{Proposition}\label{Prop-invol-tab-erase-one-ind}~[Delaying Bottom Cell Deletion]
Assume that $n\geq3$, $\lambda_n=1$, the cell $(n,1)$ of $S$ is nonempty and that $S(n,1)\leq n-2$.
Let $\hat\delta_{n,1}=\hat\delta_{n,1}^{\tikz\node[draw,inner sep=2pt]{$\sc e$};}$ and $\hat\delta_{n-1,1}=\hat\delta_{n-1,1}^{\tikz\node[draw,inner sep=2pt]{$\sc e$};}$ be as defined in part (4) of Definition~\ref{Def-compactifying-notation} with $e=S(n,1)$.
Note that we have $\delta_nS=\delta_{n,1}S=\hat\delta_{n,1}S^-$.
\begin{enumerate}
\item[(1)]
We have $\delta_{n-1}\hat\delta_{n,1}S^-=\hat\delta_{n-1,1}\delta_{n-1}S^-$.
\item[(2)]
Let $l'$ be the terminating row number of the operation $\hat\delta_{n,1}$ applied to $S^-$.
Moreover, let $R'=(k'_1,\dots,k'_{\lambda_{n-1}})$ and $R''=(k''_1,\dots,k''_{\lambda_{n-1}})$ be the sequences of terminating row numbers produced by the operation $\delta_{n-1}$ applied to $\hat\delta_{n,1}S^-$ ($=\delta_{n,1}S$) and $S^-$ respectively (which will constitute the rows $n-1$ of $T'$ and $T''$ respectively).
Furthermore, let $l''$ be the rightmost entry in $R'$ strictly smaller than $l'$.
Then we have $l''>0$, and $R''$ coincides with what is obtained from $R'$ by replacing this entry by $l'$.
Moreover, $l''$ equals the terminating row number of the operation $\hat\delta_{n-1,1}$ applied to $\delta_{n-1}S^-$.
\end{enumerate}
\end{Proposition}

In this section, we grant the truth of Proposition~\ref{Prop-invol-tab-erase-one-ind} and continue with the proof of Lemma~\ref{Lem-invol-tab-erase-one-gen}, then Proposition~\ref{Prop-invol-tab-erase-row}, which yields Theorem~\ref{The-invol-tab-main} as we saw after the statament of Proposition~\ref{Prop-invol-tab-erase-row}.
The proof of Proposition~\ref{Prop-invol-tab-erase-one-ind} will be given in Section~\ref{Sec-prf-delaying}.

\begin{proof}[Proof of Lemma~\ref{Lem-invol-tab-erase-one-gen}]
Set $e=S(n,1)$ as in the statement of Proposition~\ref{Prop-invol-tab-erase-one-ind}.
If $n-e=1$, then Lemma~\ref{Lem-invol-tab-erase-one-fin} readily shows the validity of Proposition~\ref{Prop-invol-tab-erase-row} in this case. So let us assume that $n-e\geq2$. In this case we can apply Proposition~\ref{Prop-invol-tab-erase-one-ind}.
What it says can be read as follows.
If one applies $\rho^{(n)}$ to $S$, then $\delta_{n,1}$ yields the terminating row number $l'$ and then $\delta_{n-1}$ yields the row $R'$ of terminating row numbers; these form ${}_{2\{}T'$, and the remaining operations yield 
$^{n-2\{}T'$, where the notation of part (1) of Definition~\ref{Def-compactifying-notation} has been used.
Note that this whole procedure can also be described as applying $\delta_1\cdots\delta_{n-2}\cdot\delta_{n-1}\cdot\hat\delta_{n,1}$ to $S^-$.
We compare this with applying the following operations to $S$:  first temporarily remove the entry $e$ at $(n,1)$, apply $\delta_{n-1}$, and then adjoin $e$ back to $(n-1,1)$ and apply $\delta_1\cdots\delta_{n-2}\delta_{n-1}$; in other words applying $\delta_1\cdots\delta_{n-2}\cdot\hat\delta_{n-1,1}\cdot\delta_{n-1}$ to $S^-$.
Then $\delta_{n-1}$ yields the row $R''$ of terminating row numbers, $\hat\delta_{n-1,1}$ yields $l''$ by part (2) of Proposition~\ref{Prop-invol-tab-erase-one-ind}, and since $\hat\delta_{n-1,1}\delta_{n-1}S^-=\delta_{n-1}\hat\delta_{n,1}S^-$ by part (1) of Proposition~\ref{Prop-invol-tab-erase-one-ind}, the operations $\delta_1\cdots\delta_{n-2}$ again yield $^{n-2\{}T'$.
This is illustrated in the middle picture of (\ref{Eq-terminating-row-tableaux}) below; note that this three part middle picture is not meant to be a tableau, although the adjunction of $l''$ to $^{n-2\{}T'$ at $(n-1,1)$ is in fact an LR tableau (this follows easily from the LR property of $T'$ and the fact that $l''$ was one of the entries in row $n-1$ of $T'$), which we temporarily denote by $T^*$.

%%%%%%%%%%%%%%%%%%%
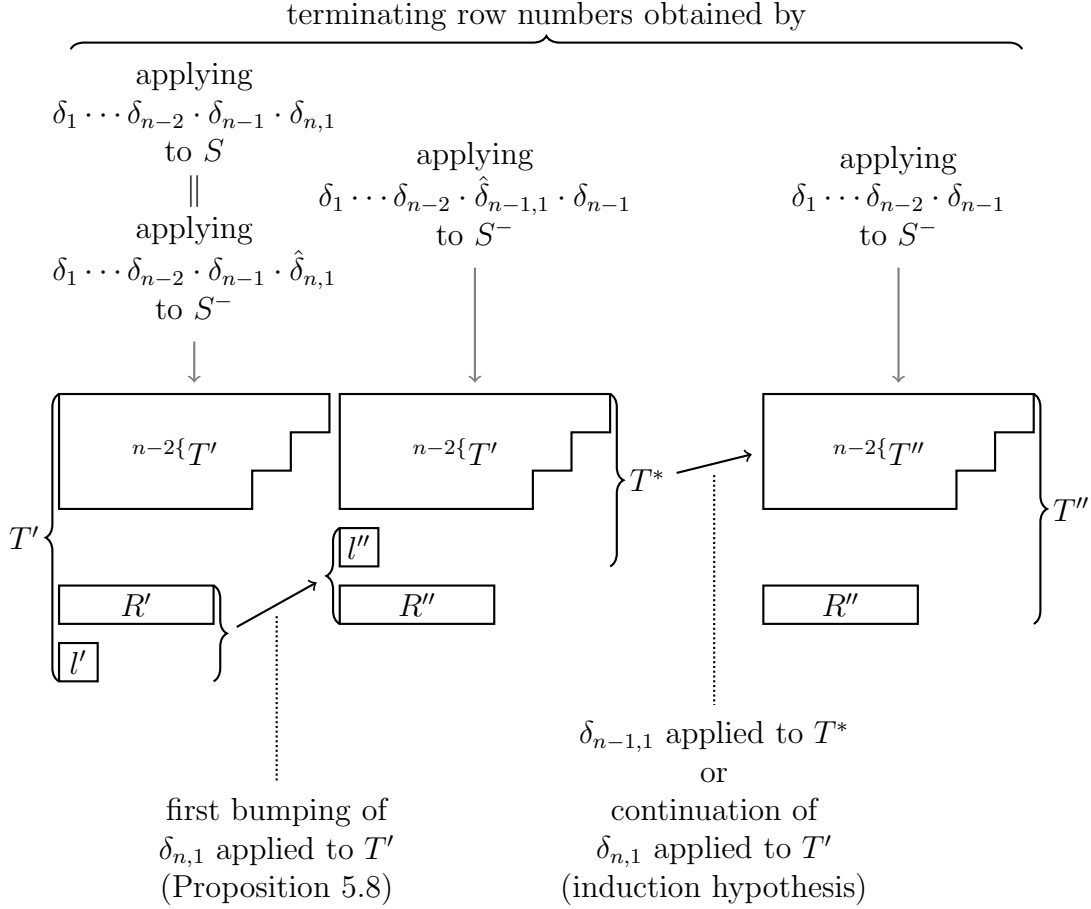
\begin{equation}\label{Eq-terminating-row-tableaux}

\end{equation}

%%%%%%%%%%%%%%%%%%%

Moreover (2) also says that, if one applies $\delta_{n,1}$ to $T'$, then $l'$ changes row $n-1$ from $R'$ to $R''$, bumping $l''$.
By the definition of deletion operation, this bumped element $l''$ continues to change $^{n-2\{}T'$ exactly as it would do if one applied the operation $\delta_{n-1,1}$ to $T^*$; hence what we are trying to show is that $_{1\{}T''=R''$ and that $^{n-2\{}T''=\delta_{n-1,1}T^*$.

That $_{1\{}T''=R''$ is clear from the definitions.
On the other hand, if we set $S^*=(\delta_{n-1}S^-)\cup_{(n-1,1)}\boxed{e}$, then $S^*$ is also an LR tableau since $S^*(n-2,1)$ equals either $S(n-2,1)$ or $S(n-1,1)$ and the weight of $^{n-2\{}S^*$ is the same as that of $^{n-1\{}S$, and the description above shows that $\rho^{(n-1)}S^*=T^*$.
Since $S^*$ contains $e$, which is $\leq n-2$, in the cell $(n-1,1)$, we can apply the induction hypothesis to $S^*$, which shows that $\delta_{n-1,1}\rho^{(n-1)}S^*=\rho^{(n-2)}({}^{n-2\{}(S^*))$, and since $^{n-2\{}(S^*)=\delta_{n-1}S^-$, we have $\rho^{(n-2)}({}^{n-2\{}(S^*))=\rho^{(n-2)}\delta_{n-1}S^-={}^{n-2\{}(\rho^{(n-1)}S^-)={}^{n-2\{}T''$, as desired.
\qedhere

\end{proof}

\begin{Example}\label{Ex-delaying-to-erase-row}

Before proceeding with the proof of Proposition~\ref{Prop-invol-tab-erase-row}, we explain in the case where $n=5$, $\lambda=(9,9,6,4,1)$, $\mu=(7,5,3)$ and $S=\vcenter{\hbox{\begin{tikzpicture}[x={(0in,-0.15in)},y={(0.15in,0in)}]
 \foreach \j in {1,...,7} \draw (0,\j-1) rectangle (1,\j);
 \foreach \j/\x in {8/1,9/1} \node[font=\footnotesize] at (0.5,\j-0.5) {$\x$};
 \foreach \j in {1,...,5} \draw (1,\j-1) rectangle (2,\j);
 \foreach \j/\x in {6/1,7/1,8/2,9/2}
  \node[font=\footnotesize] at (1.5,\j-0.5) {$\x$};
 \foreach \j in {1,2,3} \draw (2,\j-1) rectangle (3,\j);
 \foreach \j/\x in {4/1,5/1,6/2}
  \node[font=\footnotesize] at (2.5,\j-0.5) {$\x$};
 \foreach \j/\x in {1/1,2/2,3/2,4/3}
  \node[font=\footnotesize] at (3.5,\j-0.5) {$\x$};
 \node[font=\footnotesize] at (4.5,0.5) {$3$};
\end{tikzpicture}}}$ how the conclusion of Lemma~\ref{Lem-invol-tab-erase-one-gen}, $\delta_5(\rho^{(5)}S)=\rho^{(4)}S^-$, where $S^-=\vcenter{\hbox{\begin{tikzpicture}[x={(0in,-0.15in)},y={(0.15in,0in)}]
 \foreach \j in {1,...,7} \draw (0,\j-1) rectangle (1,\j);
 \foreach \j/\x in {8/1,9/1} \node[font=\footnotesize] at (0.5,\j-0.5) {$\x$};
 \foreach \j in {1,...,5} \draw (1,\j-1) rectangle (2,\j);
 \foreach \j/\x in {6/1,7/1,8/2,9/2}
  \node[font=\footnotesize] at (1.5,\j-0.5) {$\x$};
 \foreach \j in {1,2,3} \draw (2,\j-1) rectangle (3,\j);
 \foreach \j/\x in {4/1,5/1,6/2}
  \node[font=\footnotesize] at (2.5,\j-0.5) {$\x$};
 \foreach \j/\x in {1/1,2/2,3/2,4/3}
  \node[font=\footnotesize] at (3.5,\j-0.5) {$\x$};
\end{tikzpicture}}}$ (so that $S=S^-\cup_{(5,1)}\boxed{3}$), is reached by first applying Proposition~\ref{Prop-invol-tab-erase-one-ind} $n-S(n,1)$ times and then finishing by Lemma~\ref{Lem-invol-tab-erase-one-fin}.

The computation of $\rho^{(5)}S
{}=T'
$ (resp.\ $\rho^{(4)}S^-
{}=T''
$) consists of transforming the tableau $S$ (resp.\ $S^-$) following the route 
\begin{align}
S {\blue\overset{\delta_{5,1}}{{}\to{}}} S^{(4)} {\red\overset{\delta_4}{{}\to{}}} S^{(3)}  {\red\overset{\delta_3}{{}\to{}}} S^{(2)}  {\red\overset{\delta_2}{{}\to{}}} S^{(1)}  {\red\overset{\delta_1}{{}\to{}}} S^{(0)}  & =\varnothing  \label{Eq-route-T-doubleprime} \\
                         (\text{resp.\ } S^-     {\red\overset{\delta_4}{{}\to{}}} S^{-(3)} {\red\overset{\delta_3}{{}\to{}}} S^{-(2)} {\red\overset{\delta_2}{{}\to{}}} S^{-(1)} {\red\overset{\delta_1}{{}\to{}}} S^{-(0)} & =\varnothing) \label{Eq-route-T-prime}
\end{align}
(see the diagram (\ref{Eq-invol-tab-erase-one-ex}) below), 
and piling up 
the terminating row numbers shown in the diagram in \textbf{\blue blue} letters accompanying the \tikz[baseline=-2.5pt,blue]\draw[-latex](0,0)--(1,0); arrows in the case of the bottom deletions, and those shown in \textbf{\red red} letters accompanying the \tikz[baseline=-2.5pt,red]\draw[-latex](0,0)--(1,0); arrows in the case of the other deletions,
into a tableau 
$T'$ (resp.\ $T''$) shown in (\ref{Eq-invol-tab-erase-one-partners}). 

The operator ${\blue\hat\delta_{5,1}}={\blue\hat\delta_{5,1}^{\tikz\node[draw,inner sep=2pt]{$\sc 3$};}}$ enables us to rewrite
  the initial portion $S {\blue\overset{\delta_{5,1}}{{}\to{}}} S^{(4)}$ of (\ref{Eq-route-T-doubleprime}) as $S^- {\blue\overset{\hat\delta_{5,1}}{{}\to{}}} S^{(4)}$ without affecting the terminating row numbers being produced, thereby making both routes to consider start from the same tableau $S^-$.

Since $S(n,1)\leq n-2$ we can apply Proposition~\ref{Prop-invol-tab-erase-one-ind} to $S$, whose purpose is twofold: 
one is
to show the validity of the 
conclusion of Proposition~\ref{Prop-invol-tab-erase-row}, $\delta_{5,1}(\rho^{(5)}S)=\rho^{(4)}S^-$,
in rows 5 and 
4; the other is 
to pass the task of showing 
its 
validity in the remaining upper rows to a recursive application of Lemma~\ref{Lem-invol-tab-erase-one-gen} to $S^*=\delta_4S^-\cup_{(4,1)}\boxed{3}
{}=S^{-(3)}\cup_{(4,1)}\boxed{3}
$ (where $3=S(5,1)$), 
namely to comparing the tableaux obtained from the routes 
\begin{align}
S^* {\blue\overset{\delta_{4,1}}{{}\to{}}}\text{ (or }S^{-(3)} {\blue\overset{\hat\delta_{4,1}}{{}\to{}}})\text{ }S^{(3)} {\red\overset{\delta_3}{{}\to{}}} S^{(2)} {\red\overset{\delta_2}{{}\to{}}} S^{(1)} {\red\overset{\delta_1}{{}\to{}}} S^{(0)} & =\varnothing\quad\text{and} \label{Eq-route-T^*}\\
S^{-(3)} {\red\overset{\delta_3}{{}\to{}}} S^{-(2)} {\red\overset{\delta_2}{{}\to{}}} S^{-(1)} {\red\overset{\delta_1}{{}\to{}}} S^{-(0)} & =\varnothing. \label{Eq-route-T'{}^{(3)}}
\end{align}
The claims in part (2) 
of Proposition~\ref{Prop-invol-tab-erase-one-ind} 
serve the first purpose: the diagram gives $l'=2$, $R'=(\foreach[count=\j]\x in{1,1,2,3}{\x\ifnum\j<4,\fi})$ and $R''=(\foreach[count=\j]\x in{1,2,2,3}{\x\ifnum\j<4,\fi})$, so that $l''=1$ (the rightmost entry $<{}l'$ in $R'$), which verifiably coincides with the terminating row number $1$ given by ${\blue\hat\delta_{4,1}}={\blue\hat\delta_{4,1}^{\tikz\node[draw,inner sep=2pt]{$\sc 3$};}}$ applied to $S^-$, and $R''$ coincides with the result of replacing the rightmost $1$ of $R'$ by $2$.
Due to the claim ${\red\delta_4}{\blue\hat\delta_{5,1}}S^-={\blue\hat\delta_{4,1}}{\red\delta_4}S^-$ in part 
(1) of Proposition~\ref{Prop-invol-tab-erase-one-ind}, 
namely the commutativity of the upper rectangle in 
(\ref{Eq-invol-tab-erase-one-ex}) in terms of the transformations of $S^-$, 
one verifies 
by chasing the route \thetag{\ref{Eq-route-T^*}} in the diagram 
that $\rho^{(4)}S^*$ piles up the upper three rows of $\rho^{(5)}S$ over the cell $(4,1)$ containing the element $l''=1$ bumped from $R'$, and since $\rho^{(3)}S^{*-}$ is just the upper three rows of $\rho^{(4)}S^-$, the conclusion of the recursively applied Lemma~\ref{Lem-invol-tab-erase-one-gen}, $\delta_{4,1}(\rho^{(4)}S^*)=\rho^{(3)}S^{*-}$, will take care of the equality $\delta_{5,1}(\rho^{(5)}S)=\rho^{(4)}S^-$ in rows 1, 2 and 3.

In general, applying Lemma~\ref{Lem-invol-tab-erase-one-gen} to $S^*$ may again
need invoking Proposition~\ref{Prop-invol-tab-erase-one-ind}, in which case it is worth noting that $S^*(n-1,1)=S(n,1)$ since it ensures that the operation $\hat\delta_{n-1,1}$ defined in the context of $S$ and that of $S^*$ are indeed the same operation.
In this example, however, $S^*$ already meets the requirements of Lemma~\ref{Lem-invol-tab-erase-one-fin} since $S^*(4,1)=3=4-1$, which directly gives $\delta_{4,1}(\rho^{(4)}S^*)=\rho^{(3)}S^{*-}$: the lower row of (\ref{Eq-invol-tab-erase-one-partners}) indeed verifies that $\delta_{4,1}$ places $l''=1$ into the empty cell $(3,2)$ of $\rho^{(4)}S^*$ and yields $\rho^{(3)}S^{*-}$.

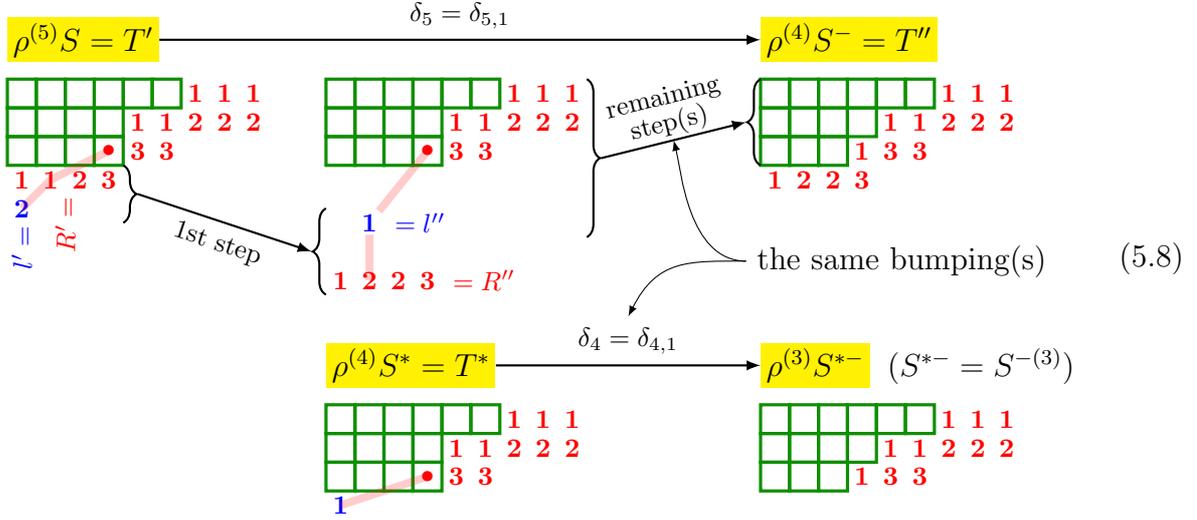
\begin{equation}\label{Eq-invol-tab-erase-one-ex}
\vcenter{\hbox{

}}
\end{equation}
This completes our Example~\ref{Ex-delaying-to-erase-row}.

\end{Example}

Combining Lemmas~\ref{Lem-invol-tab-erase-one-emp}, \ref{Lem-invol-tab-erase-one-sp} and \ref{Lem-invol-tab-erase-one-gen}, we have

\begin{Lemma}\label{Lem-invol-tab-erase-one}
If $\lambda_n=1$, then Proposition~\ref{Prop-invol-tab-erase-row} holds.
\end{Lemma}

Finally, using Lemma~\ref{Lem-invol-tab-erase-one}, we give 
the following inductive step for proving Proposition~\ref{Prop-invol-tab-erase-row} by induction on $\lambda_n$.

%%%%%%%%%%%%%%%%%%%%

\begin{Lemma}\label{Lem-invol-tab-shift}
Let $S$ be as in Proposition~\ref{Prop-invol-tab-erase-row}, and assume that $\lambda_n>0$.
Let $S_\leftarrow$ denote the tableau obtained from $S$ by erasing the cell $(n,1)$ and shifting each of the remaining cells in row $n$ to the left by one column.
Note that $S_\leftarrow\in\mathcal{LR}^{(n)}$.
Then we have $\rho^{(n)}S_\leftarrow=\delta_{n,\lambda_n}\rho^{(n)}S$.
\end{Lemma}

\begin{proof}
We are comparing the partner tableaux of $S$ and $S_\leftarrow$.
The drawings in Eq.~(\ref{Eq-invol-tab-shift}) below endeavours to illustrate the relationship.
\numero{1} and \numero{5} are the tableaux $S$ and $S_\leftarrow$ respectively.
The fat black arrows denote migration to the respective partner tableaux: \numero{1} to \numero{4} and \numero{5} to \numero{8}.
The claim is that the partner \numero{8} for $S_\leftarrow$ coincides with the transform of the partner \numero{4} for $S$ by 
a single deletion $\delta_{n,\lambda_n}$, starting with the entry denoted by $x$ in the drawing.
Note that row $n$ of \numero{4} is drawn detached from the rest of the tableau.
This is to show an arrow connecting the cell $(n,\lambda_n)$ to row $n-1$.
As we shall see towards the end of the proof, that is one point to which we shall return. %where we may call for some attention.
%whose path is indicated by the blue dashed wavy line in \numero{4}, and whose afterimage, actually the same path with each of its entries shifted up along the path by one row, is shown in \numero{8}

As an intermediary, consider \numero{2} $S':=\delta_n^{\lambda_n-1}\cdots\delta_n^1S$, the result of applying the $n$-deletion $\delta_n$ without the final deletion $\delta_n^{\lambda_n}$.
The symbol $\delta_n'$ above the arrow going from \numero{1} to \numero{2} represents this shorter combination, and the three red wavy lines in \numero{1} and the three black wavy lines in \numero{2} indicate the deletion paths associated to the application of $\delta'_n$ to $S$, assuming in the picture that all its consituents are type (ii) deletions, and with the convention that $b\sp*,c\sp*$ and $d\sp*$ denote the uppermost nonzero entries of the respective paths starting from the cells containing $b$, $c$ and $d$, and that the empty red squares above them in \numero{1}, which become dotted and filled squares in \numero{2}, are their terminating cells.
If some of these deletions are of type (i) or type (iii), the corresponding wavy lines in the pictures are just symbolic, the actual changes being limited to the removal of the starting cells.

If we apply $\delta_n^{\lambda_n-1}\cdots\delta_n^1$ to $S_\leftarrow$ instead, which amounts to applying the full $n$-deletion $\delta_n$ in this case, the deletions carve the same changes except for their starting cells being shifted to the left, as indicated by the red wavy lines in \numero{5} and the black wavy lines in \numero{6}, possibly symbolic, and leave the tableau \numero{6} which coincides with the result, $^{n-1\{}S'$, of the removal of the cell $(n,1)$ from $S'$.

Since $S'$ has only one box in row $n$, Lemma~\ref{Lem-invol-tab-erase-one} assures, by way of Proposition~\ref{Prop-invol-tab-erase-row}, that the partner on the lower level \numero{7} coincides with the transform of the partner \numero{3} for $S'$ by a single deletion $\delta_{n,1}$ indicated by the blue wavy line with the same convention as the red wavy lines in \numero{1} and the black wavy line in \numero{7} indicating the result of shifting.
Again $\delta_{n,1}$ can be of type (i) or (iii), and then the wavy line is symbolic with the change actually limited to the removal of the starting cell.
Note that such a case occurs exactly when the deletions $\delta_{n,1}$ from \numero{1} and \numero{2} are of type (iii) or (i) respectively.

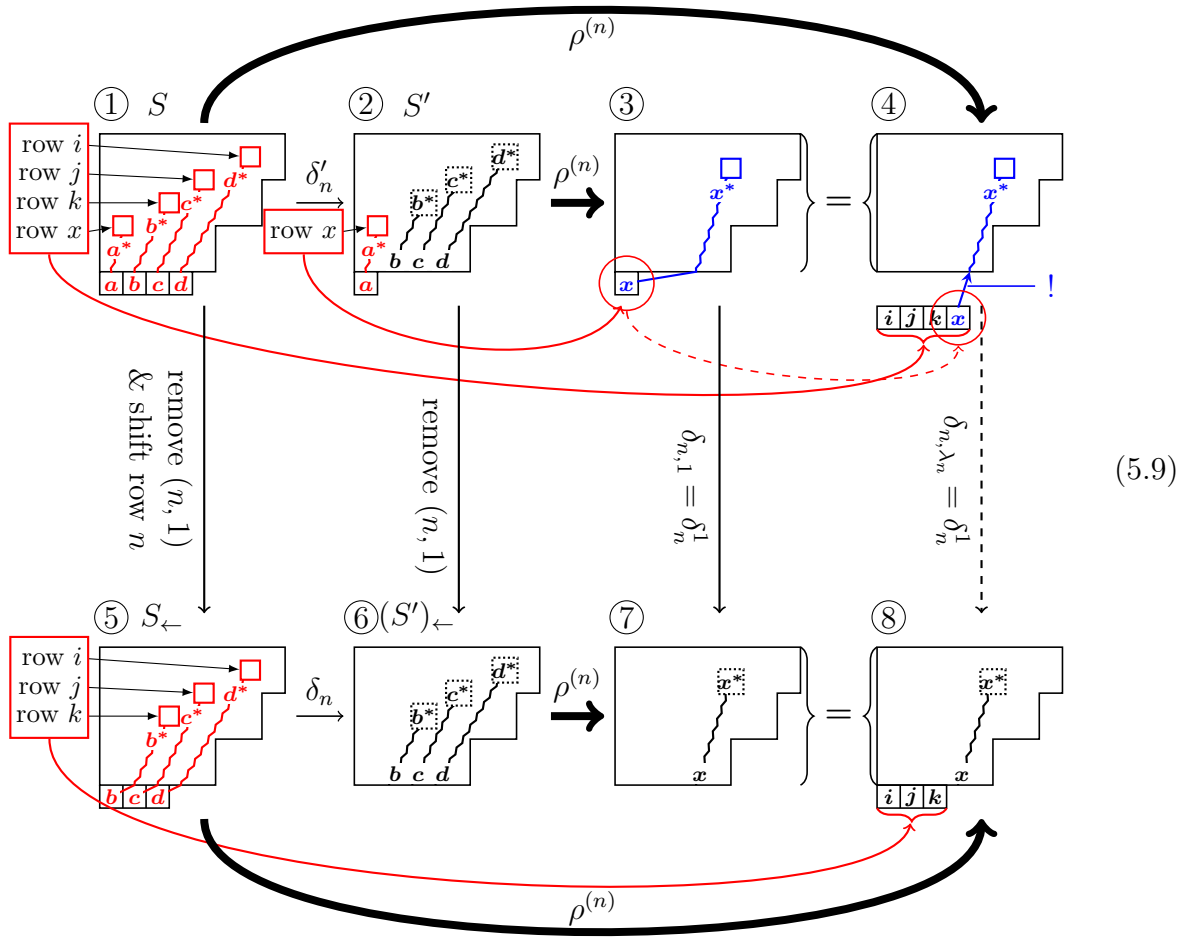
\begin{equation}\label{Eq-invol-tab-shift}
\vcenter{\hbox{

}}
\end{equation}

Next we compare the partners \numero{4} for $S$ and \numero{3} for $S'$.
Note that $S'$ appears in the process of applying $\Delta_n$ to $(S,\varnothing)$ as a left-hand side component.
At that point, the cells $(n,1)$ through $(n,\lambda_n-1)$ on the right-hand side component, namely those cells of the partner \numero{4}, have been constructed by placing the terminating row numbers of the red wavy paths shown in \numero{1}, indicated by $i,j$ and $k$ in the drawing.
If we continue, the remaining deletion $\delta_{n,1}$ applied to $S'$ produces a terminating row number, denoted by $x$ in the drawing, to be placed at $(n,\lambda_n)$ of \numero{4}, and then $\delta_{n-1}$, \dots, $\delta_1$ produce rows $n-1$ through 1 of \numero{4}.
If we start from $S'$ to produce \numero{3} instead, the transformations of the left-hand side components are exactly the same, but the terminating row number of the first deletion $\delta_{n,1}$ for $S'$, which is also $x$, is placed at $(n,1)$, and the rows $n-1$ through 1 of \numero{3}, produced by the continuing deletions $\delta_{n-1}$, \dots, $\delta_1$, are exactly the same as those of \numero{4} as indicated by the $=$ sign in the drawing.
Thus, the tableau \numero{4} can be described as being obtained from \numero{3} by moving the cell $(n,1)$ to $(n,\lambda_n)$ and placing the terminating row numbers of the possibly symbolic red wavy lines in \numero{1} in cells $(n,1)$ through $(n,\lambda_n-1)$.
Similarly, or even more simply, the partner \numero{8} for $S_\leftarrow$ can be described as the partner \numero{7} for \numero{6} plus a row consisting of the terminating row numbers of the red wavy lines shown in \numero{5}, attached as row $n$.
Note again that these terminating row numbers, indicated by $i,j$ and $k$ in the drawing, are exactly the same as those of the red wavy lines shown in \numero{1}, in other words the contents of the cells $(n,1)$ through $(n,\lambda_n-1)$ of \numero{4}.

Since the topmost $n-1$ rows of \numero{4} are exactly the same as those of \numero{3}, the first deletion $\delta_{n,\lambda_n}$ for \numero{4} and the first deletion $\delta_{n,1}$ for \numero{3}, both starting with the entry denoted by $x$ in the drawing and hence belonging to the same type, follow the same path, if they are of type (ii), and incur the same changes (no changes in the case of type (i) or (iii)) in the topmost $n-1$ rows of \numero{4} and \numero{3} respectively.
This claim in the case of type (ii) is verified by consulting Definition~\ref{Def-dp}, according to which the deletion path proceeds iteratively from the current cell by looking for the rightmost cell in the row immediately above which contains an entry stricty smaller than that of the current cell, which ensures that the deletion path is determined by the value of the starting entry and the tableau above the starting row, regardless of the column position of the starting entry.
However, at first sight it may sound suspicious to claim that the path obtained from the blue wavy line in \numero{3} by moving the starting cell to the right still constitutes the beginning of a valid deletion path if one worries about the possibility of moving the entry $x$ in \numero{3} too far to the right to form a vertical strip.
The point is that how far $x$ is moved to the right is bounded by the semistandardness of the tableau \numero{4} that is itself constructed as the partner of the originally given tableau $S$. Then as reviewed in the proof of Lemma~\ref{Lem-ss} this automatically assures that the path generated according to Definition~\ref{Def-dp} is a vertical strip.
\end{proof}%%% end of proof of Lemma 29
\medskip
%%%%%%%%%%%%%%%

Now we can state the proof of Proposition~\ref{Prop-invol-tab-erase-row}, thereby concluding the proof of Theorem~\ref{The-invol-tab-main}.

\begin{proof}[Proof of Proposition~\ref{Prop-invol-tab-erase-row}]
If $\lambda_n=0$, then $S$ and $S^-$ coincide, and $T'$ and $\delta_nT'$ coincide, so the conclusion is immediate.
If $\lambda_n>0$, then we can apply Lemma~\ref{Lem-invol-tab-shift} successively to $S$, $S_\leftarrow$, $(S_\leftarrow)_\leftarrow$ and so on, to show that $\rho^{(n-1)}S^-$ is obtained from $\rho^{(n)}S$ by applying $\delta_{n,\lambda_n}$, $\delta_{n,\lambda_n-1}$, $\delta_{n,\lambda_n-2}$ and so on, namely by applying $\delta_n$, as required.
%%%%% END OF Proposition~\ref{Prop-invol-tab-erase-row}
\end{proof}
%%%%%

\section{Proof of Delaying Bottom Cell Deletion lemma}\label{Sec-prf-delaying}

This whole section constitutes the proof of Proposition~\ref{Prop-invol-tab-erase-one-ind}.
This rather long proof tries to organize its course of logic by putting some intermediate claims in lemmas, which form layers as in \thetag{\ref{Eq-prf-delaying-str}}.
This means that the proof of Proposition~\ref{Prop-invol-tab-erase-one-ind} states, proves and uses Lemmas~\ref{Lem-bumping-position}, \ref{Lem-invol-tab-one-exchange} and \ref{Lem-del-seq-final-part} (plus some narrative text augmenting some minor points), and the proof of Lemma~\ref{Lem-invol-tab-one-exchange} has a similar structure, and so on.

\begin{equation}\label{Eq-prf-delaying-str}
\begin{tikzpicture}
\path[grow'=east,level distance=2pc,growth parent anchor=east,
      every node/.style={anchor=west},
%%% the distance set by the option "level distance=" is the distance between
%%% the anchor of the parent node, as set by "growth parent anchor=" option,
%%% and the child node's anchor. as set by "anchor=" option in the "node"
%%% command of the child (the above sets its default by "every node/.style=")
%%% (see 18.5.4 of manual for version 2.10)
      edge from parent path={
       (\tikzparentnode.east)--+(1pc,0)|-(\tikzchildnode.west)
      }]
%%% the above path is defined by modifyingthe definition of 
%%% "edge from parent fork right" in the file "tikzlibrarytrees.code.tex"
%%% which provides some ready made paths for "edge from parent" if included via
%%% "\usetikzlibrary{trees}" (the above definition works without including it)
%%% (see 18.6 and 53.2 of manual for version 2.10)
 node{Proposition~\ref{Prop-invol-tab-erase-one-ind}} [sibling distance=2.5pc]
  child { node{Lemma~\ref{Lem-bumping-position}} }
  child { node{Lemma~\ref{Lem-invol-tab-one-exchange}} [sibling distance=1.5pc]
   child { node{Lemma~\ref{Lem-invol-tab-one-exchange-before-contact}} }
   child { node{Lemma~\ref{Lem-invol-tab-one-exchange-first-contact}}
    [sibling distance=1pc]
    child { node{Lemma~\ref{Lem-invol-tab-one-exchange-first-contact-lower}} }
    child { node{Lemma~\ref{Lem-invol-tab-one-exchange-first-contact-key-row}} }
    child { node{Lemma~\ref{Lem-invol-tab-one-exchange-first-contact-upper}} }
   }
   child { node{Lemma~\ref{Lem-invol-tab-one-exchange-after-contact}} }
  }
  child { node{Lemma~\ref{Lem-del-seq-final-part}} }
;
\end{tikzpicture}
\end{equation}

We begin by applying Vertical Path Comparison Lemma~\ref{Lem-vp}
to $S$, where $\delta_{n,1}$ is a Phase 2B deletion, which shows that $R'$ contains at least one nonzero entry strictly smaller than $l'$.
Hence we have $l''>0$.

Now note that $\delta_{n-1}\hat\delta_{n,1}S^-=\delta_{n-1,1}\cdots\delta_{n-1,\lambda_{n-1}}\hat\delta_{n,1}S^-$.
Recall from part (2) of Definition~\ref{Def-compactifying-notation} that we also write $\delta_{n-1}=\delta_{n-1}^{\lambda_{n-1}}\cdots\delta_{n-1}^1$.
Then $k'_i$ (resp.\ $k''_i$) is the terminating row number of the operation $\delta_{n-1}^i$ in $\delta_{n-1}$ applied to $\hat\delta_{n,1}S^-=\delta_{n,1}S$ (resp.\ applied to $S^-$).
Moreover, for $\nu_{n-1}<i\leq\lambda_{n-1}$, let $P'_i$ (resp.\ $P''_i$) be the deletion path of this operation 
(note that both $\hat\delta_{n,1}S^-$ and $S^-$ contain exactly $\nu_{n-1}$ entries $n-1$).
Furthermore, let $P'$ be the deletion path of the operation $\hat\delta_{n,1}$ (actually that of the $\delta_{n,1}$ part of this operation) applied to $S^-$.

Let $i_*$ be the smallest $i$ such that the paths $P'_i$ and $P'$ have at least one cell in common, and call $i_*$ the \textit{first contact index}.
Such $i$ exists since the deletion starting from $P'[n-1]$ is in Phase 2 and does share the cell $P'[n-1]$ with $P'$; moreover the minimality of $i_*$ implies that the starting cell of $P'_{i_*}$ is weakly to the right of $P'[n-1]$.

With $i_*$ thus defined, let us relate it with the place where the replacement of $l''$ by $l'$ occurs in the transition from $R'$ to $R''$.

\begin{Lemma}\label{Lem-bumping-position}
In the setting of Proposition~\ref{Prop-invol-tab-erase-one-ind}, $k'_{i_*}$ is the rightmost entry in $R'$ which is strictly smaller than $l'$, and hence $l''=k'_{i_*}$.
\end{Lemma}

%%%%%
\begin{proof}
We can show that $k'_{i_*}<l'$ in a manner similar to that of the proof of Vertical Path Comparison Lemma~\ref{Lem-vp}. Namely, we first 
argue inductively for $\nu_{n-1}<i<i_*$ that the path $P'_i$ lies strictly to the right of $P'$ and that the shape of $\delta_{n-1}^{i}\cdots\delta_{n-1}^{1}\hat\delta_{n,1}S^-$ is such that $P'[n-1]$ (resp.\ $P'[l']$) is still bounded below (resp.\ above) by a horizontal segment of the outer (resp.\ inner) border,
with contents of $P'-\{(n,1)\}$ unaltered after $\hat\delta_{n,1}$: this is because the starting cell of $P'_i$ is strictly to the right of $P'[n-1]$ and, by induction, the earlier deletions in $\delta_{n-1}$ have occurred only strictly to the right of $P'$, keeping $P'[n-1]$ (resp.\ $P'[l']$) bounded below (resp.\ above) by a horizontal segment of the outer (resp.\ inner) border 
and also keeping the contents of $P'-\{(n,1)\}$ unaltered, so that part (2) of Lemma~\ref{Lem-ep} assures that $P'_i$ still stays weakly to the right of $P'$; in addition, $P'_i$ and $P'$ do not overlap by assumption, which locates $P'_i$ strictly to the right of $P'$, from which the remaining inductive claims follow.
Thus, when $\delta_{n-1}^{i_*}$ starts, this situation is maintained, and the starting cell of $P'_{i_*}$ is still weakly to the right of $P'[n-1]$, so that part (2) of Lemma~\ref{Lem-ep}
can again be applied to conclude that the path $P'_{i_*}$ stays weakly to the right of $P'$, and that $k'_{i_*}<l'$.
In the case where $i_*<\lambda_{n-1}$, we further want to show that $k'_{i_*+1}\geq l'$.
Let $(r,c)=P'_{i_*}[r]=P'[r]$ be the cell where $P'_{i_*}$ and $P'$ last meet (i.e.\ their highest meeting point).
If the path $P'_{i_*+1}$ terminates before reaching row $r$, then we already have $k'_{i_*+1}>l'$.
Now assume that $P'_{i_*+1}$ does reach row $r$, and set $S^!=\delta_{n-1}^{i_*}\cdots\delta_{n-1}^{1}\hat\delta_{n,1}S^-$.
Note that, since the deletion path $P'_{i_*}$ of $\delta_{n-1}^{i_*}$ runs strictly to the right of $P'$ above row $r$ (so do all 
paths $P'_i$ with $i<i_*$), the contents of $P'$ above row $r$ are maintained (i.e.\ the same as those in $\hat\delta_{n,1}S^-=\delta_{n,1}S$).
Now consider an auxiliary tableau $({}^{r-1\{}S)\cup_{(r,1)}\boxed{S(P'_{i_*+1}[r])}\cup_{(r,2)}\boxed{S(P'[r])}$.
That $S$ is an LR tableau assures that this is also an LR tableau.
If one applies $\delta_{r,2}$ to this tableau, performing the deletion starting from the cell containing $S(P'[r])$, then it proceeds along $^{r-1\{}P'$, yielding $({}^{r-1\{}\delta_{n,1}S)\cup_{(r,1)}\boxed{S(P'_{i_*+1}[r])}$.
If one compares this tableau with $({}^{r-1\{}S^!)\cup_{(r,1)}\boxed{S(P'_{i_*+1}[r])}$, the differences are limited to the region to the right of $^{r-1\{}P'$.
Furthermore, if one applies $\delta_{r,1}$ to this latter tableau, performing the deletion starting from the cell containing $S(P'_{i_*+1}[r])$, it proceeds along $^{r-1\{}P'_{i_*+1}$.
Hence one can apply part (1) of Lemma~\ref{Lem-ep} with $T$ being $({}^{r-1\{}S)\cup_{(r,1)}\boxed{S(P'_{i_*+1}[r])}\cup_{(r,2)}\boxed{S(P'[r])}$, $P$ being $\{(r,2)\}\cup{}^{r-1\{}P'$, $t$ being $l'$, $T'$ being $({}^{r-1\{}\delta_{n,1}S)\cup_{(r,1)}\boxed{S(P'_{i_*+1}[r])}$, $T^\#$ being $({}^{r-1\{}S^!)\cup_{(r,1)}\boxed{S(P'_{i_*+1}[r])}$, $P^\#$ being $\{(r,1)\}\cup{}^{r-1\{}P'_{i_*+1}$ and $t'$ being $k'_{i_*+1}$, and conclude that $k'_{i_*+1}\geq l'$ ($>k'_{i_*}$), and therefore $l''=k'_{i_*}$.
%%%%%
\end{proof}
%%%%%

%%%%%%%%%%%%%%%%%%%%%%%
Now we consider the following operations $D^{(i)}$, $0\leq i<\lambda_{n-1}$, that interpolate between $\delta_{n-1}\hat\delta_{n,1}$ and $\hat\delta_{n-1,1}\delta_{n-1}$ (to be applied to $S^-$), and observe how the resulting tableaux change (or rather do not change) and how the sequences of terminating row numbers change.

\begin{Lemma}\label{Lem-invol-tab-one-exchange}
Retain the situation of Proposition~\ref{Prop-invol-tab-erase-row} and Proposition~\ref{Prop-invol-tab-erase-one-ind}.
Furthermore, for each $i$ with $0\leq i<\lambda_{n-1}$, set
\begin{align*}
D^{(i)} & = \overbrace{\delta_{n-1,1}\cdots\delta_{n-1,\lambda_{n-1}-i}}^{\lambda_{n-1}-i}
            \hat\delta_{n,1}
            \overbrace{\delta_{n-1,\lambda_{n-1}-i+1}\cdots\delta_{n-1,\lambda_{n-1}}}^i \\
        & = \delta_{n-1}^{\lambda_{n-1}}\cdots\delta_{n-1}^{i+1}\hat\delta_{n,1}\delta_{n-1}^i\cdots\delta_{n-1}^1,
\end{align*}
with a slightly streched interpretation of the applicability of the notation $\delta_{n-1}^j$ introduced in part (2) of Definition~\ref{Def-compactifying-notation}, and
\begin{align*}
R^{(i)}=( & \text{sequence of terminating row numbers obtained} \\
          & \text{by applying $D^{(i)}$ to $S^-$}).
\end{align*}
Then we have
\begin{enumerate}
\item[(1)]
$D^{(i)}S^-=\delta_{n-1}\hat\delta_{n,1}S^-\quad(0\leq i<\lambda_{n-1})$
\end{enumerate}
and
\begin{enumerate}
\item[(2)]
$R^{(i)}=\begin{cases}
 (k'_1,\dots,k'_i,l',k'_{i+1},\dots,k'_{i_*-1},l'',k'_{i_*+1},\dots,k'_{\lambda_{n-1}}) & (0\leq i<i_*), \\
 (k'_1,\dots,k'_{i_*-1},l',k'_{i_*+1},\dots,k'_{i+1},l'',k'_{i+2},\dots,k'_{\lambda_{n-1}}) & (i_*\leq i<\lambda_{n-1}).
\end{cases}$
\end{enumerate}
\end{Lemma}

\begin{Remark}\label{Rem-del-seq-final-part}
Note that $D^{(0)}=\delta_{n-1}\hat\delta_{n,1}$ and that, by definition, $R^{(0)}=l'\cdot R'=(l',k'_1,\dots,k'_{\lambda_{n-1}})$, as claimed in part (2) above.
This is one end of the interpolation.
On the other hand, the operation $\hat\delta_{n-1,1}\delta_{n-1}$, which should be the other end of the interpolation, does not appear to be present among the operations $D^{(i)}$, but we shall see below in Lemma~\ref{Lem-del-seq-final-part} that it has the same effect on $S^-$ as $D^{(\lambda_{n-1}-1)}$, and also that the terminating row numbers produced by applying $\hat\delta_{n-1,1}\delta_{n-1}$ to $S^-$ coincides with $R^{(\lambda_{n-1}-1)}=R''\cdot l''=(k''_1,\dots,k''_{\lambda_{n-1}},l'')$.
\end{Remark}

%%%%%%%%%%%%%%%%%%%%%%%

\begin{proof}
The proof of Lemma~\ref{Lem-invol-tab-one-exchange} will be further divided into sublemmas.
Let us begin by verifying the statements (1) and (2) for $i<i_*$, where $i_*$ is the first contact index introduced before Lemma~\ref{Lem-bumping-position}.

\begin{Lemma}\label{Lem-invol-tab-one-exchange-before-contact}
For $i<i_*$, each of the operations $\delta_{n-1}^{1}$, \dots, $\delta_{n-1}^{\lambda_{n-1}}$ and $\hat\delta_{n,1}$ in $D^{(i)}=\delta_{n-1}^{\lambda_{n-1}}\linebreak[1]\cdots\delta_{n-1}^{i+1}\hat\delta_{n,1}\delta_{n-1}^{i}\cdots\delta_{n-1}^{1}$ incurs exactly the same transformation as the one represented by the same name in $D^{(0)}=\delta_{n-1}^{\lambda_{n-1}}\cdots\delta_{n-1}^{i+1}\delta_{n-1}^{i}\cdots\delta_{n-1}^{1}\hat\delta_{n,1}$ does.
We express this situation by saying that $\hat\delta_{n-1}$ commutes with $\delta_{n-1}^{i}\cdots\delta_{n-1}^{1}$.

Hence the statements (1) and (2) of Lemma~\ref{Lem-invol-tab-one-exchange} hold for these $i$.
\end{Lemma}

\begin{proof}
This is clear for $i\leq\nu_{n-1}$, namely for those $i$ such that $\delta_{n-1}^i$ in $D^{(0)}$ is of type (i), since $\hat\delta_{n,1}$ passes row $n-1$ by carrying $e\leq n-2$ into a cell to the left of all ($n-1$)'s in row $n-1$ of $S^-$ and hence deferring this operation until after the type (i) deletions from row $r-1$ does not change anything.

Next assume that $\nu_{n-1}<i<i_*$ and that, by induction, $\hat\delta_{n,1}$ commutes with $\delta_{n-1}^{i-1}\cdots\delta_{n-1}^1$.
In $D^{(i-1)}=\delta_{n-1}^{\lambda_{n-1}}\cdots\delta_{n-1}^{i}\hat\delta_{n,1}\delta_{n-1}^{i-1}\cdots\delta_{n-1}^{1}$, the operations $\delta_{n-1}^{i-1}$, $\hat\delta_{n,1}$ and $\delta_{n-1}^i$ are applied one after another in this order, and the deletion paths of $\hat\delta_{n,1}$ and $\delta_{n-1}^i$ are $P'$ and $P'_i$ respectively, due to the induction hypothesis.
Consider exchanging the order of $\hat\delta_{n,1}$ and $\delta_{n-1}^i$ applied to $\delta_{n-1}^{i-1}\cdots\delta_{n-1}^1S^-$.
Since $i<i_*$, $P'_i$ is still strictly to the right of $P'$, and so if we apply $\delta_{n-1}^i$ before $\hat\delta_{n,1}$, it also proceeds along $P'_i$ since its contents are nonetheless the rightmost entry strictly smaller than the entry coming from the row below at each stage, whether something to its left has changed or not.
This makes the contents of $P'_i$, which is strictly to the right of $P'$, larger than before, which again does not affect the path taken by $\hat\delta_{n,1}$ when its application is deferred until after that of $\delta_{n-1}^i$ since at every stage the rightmost entry strictly smaller than the entry coming from the row below is still the same entry in $P'$ whether something to its right has been made larger or not.
So these two operations, upon exchanging the order of applications, follow the same paths bumping exactly the same elements.
Combined with the induction hypothesis, we have $\hat\delta_{n,1}\delta_{n-1}^i\cdots\delta_{n-1}^1S^-=\delta_{n-1}^i\cdots\delta_{n-1}^1\hat\delta_{n,1}S^-$, so that the remaining operations $\delta_{n-1}^{i+1}$, \dots, $\delta_{n-1}^{\lambda_{n-1}}$ incur the same transformations on these two.
Hence $\hat\delta_{n,1}$ also commutes with $\delta_{n-1}^i\delta_{n-1}^{i-1}\cdots\delta_{n-1}^1$.

This immediately implies the statement (1): $D^{(i)}S^-=D^{(0)}S^-$.
Moreover, since the paths are the same, the terminating row numbers are also the same, including the zeros obtained from type (i) deletions, except that $l'$, obtained from $\hat\delta_{n,1}$, comes first in $D^{(0)}$ whereas in $D^{(i)}$ it comes in between $k'_i$ and $k'_{i+1}$ obtained from $\delta_{n-1}^{i}$ and $\delta_{n-1}^{i+1}$ respectively.
This validates the statement (2).
\end{proof}
%%%%%

Reverting to the proof of Lemma~\ref{Lem-invol-tab-one-exchange}, what we have just shown, namely Lemma~\ref{Lem-invol-tab-one-exchange-before-contact}, finishes the proof of Lemma~\ref{Lem-invol-tab-one-exchange} in the special case where $i_*=\lambda_{n-1}$, since it covers all values of $0\leq i<\lambda_{n-1}$ in the statements (1) and (2) of Lemma~\ref{Lem-invol-tab-one-exchange} in this case.

So we assume that $i_*<\lambda_{n-1}$ in the remainder of the proof of Lemma~\ref{Lem-invol-tab-one-exchange}.

Then we reach a crucial step which we call the \textit{case of first contact}, namely changing $D^{(i_*-1)}$ into $D^{(i_*)}$ by exchanging the order of $\hat\delta_{n,1}$ and $\delta_{n-1}^{i_*}$, which are their $i_*$th and ($i_*+1$)th operations in the order of application.
Note that the sequence expected for $R^{(i_*)}$ in the statement (2), $(k'_1,\dots,k'_{i_*-1},l',k'_{i_*+1},l'',k'_{i_*+2},\dots,k'_{\lambda_{n-1}})$, would derive from the already established $R^{(i_*-1)}=(k'_1,\dots,k'_{i_*-1},l',l'',k'_{i_*+1},k'_{i_*+2},\dots,k'_{\lambda_{n-1}})$ by switching its ($i_*+1$)th and ($i_*+2$)th terms rather than the $i_*$th and $(i_*+1$)th terms.
Thus, setting $S'_0=\delta_{n-1}^{i_*-1}\cdots\delta_{n-1}^{1}\hat\delta_{n,1}S^-$, we are to observe and compare the three, not just two, consecutive operations applied to $S'_0$ as part of $D^{(i_*-1)}$ and as part of $D^{(i_*)}$.
Let 
\begin{equation}\label{Eq-Di_*-1andi_*3}
\vcenter{\hbox{

}}
\end{equation}
denote the relevant intermediate tableaux in the application of both $D^{(i_*-1)}$ and $D^{(i_*)}$ to $S^-$.

Recall that, due to Lemma~\ref{Lem-invol-tab-one-exchange-before-contact}, the operations $\hat\delta_{n,1}$, $\delta_{n-1}^{i_*}$ and $\delta_{n-1}^{i_*+1}$ in $D^{(i_*-1)}$ take the paths $P'$, $P'_{i_*}$ and $P'_{i_*+1}$, giving the terminating row numbers $l'$, $k'_{i_*}=l''$ and $k'_{i_*+1}$ respectively.

\begin{Lemma}\label{Lem-invol-tab-one-exchange-first-contact}
We keep the assumption that $i_*<\lambda_{n-1}$, and the names given in Eq. (\ref{Eq-Di_*-1andi_*3}).
Then we have $S'_3=S''_3$, and the operations $\delta_{n-1}^{i_*}$, $\hat\delta_{n,1}$ and $\delta_{n-1}^{i_*+1}$ in $D^{(i_*)}$ (here arranged in the order of application in $D^{(i_*)}$) give terminating row numbers $l'$, $k'_{i_*+1}$ and $l''$ respectively.

Since this also ensures that any remaining operations $\delta_{n-1}^{(i_*+2)},\dots,\delta_{n-1}^{\lambda_{n-1}}$ in $D^{(i_*-1)}$ and $D^{(i_*)}$ incur the same transformations, the statements (1) and (2) of Lemma~\ref{Lem-invol-tab-one-exchange} hold for $i=i_*$.
\end{Lemma}

\begin{proof}
We prove this by investigating closely the paths taken by the three relevant operations in $D^{(i_*)}$ over Lemmas~\ref{Lem-invol-tab-one-exchange-first-contact-lower}, \ref{Lem-invol-tab-one-exchange-first-contact-key-row} and \ref{Lem-invol-tab-one-exchange-first-contact-upper} below.

For that purpose, let $(r_*,c_*)$ be the cell where $P'_{i_*}$ and $P'$ first meet, i.e.\ the lowest cell of $P'_{i_*}\cap P'$, which we call the \textit{cell of first contact}.
We divide the tableaux into three parts: the part below row $r_*$ (which we call the \textit{lower part}), the part consisting of a single row $r_*$ and the part above row $r_*$ (the \textit{upper part}).

Let us start by looking at the lower part, which is easy.

\begin{Lemma}\label{Lem-invol-tab-one-exchange-first-contact-lower}
Below row $r_*$, the operations $\delta_{n-1}^{i_*}$ and $\hat\delta_{n,1}$ in $D^{(i_*)}$ take paths $P'_{i_*}$ and $P'$ respectively, bumping the same entries as in $D^{(i_*-1)}$ including the ones bumped from row $r_*+1$ into row $r_*$, with the order of ocurrence switched from that in $D^{(i_*-1)}$.

Hence the lower parts of $S'_2$ and $S''_2$ coincide, and accordingly the next operation $\delta_{n-1}^{i_*+1}$ in $D^{(i_*-1)}$ in $D^{(i_*)}$ utterly coincides below $r_*$, including the entry it bumps into row $r_*$.
In particular, the lower parts of $S'_3$ and $S''_3$ also coincide.
\end{Lemma}

\begin{proof}
The first claim can be proved by arguing similarly to the case $i<i_*$ since, below the cell of their first contact, the path $P'_{i_*}$ lies strictly to the right of $P'$.
The remaining claims follow easily from the first one.
\end{proof}

Continuing with the proof of Lemma~\ref{Lem-invol-tab-one-exchange-first-contact}, let us next focus on row $r_*$.
Note that, due to the assumption $i_*<\lambda_{n-1}$, we have $c_*>1$ since, while the cell $(r_*,c_*)$ lies in $P'_{i_*}$, there is at least one more path $P'_{i_*+1}$ running strictly to the left of $P'_{i_*}$ due to the Horizontal Path Comparison Lemma~\ref{Lem-hp}.

\begin{Lemma}\label{Lem-invol-tab-one-exchange-first-contact-key-row}
The operations $\hat\delta_{n,1}$, $\delta_{n-1}^{i_*}$ and $\delta_{n-1}^{i_*+1}$ in $D^{(i_*-1)}$, respectively, go through the cells $(r_*,c_*)$, $(r_*,c_*)$ and $(r_*,c_*-1)$ in row $r_*$.
The operations $\delta_{n-1}^{i_*}$, $\hat\delta_{n,1}$ and $\delta_{n-1}^{i_*+1}$ in $D^{(i_*)}$ (now arranged in the order of application in $D^{(i_*)}$) go through the cells $(r_*,c_*)$, $(r_*,c_*-1)$ and $(r_*,c_*-1)$ in row $r_*$, respectively.

The rows $r_*$ of $S'_3$ and $S''_3$ coincide.
Moreover, the entries bumped to row $r_*-1$ by the three operations in $D^{(i_*-1)}$ and $D^{(i_*)}$ are switched between the second and the third operations in the order of application.
\end{Lemma}

\begin{proof} % of Lemma~\ref{Lem-invol-tab-one-exchange-first-contact-key-row}
The statements about the two operations $\hat\delta_{n,1}$ and $\delta_{n-1}^{i_*}$ in $D^{(i_*-1)}$ is immediate since $P'$ and $P'_{i_*}$ meet at the cell $(r_*,c_*)$ in row $r_*$.
The rest requires a rather detailed analysis.
We set 
\begin{align*}
& x=S'_0(r_*,c_*-1),\quad y=S'_0(r_*,c_*),\quad z=S'_0(r_*,c_*+1), \\
& u=S'_0(P'[r_*+1]),\quad v=S'_0(P'_{i_*}[r_*+1])\quad\text{and}\quad w=S'_1(P'_{i_*+1}[r_*+1])
\end{align*} 
(see $S'_0$ in (\ref{Eq-x3r}); this drawing may be consulted while chasing the following description of the changes made in row $r_*$).
Note that the cells $(r_*,c_*-1)$ and $(r_*,c_*)$ (or possibly only the cell $(r_*,c_*-1)$) may be inside the inner shape of $S'_0$, or in other words $x$ or $y$ can be $0$.
The notation such as $\langle x\rangle$ and $\langle y\rangle$ is used in (\ref{Eq-x3r}) as a reminder of this situation.
The position $(r_*,c_*+1)$, represented by a dotted box in (\ref{Eq-x3r}), may be outside the outer shape, and notation such as $(z)$ used in (\ref{Eq-x3r}) signifies this situation.
The blue box $P'[r_*+1]$ always exists in $S'_0$, and it is in column $c_*-1$ or to its left as we see at the end of this paragraph (the delicate choice of the position of this blue box in the drawing is meant to reflect such relationship).
The cell $P'_{i_*}[r_*+1]$ (dashed red box) and $P'_{i_*+1}[r_*+1]$ (dashed magenta box) exist if and only if $r_*<n-1$.
In this case, $P'_{i_*}[r_*+1]$ is strictly to the right of $P'[r_*+1]$, and $P'_{i_*+1}[r_*+1]$ is strictly to the left of $P'_{i_*}[r_*+1]$ but weakly to the right of $P'[r_*+1]$.
Since $P'_{i_*}[r_*+1]$ is in column $c_*$ or to its left, $P'[r_*+1]$ is in column $c_*-1$ or to its left in this case.
Meanwhile, if $r_*=n-1$, we have $P'[r_*+1]=(n,1)$, while $c_*>1$ as noted before Lemma~\ref{Lem-invol-tab-one-exchange-first-contact-key-row}, so that $P'[r_*+1]$ is also in column $c_*-1$ or to its left.

\begin{equation}\label{Eq-x3r}
\vcenter{\hbox{
}}
\end{equation}

First consider the case where $r_*<n-1$.
Then $v$ and $w$ exist.
Moreover, the blue cell $P'[r_*+1]$ still contains an entry in $S'_1$ since $\hat\delta_{n,1}$ starts from at least 2 rows below row $r_*$.
Set $t=S'_1(P'[r_*+1])=S'_0(P'[r_*+2])$ ($>0$).

In $D^{(i_*-1)}$, the first operation to be applied to $S'_0$, namely $\hat\delta_{n,1}$, bumps $S'_0(P'[r_*+1])=u$ into the cell $P'[r_*]=(r_*,c_*)$, which bumps $S'_0(r_*,c_*)=y$ to row $r_*-1$ (see $S'_1$ in (\ref{Eq-x3r})).
This implies $u>y$.
When we say ``bumps $\langle y\rangle$ to row $r_*-1$'', it is to be understood that if $y=0$ no actual entry is bumped out of row $r_*$ but the deletion terminates in row $r_*$.
Then the second operation $\delta_{n-1}^{i_*}$ bumps $S'_0(P'_{i_*}[r_*+1])=v$ also into $P'_{i_*}[r_*]=(r_*,c_*)$, which bumps $S'_1(r_*,c_*)=u$ to row $r_*-1$ (see $S'_2$ in (\ref{Eq-x3r})).
This implies $v\leq z$ if $z$ exists.
Note that we have $S'_2(P'_{i_*+1}[r_*+1])=S'_1(P'_{i_*+1}[r_*+1])=w$ since the path $P'_{i_*}$ of $\delta_{n-1}^{i_*}$ do not pass $P'_{i_*+1}[r_*+1]$.
Finally the third operation $\delta_{n-1}^{i_*+1}$ bumps $S'_2(P'_{i_*+1}[r_*+1])=w$ into row $r_*$.
Then $w$ enters $(r_*,c_*-1)$ since \begin{equation}\label{Eq-w}
S'_2(r_*,c_*)=\underbrace{v\geq w\geq t}_{\substack{\text{horizontal}\\\text{configuration}\\\text{in $S'_1$}}}=\underbrace{t>u>x}_{\substack{\text{entries bumping}\\\text{one another in}\\\text{$P'$ in $S'_0$}}}=S'_2(r_*,c_*-1).
\end{equation}
This confirms the statement for $\delta_{n-1}^{i_*+1}$ in $D^{(i_*-1)}$, and it bumps $\langle x\rangle$ to row $r_*-1$.
Note, for use in the next paragraph, that we have seen in (\ref{Eq-w}) that $w>u$.

Now, in $D^{(i_*)}$, the first operation $\delta_{n-1}^{i_*}$, proceeding along $P'_{i_*}$ in the lower part, bumps $S'_0(P'_{i_*}[r_*+1])=v$ into row $r_*$.
Since $v\geq u>y=S'_0(r_*,c_*)$ (see the horizontal configuration of $u$ and $v$ and the deletion path configuration of $u$ and $y$ in $S'_0$) and $v\leq z=S'_0(r_*,c_*+1)$ if $z$ exists as we saw (see the transition from $S'_1$ to $S'_2$), $v$ enters the cell $(r_*,c_*)$, confirming the statement for $\delta_{n-1}^{i_*}$ in $D^{(i_*)}$, and it bumps $\langle y\rangle$ to row $r_*-1$.
Note that, since $P'_{i_*}$ and $P'$ are disjoint in the lower part, we have $S''_1(P'[r_*+1])=S'_0(P'[r_*+1])=u$.
Then the second operation $\hat\delta_{n,1}$, proceeding along $P'$ in the lower part, bumps $S''_1(P'[r_*+1])=u$ into row $r_*$.
Since $u>y\geq x=S''_1(r_*,c_*-1)$ and $u\leq v=S''_1(r_*,c_*)$, the letter $u$ enters the cell $(r_*,c_*-1)$, confirming the statement for $\hat\delta_{n,1}$ in $D^{(i_*)}$, and it bumps $\langle x\rangle$ to row $r_*-1$.
Finally the third operation $\delta_{n-1}^{i_*+1}$, proceeding along $P'_{i_*+1}$ in the lower part, bumps $S''_2(P'_{i_*+1}[r_*+1])$ into row $r_*$.
Since the lower parts of $S'_2$ and $S''_2$ coincide, we have $S''_2(P'_{i_*+1}[r_*+1])=S'_2(P'_{i_*+1}[r_*+1])=w$.
Moreover, since $w>u$ (see (\ref{Eq-w})) $=S''_2(r_*,c_*-1)$ and $w\leq v$ (see $S'_1$) $=S''_2(r_*,c_*)$, the letter $w$ enters the cell $(r_*,c_*-1)$, confirming the statement for $\delta_{n-1}^{i_*+1}$ in $D^{(i_*)}$, and it bumps $S''_2(r_*,c_*-1)=u$ to row $r_*-1$.

We have seen that both $S'_3$ and $S''_3$ have $w$ in the cell $(r_*,c_*-1)$, $v$ in the cell $(r_*,c_*)$ and the same entries as $S'_0$ elsewhere in row $r_*$, and that the entries bumped to to row $r_*-1$ by the three operations in the order of application, $\langle y\rangle$, $u$ and $\langle x\rangle$ in the case of $D^{(i_*-1)}$ and $\langle y\rangle$, $\langle x\rangle$ and $u$ in the case of $D^{(i_*)}$, are switched between the second and the third operations.
Thus the proof of Lemma~\ref{Lem-invol-tab-one-exchange-first-contact-key-row} is finished in the case $r_*<n-1$.

If $r_*=n-1$, the path $P'$ starts from row $r_*+1$, while $P'_{i_*}$ and $P'_{i_*+1}$ start from row $r_*$.
So $v$ and $w$ do not exist, and the cell $(r_*,c_*+1)$, if it exists in $S^-$, has been removed by the operation $\delta_{n-1}^{i_*-1}$, so that $z$ does not exist either.
One can easily see that the transformation also proceeds as shown in (\ref{Eq-x3r}) if one ignores the cells containing parenthesized entries, so the same conclusions as in the case $r_*<n-1$ hold.
\end{proof}

Now we look at the upper part. 

\begin{Lemma}\label{Lem-invol-tab-one-exchange-first-contact-upper}
Above row $r_*$, the operations $\delta_{n-1}^{i_*}$, $\hat\delta_{n,1}$ and $\delta_{n-1}^{i_*+1}$ in $D^{(i_*)}$ take the paths $P'$, $P'_{i_*+1}$ and $P'_{i_*}$ respectively, bumping the same entries as the operation in $D^{(i_*-1)}$ that takes the same path, namely $\hat\delta_{n,1}$, $\delta_{n-1}^{i_*+1}$ and $\delta_{n-1}^{i_*}$ respectively.

The upper parts of $S'_3$ and $S''_3$ coincide.
The terminating row numbers given by the three operations in $D^{(i_*-1)}$ and $D^{(i_*)}$, arranged in the respective orders of occurrence, are switched between the second and the third.
\end{Lemma}

\begin{proof}
We first deal with some easy cases.
If both $x$ and $y$ are zero,  
the only letter sent into row $r_*-1$ during the three operations is $u$, by $\delta_{n-1}^{i_*}$ in the case of $D^{(i_*-1)}$ and by $\delta_{n-1}^{i_*+1}$ in the case of $D^{(i_*)}$; in both cases the upper part of $S'_0$ is maintained until accepting the letter $u$, so it causes the same deletion paths in both cases (which must be $P'_{i_*}$ since it is so in $D^{(i_*-1)}$), leaving the upper part the same.
Thus the statement is validated with the upper parts of both $P'$ and $P'_{i_*+1}$ being empty in this case.
If only $x$ is zero, the letters sent into row $r_*-1$ during the three operations are $y$ and $u$, in this order in both $D^{(i_*-1)}$ and $D^{(i_*)}$.
So they incur the same deletion paths in both cases (which must be $P'$ and $P'_{i_*}$ respectively since they are so in $D^{(i_*-1)}$), leaving the upper parts the same again.
The statement is again validated with the upper part of $P'_{i_*+1}$ being empty.

Now we assume that both $x$ and $y$ are positive.
Both in $D^{(i_*-1)}$ and $D^{(i_*)}$, the first of the three relevant consecutive operations applied to $S'_0$ sends the same letter $y$ to row $r_*-1$ as we saw in Lemma~\ref{Lem-invol-tab-one-exchange-first-contact-key-row}.
Hence it incurs the same bumping in both cases.
Since we know that it follows $P'$ in $D^{(i_*-1)}$, it also follows $P'$ in $D^{(i_*)}$, involving the same entries. 
Hence the upper parts of $S'_1$ and $S''_1$ are identical.
Now we compare the effect of the remaining two operations in the upper part.
The second operation in $D^{(i_*-1)}$ starts by bringing $u$ ($>y$) into row $r_*-1$, and we know that
its path $P'_{i_*}$ lies weakly to the right of $P'$. On the other hand, 
the second operation in $D^{(i_*)}$ starts by bringing $x$ ($\leq y$) into row $r_*-1$.  
By applying the Horizontal Path Comparison Lemma~\ref{Lem-hp} to an auxiliary tableau $({}^{r_*-1\{}S'_0)\cup_{(r_*,1)}\boxed{x}\cup_{(r_*,2)}\boxed{y}$, we see that 
its path, which we temporarily call $P''$ (actually $P'_{i_*+1}$ as we shall see instantly), lies strictly to the left of $P'$.
So the latter operation keeps everything weakly to the right of $P'$ unchanged, in particular
the contents of $P'_{i_*}$ and its right-hand side neighbours if any, so that the third operation in $D^{(i_*)}$, starting with the same letter $u$ as the second operation of $D^{(i_*-1)}$, also follows $P'_{i_*}$; making the contents of the upper part of $P'_{i_*}$ in $S''_3$ the same as those in $S'_2$.
On the other hand, the second operation in $D^{(i_*-1)}$ does not change anything strictly to the left of $P'$, including the contents of $P''$, and each of its right-hand side neighbours either stays the same or grows larger (the latter possibility can only occur if it is contained in $P'$), so that the third operation in $D^{(i_*-1)}$, starting with the same letter $x$ as the second operation of $D^{(i_*)}$, also follows $P''$ and hence we have $P'_{i_*+1}=P''$ in the upper part; making the contents of the upper part of $P'_{i_*+1}$ in $S'_3$ the same as those in $S''_2$.
We have seen that the upper parts of $S'_3$ and $S''_3$ coincide, and that the terminating row numbers produced by the three operations and $D^{(i_*-1)}$ and $D^{(i_*)}$ are switched between the second and the third operations, so the statement of Lemma~\ref{Lem-invol-tab-one-exchange-first-contact-upper} is validated in this case.
\end{proof}

The conclusions of Lemmas~\ref{Lem-invol-tab-one-exchange-first-contact-lower}, \ref{Lem-invol-tab-one-exchange-first-contact-key-row} and \ref{Lem-invol-tab-one-exchange-first-contact-upper} validate Lemma~\ref{Lem-invol-tab-one-exchange-first-contact}, and hence the statements (1) and (2) of Lemma~\ref{Lem-invol-tab-one-exchange} for $i=i_*$.
\end{proof}

Continuing with the proof of Lemma~\ref{Lem-invol-tab-one-exchange}, we show that the remaining cases where $i_*<i<\lambda_{n-1}$ can be proved by repeating the same argument as in the case of first contact.

\begin{Lemma}\label{Lem-invol-tab-one-exchange-after-contact}
Let $i$ be such that $i_*<i<\lambda_{n-1}$.
Then $(\delta_{n-1}^{i-1}\cdots\delta_{n-1}^{1}S^-)\cup_{(n,1)}\boxed{e}$, $e=S(n,1)$, is an LR tableau satisfying the assumptions of Proposition~\ref{Prop-invol-tab-erase-one-ind}, with first contact index equal to $1$.
Moreover, the statements (1) and (2) of Lemma~\ref{Lem-invol-tab-one-exchange} for $S$ hold for these values of $i$.
\end{Lemma}

\begin{proof}
Note that the tableau in the statement can be written as $\delta_{n-1}^{i-1}\cdots\delta_{n-1}^{1}S$.
The initial part of this sequence of deletions may be of type (i), namely removals of letters $n-1$ from row $n-1$, but this cannot cause violation of the LR property since there are no entries $n$ in row $n$ by the assumption on $S$ in Proposition~\ref{Prop-invol-tab-erase-one-ind}.
The remaining operations preserve the LR property due to Lemmas~\ref{Lem-ss} and \ref{Lem-lp}.
Since row $n$ of this tableau is the same as that of $S$, it also satisfies the assumption of Proposition~\ref{Prop-invol-tab-erase-one-ind}.

Now we consider the case $i=i_*+1$, and denote this tableau by $S^*$ for simplicity.
If we apply Proposition~\ref{Prop-invol-tab-erase-one-ind} to $S^*$, then $S^-$ in the statement of Proposition~\ref{Prop-invol-tab-erase-one-ind} is $\delta_{n-1}^{i_*}\cdots\delta_{n-1}^{1}S^-$ which appears as $S''_1$ in Eq.~(\ref{Eq-Di_*-1andi_*3}), and the paths for $S^*$ playing the roles of $P'$, $P'_1$ and $P'_2$ are the deletion path of $\hat\delta_{n,1}$ applied to $S''_1$, that of $\delta_{n-1}^{i_*+1}$ applied to $S''_2$ and that of $\delta_{n-1}^{i_*+2}$ applied to $S''_3$ respectively, with $S''_2$ and $S''_3$ also as in Eq.~(\ref{Eq-Di_*-1andi_*3}).
This third path is actually $P'_{i_*+2}$ since $S''_3=S'_3$.
Hence, by what we have seen in Lemma~\ref{Lem-invol-tab-one-exchange-first-contact-key-row} applied to $S$, the paths playing the roles of $P'$ and $P'_1$ for $S^*$ share a cell, $(r_*,c_*-1)$ in the notation for the proof of Lemma~\ref{Lem-invol-tab-one-exchange-first-contact} (even though the cell of first contact may be in a lower row), and so the first contact index for $S^*$ equals $1$.
By applying Lemma~\ref{Lem-invol-tab-one-exchange-first-contact} to $S^*$, we conclude that exchanging the order of $\hat\delta_{n,1}$ and $\delta_{n-1}^{i_*+1}$ applied to $S''_1$ results in the same tableau after further applying $\delta_{n-1}^{i_*+2}$, with the terminating row numbers of the second and the third of these three operations switched; namely that the statements (1) and (2) of Lemma~\ref{Lem-invol-tab-one-exchange} hold for $i=i_*+1$.

If $i_*+1<\lambda_{n-1}-1$, then we can repeat the same argument since the situation can again be regarded as that after the case of first contact, until we have shown the statements (1) and (2) of Lemma~\ref{Lem-invol-tab-one-exchange} for all $i\leq\lambda_{n-1}-1$.
\end{proof}

Lemmas~\ref{Lem-invol-tab-one-exchange-before-contact}, \ref{Lem-invol-tab-one-exchange-first-contact} and \ref{Lem-invol-tab-one-exchange-after-contact} cover all values of $0\leq i<\lambda_{n-1}$ in statements (1) and (2) of Lemma~\ref{Lem-invol-tab-one-exchange}, and so the proof of Lemma~\ref{Lem-invol-tab-one-exchange} is finished in the case $i_*<\lambda_{n-1}$ as well.
\end{proof}

Finally, as we have anticipated in Remark~\ref{Rem-del-seq-final-part}, we have to compare
\[
D^{(\lambda_{n-1}-1)}=\delta_{n-1}^{\lambda_{n-1}}\hat\delta_{n,1}\delta_{n-1}^{\lambda_{n-1}-1}\cdots\delta_{n-1}^{1}\quad\text{with}\quad\hat\delta_{n-1,1}\delta_{n-1}=\hat\delta_{n-1,1}\delta_{n-1}^{\lambda_{n-1}}\delta_{n-1}^{\lambda_{n-1}-1}\cdots\delta_{n-1}^{1}.
\]

\begin{Lemma}\label{Lem-del-seq-final-part}
We have $D^{(\lambda_{n-1}-1)}S^-=\hat\delta_{n-1,1}\delta_{n-1}S^-$.
Moreover, in applying $\hat\delta_{n-1,1}\delta_{n-1}$ to $S^-$, the sequence of terminating row numbers produced by $\delta_{n-1}$, $R''$ by definition, coincides with the sequence obtained by replacing the rightmost $l''$ in $R'$ with $l'$, and the terminating row number produced by $\hat\delta_{n-1,1}$ coincides with $l''$.
\end{Lemma}

\begin{proof}
Note that the part $\delta_{n-1}^{\lambda_{n-1}-1}\cdots\delta_{n-1}^{1}$ is common to both $D^{(\lambda_{n-1}-1)}$ and $\hat\delta_{n-1,1}\delta_{n-1}$.
Let $S^\#$ denote the result of applying this part to $S^-$.
Then $S^\#$ has only one cell $(n-1,1)$ in row $n-1$.
Hence, in $D^{(\lambda_{n-1}-1)}$, the next operation $\hat\delta_{n,1}$ sends $e$, which it first places in the cell $(n,1)$, to $(n-1,1)$, which continues to incur the same transformation as the next operation $\delta_{n-1}^{\lambda_{n-1}}$ in $\hat\delta_{n-1,1}\delta_{n-1}$ incurs on $S^\#$.
So these two operations produce the same terminating row numbers, and we have $\hat\delta_{n,1}S^\#=(\delta_{n-1}^{\lambda_{n-1}}S^\#)\cup_{(n-1,1)}\boxed{e}$.
Hence the final operation $\delta_{n-1}^{\lambda_{n-1}}$ in $D^{(\lambda_{n-1}-1)}$, applied to $\hat\delta_{n,1}S^\#$, is identical with the $\delta_{n-1,1}$ part of the final operation $\hat\delta_{n-1,1}$ in $\hat\delta_{n-1,1}\delta_{n-1}$, applied to $\delta_{n-1}^{\lambda_{n-1}}S^\#$, so that we have $\delta_{n-1}^{\lambda_{n-1}}\hat\delta_{n,1}S^\#=\hat\delta_{n-1,1}\delta_{n-1}^{\lambda_{n-1}}S^\#$ namely $D^{(\lambda_{n-1}-1)}S^-=\hat\delta_{n-1,1}\delta_{n-1}S^-$, with the final operations also giving the same terminating row numbers.
Hence the sequences of terminating row numbers of $D^{(\lambda_{n-1}-1)}$ and $\hat\delta_{n-1,1}\delta_{n-1}$, both applied to $S^-$, are the same.
The former sequence is $R^{(\lambda_{n-1}-1)}$ by definition, which is obtained from $R'$ by replacing the rightmost $l''$ with $l'$ and then appending $l''$ due to part (2) of Lemma~\ref{Lem-invol-tab-one-exchange} with $i=\lambda_{n-1}-1$.
Hence the part concerning the terminating row numbers is also verified.
\end{proof}%%% end of Lemma~\ref{Lem-del-seq-final-part}

This completes the rather long proof of Proposition~\ref{Prop-invol-tab-erase-one-ind}.

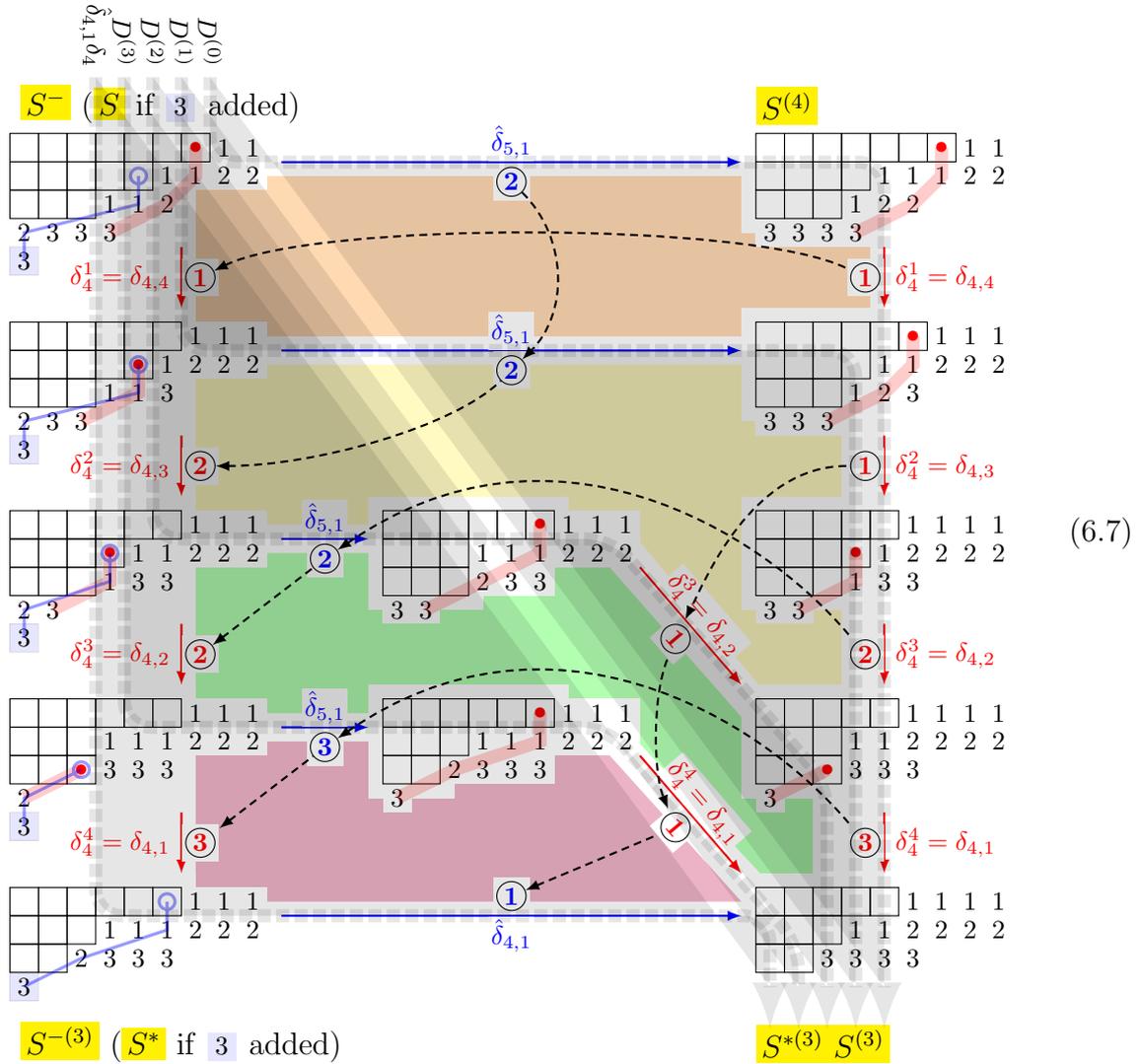
\begin{Example}\label{Ex-delaying-subdiv}
We use the same $S$ as in Example~\ref{Ex-delaying-to-erase-row}.
We illustrate in this case how the delay of the bottom cell deletion from $\delta_4\hat\delta_{5,1}$ to $\hat\delta_{4,1}\delta_4$, namely the act of diverting around the upper rectangle in \thetag{\ref{Eq-invol-tab-erase-one-ex}}, accompanied by the mutation of the resulting sequence of terminating row numbers, can be interpolated by considering the operations $D^{(i)}$ introduced in Lemma~\ref{Lem-invol-tab-one-exchange}.

Let us first show in the diagram \thetag{\ref{Eq-delaying-subdiv-paths}} the dim red paths $P'_1$, $P'_2$, $P'_3$ and $P'_4$, from right to left, taken by the operations $\delta_4^1$, $\delta_4^2$, $\delta_4^3$ and $\delta_4^4$ to be applied successively to $S^-$ and, for easy visual recognition, an additional very dim blue cell $(5,1)$ with entry $3$ to make it $S$ and the slightly dim blue path $P'$ taken by the operation $\delta_{5,1}$ applied to $S$.
This shows that $i_*=2$ with $(3,5)$ being the cell of first contact.

\begin{equation}\label{Eq-delaying-subdiv-paths}
\vcenter{\hbox{
\begin{tikzpicture}
[x={(0in,-0.15in)},y={(0.15in,0in)},
 deletion path/.style={red,line width=3pt,opacity=0.2},
 bottom deletion path/.style={blue,thick,rounded corners=6pt},
 application path/.style={dotted,rounded corners=10pt,opacity=0.25,
                          line width=3pt}
]
%%%%% S^-/S
 \foreach \j in {1,...,7} \draw (0,\j-1) rectangle (1,\j);
 \foreach \j/\x in {8/1,9/1} \node[font=\footnotesize] at (0.5,\j-0.5) {$\x$};
 \foreach \j in {1,...,5} \draw (1,\j-1) rectangle (2,\j);
 \foreach \j/\x in {6/1,7/1,8/2,9/2}
  \node[font=\footnotesize] at (1.5,\j-0.5) {$\x$};
 \foreach \j in {1,2,3} \draw (2,\j-1) rectangle (3,\j);
 \foreach \j/\x in {4/1,5/1,6/2}
  \node[font=\footnotesize] at (2.5,\j-0.5) {$\x$};
 \foreach \j/\x in {1/2,2/3,3/3,4/3}
  \node[font=\footnotesize] at (3.5,\j-0.5) {$\x$};
 \draw[deletion path](3.5,3.5)--(2.5,5.5)--(1.5,6.5)--(0.5,6.5)
  node[font=\footnotesize,opaque]{$\bullet$};
 \draw[deletion path](3.5,2.5)--(2.5,4.5)--(1.5,4.5)
  node[font=\footnotesize,opaque]{$\bullet$};
 \draw[deletion path](3.5,1.5)--(2.5,3.5)--(1.5,3.5)
  node[font=\footnotesize,opaque]{$\bullet$};
 \draw[deletion path](3.5,0.5)--(2.5,2.5)
  node[font=\footnotesize,opaque]{$\bullet$};
 \draw[fill=blue,opacity=0.1](4,0) rectangle +(1,1); 
 \node[font=\footnotesize] at (4.5,0.5) {3};
 \draw[bottom deletion path](4.5,0.5)--(3.5,0.5)--(2.5,4.5)--(1.5,4.5)
  circle[radius=3pt];
\end{tikzpicture}
}}
\end{equation}

The diagram \thetag{\ref{Eq-delaying-subdiv-tableaux}} shows the mechanism mutating the sequences of terminating row numbers obtained from varying routes interpolating between $\delta_{4}{\blue\hat\delta_{5,1}}=\delta_{4}^{4}\delta_{4}^{3}\delta_{4}^{2}\delta_{4}^{1}{\blue\hat\delta_{5,1}}$ and ${\blue\hat\delta_{4,1}}\delta_{4}={\blue\hat\delta_{4,1}}\delta_{4}^{4}\delta_{4}^{3}\delta_{4}^{2}\delta_{4}^{1}$ by subdividing the process of delaying the bottom cell deletion, namely \begin{equation}\label{Eq-exch-pts}
\vcenter{\hbox{
\begin{tikzpicture}
\node{$
\begin{alignedat}{3}
R^{(0)}={}&(2,1,1,2,3)\quad\text{from}\quad&&\text{$D^{(0)}$ i.e.\ }&&\text{${\blue\hat\delta_{5,1}},\delta_{4}^{1},\delta_{4}^{2},\delta_{4}^{3},\delta_{4}^{4}$ in the order of application},\\
R^{(1)}={}&(1,2,1,2,3)&&D^{(1)}&&\delta_{4}^{1},{\blue\hat\delta_{5,1}},\delta_{4}^{2},\delta_{4}^{3},\delta_{4}^{4},\\
R^{(2)}={}&(1,2,2,1,3)&&D^{(2)}&&\delta_{4}^{1},\delta_{4}^{2},{\blue\hat\delta_{5,1}},\delta_{4}^{3},\delta_{4}^{4},\\
R^{(3)}={}&(1,2,2,3,1)&&D^{(3)}&&\text{$\delta_{4}^{1},\delta_{4}^{2},\delta_{4}^{3},{\blue\hat\delta_{5,1}},\delta_{4}^{4}$ and}\\
&(1,2,2,3,1)&&\hat\delta_{4,1}\delta_{4}&&\delta_{4}^{1},\delta_{4}^{2},\delta_{4}^{3},\delta_{4}^{4},{\blue\hat\delta_{4,1}},
\end{alignedat}
$};
\foreach \position in {(-5.4,1.05),(-0.05,1.075),
                       (-4.625,0.35),(0.5,0.375),
                       (-4.22,-0.35),(1.05,-0.325),
                                      (1.6,-1.025)}
 \node at \position {$\times$};
\end{tikzpicture}
}}
\end{equation} whose exchange points, indicated by $\times$ in (\ref{Eq-exch-pts}) on the left, differ delicately from 
those of the operations themselves, indicated by $\times$ in (\ref{Eq-exch-pts}) on the right, in the case of first contact and afterwards.
The mechanism involves diverting around the dim orange rectangle as described in Lemma~\ref{Lem-invol-tab-one-exchange-before-contact} in the transition from $\delta_{4}\hat\delta_{5,1}=D^{(0)}$ to $D^{(1)}$ (case $i<i_*=2$), diverting around the dim yellow hexagon as described in Lemma~\ref{Lem-invol-tab-one-exchange-first-contact} in the transition from $D^{(1)}$ to $D^{(2)}$ (case $i=i_*$), diverting around the dim green hexagon again as described in Lemma~\ref{Lem-invol-tab-one-exchange-first-contact} invoked through Lemma~\ref{Lem-invol-tab-one-exchange-after-contact} in the transition from $D^{(2)}$ to $D^{(3)}$ (case $i>i_*$), and finally diverting around the dim purple trapezium as described in Lemma~\ref{Lem-del-seq-final-part} in the transition from $D^{(3)}$ to $\hat\delta_{4,1}\delta_{4}$.

\begin{equation}\label{Eq-delaying-subdiv-tableaux}
\vcenter{\hbox{
\pgfdeclarelayer{background}
\pgfsetlayers{background,main}

}}
\end{equation}

The $\vcenter{\hbox{\tikz\draw[-latex,black,thick,densely dashed](0,0)--(1,0);}}$ arrows in \thetag{\ref{Eq-delaying-subdiv-tableaux}} show how terminating row numbers exchange along with the change of order in which operators are applied.
In the case of diversion around a rectangle, such as the dim orange patch in \thetag{\ref{Eq-delaying-subdiv-tableaux}}, the exchange of the terminating row numbers exactly accompanies the exchange of operators, while in the case of diversion around a hexagon such as the dim yellow and green patches, representing the exchange between the first and second terms in the three term sequence of operators to apply, the exchange of terminating row numbers occurs between the second and third terms in the corresponding three term sequence of terminating row numbers.
The diversion around the lowermost dim purple trapezium results in essentially the same sequence of transformations, causing no change in the resulting terminating row numbers.

This whole diagram \thetag{\ref{Eq-delaying-subdiv-tableaux}} is to fit inside the top rectangle of the drawing \thetag{\ref{Eq-invol-tab-erase-one-ex}}.
In the notation of Proposition~\ref{Prop-invol-tab-erase-one-ind} and Example~\ref{Ex-delaying-to-erase-row}, the terminating row number produced from the top $\blue\hat\delta_{5,1}$ arrow is $l'$, those produced from the arrows on the right-hand (resp.\ left-hand) edge of \thetag{\ref{Eq-delaying-subdiv-tableaux}} form the sequence $R'$ (resp.\ $R''$), $l''$ is the term of $R'$ produced by the second arrow from the top in this case.
One can observe that, through the local changes described above, $l'$ sneaks into the position of $l''$ in $R'$ while most of $R'$ is copied to $R''$, $l''$ being bumped out of the sequence and sent to the bottom $\blue\hat\delta_{4,1}$ edge.
\end{Example}

%%%%%%%%%%%%%%%%%%%%%%%%%%%%%%%%%%%%%%%%%%%%%%%%%%%%%%%%%%%%%%%%%%%%%%
%%%% chapters 6,7,8 from master-akt20oct14.tex on hive based proofs
%%%%%%%%%%%%%%%%%%%%%%%%%%%%%%%%%%%%%%%%%%%%%%%%%%%%%%%%%%%%%%%%%%%%%%%
%%%%%%%%%%%%%%%%%%%%%%%%%%%%%%%%%%%%%%%%%%
\section{Path removals from a hive}
\label{Sec-hive-path-removal}

By virtue of the bijection between Littlewood--Richardson tableaux and hives, the action
of deletion operators on tableaux corresponds to a path removal procedure in hives. 
Under this bijection the three types of action (i), (ii) and (iii) of the deletion 
operator $\delta_{r,\lambda_r}$ on an LR tableau $T$ with a corner cell at $(r,\lambda_r)$ as
described in Definition~\ref{Def-dp} correspond to %the 
three types of hive path removal that we now define. 

It should be recalled first that in any given hive a \textit{diagonal}
consists of all the triangles in a strip parallel to the 
right hand boundary of the hive lying between an edge on the base and a corresponding edge
on the left hand boundary. Our convention is that the $p$th diagonal is the one bounded by
the edges $\lambda_p$ and $\mu_p$. Then by a \textit{path} we mean a connected set of edges
extending from an edge on the base to an edge on either the left or right hand boundary.
Such paths are generally zig-zag in nature and proceed either up a
diagonal or horizontally leftwards from one diagonal to another. 
They consist of pairs of edges taken from a sequence of neighbouring triangles always
including their common edge.

\begin{Definition}\label{Def-hive-paths}
For any given $H\in{\cal H}^{(r)}(\lambda,\mu,\nu)$ with $r=\ell(\lambda)$ we may define 
three path removal operators $\chi_r$, $\phi_r$ and $\omega_r$ whose action on $H$
is to reduce or increase edge labels by $1$ along a path starting from
the edge labelled $\lambda_r$ on the base of the hive and specified as follows:
\begin{enumerate}
\item[(i)] $\chi_r$: if $\nu_r>0$ then the path consists of the edges labelled $\lambda_r$ 
and $\nu_r$, with both edge labels decreased by $1$;
\item[(ii)] $\phi_r$: if $\lambda_r-\mu_r-\nu_r>0$ so that $U_{ir}>0$ for some $i<r$ then the path proceeds up the $r$th diagonal 
from the edge labelled $\lambda_r$ through upright rhombi of gradient $0$ until it encounters an 
upright rhombus of positive gradient, at which point it moves horizontally to the left 
into the $(r-1)$th diagonal and proceeds up this diagonal or to the left as before, 
and so on until it terminates on the left-hand boundary at the top of the $k$th diagonal, 
that is at the edge labelled $\mu_k$ for some $k$ such that $1\leq k<r$,
with all path $\alpha$ and $\gamma$ edge
labels being decreased by $1$ and all path $\beta$ edge labels increased by $1$;
\item[(iii)] $\omega_r$: if $\mu_r>0$ then the path proceeds directly up the $r$th
diagonal until it terminates on the left-hand 
boundary at level $r$, that is at the edge labelled $\mu_r$, with all path edge labels decreased by $1$.
\end{enumerate}
\end{Definition}

The three types of path are illustrated below. In each case the action of $\chi_r$,
$\phi_r$ and $\omega_r$ on the hive $H$ is to decrease the label of 
each {\bf\red red} edge by $1$ and to increase that of each {\bf\blue blue} edge by $1$. 
In particular, under this action the edge label $\lambda_r$
and one or other of $\nu_r$, $\mu_k$ (with $k<r$) or $\mu_r$
are each reduced by $1$ to $\lambda_r-1$ and $\nu_r-1$, $\mu_k-1$ or $\mu_r-1$, respectively.
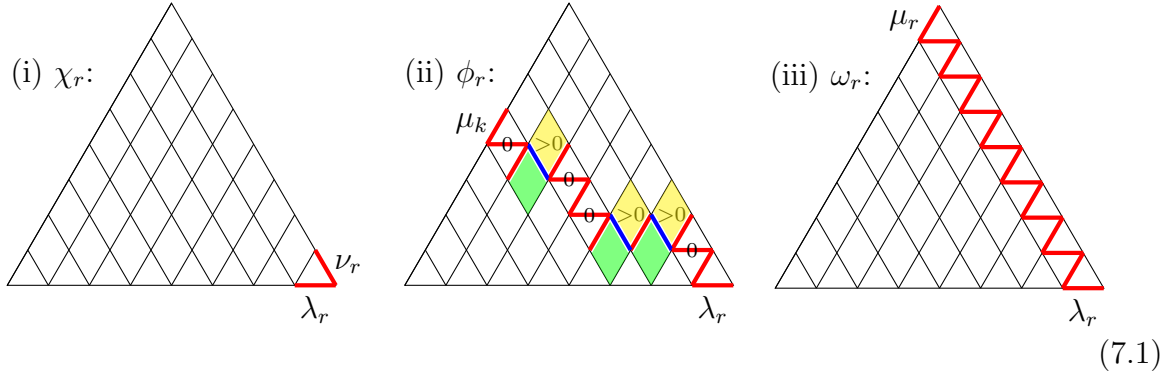
\begin{equation}\label{Eq-hpr}
%(i)
%\hskip-0.5cm
\vcenter{\hbox{

}}
\end{equation}
While the structure of the paths is rather simple in cases (i) and (iii), %it might be noted that 
the structure of the path in case (ii) is more complicated and consists of a sequence of \textit{ladders} in each diagonal from
the $r$th to the $k$th. Each ladder consists of a continuous zig-zag of red $\alpha$- or $\gamma$-edges passing through upright rhombi
of gradient $0$ that extends up the diagonal from an edge that is either on the base of the hive or the top of an upright rhombus 
that we call the \textit{foot rhombus} (coloured green), to an edge that is either on the left hand boundary
or the bottom of an upright rhombus that we call the \textit{head rhombus} (coloured yellow), as illustrated below:
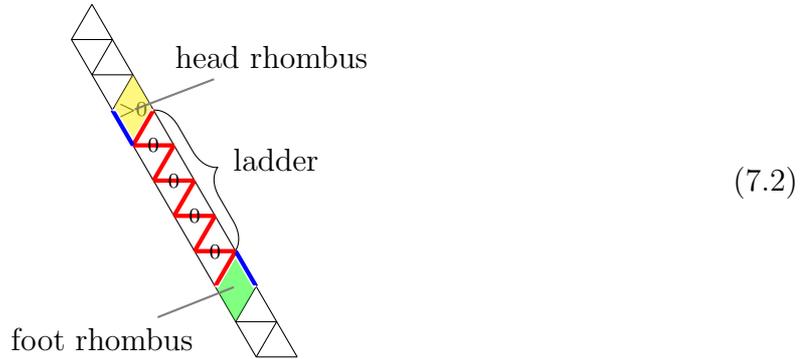
\begin{equation} \label{Eq-ladder}
\vcenter{\hbox{
\begin{tikzpicture}[x={(1cm*0.27,-\rootthree cm*0.27)},
                     y={(1cm*0.27,\rootthree cm*0.27)}]
%\tikz[x={(0.4cm,-0.4cm)},y={(0.4cm,0.4cm)}]{
\foreach\n/\f/\fmone/\fmtwo/\fp in {11/9/8/7/5} {   %{13/9/8/7/5}
 \draw[decorate,decoration={brace,amplitude=10pt}](\fp,\n+1)--(\f,\n+1) node[pos=0.35,right=15pt]{ladder};
 \draw(2,\n)--(\n,\n) (2,\n+1)--(\n+1,\n+1);
 \foreach\i in {2,3,...,\n} \draw(\i,\n)--(\i,\n+1) (\i,\n)--(\i+1,\n+1);
 \draw[red,ultra thick](\f,\n)--(\f,\n+1);
 \draw[blue,ultra thick](\f,\n+1)--(\f+1,\n+1);
 \foreach\i in {\fmone,\fmtwo,...,\fp} {
  \draw[red,ultra thick](\i+1,\n+1)--(\i,\n);
	\draw[red,ultra thick](\i,\n)--(\i,\n+1); }
 \draw[blue,ultra thick](\fp-1,\n)--(\fp,\n);
 \draw[white,thick](\f,\n)--(\f+1,\n+1) (\fp-1,\n)--(\fp,\n+1);
 %ladder rhombi
 \foreach\i in {\fmone,\fmtwo,...,\fp} {\path(\i+1,\n+1)--node{$\sc{0}$}(\i,\n); }
 \path(\fp,\n+1)--node{$\sc{>0}$}(\fp-1,\n);
 %foot rhombus
 \path[fill=green,fill opacity=0.5] (\f+0.1,\n)--(\f+0.1,\n+1-0.1)--(\f+1,\n+1-0.1)--(\f+1,\n)--cycle;
 %head rhombus
 \path[fill=yellow,fill opacity=0.5] (\fp-1,\n+0.1)--(\fp-1,\n+1)--(\fp-0.1,\n+1)--(\fp-0.1,\n+0.1)--cycle;
 %\fill[pattern=horizontal lines] (\f,\n)--(\f,\n+1)--(\f+1,\n+1)--(\f+1,\n)--cycle;
 %\fill[pattern=horizontal lines] (\fp-1,\n)--(\fp-1,\n+1)--(\fp,\n+1)--(\fp,\n)--cycle;
 \node[inner sep=0pt,pin={[pin edge=thick]below left:{foot rhombus}}]
   at (\f+0.5,\n+0.5) {};
 \node[inner sep=0pt,pin={[pin edge=thick]above right:{head rhombus}}]
   at (\fp-0.5,\n+0.5) {};
}
\end{tikzpicture}
}}
\end{equation}
The length of a ladder is the number of its rungs, that is to say the number of horizontal {\bf\red red} $\gamma$-edges.
Thus the ladders starting at the base of the hive lack a foot rhombus, while those ending on the left hand 
boundary lack a head rhombus. It might be noted that all ladders, including these, may consist
of a single {\bf\red red} $\alpha$-edge, in which case the length is $0$. 
The passage between one diagonal and the next is by way of a {\bf\blue blue}
$\beta$-edge parallel to the right hand boundary. If such an edge is the $l$th from the top 
of the diagonal then its \textit{level} is said to be $l$.

Before using these path removals we first establish that the action of each of the path 
removal operators preserves the hive conditions, as follows:

\begin{Lemma}\label{Lem-hpr-hive}
Let $H$ be a hive in ${\cal H}^{(r)}(\lambda,\mu,\nu)$ with $r=\ell(\lambda)$.
Then the path removal operators are such that if we set $\tilde{H}=\chi_r\,H$, $\phi_r\,H$ or $\omega_r\,H$
with $\nu_r>0$, $\lambda_r-\mu_r-\nu_r>0$ and $\mu_r>0$, respectively, then in each case $\tilde{H}$ is a hive. 
\end{Lemma}

\noindent{\bf Proof}:~~
It is easy to see that the triangle conditions (\ref{Eq-triangle-condition}) are preserved under 
any of the three path removal procedures mapping a hive $H$ to $\tilde{H}$ by examining the  
changes to the edge labels of elementary triangles as illustrated by: 
\begin{equation}\label{Eq-triangles}
%1
\vcenter{\hbox{

}}
\end{equation}

Moreover all edge labels remain non-negative under the map from $H$ to $\tilde{H}$. This is obvious in the case of
all blue edge labels since they are increased by $1$ in the passage from $H$ to $\tilde{H}$. On the other
hand red edge labels are each decreased by $1$, so it is necessary to show that each red edge label in 
$H$ is positive. 
 
This is clear in case (i) since the only two red edges are those labelled $\lambda_r$ 
and $\nu_r$, which by hypothesis are positive. 
In case (ii) the fact that $\lambda_r>0$ implies that all $\gamma$-edges %horizontal edges 
in the first $r$ diagonals of $H$ carry positive labels, this includes all the red $\gamma$-edges %horizontal edge 
labels. This only leaves the red $\alpha$-edges. %{\red parallel to the left hand boundary}.
On any given $p$th diagonal these are %form the rungs of a ladder
separated by upright rhombi of gradient $0$, so they all carry the same label.  For $p>k$
this label must be positive since at the top of each the ladder there appears a head rhombus of 
positive gradient. 
%$>0$. This covers all diagonals beyond the $k$th. 
In the $k$th diagonal the 
shared red $\alpha$-edge label is $\mu_k$, and the fact that $\mu_k$ is also positive follows from the 
application of the betweenness condition to the configuration
\begin{equation}
\vcenter{\hbox{
\begin{tikzpicture}[x={(1cm*0.5,-1.7320508cm*0.5)},y={(1cm*0.5,1.7320508cm*0.5)}]
\path(0,0)--node[pos=0.4,left]{$\mu_k$}(0,1); %\alpha
\draw[red, ultra thick](0,0)--(0,1); \draw[black, thin](0,1)--(0,2); \draw[red, ultra thick](1,1)--(1,2); %\alpha
\draw[blue, ultra thick](0,1)--(1,1); \draw[black, thin](0,2)--(1,2); %\beta
\draw[black, thin](0,0)--(1,1); %\gamma
\path(0,1)--node{$>0$}(1,2); %\gamma
%head rhombi filled
\foreach \i/\j/\k in {0/1/0.05}
\path[fill=yellow,fill opacity=0.5] (\i,\j+\k)--(\i,\j+1)--(\i+1-\k,\j+1)--(\i+1-\k,\j+\k)--cycle;
\end{tikzpicture}
}}
\end{equation}
In case (iii) the hive conditions imply that the labels of all red $\gamma$-edges % parallel to the base 
are $\geq \lambda_r$ and those of all red $\alpha$-edges %parallel to the left hand boundary 
are $\geq \mu_r$. These red edge labels must therefore all be positive since by hypothesis
$\lambda_r>0$ and $\mu_r>0$. 
Thus all red edge labels of $H$ are positive in all three cases, and as 
a consequence all edge labels of $\tilde{H}$ are non-negative integers.

Turning next to rhombi, as a result of the path removals the requisite changes in red and blue edge 
labels automatically change the gradients of certain elementary rhombi. In order that $\tilde{H}$ 
be an LR hive it is necessary to show that all these gradients remain non-negative.

In the case of a type (i) hive path removal the action of $\chi_r$ affects only one rhombus and does so as shown below:
\begin{equation}\label{Eq-LtoL-1i}
\vcenter{\hbox{
\begin{tikzpicture}[x={(1cm*0.5,-\rootthree cm*0.5)},
                    y={(1cm*0.5,\rootthree cm*0.5)}]
\draw(0,1)--(1,1);  
\draw[red,ultra thick](1,1)--node[black,below]{$\lambda_r$}(2,2); 
\draw(0,1)--(1,2); 
\draw[red,ultra thick](1,2)--node[black,right]{$\nu_r$}(2,2); 
\path(1,1)--node{$L$}(1,2);	
\end{tikzpicture}
}}
\quad \overset{\chi_r}{\longmapsto}\quad 
\vcenter{\hbox{
\begin{tikzpicture}[x={(1cm*0.5,-\rootthree cm*0.5)},
                    y={(1cm*0.5,\rootthree cm*0.5)}]
\draw(0,1)--(1,1);  
\draw[red,ultra thick](1,1)--node[black,below]{$\lambda_r\!\!-\!\!1$}(2,2); 
\draw(0,1)--(1,2); 
\draw[red,ultra thick](1,2)--node[black,right]{$\nu_r\!\!-\!\!1$}(2,2); 
\path(1,1)--node{$L\!\!+\!\!1$}(1,2);
\end{tikzpicture}
}}
\end{equation}
Clearly, the gradient of this rhombus is increased. It follows that all rhombus gradients
remain non-negative under the action of $\chi_r$.

The situation is more complicated for type (ii) hive path removals under the action of $\phi_r$. 
From the definition of a type (ii) path, the upright rhombi between
the foot and the head rhombi all have gradient $0$.
The only upright rhombus gradients that are changed under the action 
of $\phi_r$ are those of the foot and head rhombi which increase and decrease by $1$, respectively.
In the case of the type (ii)
example illustrated in (\ref{Eq-hpr}) the changes in upright rhombus gradients are then as indicated
below:
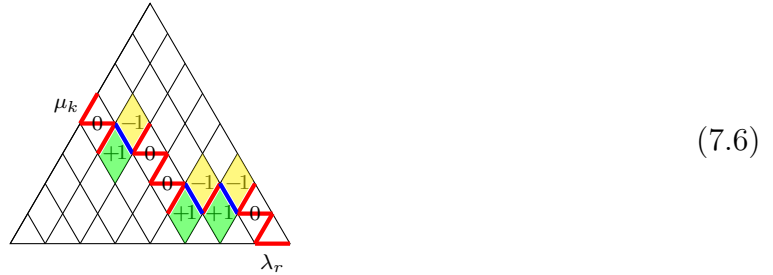
\begin{equation}\label{Eq-Uchanges}
%(ii)
\vcenter{\hbox{

}}
\end{equation}
The only reductions in gradient are thus of the form:   
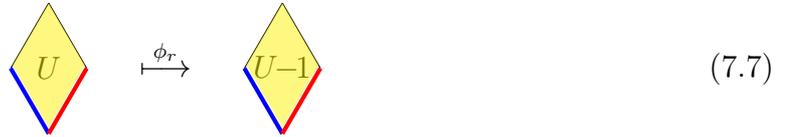
\begin{equation}\label{Eq-UtoU-1}
%UtoU-1
\vcenter{\hbox{
%
}}
\end{equation}
However, since each head rhombus necessarily has positive gradient this configuration 
only arises if $U>0$ so that $U-1\geq0$. 
As a result all upright rhombus gradients remain non-negative 
after a type (ii) hive path removal.

The only right-leaning rhombi that undergo a reduction in gradient under
a type (ii) hive path removal generated by the action of $\phi_r$ are those subject to the following transformation:
\begin{equation}\label{Eq-RtoR-1ii}
\vcenter{\hbox{
%
}}
\end{equation}
To preserve the validity of the corresponding hive condition
it is therefore necessary to show that in this case the initial
gradient $R=\alpha-\alpha'$ is positive. 
It can be seen from (\ref{Eq-Uchanges}) that the only cases
that arise are those of the following type:
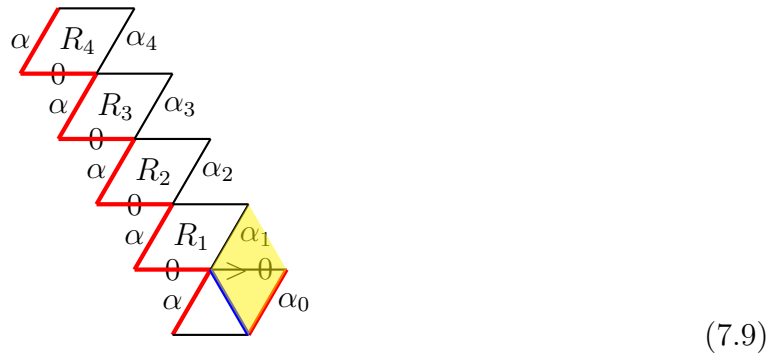
\begin{equation}
%
\end{equation}
The fact that the parallel red edge labels $\alpha$ on the left of the diagram are identical,
is a consequence of the gradients of the intervening upright rhombi all being $0$ as 
indicated by the $0$'s appearing on the horizontal red edges. On the right of the diagram
the fact that the relevant upright rhombi have non-negative gradients implies that
$\alpha_1\geq\alpha_2\geq\alpha_3\geq\alpha_4$, while the positivity of the gradient of the
upright rhombus indicated in the diagram by $>0$ on a horizontal black edge and coloured yellow
ensures from the betweenness conditions that $\alpha\geq\alpha_0>\alpha_1$. 
Hence $\alpha\geq\alpha_0>\alpha_1\geq\alpha_2\geq\alpha_3\geq\alpha_4$ so that
$R_k=\alpha-\alpha_k>0$ for all $k$, as required for the right-leaning rhombus condition to be
maintained after the path removal.

Still dealing with type (ii) hive path removals generated by the action of $\phi_r$, the only left-leaning rhombi that undergo a decrease 
in their gradients are those subject to the following transformation:
\begin{equation}\label{Eq-LtoL-1ii}
\vcenter{\hbox{

}}
\end{equation}
This case can only arise as part of the following hexagonal initial configuration:
\begin{equation}
%
\end{equation}
In terms of the given edge labels the initial gradient of the lower left-leaning rhombus of 
the type illustrated above is given by $L=\beta'-\beta$.
However the gradients of the two upright rhombi specified in the diagram as $0$ 
and $>0$ ensure that $\beta'=\beta''$ and $\beta'''>\beta$. 
In addition the gradient $\beta''-\beta'''$
of the upper left-leaning rhombus must be non-negative. Hence $\beta'=\beta''\geq\beta'''>\beta$.
It follows that $L=\beta'-\beta>0$ so that $L-1\geq0$.
This confirms that all rhombus gradients remain non-negative under the action of $\phi_r$.

Finally, for a hive path removal of type (iii) as illustrated on the right of (\ref{Eq-hpr})
the only rhombi whose gradients change are those undergoing the map:
\begin{equation}\label{Eq-RtoR-1iii}
\vcenter{\hbox{
\begin{tikzpicture}[x={(1cm*0.5,-\rootthree cm*0.5)},
                    y={(1cm*0.5,\rootthree cm*0.5)}]
\draw(0,0)--(0,1);  
\draw(0,0)--(1,1); 
\draw[red,ultra thick](0,1)--(1,2); 
\draw[red,ultra thick](1,1)--(1,2); 
\path(0,1)--node{$R$}(1,1);
\end{tikzpicture}
}}
\quad \overset{\omega_r}{\longmapsto}\quad %\ \mapsto\ 
\vcenter{\hbox{
\begin{tikzpicture}[x={(1cm*0.5,-\rootthree cm*0.5)},
                    y={(1cm*0.5,\rootthree cm*0.5)}]
\draw(0,0)--(0,1);  
\draw(0,0)--(1,1); 
\draw[red,ultra thick](0,1)--(1,2); 
\draw[red,ultra thick](1,1)--(1,2); 
\path(0,1)--node{$R\!\!+\!\!1$}(1,1);
\end{tikzpicture}
}}
\end{equation}
Thus all rhombus gradients remain non-negative under the action of $\omega_r$. 
\qed
\medskip

Very much as in the tableau case we next introduce 
\begin{Definition}\label{Def-hpr}
For any given $n$-hive $H\in{\cal H}^{(r)}(\lambda,\mu,\nu)$ with $r=\ell(\lambda)$,
the path removal operator $\theta_{r,\lambda_r}$ is defined by
\begin{equation}
       \theta_{r,\lambda_r} = \left\{ 
			\begin{array}{ll}
			  \chi_r&\quad\hbox{if $\nu_r>0$};\cr
				\phi_r&\quad\hbox{if $\nu_r=0$ and $U_{ir}>0$ for some $i<r$};\cr
				\omega_r&\quad\hbox{if $\nu_r=0$, $U_{ir}=0$ for all $i<r$, and $\mu_r>0$}.\cr
			\end{array} \right.
\end{equation}
\end{Definition}

With this definition of $\theta_{r,\lambda_r}$ we have

\begin{Lemma}\label{Lem-hpr}
Let $\ell(\lambda)=r$. Then the hive path removal operator $\theta_{r,\lambda_r}$ maps 
any given $H\in{\cal H}^{(r)}(\lambda,\mu,\nu)$ to $\tilde{H}$ with
(i) $\tilde{H}\in{\cal H}^{(r)}((\lambda-\epsilon_r),\mu,(\nu-\epsilon_r))$ if $\nu_r>0$;
(ii) $\tilde{H}\in{\cal H}^{(r)}((\lambda-\epsilon_r),(\mu-\epsilon_k),\nu)$ with $1\leq k<r$ if $\nu_r=0$
and $U_{ir}>0$ for some $i<r$;
(iii) $\tilde{H}\in{\cal H}^{(r)}((\lambda-\epsilon_r),(\mu-\epsilon_r),\nu)$ if $\nu_r=0$, $U_{ir}=0$ for all $i<r$.
\end{Lemma}

\noindent{\bf Proof}:
In each case we have established in Lemma~\ref{Lem-hpr-hive} that $\tilde{H}$ is a hive. It only remains to
consider the changes of the hive boundary labels. It can be seen from (\ref{Eq-hpr}) that the only changes are 
a reduction by $1$ of the labels $\lambda_r$ and $\nu_r$ in case (i), $\lambda_r$ and $\mu_k$ in case (ii), and 
$\lambda_r$ and $\mu_r$ in case (iii).
\qed 
\medskip

Again very much as in the tableau case, we next introduce 
\begin{Definition}\label{Def-full-r-hpr}
For any given hive $H\in{\cal H}^{(r)}(\lambda,\mu,\nu)$ with $\ell(\lambda)\leq r$ 
the {\it full $r$-hive path removal operator\/} $\theta_r$ is defined by
\begin{equation}\label{Eq-theta-r1r2}
\theta_{r}:=\theta_{r,1}\theta_{r,2}\cdots\theta_{r,\lambda_r}\,,
\end{equation}
where it is to be understood that $\theta_r$ acts as the identity on $H$ if $\ell(\lambda)<r$.
In general its action on $H$ proceeds in three successive phases: 
Phase 1 consists of $\nu_r$ hive path removals of type (i);
Phase 2 consists of $\lambda_r-\nu_r-\mu_r$ hive path removals of type (ii);
and Phase 3 consists of $\mu_r$ hive path removals of type (iii), so that 
\begin{equation}\label{Eq-theta-omega-phi-chi}
    \theta_{r}= \omega_r^{\mu_r}\ \phi_r^{\lambda_r-\mu_r-\nu_r}\ \chi_r^{\nu_r}\,.
\end{equation}
\end{Definition}

With this definition we have

\begin{Theorem}\label{The-Hr-theta-r}
For a hive $H\in{\cal H}^{(r)}(\lambda,\mu,\nu)$ with $\ell(\lambda)\leq r$ 
let $\theta_r H=\tilde{H}$ then 
$\tilde{H}\in{\cal H}^{(r)}(\tilde{\lambda},\tilde{\mu},\tilde{\nu})$
with $\tilde{\lambda}=(\lambda_1,\ldots,\lambda_{r-1},0)$, $\tilde{\mu}=(\mu_1\!-\!V_{1r},\ldots,\mu_{r-1}\!-\!V_{r-1,r},0)$ and 
$\tilde{\nu}=(\nu_1,\ldots,\nu_{r-1},0)$, where $V_{kr}$ is the number of type (ii)
hive path removals from $H$ that extend from the boundary edge initially labelled $\lambda_r$ 
to that initially labelled $\mu_k$ for $1\leq k<r$.
\end{Theorem} 

\noindent{\bf Proof}:~~ If $\ell(\lambda)<r$ then $\lambda_r=\mu_r=\nu_r=0$ and there are no path removals
so that $V_{kr}=0$ for all $k=1,2,\ldots,r-1$. This implies that $\tilde{\lambda}=\lambda$, $\tilde{\mu}=\mu$, $\tilde{\nu}=\nu$
and $\tilde{H}=\theta_r H=H\in{\cal H}^{(r)}(\tilde{\lambda},\tilde{\mu},\tilde{\nu})$, as required.
For $\ell(\lambda)=r$ the required result 
is an easy consequence of the iterated use of Lemma~\ref{Lem-hpr}. The edge label 
$\nu_r$ is reduced to $0$ in Phase 1. At the same time the edge label $\lambda_r$ is reduced
to $\lambda-\nu_r$ in Phase 1, then to $\lambda-\nu_r-\mu_r$ in Phase 2, and finally
to $0$ in Phase 3, in which the edge label $\mu_r$ is also reduced to $0$. Meanwhile
in Phase 2 the parameters $V_{kr}$ give the number of hive removal paths that reach the 
left-hand boundary edge initiall labelled $\mu_k$ thereby reducing this label to 
$\mu_k-V_{kr}$ for $k=1,2,\ldots,r-1$.
\qed

Two further observations are of use in what follows.

\begin{Lemma}\label{Lem-obs1}
In the action of $\theta_r$ on $H\in{\cal H}^{(r)}(\lambda,\mu,\nu)$ with $\ell(\lambda)=r$
if a hive path removal of type (ii) follows a path $P$ and reaches the left-hand boundary 
at level $k$, then the next such removal follows a path $P'$ lying weakly above the path
$P$ in each diagonal and reaches level $k'$ with $k'\geq k$.
\end{Lemma}

\noindent{\bf Proof}:~~
This can be seen by consideration of the following hive path removal diagram in which
two successive removal paths $P$ and $P'$ are illustrated, $P$ in {\bf\red red}, $P'$ in {\bf\blue blue}
and their overlap in {\bf\magenta magenta}.
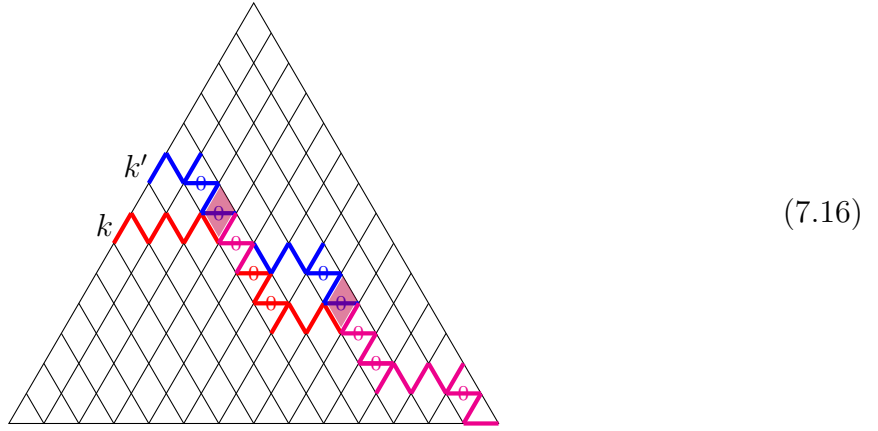
\begin{equation}\label{Eq-obs1}
\vcenter{\hbox{

}}
\end{equation}
The interpretation of this is that the first path $P$ is the union of the {\bf\red red} and {\bf\magenta magenta}
edges, while the second path $P'$ is the union of the {\bf\blue blue} and {\bf\magenta magenta} edges. 
Working from right to left the route of $P'$ necessarily coincides with that of $P$,
as shown in the initial {\bf\magenta magenta}section, unless and until it reaches 
an upright rhombus, indicated in {\bf\purple purple}, whose initial gradient $1$ was reduced to $0$ 
by the removal of the path $P$. At this point the paths separate but they may meet again, 
whereupon the path $P'$ again must coincide with that of $P$, as shown in the second
{\bf\magenta magenta} section, unless and until it reaches another such upright rhombus,
again indicated in {\bf\purple purple}, 
whose initial gradient $1$ was reduced to $0$ by the removal of the path $P$. 
In such a case the paths separate again, and so on. 
This observation 
guarantees that the path $P'$ lies weakly above that of $P$ and also that it
meets the left-hand boundary at an edge weakly above that reached by $P$.
\qed

\begin{Corollary}\label{Cor-obs1}
Under the action of $\theta_{r}$ on $H\in{\cal H}^{(r)}(\lambda,\mu,\nu)$ 
with $\ell(\lambda)=r$ let $V_{kr}$ be the number of type (ii) hive path removals 
reaching  the left-hand boundary at level $k$ for $1\leq k<r$. Then for each such $k$
\begin{equation}\label{Eq-cor-obs1}
        \mu_k\geq \mu_k-V_{kr} \geq \mu_{k+1}.
\end{equation}

\end{Corollary}

\noindent{\bf Proof}:~~
The first inequality follows from the fact that $V_{kr}\geq0$. Then, by virtue of Lemma~\ref{Lem-obs1}, 
the left-hand boundary edge label $\mu_{k+1}$ must remain unchanged under all $V_{kr}$ hive path removals 
reaching the edge labelled $\mu_k$. Each one of these decreases this label by $1$. The fact that the hive 
conditions are preserved under these hive path removals implies that on the left-hand boundary the $k$th edge label 
must remain greater than or equal to that of the $(k+1)$th edge label. Thus after all $V_{kr}$ path removals
we have $\mu_k-V_{kr} \geq \mu_{k+1}$, that is the second inequality of (\ref{Eq-cor-obs1}).
\qed

Our second observation is the following:

\begin{Lemma}\label{Lem-obs2}
Under the action of $\theta_{r}$ followed by $\theta_{r-1}$ on $H\in{\cal H}^{(r)}(\lambda,\mu,\nu)$ 
with $\ell(\lambda)=r$, let $N_{kr}$ and $N_{k-1,r-1}$ be the number of type (ii) hive path removals 
occurring in the action of $\theta_r$ and $\theta_{r-1}$ that reach 
the left-hand boundary at or below levels $k$ and $k-1$, respectively. Then 
\begin{equation}\label{Eq-obs2}
    N_{k-1,r-1}\geq N_{kr}-U_{r-1,r},
\end{equation}
where $U_{r-1,r}$ is the upright rhombus gradient at the foot of the $r$th diagonal of $H$.   
\end{Lemma}

\noindent{\bf Proof}:~~
It should be noted that the impact of type (ii) hive path removals in Phase 2 of the action 
of $\theta_r$ on $H$ is typically as shown below.
\begin{equation}\label{Eq-obs2-hives}
\vcenter{\hbox{

}}
\end{equation}
We have illustrated two distinct types of path: on the left a 
path that passes from the $r$th to the $(r-1)$th diagonal beneath
the upright rhombus at the foot of the $r$th diagonal with initial
gradient $U_{r-1,r}>0$ reduced to $U_{r-1,r}-1\geq0$ by the hive path removal, 
and on the right a path that first passes up the $r$th diagonal 
through at least one upright rhombus of gradient $0$ before entering
the $(r-1)$th diagonal. The number of the former is precisely $U_{r-1,r}$
and these are necessarily removed before any of the latter.
Each of the latter leaves immediately below it in each diagonal an upright rhombus
whose gradient is increased by $1$. These increases are cumulative
since successive type (ii) hive path removals always lie weakly above their 
predecessors, by virtue of Lemma~\ref{Lem-obs1}. 

Now consider the $N_{kr}$ type (ii) hive path removals that extend from the $r$th 
to the $j$th diagonal with $1\leq j\leq k$.  
If $N_{kr}\leq U_{r-1,r}$ then all these path removals are of the form shown above on the left. 
In such a case $N_{k-1,r-1}\geq 0\geq N_{kr}-U_{r-1,r}$ as required. If $N_{kr}>U_{r-1,r}$ then 
the first $U_{r-1,r}$ of these are again as shown on the left. Let the sequence of 
the remaining $n_{kr}=N_{kr}-U_{r-1,r}$ of path removals of the type shown on the right be
$P_1,P_2,\ldots,P_{n_{kr}}$, with $P_{i+1}$ lying weakly above $P_i$ in accordance with Lemma~\ref{Lem-obs1}
for $i=1,2,\ldots,n_{kr}-1$. 
It follows that the sum of upright rhombus gradients in the $(r-1)$th diagonal
of $\theta_r H$ is at least $n_{kr}$.
The action of $\theta_{r-1}$ must reduce this sum of gradients to $0$. 
It does so through a sequence of at least $n_{kr}$ hive path removals $P'_1,P'_2,\ldots,P'_{n_{kr}}$ 
of type (ii), again with $P'_{i+1}$ lying weakly above $P'_i$ in accordance wth  Lemma~\ref{Lem-obs1}.
Consider first the relationship between $P_1$ and $P'_1$. We wish to show that $P'_1$ lies strictly below $P_1$.
Under the action of $\theta_r$, in each diagonal 
it enters, the path removal $P_1$ leaves immediately beneath it an upright rhombus of gradient $\geq1$,
moving weakly upwards as the path traverses the diagonals westwards (see the green upright rhombi in the right-hand picture of (\ref{Eq-obs2-hives})).
This situation is unaltered by the removal of all subsequent type (ii) paths $P_j$ with $j>1$ since these 
all lie weakly above $P_1$. It also remains unaltered by the removal of any necessary type (iii) paths 
under the action of $\theta_r$ and any necessary type (i) paths under the action of $\theta_{r-1}$. 
Hence the first type (ii) path $P'_1$ removed under the action of $\theta_{r-1}$, starting from the bottom edge of diagonal $r-1$, 
must leave that diagonal below the green rhombus in that diagonal.
This puts $P'_1$ in diagonal $r-2$ below the green rhombus in diagonal $r-2$, and forces it to leave that diagonal below that rhombus.
The situation persists, and we conclude, as intended, that $P'_1$ does lie strictly below $P_1$.
This argument may be extended. Note that, following the removal of $P'_1$, there will 
still remain in each diagonal entered by $P_2$ an upright rhombus of gradient $\geq1$ immediately below
the path $P_2$. 
This is easy to see if the transition of $P_2$ into that diagonal occurs strictly above that of $P_1$.
If $P_2$ enters that diagonal at the same level as $P_1$, then the upright rhombus in that diagonal immediately 
below them will have gradient $\geq2$, and even if the removal of $P'_1$ reduces that gradient by one it remains $\geq1$.
It then follows as before that $P'_2$ lies strictly below $P_2$. Continuing
in this way it is clear that $P'_i$ lies strictly below $P_i$ for all $i=1,2,\ldots,n_{kr}$.
Therefore at least $n_{kr}$ type (ii) paths removed under the action of $\theta_{r-1}$
meet the left hand boundary at a level strictly below the level reached by $P_{n_{kr}}$, which is itself $\leq k$.
It follows that $N_{k-1,r-1}\geq n_{kr}=N_{kr}-U_{r-1,r}$, as claimed.
\qed 

%%%%%%%%%%%Startof section 8

\section{Path removal map from ${\cal H}^{(n)}(\lambda,\mu,\nu)$ to ${\cal H}^{(n)}(\lambda,\nu,\mu)$}
\label{Sec-hive-bijection}

Armed with our path removal procedures we are able to exploit them to construct
from any hive $H\in{\cal H}^{(n)}(\lambda,\mu,\nu)$ a partner hive $K\in{\cal H}^{(n)}(\lambda,\nu,\mu)$.
To do so it is only necessary to evacuate the initial hive $H$ by performing
a sequence of path removals that render it empty and to build the final hive $K$
from the data on the location of the first and last edges on each path that is removed.
In doing so one constructs a sequence of pairs $(H^{(r)},K^{(n-r)})$ 
for each $r=n,n-1,\ldots,0$ where $H^{(r)}$ is an $r$-hive and $K^{(n-r)}$ is an \textit{$r$-truncated $n$-hive} 
consisting of the rightmost $n-r$ diagonals of some $n$-hive. These pairs are such that
$H^{(n)}=H$ and $H^{(0)}$ is the empty hive, signified here by a single point, 
while $K^{(n)}=K$ and $K^{(0)}$ is an empty $n$-truncated $n$-hive, signified by a 
single boundary line consisting of $\beta$-edges with labels $\mu_1,\mu_2,\ldots,\mu_r$.

\begin{Example} 
In the case $n=4$ the map we are seeking is of the following type
from $(H^{(4)},K^{(0)})$ to $(H^{(0)},K^{(4)})$:
\[
\left(
\vcenter{\hbox{

}} \right)
\]
\end{Example}
 
\begin{Definition}\label{Def-HtoK}
Given any LR hive $H\in{\cal H}^{(n)}(\lambda,\mu,\nu)$, let $H^{(n)}=H$ and let $K^{(0)}$ be the empty $0$-hive. 
Then let $\Theta^{(n)}:=\Theta_1\cdots\Theta_{n-1}\Theta_n$ denote the operation which transforms the pair $(H^{(n)},K^{(0)})$ 
to the pair $\Theta^{(n)}(H^{(n)},K^{(0)}):=(H^{(0)},K^{(n)})$ through the action of 
a succession of $n$ operators 
that produce pairs $(H^{(r)},K^{(n-r)})$, with $r=n-1,n-2,\dots,0$, respectively, as indicated by
\begin{equation}\label{Eq-ThetaHK}
(H^{(n)},K^{(0)})\overset{\Theta_n}{\longmapsto}(H^{(n-1)},K^{(1)})\overset{\Theta_{n-1}}{\longmapsto}\cdots\overset{\Theta_2}{\longmapsto}(H^{(1)},K^{(n-1)})\overset{\Theta_1}{\longmapsto}(H^{(0)},K^{(n)}).
\end{equation}
Here $H^{(r)}$ is in the form of an $r$-hive with boundary edge labels
$(\lambda_1,\ldots,\lambda_r)$, $(\mu_1^{(r)},\ldots,\mu_r^{(r)})$ and $(\nu_1,\ldots\nu_r)$,
while $K^{(n-r)}$ is in the form of an $r$-truncated $n$-hive with boundary edge 
labels $(\lambda_{r+1},\ldots,\lambda_n)$, $(\mu_1^{(r)},\ldots,\mu_r^{(r)})$, $(\nu_{r+1},\ldots\nu_n)$
and $(\mu_1,\ldots,\mu_n)$, as exemplified in the case $n=7$ and $r=5$ by
\begin{equation}\label{Eq-HrKnr}
%Hr
\vcenter{\hbox{
\begin{tikzpicture}[x={(1cm*0.4,-\rootthree cm*0.4)},
                    y={(1cm*0.4,\rootthree cm*0.4)}]
%label \node at (-3,6){(ii)};  
%boundary edges
\draw(0,0)--(0,5)--(5,5)--cycle;
%boundary edge labels
\foreach\i in{1,...,5}\path(0,\i-1)--node[pos=0.75,left]{$\mu_\i^{(5)}$}(0,\i); % alpha
\foreach\i in{1,...,5}\path(\i-1,5)--node[pos=0.4,right]{$\nu_\i$}(\i,5); %{$\beta_{\i}$}
\foreach\i in{1,...,5}\path(\i-1,\i-1)--node[below]{$\lambda_\i$}(\i,\i);  %{$\gamma_{\i}$}
%hive edges
\foreach\i in{1,...,4}\draw(\i,\i)--(\i,5); % ascending alpha edges
\foreach\i in{1,...,5}\draw(0,\i)--(\i,\i); % descending beta edges
%upright rhombus entries 
\foreach\i in{1,...,4} \foreach\j in{2,...,5} {\ifnum \i<\j \path(\i-1,\j-1)--node{$U_{\i\j}$}(\i,\j); \fi}
%Knr shifted to right by 5 units
%boundary edges
\draw(0+4,5+4)--(0+4,7+4)--(7+4,7+4)--(5+4,5+4)--cycle;
%boundary edge labels
\foreach\i in{6,...,7}\path(0+4,\i-1+4)--node[pos=0.6,left]{$\nu_{\i}$}(0+4,\i+4); % alpha
\foreach\i in{1,...,5}\path(\i-1+4,7+4)--node[pos=0.4,right]{$\mu_{\i}$}(\i+4,7+4); %{$\beta_{\i}$}
\foreach\i in{6,...,7}\path(\i-1+4,7+4)--node[pos=0.4,right]{$\mu_{\i}$}(\i+4,7+4); %{$\beta_{\i}$}
\foreach\i in{1,...,5}\path(\i-1+4,5+4)--node[pos=0.75,left]{$\mu_\i^{(5)}$}(\i+4,5+4); %{$\beta_{\i}$}
\foreach\i in{6,...,7}\path(\i-1+4,\i-1+4)--node[below]{$\lambda_{\i}$}(\i+4,\i+4);  %{$\gamma_{\i}$}
%hive edges
\foreach\i in{1,...,5}\draw(\i+4,5+4)--(\i+4,7+4); \draw(6+4,6+4)--(6+4,7+4);% ascending alpha edges
\foreach\i in{5,...,6}\draw(0+4,\i+4)--(\i+4,\i+4); % descending beta edges
%upright rhombus entries 
\foreach\i in{1,...,6} \foreach\j in{6,...,7}{\ifnum\i<\j \path(\i-1+4,\j-1+4)--node{$V_{{\i}{\j}}$}(\i+4,\j+4); \fi}
\end{tikzpicture}
}}
\end{equation}
The operator $\Theta_r$ is defined to map the pair $(H^{(r)},K^{(n-r)})$ to $(H^{(r-1)},K^{(n-r+1)})$
as follows. First $H^{(r-1)}$ is obtained from $\theta_r\,H^{(r)}$ by the action of the operator $\kappa_r$
introduced at the end of section~\ref{Sec-notation}. This deletes the $r$th diagonal, which by
construction is empty, so that  $H^{(r-1)}=\kappa_r\theta_r H^{(r)}$. 
In parallel with this, $K^{(n-r+1)}$ is obtained from $K^{(n-r)}$, 
by adding to its left-hand boundary an $r$th diagonal of upright rhombi having gradients $V_{kr}$, 
with boundary edge labels $\nu_r$ and $\lambda_r$ at its top and bottom, respectively,
and with right- and left-hand boundary edge labels $(\mu_1^{(r)},\ldots,\mu_{r}^{(r)})$ and
$(\mu_1^{(r-1)},\ldots,\mu_{r-1}^{(r-1)})$, respectively,
where $\mu_k^{(r-1)}=\mu_k^{(r)}-V_{kr}$ and $V_{kr}$ is as defined in Theorem~\ref{The-Hr-theta-r}.

For example $\Theta_5$ maps the pair $(H^{(5)},K^{(2)})$ displayed in (\ref{Eq-HrKnr})
to the pair $(H^{(4)},K^{(3)})$ given by
\begin{equation}\label{Eq-HrKnr-1}
%H4
\vcenter{\hbox{
\begin{tikzpicture}[x={(1cm*0.4,-\rootthree cm*0.4)},
                    y={(1cm*0.4,\rootthree cm*0.4)}]
%label \node at (-3,6){(ii)};  
%boundary edges
\draw(0,0)--(0,4)--(4,4)--cycle;
%boundary edge labels
\foreach\i in{1,...,4}\path(0,\i-1)--node[pos=0.75,left]{$\mu_\i^{(4)}$}(0,\i); % alpha
\foreach\i in{1,...,4}\path(\i-1,4)--node[pos=0.4,right]{$\nu_\i$}(\i,4); %{$\beta_{\i}$}
\foreach\i in{1,...,4}\path(\i-1,\i-1)--node[below]{$\lambda_\i$}(\i,\i);  %{$\gamma_{\i}$}
%hive edges
\foreach\i in{1,...,3}\draw(\i,\i)--(\i,4); % ascending alpha edges
\foreach\i in{1,...,4}\draw(0,\i)--(\i,\i); % descending beta edges
%upright rhombus entries 
\foreach\i in{1,...,3} \foreach\j in{2,...,4} {\ifnum \i<\j \path(\i-1,\j-1)--node{$\tilde{U}_{\i\j}$}(\i,\j); \fi}
%Knr shifted to right by 4 units
%boundary edges
\draw(0+4,4+4)--(0+4,7+4)--(7+4,7+4)--(4+4,4+4)--cycle;
%boundary edge labels
\foreach\i in{5,...,7}\path(0+4,\i-1+4)--node[pos=0.6,left]{$\nu_{\i}$}(0+4,\i+4); % alpha
\foreach\i in{1,...,4}\path(\i-1+4,7+4)--node[pos=0.4,right]{$\mu_{\i}$}(\i+4,7+4); %{$\beta_{\i}$}
\foreach\i in{5,...,7}\path(\i-1+4,7+4)--node[pos=0.4,right]{$\mu_{\i}$}(\i+4,7+4); %{$\beta_{\i}$}
\foreach\i in{1,...,4}\path(\i-1+4,4+4)--node[pos=0.75,left]{$\mu_\i^{(4)}$}(\i+4,4+4); %{$\beta_{\i}$}
\foreach\i in{5,...,7}\path(\i-1+4,\i-1+4)--node[below]{$\lambda_{\i}$}(\i+4,\i+4);  %{$\gamma_{\i}$}
%hive edges
\foreach\i in{1,...,4}\draw(\i+4,4+4)--(\i+4,7+4); 
\draw(5+4,5+4)--(5+4,7+4); \draw(6+4,6+4)--(6+4,7+4);% ascending alpha edges
\foreach\i in{5,...,6}\draw(0+4,\i+4)--(\i+4,\i+4); % descending beta edges
%upright rhombus entries 
\foreach\i in{1,...,6} \foreach\j in{5,...,7}{\ifnum\i<\j \path(\i-1+4,\j-1+4)--node{$V_{{\i}{\j}}$}(\i+4,\j+4); \fi}
\end{tikzpicture}
}}
\end{equation}
where, as a result of the type (ii) path removals, the upright rhombus gradients $U_{ij}$ of $H^{(5)}$
have been replaced by $\tilde{U}_{ij}$ in $H^{(4)}$. 
\end{Definition}

Following this lengthy definition, it is convenient to exemplify each phase of the action
of $\Theta_r$ on the pair $(H^{(r)},K^{(n-r)})$. 
Before doing this it is helpful to introduce an operator $\zeta_r$ whose action on a truncated hive $K^{(n-r)}$ 
with left hand boundary edge labels $\mu_k^{(r)}$ for $k=1,2,\ldots,r$ is to 
add to the left-hand boundary of $K^{(n-r)}$ an $r$th diagonal with upper boundary edge label $0$, left hand
boundary edge labels $\mu_k^{(r)}$ for $k=1,2,\ldots,r-1$, lower boundary edge label $\mu_r^{(r)}$, and 
upright rhombus gradients all $0$, as illustrated by:
\begin{equation}
\label{Eq-zeta}
\vcenter{\hbox{

}}
\end{equation} 
Here it might be noted that the label $\mu_5^{(5)}$ has been added to the the leftmost lower boundary edge 
so as to preserve the triangle condition.

Phase 1 then arises if $\nu_r>0$ in which case 
$\theta_r$ involves $\nu_r$ type (i) hive path removals from $H^{(r)}$ and the same number 
of hive path additions to $\zeta_r K^{(n-r)}$ as illustrated by
\begin{equation}\label{Eq-HKphase1}
%Hr
\vcenter{\hbox{
%
}}
\end{equation}
%%In this Phase 1 it will be noticed that on the right what appears to be an extra term $\mu_5^{(5)}$ has been 
%%added to the the leftmost lower boundary edge label. This has been done so as to preserve the triangle condition.

Phase 2 involves a sequence of $\lambda_r-\nu_r-\mu_r$ type (ii) hive path removals from $H^{(r)}$ and the 
same number of hive path additions to $\zeta_r K^{(n-r)}$, of which one such removal and addition is illustrated by
\begin{equation}\label{Eq-HKphase2}
%Hr
\vcenter{\hbox{
%
}}
\end{equation}
In this example it has been assumed that the hive path removal illustrated
on the left is the first that terminates at the $3$rd edge on the left-hand boundary, thereby
reducing the edge label $\mu_3^{(5)}$ by $1$. This means that the edge labels
$\mu_1^{(5)}$ and $\mu_2^{(5)}$ have been reduced to $\mu_1^{(4)}$ and $\mu_2^{(4)}$, 
respectively, by the previous Phase 2 removal of two sequences of $V_{15}$ and $V_{25}$ paths 
starting at the $5$th edge on the lower boundary and terminating at the $1$st and $2$nd edges, respectively, 
on the left-hand boundary. As a result the edge label $\lambda_5-\nu_5$ has been reduced
to $\sigma_5:=\lambda_5-\nu_5-V_{15}-V_{25}$. These previous hive path removals on the left
have dictated the insertion of the two upright rhombus gradients $V_{15}$ and $V_{25}$ appearing 
on the right, and as a consequence the increase in the lower boundary edge label from 
$\nu_5+\mu_5^{(5)}$ to $\tau_5:=\nu_5+\mu_5^{(5)}+V_{15}+V_{25}$. Then the hive path removal on the left,
whose route is determined by the $0$ upright rhombus gradients, gives rise to the 
insertion of an upright rhombus gradient of $1$ on the right along with a further
increase of the lower boundary edge label from $\tau_5$ to $\tau_5+1$. In Phase 2
this process continues until on the left the upright rhombi in the $5$th diagonal
all have gradients equal to $0$, with the initial edge label $\lambda_5-\nu_5$ 
reduced to $\mu_5^{(5)}$. On the right the upright rhombi in the leftmost diagonal
have gradients $V_{k5}$, with the initial edge label $\nu_5+\mu_5^{(5)}$ increased
to $\nu_5+\mu_5^{(5)}+V_{15}+V_{25}+V_{35}+V_{45}=\lambda_5$. 

Phase 3 then involves a succession of $\mu_r^{(r)}$ type (iii) hive path removals 
from $H^{(r)}$. However, no corresponding hive path additions to $\zeta_rK^{(n-r)}$ are required
because the addition of $\mu_r^{(r)}$ to the leftmost lower boundary edge label has
already taken place in Phase 1. Thus Phase 3 is illustrated in our example by:
\begin{equation}\label{Eq-HKphase3}
%Hr
\vcenter{\hbox{

}}
\end{equation}
The removal of the redundant $5$th diagonal on the left by means of the action of $\kappa_5$ then yields (\ref{Eq-HrKnr-1}) as required.

%%%%%%%%%%%%%%%%%%%%%%%%%%%%%%%
\begin{Example}\label{Ex-hive}
An explicit exemplification of the map from $(H^{(4)},K^{(0)})$ to $(H^{(0)},K^{(4)})$ is provided by the following:\\
%%%%%%%%%% H4K0 to H3K1
\noindent $(H^{(4)},K^{(0)})~\overset{\Theta_4}{\longmapsto}~(H^{(3)},K^{(1)})$:
\begin{equation*}%%%%%%%%%%%%%%%%%%%%%%%%%%%%%%%% MENDED\label{Eq-H4K0-H3K1}
%
 %%%%%%%%%%%%%%%%%%%%%%%%%%%%%%%% MENDED  \nonumber
\end{equation*}%%%%%%%%%%%%%%%%%%%%%%%%%%%%%%%% MENDED 

\end{Example}

%%%%%%%%%%%%%%%%%%%%%%%%%%%%%%%%

We are now in a position to state
\begin{Theorem}\label{The-HtoK}
Let $n$ be a positive integer and let $\lambda$, $\mu$ and $\nu$ be partitions such that $\ell(\lambda)\leq n$ 
and $\mu,\nu\subseteq\lambda$ with $|\lambda|=|\mu|+|\nu|$.
For each LR hive $H\in\mathcal{H}^{(n)}(\lambda,\mu,\nu)$ let $H^{(n)}=H$ and let $K^{(0)}$ 
be an $n$-truncated $n$-hive with edge labels $\mu$.
If we let $\Theta^{(n)}(H^{(n)},K^{(0)})=(H^{(0)},K^{(n)})$ as in Definition~\ref{Def-HtoK}, then 
$H^{(0)}$ is an empty hive and $K=K^{(n)}$ is an LR hive $K\in\mathcal{H}^{(n)}(\lambda,\nu,\mu)$.
In such a case we write $K=\sigma^{(n)} H$.
\end{Theorem}

\noindent{\bf Proof}:~~
From the iteration scheme of Definition~\ref{Def-HtoK} and the illustration of the action
of $\Theta_r$ in mapping (\ref{Eq-HrKnr}) to (\ref{Eq-HrKnr-1}) it is clear that $K=K^{(n)}$ will 
have boundary edge labels $\lambda$, $\nu$ and $\mu$. In addition it can 
be seen that each Phase of the action of $\Theta_r$ preserves the triangle condition at every stage.
The possibility of an edge label becoming negative arises only in the case of a Phase 2 type (ii) 
hive path removal as illustrated by the red edge in (\ref{Eq-HKphase2}). However, arguing inductively,
the edge labels $\mu_k^{(n)}=\mu_k$ are non-negative for $1\leq k<n$. If, by the induction
hypothesis the edge labels $\mu_k^{(r)}$ are all non-negative for $1\leq k<r$ then the $V_{kr}$
Phase 2 removals have the effect of reducing the edge initially labelled $\mu_k^{(r)}$ to  
$\mu_k^{(r-1)}=\mu_k^{(r)}-V_{kr}\geq\mu_{k+1}^{(r)}\geq0$ for $1\leq k<r-1$, where the penultimate 
step follows from Corollary~\ref{Cor-obs1}. Thus the edges labels $\mu_k^{(r-1)}$ for $1\leq k<r-1$
are all non-negative. This completes the induction argument. 

As far as the gradients of elementary rhombi are concerned, all the upright rhombus gradients $V_{kr}$ 
are non-negative as they count the number of certain type (ii) hive path removals. It remains 
to show that the left and right-leaning rhombus gradients are also non-negative.

The $(k,r)$th left-leaning rhombus of $K$ will have the edge labels indicated below
\begin{equation}\label{Eq-VLrhombi}
%Hr
\vcenter{\hbox{

}}
\end{equation}
The gradient of the illustrated left-leaning rhombus is then given by
\begin{equation}
      L_{kr}=\mu_k^{(r-1)}-\mu_{k+1}^{(r)}=\mu_k^{(r)}-V_{kr}-\mu_{k+1}^{(r)}\geq0,
\end{equation}
where the final step is a consequence of Corollary~\ref{Cor-obs1}
being applied in the case $H=H^{(r)}$. Thus all left-leaning rhombi in $K$ 
are non-negative.

Similarly, the $(k,r)$th right-leaning rhombus of $K$ will have the edge labels
determined by the diagram shown below on the left.
\begin{equation}\label{Eq-Rrhombus}
\vcenter{\hbox{
%
}}
\end{equation}
In this diagram on the left, the edge labels $\nu_r$ and $\nu_{r-1}$ are constrained by the
hive conditions for the sub-diagram of $H^{(r)}$ shown on the right. These
imply that
\begin{equation}
      \nu_{r-1}-U_{r-1,r}\geq \nu_{r}.
\end{equation}
It follows that the right-leaning rhombus gradient is given by
\begin{eqnarray}
R_{kr} &=& (\nu_{r-1}+\sum_{i=1}^{k-1} V_{i,r-1})-(\nu_r+\sum_{i=1}^k V_{ir})\cr && \cr
       &\geq& U_{r-1,r}+N_{k-1,r-1}-N_{kr}\ \geq\ 0 \,.
\end{eqnarray}
where use has been made of Lemma~\ref{Lem-obs2}.

Thus all elementary rhombus gradients of $K$ are non-negative, and this completes the proof that
$K$ is an LR hive. The boundary edge labels then ensure that $K\in{\cal H}^{(n)}(\lambda,\nu,\mu)$.
\qed
%%%%%%%%%%%%%%%%%%%%%%%%%%%%%%%%%%%%%%%%%%%%%

\section{Creation of a hive by path additions}
\label{Sec-hive-path-addition}

Having used a path removal procedure to provide a map from any $H\in{\cal H}^{(n)}(\lambda,\mu,\nu)$ to 
some $K\in{\cal H}^{(n)}(\lambda,\nu,\mu)$ we now wish to show that a path addition procedure may be
used to provide a map from any $K\in{\cal H}^{(n)}(\lambda,\nu,\mu)$ to some $H\in{\cal H}^{(n)}(\lambda,\mu,\nu)$. 
The aim is to show that these two maps are mutually inverse to one another, thereby proving that 
each is a bijection. 

Our approach is to move successively from an $(r-1)$-hive $H^{(r-1)}$ to an $r$-hive $H^{(r)}$ under a 
procedure dictated by the $r$th diagonal of $K$. In doing so it is necessary to exploit first 
the operator $\kappa_r^{-1}$ that adds to $H^{(r-1)}$ an $r$th diagonal consisting of a sequence of upright rhombi all 
of gradient $0$, with its upper and lower boundary edge labels $0$, and with its remaining new edges 
given the unique labels that preserve the hive conditions. This is the inverse of the bijection $\kappa_r$ introduced and exemplified 
in the case $r=4$ at the end of section~\ref{Sec-notation}.
%To this end 

To proceed further, let ${\cal UH}^{(r)}(\lambda,\mu,\nu)$ be the set of all triangular arrays of side length $r$ satisfying the hive triangle 
relations and the condition that all upright rhombus labels are non-negative integers, with boundary edges labelled by $\lambda$, $\mu$ and $\nu$. 
For each $H\in{\cal UH}^{(r)}(\lambda,\mu,\nu)$
we define path addition operators $\ov{\chi}_r$, $\ov{\phi}_k$ and $\ov{\omega}_r$
analogous to the path removal operators $\chi_r$, $\phi_r$ and $\omega_r$ of section~\ref{Sec-hive-path-removal}.

\begin{Definition}\label{Def-hive-path-additions}
For any given $H\in{\cal UH}^{(r)}(\lambda,\mu,\nu)$ with $\ell(\lambda)\leq r$ we define 
three path addition operators $\ov{\chi}_r$, $\ov{\phi}_k$ and $\ov{\omega}_r$ whose action on $H$
is to increase or reduce edge labels by $1$ along paths specified as follows:
\begin{enumerate}
\item[\tov{(i)}] $\ov{\chi}_r$: the path consists of the boundary edges labelled $\nu_r$ 
and $\lambda_r$, with each of these labels being increased by $1$;
\item[\tov{(ii)}] $\ov{\phi}_k$: for any $k<r$ the path proceeds down the $k$th diagonal 
from the edge labelled $\mu_k$ through upright rhombi of gradient $0$ until it encounters an 
upright rhombus of positive gradient, at which point it moves horizontally to the right 
into the $(k+1)$th diagonal and proceeds down this diagonal or to the right as before, 
and so on until it either reaches the base of the hive and then moves to the right
or reaches the right hand boundary and then moves down the $r$th diagonal until, in 
both cases, it terminates at the edge labelled $\lambda_r$, with all path $\alpha$ and $\gamma$ edge
labels being increased by $1$ and all path $\beta$ edge labels decreased by $1$ ;
\item[\tov{(iii)}] $\ov{\omega}_r$: the path proceeds directly down the $r$th
diagonal until it terminates at the base at the edge labelled $\lambda_r$, with all path
edge labels increased by $1$.
\end{enumerate}
\end{Definition}

These three types of path addition are illustrated below. In each case every {\\bf\blue blue} edge label is increased 
by $1$ and every {\bf\red red} edge label is decreased by $1$.  
In the case of $\ov{\chi}_r$ and $\ov{\omega}_r$ the path additions are confined to the rightmost $r$th diagonal. 
On the other hand the path addition route ascribed to the action of $\ov{\phi}_k$ with $1\leq k<r$
consists of a sequence of ladders through upright rhombi of gradient $0$ in each diagonal from the $k$th 
to the $r$th, with the passage from each diagonal to the next taking place through a {\bf\red red} $\beta$-edge.
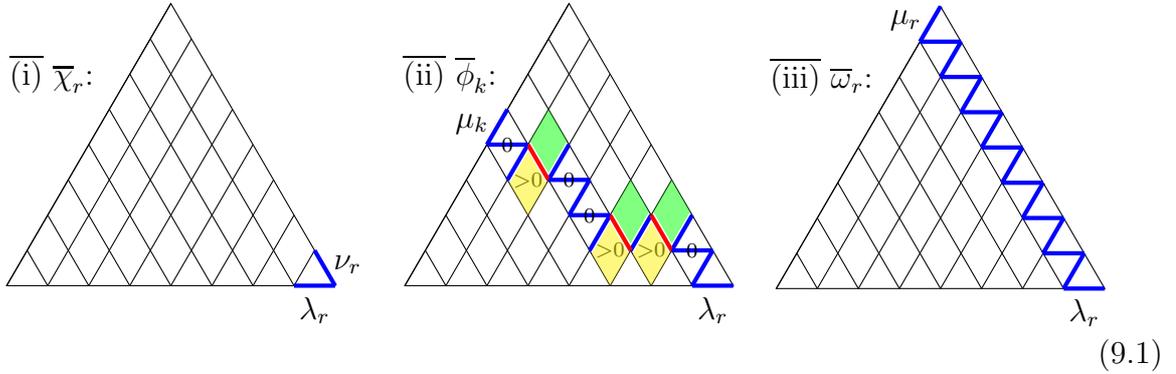
\begin{equation}\label{Eq-path-additions}
%(i)
%\hskip-0.5cm
\vcenter{\hbox{

}}
\end{equation}

The ladders are of the following form:
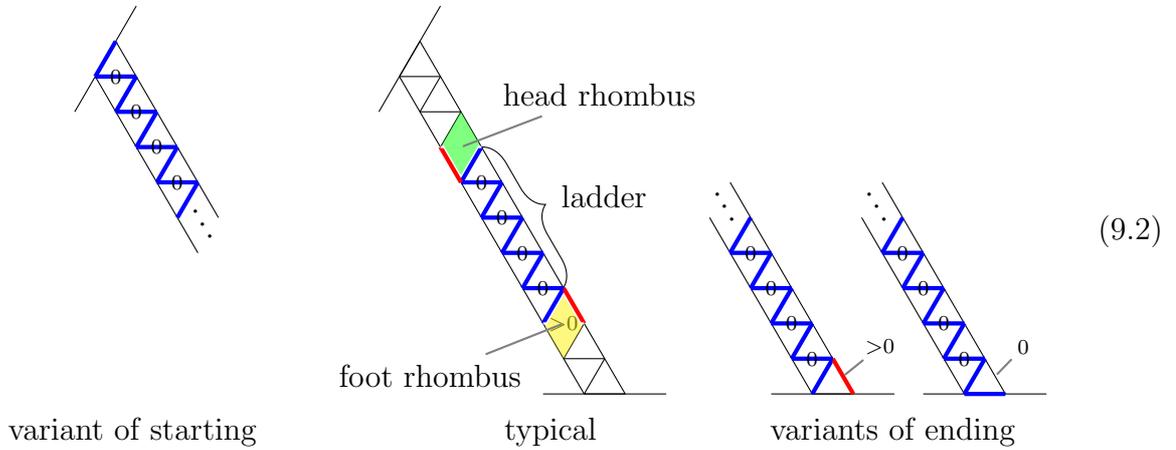
\begin{equation}
\vcenter{\hbox{
%
}}
\end{equation}

Each ladder (coloured blue) extends down its diagonal either from the top left hand boundary 
or from an upright rhombus (coloured green) until it 
either reaches an upright rhombus of positive gradient (coloured yellow) or the base of the triangular array. Thus,
the ladders starting on the left hand boundary lack a head rhombus, while those ending at the base lack a foot rhombus. 
As shown in the above example of type \tov{(ii)}, ladders may consist of a single blue edge. 
In such a case the ladder is said to have length $0$. More generally, its length is the number of upright rhombi it
crosses, that is to say the number of its horizontal blue edges.
The path addition is such that the gradient of the foot rhombus is decreased by $1$ from its initial positive value, 
while the gradient of the head rhombus is increased by $1$. 

It might be noted here that there are two distinct manners in which type \tov{(ii)} paths may terminate.
They are illustrated below, with the figures on the left and right applying to cases in which the path addition reaches
the base in the diagonal with $j=r$ and $j<r$, respectively.

\begin{equation}\label{Eq-iib-iia}
%(\ov{iib})
\vcenter{\hbox{

}}
\end{equation}

%%%%%%Insertion from UH_addition_lemma
\begin{Lemma}
Let $H$ be an array in $\mathcal{UH}\sp{(r)}(\lambda,\mu,\nu)$ with $\ell(\lambda)\leq r$.
Then an application of the path addition operator to $H$ produces another such array $\tilde H$ with boundary edge labels 
as appropriate to each type of addition operator. More precisely,
%That is to say:

\begin{tabular}{c@{ }lc@{ }l}
if & $\tilde H=\overline\chi_rH$ & then & $\tilde H\in\mathcal{UH}\sp{(r)}((\lambda+\epsilon_r),\mu,(\nu+\epsilon_r))$; \\
if & $\tilde H=\overline\phi_kH$, with $1\leq k<r$, & then &$\tilde H\in\mathcal{UH}\sp{(r)}((\lambda+\epsilon_r),(\mu+\epsilon_k),\nu)$; \\
if & $\tilde H=\overline\omega_rH$ & then &$\tilde H\in\mathcal{UH}\sp{(r)}((\lambda+\epsilon_r),(\mu+\epsilon_r),\nu)$.
\end{tabular}
\end{Lemma}

\noindent{\bf Proof}:~~
As far as elementary triangles are concerned the path additions give rise to the following possibilities.
\begin{equation}\label{Eq-triangles-add}
%1
\vcenter{\hbox{
%
}}
\end{equation}
It is clear that the triangle conditions are preserved in every case.

As far as upright rhombus gradients are concerned the only configuration that gives rise to a reduction 
in an upright rhombus gradient is that shown below:
\begin{equation}
\label{Eq-Urhombi-add}
\begin{array}{c}\vcenter{\hbox{
\begin{tikzpicture}[x={(1cm*0.5,-1.7320508cm*0.5)},y={(1cm*0.5,1.7320508cm*0.5)}]
\draw[blue, ultra thick](0,1)--(0,2); \draw[black,thin](1,1)--(1,2); %\alpha
\draw[black,thin](0,1)--(1,1); \draw[red, ultra thick](0,2)--(1,2); %\beta
%\draw[black, thin](0,1)--(1,2);  %\gamma
\path(0,1)--node[pos=0.6,left]{$+1$}(0,2); %\alpha
\path(0,2)--node[pos=0.4,right]{$-1$}(1,2); %\beta
%foot rhombi filled
\foreach \f/\n in {0/1}\path[fill=yellow,fill opacity=0.5] (\f+0.05,\n)--(\f+0.05,\n+1-0.05)--(\f+1,\n+1-0.05)--(\f+1,\n)--cycle;
%rhombus gradient
\path(0,1)--node{$>0$}(1,2);
\end{tikzpicture}
}}
\cr\cr
U\longrightarrow U\!-\!1\cr
\end{array}
\end{equation}
This situation only occurs in case (ii) of (\ref{Eq-path-additions}), 
where the transition in the case of the foot rhombi is from $U$ to $U-1$ with $U>0$.
Thus all upright rhombus gradients remain non-negative after all possible path additions.

This ensures that $\tilde H$ belongs to $\mathcal{UH}(\lambda^*,\mu^*,\nu^*)$ for some boundary edge labels $\lambda^*,\mu^*,\nu^*$.
The fact that these boundary edge labels are as specified in the Lemma follows directly from Definition~\ref{Def-hive-path-additions} 
of the path addition operations.
\qed

In the case of left- and right-leaning elementary rhombi the only configurations that give rise to a reduction 
in a rhombus gradient are those shown below:
\begin{equation}
\label{Eq-rhombi-add}
\begin{array}{ccccc}
%R
\vcenter{\hbox{
\begin{tikzpicture}[x={(1cm*0.5,-1.7320508cm*0.5)},y={(1cm*0.5,1.7320508cm*0.5)}]
\draw[black,thin](0,0)--(0,1); \draw[blue, ultra thick](1,1)--(1,2); %\alpha
\draw[black, thin](0,1)--(1,1); %\beta
\draw[black, thin](0,0)--(1,1); \draw[blue,ultra thick](0,1)--(1,2); %\gamma
\path(1,1)--node[pos=0.4,right]{$+1$}(1,2); %\alpha
%\path(1,1)--node[pos=0.6,right]{$-1$}(1,2); %\beta
\path(0,1)--node[above]{$+1$}(1,2); %\gamma
\end{tikzpicture}
}}
&\quad&
%L
\vcenter{\hbox{
\begin{tikzpicture}[x={(1cm*0.5,-1.7320508cm*0.5)},y={(1cm*0.5,1.7320508cm*0.5)}]
\draw[blue, ultra thick](1,1)--(1,2);  %\alpha
\draw[red, ultra thick](0,1)--(1,1); \draw[black, thin](1,2)--(2,2); %\beta
\draw[black, thin](0,1)--(1,2); \draw[blue,ultra thick](1,1)--(2,2); %\gamma
\path(1,1)--node[pos=0.8,left]{$+1$}(1,2); %\alpha
\path(0,1)--node[pos=0.7,left]{$-1$}(1,1); %\beta
\path(1,1)--node[below]{$+1$}(2,2); %\gamma
\end{tikzpicture}
}}
&\quad&
%L
\vcenter{\hbox{
\begin{tikzpicture}[x={(1cm*0.5,-1.7320508cm*0.5)},y={(1cm*0.5,1.7320508cm*0.5)}]
\draw[black, thin](1,1)--(1,2);  %\alpha
\draw[black,thin](0,1)--(1,1); \draw[blue, ultra thick](1,2)--(2,2); %\beta
\draw[black, thin](0,1)--(1,2); \draw[blue,ultra thick](1,1)--(2,2); %\gamma
%\path(1,1)--node[pos=0.8,left]{$+1$}(1,2); %\alpha
\path(1,2)--node[pos=0.4,right]{$+1$}(2,2); %\beta
\path(1,1)--node[below]{$+1$}(2,2); %\gamma
\end{tikzpicture}
}}
\cr\cr
R\longrightarrow R\!-\!1&\quad&L\longrightarrow L\!-\!1&\quad&L\longrightarrow L\!-\!1\cr
\end{array}
\end{equation}
It is not immediately clear that $L$ and $R$ are always $>0$ in these configuration.
%the same is true of left- and right-leaning rhombus gradients.
In fact this depends crucially on both the number of path additions of each type
and the order in which they are applied.
However, we claim the validity of the following
\begin{Theorem}\label{The-KtoH}
Let $n$ be a positive integer and let $\lambda$, $\mu$ and $\nu$ be partitions such that $\ell(\lambda)\leq n$ and 
$\mu,\nu\subseteq\lambda$ with $|\lambda|=|\mu|+|\nu|$.
For each LR hive $K\in{\cal H}^{(n)}(\lambda,\nu,\mu)$ with upright rhombus gradients $V_{ij}$ for $1\leq i< j\leq n$ let
\begin{equation}\label{Eq-ovtheta}
  \ov{\theta}_r = \ov{\chi}_r^{\,\nu_r}\ \ov{\phi}_1^{\,V_{1,r}}\ \ov{\phi}_2^{\,V_{2r}}\ \ov{\phi}_{r-1}^{\,V_{r-1,r}}\ 
	                     \ov{\omega}_r^{\,\mu_r-V_{r,n}-V_{r,n-1}-\cdots-V_{r,r+1}}\,. %, \kappa_r^{-1}\,.
\end{equation}  
Then the successive application of $\ov{\theta}_r\kappa_r^{-1}$ for $r=1,2,\ldots,n$ to the empty hive $H^{(0)}$ 
yields
\begin{equation}\label{Eq-H-hive}
     H:=\ov{\theta}_n\kappa_n^{-1}\,\cdots\,\ov{\theta}_2\kappa_2^{-1}\,\ov{\theta}_1\kappa_1^{-1}\,H^{(0)}\, \in{\cal H}^{(n)}(\lambda,\mu,\nu) \,.
\end{equation}
In such a case we write $H=\ov{\sigma}^{(n)} K$. 
\end{Theorem}

\noindent{\bf Proof}:

It is convenient to set $\lambda^{(r)}=(\lambda_1,\lambda_2,\ldots,\lambda_r)$ and $\nu^{(r)}=(\nu_1,\nu_2,\ldots,\nu_r)$
for $r=1,2,\ldots,n$ and to remind ourselves of the notation already used in connection with $K$ whereby
$\mu^{(r)}=(\mu^{(r)}_1,\mu^{(r)}_2,\ldots,\mu^{(r)}_r)$ with $\mu^{(r)}_k=\mu_k-V_{k,n}-V_{k,n-1}-\cdots-V_{k,r+1}$
for $k=1,2,\ldots,r$. Notice in particular that this means that the exponent of $\ov{\omega}_r$ in (\ref{Eq-ovtheta}) is just $\mu_r^{(r)}$.
The notation also allows us to define $K_{(r)}\in{\cal H}^{(r)}(\lambda^{(r)},\nu^{(r)},\mu^{(r)})$ to be the subhive
of $K\in{\cal H}^{(n)}(\lambda,\nu,\mu)$ consisting of its leftmost $r$ diagonals, for $r=1,2,\ldots,n$. Thus $K_{(r)}$ is essentially 
the complement of the truncated hive $K^{(n-r)}$ in $K=K^{(n)}$.

Our strategy is to proceed by induction with respect to $r$ in showing that $H^{(r)}\in{\cal H}^{(r)}(\lambda^{(r)},\mu^{(r)},\nu^{(r)})$,
where $H^{(r)}=\ov{\theta}_r\kappa_r^{-1}H^{(r-1)}$ with $\ov{\theta}_r$ determined by $K_{(r)}$. 

In the case $r=1$ we have $K_{(1)}=\vcenter{\hbox{

}}$ so that $\ov{\theta}_1= \ov{\chi}_1^{\,\nu_1}\ \ov{\omega}_1^{\,\mu_1^{(1)}}$. 
This yields $H^{(1)}=\ov{\theta}_1\kappa_1^{-1}H^{(0)}$ as shown below:
\begin{equation}
\begin{array}{cccccccccc} %ccc}
H^{(0)}=\ \cdot\ \ 
&\overset{\kappa_1^{-1}}{\longmapsto}
&\vcenter{\hbox{
\begin{tikzpicture}[x={(1cm*0.5,-\rootthree cm*0.5)},
                    y={(1cm*0.5,\rootthree cm*0.5)}]
\hiveH{1}{0};
%boundary edge labels
\path(0,0)--node[pos=0.65,left]{$\sc0$}(0,1); % mu1
\path(0,1)--node[pos=0.35,right]{$\sc0$}(1,1); % nu1
\path(0,0)--node[below]{$\sc0$}(1,1);  % lambda1
\end{tikzpicture}
}}
&
\overset{\ov{\omega}_1^\ast}{\underset{\ast=\mu_1^{(1)}}{\longmapsto}}
%H4
&\vcenter{\hbox{
\begin{tikzpicture}[x={(1cm*0.5,-\rootthree cm*0.5)},
                    y={(1cm*0.5,\rootthree cm*0.5)}]
\hiveH{1}{0}; %\node at (4.5,4.5) {,};
%boundary edge labels
\path(0,0)--node[pos=0.65,left]{$\mu_1^{(1)}$}(0,1); % mu1
\path(0,1)--node[pos=0.35,right]{$\sc0$}(1,1); % nu1
\path(0,0)--node[below]{$\mu_1^{(1)}$}(1,1);  % lambda1
%path edges
%blue alpha edges
\draw[blue,ultra thick] (0,0)--(0,1);
%red beta edges
%\foreach \i/\j in {1/3} \draw[blue,ultra thick] (\i+0,\j+0)--(\i+1+0,\j+0);
%red gamma edges
\draw[blue,ultra thick] (0,0)--(1,1);
\end{tikzpicture}
}}
&
\overset{\ov{\chi}_1^\ast}{\underset{\ast={\nu_1}}{\longmapsto}}
&\vcenter{\hbox{
\begin{tikzpicture}[x={(1cm*0.5,-\rootthree cm*0.5)},
                    y={(1cm*0.5,\rootthree cm*0.5)}]
\hiveH{1}{0}; %\node at (4.5,4.5) {,};
%boundary edge labels
\path(0,0)--node[pos=0.8,left]{$\mu_1^{(1)}$}(0,1); % mu1
\path(0,1)--node[pos=0.25,right]{$\nu_1$}(1,1); % nu1
\path(0,0)--node[below]{$\lambda_1$}(1,1);  % lambda1
%path edges
%blue alpha edges
%\draw[blue,ultra thick] (0,0)--(0,1);
%red beta edges
\draw[blue,ultra thick] (0,1)--(1,1);
%red gamma edges
\draw[blue,ultra thick] (0,0)--(1,1);
\end{tikzpicture}
}}
&\!\!\!=\ H^{(1)}\,.
\end{array}
\end{equation}
where in the final step use has been made of the fact that $\mu_1^{(1)}+\nu_1=\lambda_1$, as implied by the hive condition
on $K_{(1)}$. This demonstrates that $H^{(1)}=\ov{\theta}_1\kappa_1^{-1}H^{(0)}\in{\cal H}^{(1)}(\lambda^{(1)},\mu^{(1)},\nu^{(1)})$,
as required. 

To complete the induction argument, we need to show that for $r\geq2$ if 
$H^{(r-1)}=\ov{\theta}_{r-1}\kappa^{-1}_{r-1}\cdots\ov{\theta}_{1}\kappa^{-1}_{1} H^{(0)}\in{\cal H}^{(r-1)}(\lambda^{(r-1)},\mu^{(r-1)},\nu^{(r-1)})$
and $H^{(r)}=\ov{\theta}_r \kappa^{-1}_{r} H^{(r-1)}$ with $\ov{\theta}_r$ determined by $K_{(r)}$ as in (\ref{Eq-ovtheta}), then
$H^{(r)}\in{\cal H}^{(r)}(\lambda^{(r)},\mu^{(r)},\nu^{(r)})$. Before showing that $H^{(r)}$ constructed in this way is a hive, 
let us first show that at least its boundary edge labels are as required. To do this it is helpful to display
$K_{(r)}$ and various intermediate states between $H^{(r-1)}$ and $H^{(r)}$. These take the form:
\begin{equation}\label{Eq-KHHH}

\end{equation}
The hive conditions on $K_{(r)}$ imply that
\begin{equation}\label{Eq-Kidentities}
    \mu_k^{(r)}=\mu_k^{(r-1)}+V_{kr} \quad\hbox{for $k=1,2,\ldots,r-1$~~~and}\quad  \sum_{k=1}^{r-1} V_{kr} = \lambda_r-\nu_r-\mu_r^{(r)} \,.
\end{equation}
%by virtue of the conditions on the rhombus of gradient $V_{kr}$ and on the lower rightmost triangle of $K$. 
In the above display (\ref{Eq-KHHH}) $H^{(r)}_{\ov{(\rm{i})}}=\ov{\omega}_r^{\mu_r^{(r)}}\kappa_r^{-1}H^{(r-1)}$
is formed by adding an $r$th diagonal of upright rhombi of gradient $0$ to $H^{(r-1)}$
and then applying all $\mu_r^{(r)}$ type \tov{(i)} path addition operators. 
Then $H^{(r)}_{\ov{(\rm{ii})}}$ is obtained from $H^{(r)}_{\ov{(\rm{i})}}$ by applying
$V_{kr}$ type \tov{(ii)} path addition operators successively in the order $k=r-1,\ldots,2,1$. 
For each $k$ the $V_{kr}$ added paths increase the $k$th left hand boundary edge label from
$\mu_k^{(r-1)}$ to $\mu_k^{(r)}$, as a consequence of the first identity in (\ref{Eq-Kidentities}). 
Moreover, each of the path additions extends as far as the foot of the $r$th diagonal adding 
precisely $1$ both to the $r$th lower boundary edge and to one or other of the upright rhombus gradients in this diagonal.
It follows from this that $\tilde{U}_{1r}+\tilde{U}_{2r}+\cdots+\tilde{U}_{r-1,r}=V_{1r}+V_{2r}+\cdots+V_{r-1,r}$
and that the $r$th lower boundary edge becomes $\mu_r^{(r)}+V_{1r}+V_{2r}+\cdots+V_{r-1,r}=\lambda_r-\nu_r$,
as shown, where use has been made of the second identity of (\ref{Eq-Kidentities}). Finally, the application of
all $\nu_r$ type \tov{(iii)} path addition operations to  $H^{(r)}_{\ov{(\rm{ii})}}$ yields  $H^{(r)}_{\ov{(\rm{iii})}}$
with boundary edge labels as shown in the last diagram of (\ref{Eq-KHHH}).
It can be seen from this $H^{(r)}=H^{(r)}_{\ov{(\rm{iii})}}$ has boundary edge labels specified by
$\lambda^{(r)}$, $\mu^{(r)}$ and $\nu^{(r)}$, as required.

%%%%%%%%%%%%%%%%%%%%%%%%%%%%%%%%%%%%%%%%%%%

It remains to determine whether or not $H^{(r)}$ satisfies all necessary hive rhombus conditions.
In applying $\ov{\theta}_r\kappa_r^{-1}$, as defined by (\ref{Eq-ovtheta}), to $H^{(r-1)}$ the first step associated
with the action of $\kappa_r^{-1}$ is to add an $r$th diagonal in which each upright rhombus has gradient $0$,
as shown below on the left. This is followed by the action of $\ov{\omega}_r^\alpha$ with $\alpha={\mu_r^{(r)}}$ as shown below on the right,
where each blue edge has its label increased by $\alpha$: that is from $0$ to $\alpha$ in the case of sloping blue edges, from $\nu_k$ to
$\nu_k+\alpha$ in the case of the $k$th horizontal blue edge counted from the top, with $k=1,2,\ldots,r-1$, and from $0$ to $\alpha$ in the
case of the horizontal blue edge on the boundary.
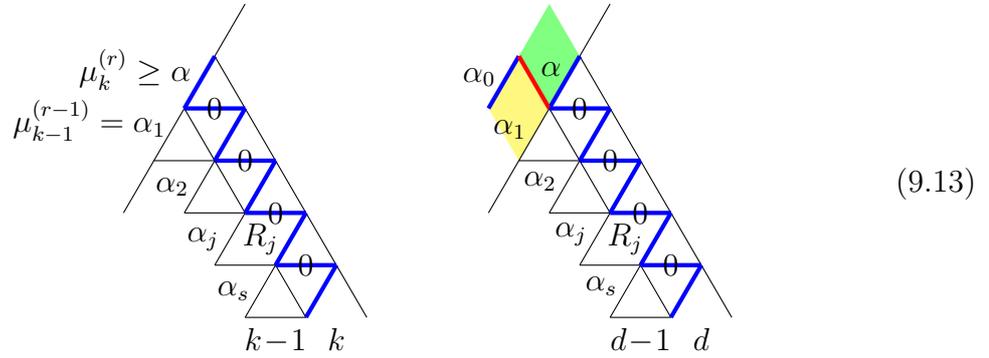
\begin{equation}\label{Eq-kappa-chi}
%H4
\kappa_r^{-1} H^{(r-1)}\ =
\hskip-4ex 
\vcenter{\hbox{

}}
\hskip-2ex
=\ \ov{\omega}_r^\alpha\kappa_r^{-1} H^{(r-1)}\,.
\end{equation}
The hive conditions on $H^{(r-1)}$ imply that $\alpha_{r-1}\geq\cdots\geq\alpha_2\geq\alpha_1\geq0$, 
while those on $K_{(r)}$ imply that $\alpha_1=\mu_{r-1}^{(r-1)}\geq\mu_r^{(r)}=\alpha$.
It follows on the left that after the action of $\kappa_r^{-1}$ the right-leaning rhombus crossing the 
$(r-1)$th and $r$th diagonals at level $j$ has gradient $R_j=\alpha_j-0\geq0$ for $j=1,2,\ldots,r-1$,
and the gradients of the left-leaning rhombi in the $r$th diagonal are $\nu_k-\nu_{k+1}\geq0$ for $k=1,2,\ldots,r-1$.
Since the upright rhombus gradients at level $k$ in the $r$th diagonal are $0$, 
all the hive conditions are satisfied.
Similarly, on the right after the action of $\ov\omega_r^\alpha$ the right-leaning rhombus crossing 
the $(r-1)$th and $r$th diagonals at level $j$ has gradient $\alpha_j-\alpha\geq\alpha_1-\alpha\geq0$ for $j=1,2,\ldots,r-1$.
Since the upright and left-leaning rhombus gradients in the $r$th diagonal are unchanged, % still $0$, 
it can again be seen that after all the path additions on the right all the hive conditions are still satisfied.

Turning next to the path additions that follow from the application of $\ov\phi_k$ with $k<r$, there are two
distinct ways in which an addition path may enter a ladder, as illustrated below in the case of a ladder of length
at least $1$. If the ladder has length $0$, in either case, it passes immediately to the next diagonal and creates no rhombus 
of the type listed in (\ref{Eq-iib-iia}).
\begin{equation}\label{Eq-phi}
%H4
\vcenter{\hbox{

}}
\end{equation}

On the left, a ladder starting on the left hand boundary in the $k$th diagonal has been shown. It has been illustrated 
immediately after the completed action any one of the $V_{kr}$ path additions by $\ov\phi_k$, %action of $\ov\phi_k^{V_{kr}}$ 
on the assumption that this leaves a ladder of length 
$s\geq1$ in the $k$th diagonal. 
This action has the effect of increasing the boundary edge label from $\mu_k^{(r-1)}$ to $\alpha\leq\mu_k^{(r-1)}+V_{kr}=\mu_k^{(r)}$. 
This case is exactly analogous to that on the right of (\ref{Eq-kappa-chi}). The hive conditions on $H^{(r-1)}$ imply that $\alpha_s\geq\cdots\geq\alpha_2\geq\alpha_1$, 
while those on $K_{(r)}$ imply that $\alpha_1=\mu_{k-1}^{(r-1)}\geq\mu_k^{(r)}\geq\alpha$. It follows as before that all necessary
right-leaning rhombus gradients such as $R_j$ remain $\geq0$, and that all hive conditions are preserved.

On the right, a typical interior ladder is shown lying in the $d$th diagonal. It has been illustrated immediately after
any one of the actions of $\ov\phi_k$. Prior to this action the edges labelled $\alpha$, $\alpha_0$ and $\alpha_1$ had labels
$\alpha-1$, $\alpha_0-1$ and $\alpha_1$, respectively. For the addition path to enter the $d$th diagonal by 
way of the red edge it is necessary that the yellow upright rhombus has an initial gradient $U=\alpha_1-(\alpha_0-1)>0$. 
The initial hive conditions also imply that both $\alpha_0-1\geq\alpha-1$ and $\alpha_s\geq\cdots\geq\alpha_2\geq\alpha_1$. 
It follows that $\alpha_1\geq\alpha_0$ and $\alpha_s\geq\cdots\geq\alpha_2\geq\alpha_1\geq\alpha$ so that after the 
path addition the upright rhombus has gradient $U-1=\alpha_1-\alpha_0\geq0$ and the $j$th the right-leaning
rhombus has gradient $R_j=\alpha_j-\alpha\geq0$ for $j=1,2,\ldots,s$. Thus all rhombus gradients remain $\geq0$.
Not only that, the red edge has an initial label that must be positive since $U>0$, so it remains non-negative after the path 
addition, and all hive conditions are therefore satisfied. However it still remains to consider the left-leaning
rhombus at the top of the ladder. To this end consider the following diagram with the edge and gradient labels
specified {\em before} the path addition. 
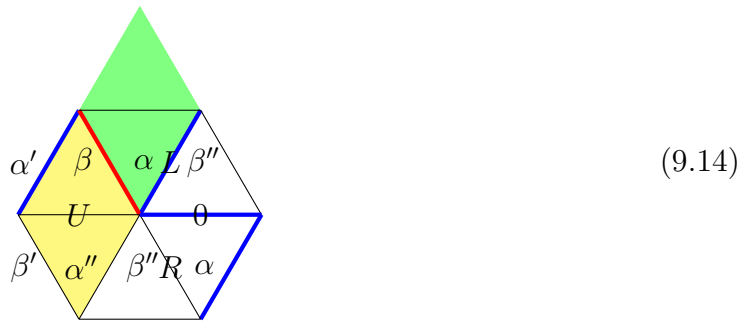
\begin{equation}\label{Eq-hexagon}
%hexagon
\vcenter{\hbox{
\begin{tikzpicture}[x={(1cm*0.4,-\rootthree cm*0.4)},
                    y={(1cm*0.4,\rootthree cm*0.4)}]
%transition hexagon scaled up by 2 and shifted to the right by 3 units
%
\path[fill=yellow,fill opacity=0.5] (0,4)--(0,6)--(2,6)--(2,4)--cycle;
\path[fill=green,fill opacity=0.5] (0,6)--(0,8)--(2,8)--(2,6)--cycle;
%
%hive edges
\draw(0,4)[blue,ultra thick]--(0,6); \draw(2,4)--(2,6); \draw(2,6)[blue,ultra thick]--(2,8); \draw(4,6)[blue,ultra thick]--(4,8);  %alpha edges
\draw(0,4)--(2,4); \draw(0,6)[red,ultra thick]--(2,6); \draw(2,6)--(4,6); \draw(2,8)--(4,8); %beta edges
\draw(0,4)--(2,6); \draw(2,6)[blue,ultra thick]--(4,8); \draw(0,6)--(2,8); \draw(2,4)--(4,6); %gamma edges
%edge labels
\path(0,4)--node[pos=0.5,left]{$\alpha'$}(0,6); % alpha
\path(2,6)--node[pos=0.5,left]{$\alpha\,$}(2,8); % alpha
\path(2,4)--node[pos=0.5,left]{$\alpha''$}(2,6); % alpha
\path(4,6)--node[pos=0.5,left]{$\alpha\,$}(4,8); % alpha
\path(0,4)--node[pos=0.5,left]{$\beta'$}(2,4); % beta
\path(0,6)--node[pos=0.5,left]{$\beta\,$}(2,6); % beta
\path(2,6)--node[pos=0.5,left]{$\beta''$}(4,6); % beta
\path(2,8)--node[pos=0.5,left]{$\beta''$}(4,8); % beta
%upright rhombus entries 
\path(2,6)--node{$0$}(4,8);
\path(0,4)--node{$U$}(2,6);
\path(2,6)--node{$L$}(2,8);
\path(2,6)--node[pos=0.5]{$R$}(4,6);
\end{tikzpicture}
}}
\end{equation}

In this situation, by hypothesis, the yellow upright rhombus has gradient $U=\beta-\beta'=\alpha''-\alpha'>0$ and the
white upright rhombus has gradient $0$. 
Advantage has been taken of the latter condition to equate two pairs of edge labels, one pair labelled $\alpha$ and the other
labelled $\beta''$. The initial hive conditions also imply that $\alpha'\geq\alpha$ and $\beta'\geq\beta''$
from which it follows that $L=\beta-\beta''>0$ and $R=\alpha''-\alpha>0$. After the path addition the edge label
$\beta$ is reduced to $\beta-1$ and the edge label $\alpha$ is increased to $\alpha+1$. All this implies that 
the rhombus gradients $U$, $L$ and $R$ are reduced to $U-1$, $L-1$ and $R-1$, respectively, all of which remain $\geq0$.
In addition the red edge label has been reduced from $\beta$ to $\beta-1=U+\beta'-1\geq\beta'$ which is 
necessarily non-negative by the initial hive conditions.
Thus yet again the hive conditions are preserved.

There is just one more concern that shows itself 
in the case of a type \tov{(iia)} path addition of the
form illustrated on the right in (\ref{Eq-iib-iia}), in which under the action of some $\ov\phi_k$ the path 
reaches the base of the hive in a diagonal $j$ with $k\leq j<r$
and then proceeds along this base to the right-hand corner.
Our previous arguments are sufficient to see that, after the addition of the type \tov{(iia)} path,
all hive rhombus gradient conditions are satisfied, including the non-negativity of the upright rhombus gradient 
signified by $U$ in (\ref{Eq-iib-iia}), except possibly the non-negativity of the gradient of the left-leaning rhombus 
of the penultimate type indicated in (\ref{Eq-rhombi-add}) that appears in
the bottom right hand corner of (\ref{Eq-iib-iia}). 
All the previous arguments about the preservation of hive conditions
have been independent of any positivity assumption about the gradient, $L$, of this
left-leaning rhombus on the base of the hive at its right-hand end.
To meet the hive condition for this rhombus we require 
that $L=\nu_{r-1}-U\geq0$ because at this stage, just before 
the application of $\ov{\chi}_r^{\,\nu_r}$, the bottom two edges on the right-hand boundary have
labels $\nu_{r-1}$ and $0$. 
Since the red edge labels are weakly decreasing from left to 
right both before and after the path addition, all such red edge labels will be non-negative if 
$\nu_{r-1}-U\geq0$. 
Thus the non-negativity of both $L$ and all red edge labels will be guaranteed
after each type \tov{(iia)} path addition if $U\leq\nu_{r-1}$. However, since 
$U$ increases steadily under these path additions from its initial value $0$ to its final value $\tilde{U}_{r-1,r}$,
we need only $\tilde{U}_{r-1,r}\leq\nu_{r-1}$ to ensure that the hive conditions are satisfied at
the conclusion of all type \tov{(ii)} path additions.
 
We will actually need to ensure that a stronger condition is met. This is necessitated by a 
consideration of the final type \tov{(iii)} path additions arising from the action of $\ov{\chi}_r^{\,\nu_r}$.

This simply adds $\nu_r$ to the labels of the two edges meeting at the lower rightmost vertex of $H^{(r)}$, 
as illustrated below, where it has been convenient to denote $\tilde{U}_{r-1,r}$ by $U$.
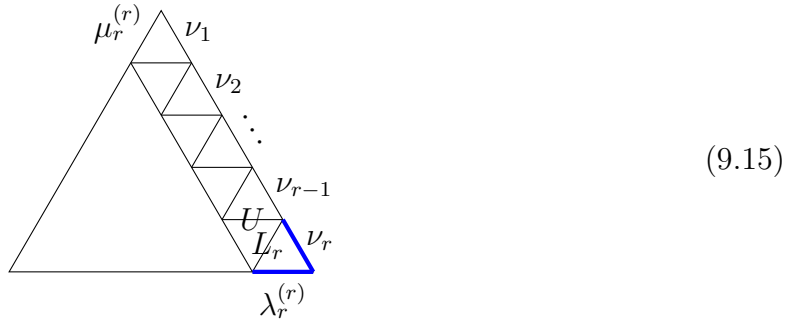
\begin{equation}\label{Eq-omega}
%H4
\vcenter{\hbox{
\begin{tikzpicture}[x={(1cm*0.4,-\rootthree cm*0.4)},
                    y={(1cm*0.4,\rootthree cm*0.4)}]
%label \node at (-3,6){(ii)};  
%boundary edges
\draw(0,0)--(0,5)--(5,5)--cycle;
%boundary edge labels
%\path(0,3)--node[pos=0.75,left]{$\mu_{r-1}^{(r-1)}=\alpha_1$}(0,4); % alpha
\path(0,4)--node[pos=0.75,left]{$\mu_r^{(r)}$}(0,5); % alpha
\foreach\i in{1,...,2}\path(\i-1,5)--node[pos=0.4,right]{$\nu_\i$}(\i,5); %{$\alpha_{\i}$}
\foreach\i in{3,...,3}\path(\i-1,5)--node[sloped,pos=0.5,above]{$\cdots$}(\i,5); % beta
%\foreach\i in{3,...,3}\path(\i-1,5)--node[pos=0.4,right]{$\cdots$}(\i,5); %{$\alpha_{\i}$}
\foreach\i in{4,...,4}\path(\i-1,5)--node[pos=0.4,right]{$\nu_{r-1}$}(\i,5); %{$\alpha_{\i}$}
\path(4,5)--node[pos=0.4,right]{$\nu_r$}(5,5); %{$\alpha_{\i}$}
%\foreach\i in{1,...,4}\path(\i-1,4)--node[pos=0.4,left]{$\nu_\i$}(\i,4); %{$\alpha_{\i}$}
%\foreach\i in{1,...,5}\path(\i-1,\i-1)--node[below]{$\lambda_\i$}(\i,\i);  %{$\gamma_{\i}$}
\path(4,4)--node[below]{$\lambda_r^{(r)}$}(5,5);  %{$\gamma_{\i}$}
%hive edges
\foreach\i in{1,...,4}\draw(\i,4)--(\i,5); % alpha edges
\draw(0,4)--(4,4); % beta edges
\foreach\i in{0,...,4}\draw(\i,4)--(\i+1,5); % gamma edges
\draw(4,5)[blue,ultra thick]--(5,5); %beta
\draw(4,4)[blue,ultra thick]--(5,5); %gamma
%upright rhombus entries 
\foreach\i in{4} \foreach\j in{5}\path(\i-1,\j-1)--node{$U$}(\i,\j);
\path(4,4)--node{$L_r$}(4,5);
\end{tikzpicture}
}}
\end{equation}
As a result of this path addition the left-leaning rhombus at the foot of the $r$th diagonal has its gradient 
$L_r$ decreased from $\nu_{r-1}-U$ to $\nu_{r-1}-U-\nu_r$, which is non-negative if and only if $U\leq \nu_{r-1}-\nu_r$. 
Since this is stronger than our previous condition $U\leq \nu_{r-1}$, it follows that all hive conditions will be satisfied 
by $H^{(r)}$ if it can be shown that $\tilde{U}_{r-1,r}\leq\nu_{r-1}-\nu_r$.

To prove this we first make the following observation regarding the addition of a succession of type \tov{(ii)} paths.
\begin{Lemma}\label{Lem-Kobs1}
During the action of $\ov{\theta}_r\kappa_r^{-1}$ on $H^{(r-1)}$ let a hive path addition of type \tov{(ii)} follow a path $P$ 
starting from the left-hand boundary at level $k$, then the next such path addition follows a path $P'$ lying weakly 
below the path $P$ in each diagonal from the $k$th to the $r$th.
\end{Lemma}

\noindent{\bf Proof}:~~
The path $P$ is generated by the action of $\ov{\phi}_k$
with $k<r$ and consists of a sequence of ladders in the diagonals $p$ for $p=k,k+1,\ldots,r$.
The definition (\ref{Eq-ovtheta}) of $\ov{\theta}_r$ is such that
the next path $P'$ must then be generated by the action of some $\ov{\phi}_{k'}$ with $k'\leq k$. It follows that
as $P'$ moves down and to the right it encounters $P$ in some diagonal $p$ with $p\geq k$. However
it can only do so by intersecting that part of $P$ consisting of a ladder in that diagonal. 
Since the upright rhombi on the ladder all have gradient $0$, the path $P'$ then necessarily follows $P$ to the 
foot of the ladder where it may either continue to follow $P$ into the next diagonal or separate from $P$ 
by continuing down its own ladder in the $p$th diagonal. This is illustrated below in a diagram intended 
to make it clear that $P'$ always remains weakly below $P$.
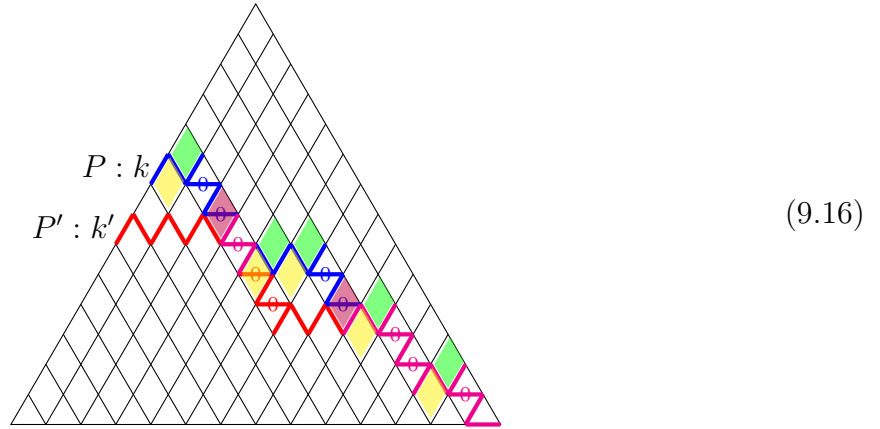
\begin{equation}\label{Eq-Kobs1}
\vcenter{\hbox{

}}
\end{equation}
The interpretation of this is that the first path $P$ is the union of the {\bf\blue blue} and {\bf\magenta magenta}
edges, while the second path $P'$ is the union of the {\bf\red red} and {\bf\magenta magenta} edges. 
Working from left to right the route of $P'$ meets that of $P$ at the site of a {\bf\purple purple} upright
rhombus of gradient $0$. The paths $P'$ and $P$ then coincide as shown in the initial {\bf\magenta magenta} section, 
until it reaches a foot rhombus at which point $P$ continues to the right. Meanwhile, if the gradient
of that yellow rhombus, which was necessarily positive, is reduced to $0$ by the 
addition of the path $P$, then $P'$ continues further down the diagonal. The process then continues
with the divergent paths coming together again in another {\bf\magenta magenta} section, but always with $P'$
weakly below $P$.
\qed

Our next observation concerns the relationship between type \tov{(ii)} path additions
$P_i$ and $Q_i$ generated by $\ov{\theta}_{r}$ and $\ov{\theta}_{r-1}$, respectively. 
Since $\ov{\theta}_{1}$ does not generate any type \tov{(ii)} path additions it may be assumed that $r\geq3$.

\begin{Lemma}\label{Lem-Kobs2} 
For $r\geq3$ let $P_i$ for $i=1,2,\ldots,\lambda_r$ [resp. $Q_i$ for $i=1,2\ldots,\lambda_{r-1}$] be the 
paths added by the operation lying in the $i$th position counted from the left in the expression~(\ref{Eq-ovtheta}) for
$\ov{\theta}_r$ [resp. the corresponding expression for $\ov{\theta}_{r-1}$], that is to say the $i$th position
in the reverse order of action of the individual constituent operations $\ov{\chi}_r$, $\ov{\phi}_k$, $\ov{\omega}_r$
in $\ov{\theta}_r$ [resp. $\ov{\chi}_{r-1}$, $\ov{\phi}_k$, $\ov{\omega}_{r-1}$ in $\ov{\theta}_{r-1}$], assuming that each power of an operator 
is expanded as a product of the same operator occurring as many times as its exponent. We denote the operator giving 
rise to this $i$th path addition by $\ov{\theta}_{r,i}$ [resp. $\ov{\theta}_{r-1,i}$].

Then for each $i$ such that $\nu_{r-1}<i\leq \nu_{r}+\sum_{j=1}^{r-1} V_{jr}$, the paths $P_i$ and $Q_i$ are both of type 
\tov{(ii)} and the path $P_i$ lies strictly above $Q_i$.
\end{Lemma}

\noindent{\bf Proof}:~~
The expansions of $\ov{\theta}_r$ and $\ov{\theta}_{r-1}$ that we have in mind here, take the form:
\begin{equation}\label{Eq-theta-expansions}
\begin{array}{rcl}
\ov{\theta}_r &=& \overbrace{\ov{\chi}_r\cdots\ov{\chi}_r}^{\nu_r} \overbrace{\ov{\phi}_1\cdots\ov{\phi}_1}^{V_{1r}} \overbrace{\ov{\phi}_2\cdots\ov{\phi}_2}^{V_{2r}}~~~\cdots~~~
            \overbrace{\ov{\phi}_{r-1}\cdots\ov{\phi}_{r-1}}^{V_{r-1,r}}\overbrace{\ov{\omega}_r\cdots\ov{\omega}_r}^{\mu_r^{(r)}}\cr\cr 
						%% \kappa_{r}^{-1}\cr\cr
\ov{\theta}_{r-1} &=& \overbrace{\ov{\chi}_{r-1}~~~~~\cdots~~~~~\ov{\chi}_{r-1}}^{\nu_{r-1}} 		\overbrace{\ov{\phi}_1\cdots\ov{\phi}_1}^{V_{1,r-1}}
              ~~~~\cdots~~~~ \overbrace{\ov{\phi}_{r-2}\cdots\ov{\phi}_{r-2}}^{V_{r-2,r-1}}\overbrace{\ov{\omega}_{r-1}\cdots\ov{\omega}_{r-1}}^{\mu_{r-1}^{(r-1)}}\cr\cr 
							%%\kappa_{r-1}^{-1}\cr\cr
									i&=&1~~\cdots~~\nu_r~~\cdots~~\nu_{r-1}~
									\underbrace{\nu_{r-1}+1~~\cdots~~\nu_r+\sum_{j=1}^{r-1} V_{jr}}_{\nu_r+\sum_{j=1}^{r-1} V_{jr} - \nu_{r-1}}
\end{array} 
\end{equation}

It follows from (\ref{Eq-ovtheta}) that $P_i$ will be a type \tov{(ii)} path if and only if $\nu_r<i\leq \nu_r+\sum_{j=1}^{r-1} V_{j,r}$.
Similarly,  $Q_i$ will be a type \tov{(ii)} path if and only if $\nu_{r-1}<i\leq \nu_{r-1}+\sum_{j=1}^{r-2} V_{j,r-1}$. Since $\nu_r\leq \nu_{r-1}$
and $\nu_r+\sum_{j=1}^{r-1} V_{j,r}\leq \nu_{r-1}+\sum_{j=1}^{r-2} V_{j,r-1}$ by virtue of the non-negativity of the right leaning rhombus gradient 
$R_{r-1,r}$ in $K$, it follows that the paths $P_i$ and $Q_i$ are both of type \tov{(ii)} if and only if $\nu_{r-1}<i\leq \nu_{r}+\sum_{j=1}^{r-1} V_{jr}$,
as illustrated above in (\ref{Eq-theta-expansions}).

Fix $i$ in the range $\nu_{r-1}<i\leq \nu_{r}+\sum_{j=1}^{r-1} V_{jr}$ and for notational simplicity set $P=P_i$ and $Q=Q_i$
with the type \tov{(ii)} paths $P$ and $Q$ starting on the left hand boundary at levels $k$ and $l$, respectively. Then
$k>l$ since $\ov{\theta}_{r,i}=\ov{\phi}_k$ implies $\nu_r+\sum_{j=1}^{k-1} V_{j,r} < i \leq \nu_r+\sum_{j=1}^{k} V_{j,r}$
which in turn is $\leq \nu_{r-1}+\sum_{j=1}^{k-1} V_{j,r-1}$ due to the non-negativity of the right leaning rhombus gradient $R_{kr}$ in 
$K$. Hence the path $Q$ passes from the $l$th diagonal to the $(r-1)$th diagonal leaving an upright rhombus of positive gradient
immediately above it in each diagonal from the $(l+1)$th to the $(r-1)$th, necessarily including the $k$th, as illustrated below.

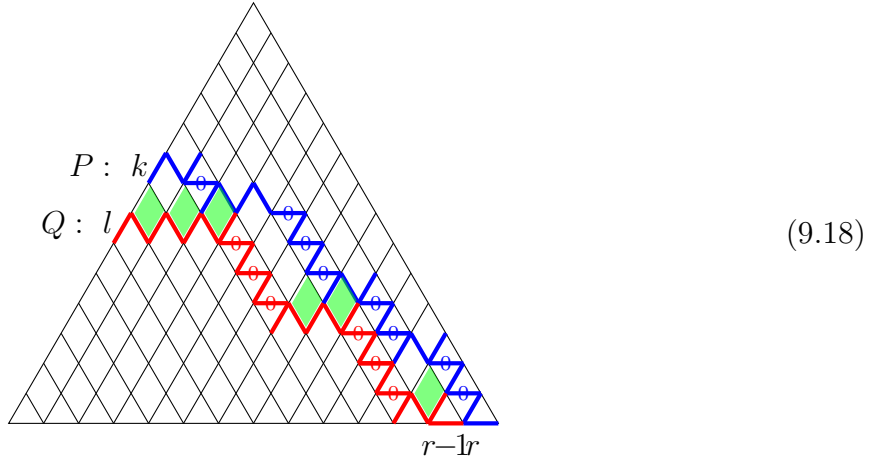
\begin{equation}\label{Eq-QP}
\vcenter{\hbox{

}}
\end{equation}
To show that $P=P_i$ lies strictly above $Q=Q_i$ it only remains to show that the positivity condition on all the 
green rhombus gradients associated with the addition of $Q_i$ remains valid up until the subsequent addition of $P_i$.
%%for all $i$ in the range $\nu_{r-1}<i\leq \nu_{r}+\sum_{j=1}^{r-1} V_{jr}$.

This can be accomplished as follows. Consider first the case $i=m$ where $m=\nu_r+\sum_{j=1}^{r-1}V_{jr}$,
that is the case in (\ref{Eq-ovtheta}) corresponding to the first type \tov{(ii)} path addition 
in the action of $\ov{\theta}_r$ and thus the maximum value of $i$ in (\ref{Eq-theta-expansions}). In this
case $P=P_m$, let the addition of $Q=Q_m$ leave a sequence of upright rhombus gradients $U_{u_m(d),d}>0$ for
$d=l+1,l+2,\ldots,r-1$ with $u_m(d)$ specifying the level at which the path $Q_m$ enters the $d$th diagonal,
with $u_m(d)$ strictly increasing as $d$ increases.
The addition of the remaining \tov{(ii)} paths $Q_i$ with $i<m$ can only increase the value of each $U_{u_m(d),d}$
since, thanks to Lemma~\ref{Lem-Kobs1}, each $Q_i$ with $i<m$ lies weakly below $Q_m$. The remaining operations
between the addition of $Q_m$ and the addition of $P_m$ leave all upright rhombus gradients unaltered. Thus
immediately prior to the addition of $P_m$ we have $U_{u_m(d),d}>0$ for all $d=k,k+1,\ldots,r-1$ so that
$P_m$ cannot pass down through any of these upright rhombi and must therefore lie stricty above $Q_m$ as claimed.

Then, in the case $i=m-1$ if we set $P=P_{m-1}$ and $Q=Q_{m-1}$ the addition of $Q_{m}$ and $Q_{m-1}$
leaves $U_{u_m(d),d}>0$ and $U_{u_{m-1}(d),d}>0$ if $u_{m-1}(d)<u_m(d)$ and $U_{u_{m-1}(d),d}>1$ if $u_{m-1}(d)=u_m(d)$
for $d=l+1,l+2,\ldots,r-1$, with $u_{m-1}(d)$ also strictly increasing as $d$ increases. 
All of these gradients are left unaltered by the addition of the remaining \tov{(ii)} paths $Q_i$ with $i<m-1$, 
and the remaining operations between the addition of $Q_{m-1}$ and the addition of $P_m$. For any $d>k$ the addition of
$P_m$ can only reduce $U_{u_m(d),d}$ by $1$ leaving $U_{u_{m-1}(d),d}>0$ in both cases described above. It follows
that $P_{m-1}$ which starts at a level above $Q_{m-1}$ remains strictly above $Q_{m-1}$ for all $d$ up to
$r-1$.

Continuing in the same way, one finds that $P_i$ lies strictly above $Q_i$ for all $i$ from $\nu_{r-1}+1$ to
$m=\nu_r+\sum_{j=1}^{r-1}V_{jr}$, as claimed.
\qed
\bigskip

We are now in a position to prove the following

\begin{Lemma}\label{Lem-tildeU}
For $r\geq2$ let the action of $\ov{\theta}_r\kappa_r^{-1}$ on $H^{(r-1)}=\ov{\theta}_{r-1}\kappa_{r-1}^{-1}H^{(r-2)}$ yield
$H^{(r)}$ with the upright rhombus gradient in the position $(r-1,r)$ given by $\tilde{U}_{r-1,r}$.
Then $\tilde{U}_{r-1,r}\leq\nu_{r-1}-\nu_r$.
\end{Lemma}

\noindent{\bf Proof}:~~
For $r=2$ there are precisely $V_{12}$ type
\tov{(ii)} path additions and each of these contributes $+1$ to $\tilde{U}_{12}$. There
are no other contributions so that $\tilde{U}_{12}=V_{12}$. However, in $K_{(2)}$ the
right-leaning rhombus hive condition $R_{12}\geq0$ implies that $V_{12}\leq \nu_1-\nu_2$,
thereby establishing the required result in the case $r=2$.

For $r\geq3$ we can exploit Lemma~\ref{Lem-Kobs2}. Given any pair of addition paths $P_i$ and $Q_i$ with 
$Q_i$ necessarily extending as far as the $(r-1)$th diagonal, the fact that $P_i$ lies
strictly above $Q_i$ means that $P_i$ enters the $r$th diagonal above its lowest upright
rhombus, and therefore makes no contribution to $\tilde{U}_{r-1,r}$.
The only possible contributions to $\tilde{U}_{r-1,r}$ are those that might arise from 
the type \tov{(ii)} path additions $P_i$ that are not paired with a corresponding 
type \tov{(ii)} path addition $Q_i$. It then follows immediately from the
alignment of the expansions of $\ov{\theta}_r$ and $\ov{\theta}_{r-1}$ in (\ref{Eq-theta-expansions})
that $\tilde{U}_{r-1,r}\leq \nu_{r-1} -\nu_r$, as required.
\qed
\bigskip

Returning to the proof of Theorem~\ref{The-KtoH} that was interrupted immediately before
we embarked on Lemmas~\ref{Lem-Kobs1}--\ref{Lem-tildeU}, we now know that
$\tilde{U}_{r-1,r}\leq\nu_{r-1}-\nu_r$, as required to prove 
that all hive conditions are satisfied under the successive addition of all paths
as dictated by the action of $\ov{\theta}_r\kappa_r^{-1}$ on the hive $H^{(r-1)}$. Under the induction hypothesis 
that $H^{(r-1)}\in{\cal H}^{(r)}(\lambda^{(r-1)},\mu^{(r-1)},\nu^{(r-1)})$, we have 
therefore established that $H^{(r)}=\ov{\theta}_r\kappa_r^{-1}H^{(r-1)}$ is a hive. Since we have already 
established that $H^{(r)}$ has the appropriate boundary edge labels, it follows that
$H^{(r)}\in{\cal H}^{(r)}(\lambda^{(r)},\mu^{(r)},\nu^{(r)})$. 

This completes the induction argument, and applying this result in the case $r=n$ we conclude
that $H=H^{(n)}\in{\cal H}^{(n)}(\lambda^{(n)},\mu^{(n)},\nu^{(n)})={\cal H}^{(n)}(\lambda,\mu,\nu)$,
as claimed in (\ref{Eq-H-hive}), thereby proving the validity of Theorem~\ref{The-KtoH}.
\qed

%\fi %(temporary \fi)

\section{Hive proofs of the bijective and involutive properties}
\label{Sec-hive-inv}

The relationship between our path removal and path addition operations allows us to establish 
the bijective nature of the maps we have encountered in Theorems~\ref{The-HtoK} and~\ref{The-KtoH}
whose domains we can combine as follows:
\begin{Definition}\label{Def-sigma}
For fixed positive integer $n$, let $\mathcal{H}^{(n)}$ be the union of all $\mathcal{H}^{(n)}(\lambda,\mu,\nu)$ for 
partitions $\lambda$, $\mu$ and $\nu$ with $\ell(\lambda)\leq n$ and $\mu,\nu\subseteq\lambda$ 
with $|\lambda|=|\mu|+|\nu|$. 
Let $\sigma^{(n)}:\mathcal{H}^{(n)}\rightarrow\mathcal{H}^{(n)}$
be such that for each $H\in\mathcal{H}^{(n)}(\lambda,\mu,\nu)$ we have $\sigma^{(n)}:H\mapsto K\in\mathcal{H}^{(n)}(\lambda,\nu,\mu)$
with $K=\sigma^{(n)} H$, as defined in Theorem~\ref{The-HtoK}.
Similarly, let $\ov{\sigma}^{(n)}:\mathcal{H}^{(n)}\rightarrow\mathcal{H}^{(n)}$
be such that for each $K\in\mathcal{H}^{(n)}(\lambda,\nu,\mu)$ we have $\ov{\sigma}^{(n)}:K\mapsto H\in\mathcal{H}^{(n)}(\lambda,\mu,\nu)$
with $H=\ov{\sigma}^{(n)} K$, 
as defined in Theorem~\ref{The-KtoH}.
\end{Definition}

With this definition we have the following:
\begin{Theorem}\label{The-hive-bijection}
The maps $\sigma^{(n)}$ and $\ov{\sigma}^{(n)}$ are mutually inverse bijections.
\end{Theorem}
 
\noindent{\bf Proof}:  
It follows from Theorems~\ref{The-HtoK} and \ref{The-KtoH} that the maps $\sigma^{(n)}$ and $\ov{\sigma}^{(n)}$ 
provide maps from ${\cal H}^{(n)}(\lambda,\mu,\nu)$ to ${\cal H}^{(n)}(\lambda,\nu,\mu)$ and 
from ${\cal H}^{(n)}(\lambda,\nu,\mu)$ to ${\cal H}^{(n)}(\lambda,\mu,\nu)$, respectively. 
Moreover, for any $H\in{\cal H}^{(n)}(\lambda,\mu,\nu)$ if $\sigma^{(n)}H=K$ with $K\in{\cal H}^{(n)}(\lambda,\nu,\mu)$, 
then $K$ records the boundary edges of a sequence of path removals 
that reduce $H$ to the empty hive $H^{(0)}$.
If $K$ is then used to specify the sequence of path additions that
create $\ov{\sigma}^{(n)}K$ from $H^{(0)}$, then this latter sequence is exactly
the reverse of the original sequence with each path removal replaced by its
inverse in the form of a path addition. The sequence 
formed in this way is precisely what is required to re-create $H$ from $H^{(0)}$. 
That is to say, if $\sigma^{(n)}H=K$ then $\ov{\sigma}^{(n)}K=H$. 
By exactly the same type of argument, if one starts with $K\in{\cal H}^{(n)}(\lambda,\nu,\mu)$
and creates $H=\ov{\sigma}^{(n)}K\in{\cal H}^{(n)}(\lambda,\mu,\nu)$, then the action of $\sigma^{(n)}$ on $H$
consists of reversing the order of the path additions and applying their inverses to $H$ in the form of 
corresponding path removals. By this means one necessarily
arrives back at $K\in{\cal H}^{(n)}(\lambda,\nu,\mu)$ as a record of the boundary edges of the sequence of path removals.
That is to say, this time, if $H=\ov{\sigma}^{(n)}K$ then $K=\sigma^{(n)}H$.
It follows that $\sigma^{(n)}$ and $\ov{\sigma}^{(n)}$ are mutually
inverse maps and that both are bijective.
\qed

Our next task is to prove that the map $\sigma^{(n)}$ is an involution. 
To do this we proceed by way of a sequence of Lemmas, in connection with which we need to introduce 
a new type (iv) path removal operation.
%%two new types (iv) and (v) of path removal operations on certain hives.

\begin{Definition}
Given any hive $H\in{\cal H}^{(n)}(\lambda,\mu,\nu)$ with $\lambda_n>0$ and $U_{in}>0$ for some $i<n$,  
the type (iv) path removal operator $\psi_{n}$ acts on $H$ by decreasing the {\red red} edge labels 
by $1$ along a zig-zag path from the edge labelled $\lambda_n$ to that labelled $\nu_k$ 
where $k=\min\{j\,|\,U_{jn}>0\}$ as illustrated below.
For convenience we have set $\beta_{k,n-1}=\beta^\ast$ and $U_{kn}=U^\ast$ with $U^\ast>0$ by the definition of $k$, 
%{\blue and for later use we have indicated one particular internal edge label $\beta^\ast$.} 

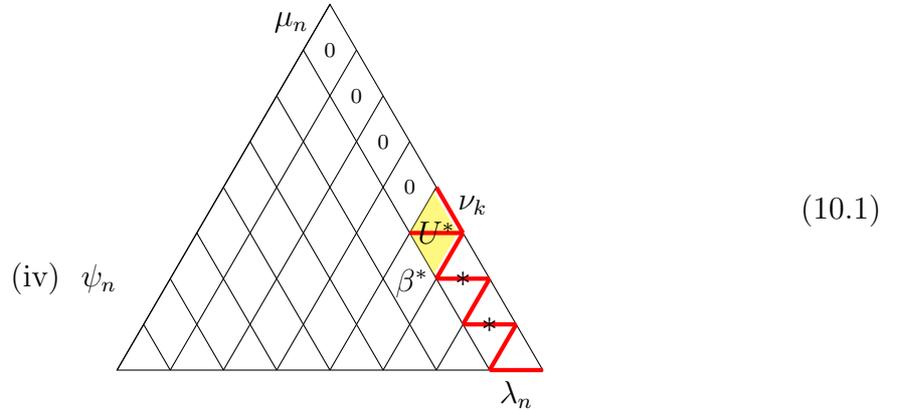
\begin{equation}
\vcenter{\hbox{
\begin{tikzpicture}[x={(1cm*0.35,-\rootthree cm*0.35)},
                     y={(1cm*0.35,\rootthree cm*0.35)}]
%label
\node at (-2,0){(iv)~~$\psi_n$};  
%boundary edges
\draw(0,0)--(0,8)--(8,8)--cycle;
%boundary edge labels
\path(0,7)--node[pos=0.6,left]{$\mu_n$}(0,8); % alpha
\path(4,8)--node[pos=0.4,right]{$\nu_k$}(5,8); \path(4,7)--node[pos=1.1,left]{$\beta^\ast$}(5,7);% beta 
\path(7,7)--node[below]{$\lambda_n$}(8,8);     % gamma
%hive edges
\foreach\i in{0,...,7}\draw(\i,\i)--(\i,8); % ascending alpha edges
\foreach\i in{1,...,7}\draw(0,\i)--(\i,\i); % descending beta edges
%fill upright rhombus yellow
\path[fill=yellow,fill opacity=0.5](4.0,7.0)--(4.0,7.9)--(4.9,7.9)--(4.9,7.0)--cycle; 
%path edges 
\foreach\i in{5,...,7} \foreach\j in{8,...,8}\draw[red,ultra thick](\i,\j-1)--(\i,\j);  %alpha edges 
\foreach\i in{5,...,5} \foreach\j in{8,...,8}\draw[red,ultra thick](\i-1,\j)--(\i,\j);  %beta edges
\foreach\i in{4,...,7} \foreach\j in{8,...,8}\draw[red,ultra thick](\i,\j-1)--(\i+1,\j);  %gamma edges
%upright rhombus entries 
\foreach\i in{1,...,4} \foreach\j in{8,...,8}\path(\i-1,\j-1)--node{$\sc{0}$}(\i,\j); 
\path(4,7)--node{$U^\ast$}(5,8); 
\foreach\i in{6,...,7} \foreach\j in{8,...,8}\path(\i-1,\j-1)--node{$\ast$}(\i,\j);  %on horizontal path edges
\end{tikzpicture}
}}
\end{equation}
\end{Definition}

Prior to the path removal the hive conditions, together with the fact that $U^\ast>0$, $\beta^\ast\geq0$, $\mu_n\geq0$ and $\lambda_n>0$,
imply that all the {\bf\red red} edge labels are positive. This can be seen by noting that the {\bf\red red} $\alpha$-edges
have labels $\geq \mu_n+U^\ast>0$, the {\bf\red red} $\beta$-edge has label $\nu_k=\beta^\ast+U^\ast>0$, and the 
{\bf\red red} $\gamma$-edges have labels $\geq \lambda_n>0$. It follows that after the type (iv) path removal 
these and all other edge labels remain non-negative. In addition the
type (iv) path removal preserves the hive triangle conditions and decreases the bounday edge labels
$\lambda_n$ and $\nu_k$ by one, and the upright rhombus gradient $U_{kn}=U^\ast$ also
by one, whilst preserving the values of all other upright rhombus gradients. 
That such a path removal preserves all the rhombus hive conditions can be seen from the following
diagrams
\begin{equation}\label{Eq-phi-rhombus-conditions}
\begin{array}{ccccc}
%R
\vcenter{\hbox{
\begin{tikzpicture}[x={(1cm*0.5,-\rootthree cm*0.5)},
                    y={(1cm*0.5,\rootthree cm*0.5)}]
\draw(0,0)--(0,1);  
\draw(0,0)--(1,1); 
\draw[red,ultra thick](0,1)--(1,2); 
\draw[red,ultra thick](1,1)--(1,2); 
\path(0,1)--node{$R$}(1,1);
\end{tikzpicture}
}}
&
%U1
\vcenter{\hbox{
\begin{tikzpicture}[x={(1cm*0.5,-\rootthree cm*0.5)},
                    y={(1cm*0.5,\rootthree cm*0.5)}]
\draw(0,1)--(0,2);  
\draw[red,ultra thick](0,2)--(1,2); 
\draw(0,1)--(1,1); 
\draw[red,ultra thick](1,1)--(1,2); 
\draw[red,ultra thick](0,1)--(1,2); 
\path(0,1)--node{$U$}(1,2);
\end{tikzpicture}
}}
&
%U2
\vcenter{\hbox{
\begin{tikzpicture}[x={(1cm*0.5,-\rootthree cm*0.5)},
                    y={(1cm*0.5,\rootthree cm*0.5)}]
\draw[red,ultra thick](0,1)--(0,2);  
\draw(0,2)--(1,2); 
\draw(0,1)--(1,1); 
\draw[red,ultra thick](1,1)--(1,2); 
\draw[red,ultra thick](0,1)--(1,2); 
\path(0,1)--node{$U$}(1,2);
\end{tikzpicture}
}}
&
%L1
\vcenter{\hbox{
\begin{tikzpicture}[x={(1cm*0.5,-\rootthree cm*0.5)},
                    y={(1cm*0.5,\rootthree cm*0.5)}]
\draw(0,1)--(1,2);  
\draw(0,1)--(1,1); 
\draw[red,ultra thick](1,2)--(2,2); 
\draw[red,ultra thick](1,1)--(2,2); 
\path(1,1)--node{$L$}(1,2);
\end{tikzpicture}
}}
&
%L2
\vcenter{\hbox{
\begin{tikzpicture}[x={(1cm*0.5,-\rootthree cm*0.5)},
                    y={(1cm*0.5,\rootthree cm*0.5)}]
\draw[red,ultra thick](0,1)--(1,2);  
\draw(0,1)--(1,1); 
\draw(1,2)--(2,2); 
\draw[red,ultra thick](1,1)--(2,2); 
\draw[red,ultra thick](1,1)--(1,2); 
\path(1,1)--node{$L$}(1,2);
\end{tikzpicture}
}}\cr\cr
~~R\mapsto R\!+\!1~~&~~U\mapsto U\!-\!1~~&~~U\mapsto U~~&~~L\mapsto L\!+\!1~~&~~L\mapsto L\cr
\end{array}
\end{equation}
 
As an immediate consequence of this we have:
\begin{Lemma}\label{Lem-HphiH}
Let $H\in{\cal H}^{(n)}(\lambda,\mu,\nu)$ with $\ell(\lambda)=n$ be such that $U_{in}>0$ for some $i<n$.  
Then $\psi_{n} H\in{\cal H}^{(n)}(\lambda-\epsilon_n,\mu,\nu-\epsilon_k)$ where $k=\min\{j\,|\,U_{jn}>0\}$.
\end{Lemma}

We are now in a position to state what turns out to be a crucial lemma {\it en route} to
Lemma~\ref{Lem-sigma2H} and our involution Theorem~\ref{The-inv}.
\begin{Lemma}\label{Lem-KthetaK}
Let $H\in{\cal H}^{(n)}(\lambda,\mu,\nu)$ with $\ell(\lambda)=n$ be such that $U_{in}>0$ for some $i<n$.  
Then $\sigma^{(n)}\psi_n\, H = \phi_n\, \sigma^{(n)} H$, 
that is to say that the following diagram commutes:
\begin{equation}
\vcenter{\hbox{
\begin{tikzpicture}
\node(tl) at (0,3){$H$};
\node(tr) at (4,3){$K$};
\node(bl) at (0,0){$\H$};
\node(br) at (4,0){$\K$};
\draw[->](tl)--node[above]{$\sigma^{(n)}$}(tr);
\draw[->](bl)--node[below]{$\sigma^{(n)}$}(br);
\draw[->](tl)--node[left]{$\psi_n$}(bl);
\draw[->](tr)--node[right]{$\phi_n$}(br);
\end{tikzpicture}
}}
\end{equation}
where $K=\sigma^{(n)}\,H$, $\H=\psi_n\,H$ and $\K=\sigma^{(n)}\,\H$. 
\end{Lemma}

This result emerges %%very simply 
as the $r=n$ special case of the following more general lemma:
\begin{Lemma}\label{Lem-KhatKr}
Let $H\in{\cal H}^{(n)}(\lambda,\mu,\nu)$ with $\ell(\lambda)=n$ be such that $U_{in}>0$ for some $i<n$,
and let $\H=\psi_n\,H$. Setting $H^{(n)}=H$ and $\H^{(n)}=\H$, let the action of $\sigma^{(n)}$ yield $K=K^{(n)}$
and $\K=\K^{(n)}$ by way of the chains
$$
  (H^{(n)},K^{(0)})\overset{\Theta_n}{\longmapsto}(H^{(n-1)},K^{(1)}) \overset{\Theta_{n-1}}{\longmapsto}(H^{(n-2)},K^{(2)})
	\overset{\Theta_{n-2}}{\longmapsto}\cdots\overset{\Theta_1}{\longmapsto}(H^{(0)},K^{(n)})
$$ 
and
$$
  (\H^{(n)},\K^{(0)})\overset{\Theta_n}{\longmapsto}(\H^{(n-1)},\K^{(1)}) \overset{\Theta_{n-1}}{\longmapsto}(\H^{(n-2)},\K^{(2)})
	\overset{\Theta_{n-2}}{\longmapsto}\cdots\overset{\Theta_1}{\longmapsto}(\H^{(0)},\K^{(n)})\,.
$$ 
Then
\begin{equation}\label{Eq-KhatKr}
         \phi_n\, K^{(r)} = \K^{(r)}   \qquad\hbox{for $r=1,2,\ldots,n$}\,,
\end{equation}		
where the action of $\phi_n$ on the truncated hive $K^{(r)}$ is exactly 
the same as its action would be on a hive except that it terminates on reaching the lower left hand boundary 
of $K^{(r)}$.			
\end{Lemma}

\noindent{\bf Proof}:~~
The action of $\Theta_r$ involves applying $\theta_r$ to $H^{(r)}$ and $\H^{(r)}$ to create $H^{(r-1)}$ and $\H^{(r-1)}$,
while recording information on the relevant path removals in $K^{(n-r+1)}$ and $\K^{(n-r+1)}$ respectively, for $r=n,n-1,\ldots,1$. 
So first we explore the results of applying
$\theta_{n}$ to $H$ and $\theta_{n}$ to $\H:=\psi_{n}\,H$. 
These procedures differ only in respect of the existence of one
extra path removal from $H$, namely the last type (ii) path removal.
Typically, this last type (ii) path removal is illustrated  in {\bf\magenta magenta} and {\bf\blue blue} 
in (\ref{Eq-path01}) below on the left:
\begin{equation}\label{Eq-path01}
\vcenter{\hbox{

}}
\end{equation}
This has been drawn on the assumption that there exists $l<k$ such that
$U_{l,n-1}=Y$ with $Y>0$ before the removal of the {\bf\blue blue} path. 
For simplicity, we refer to this as the $Y>0$ situation. 
If there is no upright rhombus gradient $U_{l,n-1}=Y$ with $l<k$ and $Y>0$, 
which we refer to as the $Y=0$ situation, then the
{\bf\blue blue} path removal extends the whole way up the $(n-1)$th diagonal to the left-hand boundary at 
level $j=n-1$.

All but the last of the type (ii) path removals from $H$ coincide with those from $\H$
and reduce all the upright rhombus gradients in the $n$th diagonal to $0$ except for
that of $U_{kn}$ in $H$ which at this stage takes the value $1$. This gradient is then reduced to $0$
by the final type (ii) path removal from $H$, represented above by the {\bf\magenta magenta} and {\bf\blue blue} path.
This goes up the {\bf\magenta magenta} ladder from its horizontal base to its head rhombus, coloured {\bf\purple purple},
at level $k$ (counted from the top) and then enters the $(n-1)$th diagonal and passes along the {\bf\blue blue} 
path to the left-hand boundary at some level $j$ (counted from the bottom) with $j\leq n-1$. 

Following this, the final $\mu_n$ type (iii) path removals that arise 
if $\mu_n>0$ are necessarily the same for both $H$ and $\H$. 
It follows that $H^{(n-1)}$ differs from $\H^{(n-1)}$ only through differences of 
$\pm1$ as appropriate in the labels of the {\bf\blue blue} edges, 
together with associated differences of $\pm1$ in the
gradients of the foot and head rhombi of the {\bf\blue blue} ladders in the diagonals from $n-1$ down to $j$. 
To be precise, in the case of the above example (\ref{Eq-path01}), if in the $(n-1)$th diagonal two upright rhombus gradients 
at levels $k$ and $l$ separated by $0$'s take values $X\geq0$ and $Y\geq1$ in $\H^{(n-1)}$, then 
they take the values $X+1$ and $Y-1$ in $H^{(n-1)}$, with further differences of $\pm1$ in the
foot and head rhombus gradients in
four more diagonals. If we are in the $Y=0$ situation
then the {\bf\blue blue} path zig-zags up the $(n-1)$th diagonal 
in the form of a type (iii) path leaving an upright rhombus gradient of $X+1$ in $H^{(n-1)}$
and decreasing the top-most edge label by $1$.  

Furthermore, if the {\bf\blue blue} path reaches the left-hand boundary at level
$j\leq n-1$ (counted from the bottom), then the path removals discussed so far create identical
entries $V_{in}\geq0$ in both $K^{(1)}$ and $\K^{(1)}$ for all $i<j$, together with $V_{jn}=V+1$ in  
$K^{(1)}$ and $V_{jn}=V$ in $\K^{(1)}$, respectively, for some $V\geq0$, as illustrated in
(\ref{Eq-path01}) on the right.
It can then be seen from this diagram that the removal of the {\bf\brown} path
from $K^{(1)}$ gives $\phi_n\,K^{(1)}=\K^{(1)}$ since the effect of this removal is 
to reduce the lowest non-vanishing upright rhombus gradient $V+1$ to $V$. 
This proves the validity of (\ref{Eq-KhatKr}) in the case $r=1$.
\bigskip

Next we consider the removal of paths in $\H^{(n-1)}$ and $H^{(n-1)}$ through the action
of $\theta_{n-1}$. There are two cases to consider: Case (i), in which all type (ii) paths
removed from $\H^{(n-1)}$ that enter the $(n-2)$th diagonal  %and $H^{(n-1)}$ starting 
at or below level $k$ pass west and 
northwest to the hive boundary without ever becoming coincident (in the sense of having an edge in common)
with the {\bf\blue blue} path; and Case (ii) in which at least one such path removed from $\H^{(n-1)}$
does become coincident with the {\bf\blue blue} path. 
\medskip1ex
%\end{document}

\noindent{\bf Case (i)}~~Consider first all paths removed from $\H^{(n-1)}$ 
that enter the $(n-2)$th diagonal at or below level $k$ to leave a hive $\H'$. The effect of their removal
is to reduce the upright rhombus gradient $U_{k,n-1}$ in $\H^{(n-1)}$ from $X$ to $0$ in $\H'$.
By hypothesis none of these paths becomes coincident with the {\bf\blue blue} 
path. They pass to the left-hand boundary at level $i<j$ along paths determined
by the upright rhombus gradients below the {\bf\blue blue} path, with the foot rhombus
gradients responsible for an inpenetrable barrier to climbing the {\bf\blue blue} ladders.
Exactly the same paths may be removed from $H^{(n-1)}$
since the upright rhombus gradients below the {\bf\blue blue} path in $H^{(n-1)}$ and $\H^{(n-1)}$
differ only through the foot rhombi of the {\bf\blue blue} ladders having a gradient $1$ greater
than those in $\H^{(n-1)}$, as signified by the symbols $+1$ in each such foot rhombus 
as shown on the left in (\ref{Eq-path01}).
They thereby present exactly the same inpenetrable barrier as
before, leaving the paths themselves unaltered. Removing these paths from $H^{(n-1)}$ to give a hive $H'$ 
reduces the upright rhombus gradient $U_{k,n-1}$ in $H^{(n-1)}$ from $X+1$ to $1$ in $H'$.
In this situation with the upright rhombus gradient $1$ in $H'$ and $0$ in $\H'$,
we say that the rhombus itself is pre-critical. This is signified in (\ref{Eq-path02}) below
by placing the symbols $0$ above $1$ in such an upright rhombus.

In the case of path removals from $H'$, this leaves a single extra type (ii) path $P$ %to enter 
entering the ($n-2$)th diagonal at level $k$, as shown in {\bf\magenta magenta} and {\bf\red red} 
on the left in (\ref{Eq-path02}). %\thetag{106} on the left.
Before its removal, the difference of $+1$ in the gradients of the foot rhombi 
of $H'$ as compared with those of $\H'$, ensures that they remain positive.
It follows that the extra path $P$ stays strictly below the {\bf\blue blue} path
and ends on the left hand boundary at level $i$ with $i<j$. %, as exemplified in (\ref{Eq-path02}) on the left below:

%%%%%%%%%%%%%%%%%%%%%%%Path02
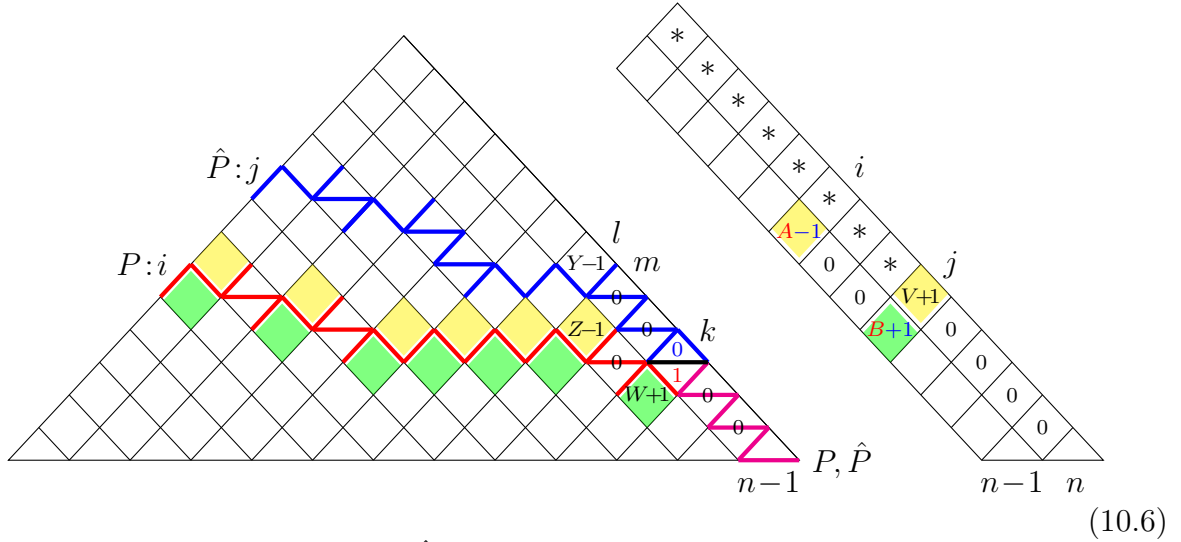
\begin{equation}\label{Eq-path02}
\vcenter{\hbox{

}}
\end{equation}
%%%%%%%%%%%%%%%%%%%%%%%%%%%%%%%

%Noting the intervening gradients $0$ in the $(n-1)$th diagonal, 
%{\red In the case $Y>0$} 
The next path removal from $\H^{(n-1)}$ has also been illustrated on the left
on the assumption that we are in the $Y>0$ situation.  
It takes the form of a type (ii) path $\hat{P}$ and
proceeds up the $(n-1)$th diagonal, following first 
the {\bf\magenta magenta} path and then the {\bf\black} edge at level $k$ across 
what is now the critical rhombus in the $(n-1)$th diagonal. 
After this it necessarily
follows the {\bf\blue blue} path that formed the only difference between
$\H^{(n-1)}$ and $H^{(n-1)}$.  
Thus the critical rhombus has produced a bifurcation of the {\bf\magenta magenta}
path into the {\bf \blue blue} and {\bf \red red } paths followed by $\hat{P}$ and $P$,
respectively.
Thereafter the rhombus has gradient $0$ in both cases and is said to be post-critical.
All the remaining pairs of paths removed from $\H^{(n-1)}$
and $H^{(n-1)}$ coincide, each passing directly through the post-critical rhombus.

On the other hand if we are in the $Y=0$ situation, 
the {\bf\blue blue} path that formed the only difference between
$\H^{(n-1)}$ and $H^{(n-1)}$ would have proceeded directly up the $(n-1)$th diagonal to its top-most edge, with
no difference between the upright rhombus gradients of $\H^{(n-1)}$ and $H^{(n-1)}$ in this diagonal
above level $k$, or indeed any other diagonal at any level.  
It follows that all the remaining type (ii) path removals from $\H^{(n-1)}$ and $H^{(n-1)}$ coincide. This leaves
just type (iii) path removals from $\H^{(n-1)}$ and $H^{(n-1)}$, all following the same path directly
up the $(n-1)$th diagonal, but with one more type (iii) path removal required for $\H^{(n-1)}$ than for $H^{(n-1)}$.
This extra path removal, like all the others of type (iii), must also follow the {\bf\blue blue} path 
terminating at level $j=n-1$, bringing the edge labels along this path into coincidence for both 
$\H^{(n-1)}$ and $H^{(n-1)}$.

To conclude, for both $Y=0$ and $Y>0$ all the path removals from $H^{(n-1)}$ and $\H^{(n-1)}$ coincide
other than the {\bf\magenta magenta}  and {\bf\red red} path removed from $H^{(n-1)}$ and the 
{\bf\magenta magenta}, {\bf\black} and {\bf\blue blue} path removed from $\H^{(n-1)}$. Since the 
{\bf\blue blue} part was already absent in $H^{(n-1)}$ by virtue of its removal 
from $H^{(n)}$ it follows that $\H^{(n-2)}$
differs from $H^{(n-2)}$ only through differences of $\pm1$ as appropriate 
in the labels of the {\bf\red red} edges, together with associated differences of $\pm1$ in the
gradients of the foot and head rhombi of the {\bf\red red} ladders in the diagonals from $n-2$ down to $i$. 

It also follows that if the numbers of paths
reaching $i$ and $j$ in passing from $H^{(n-1)}$ to $H^{(n-2)}$ are
$A$ and $B$, respectively, then the corresponding numbers for the passage
from $\H^{(n-1)}$ to $\H^{(n-2)}$ are $A-1$ and $B+1$, respectively, with
$A\geq1$ and $B\geq0$. There are no paths in either case reaching levels 
between $i$ and $j$. The consequence for this for the upright rhombus gradients of  
$K^{(2)}$ and $\K^{(2)}$ is illustrated on the right in (\ref{Eq-path02}),
where the entries $A$, $B$ and $V+1$ apply to $K^{(2)}$ 
and $A-1$, $B+1$ and $V$ apply to $\K^{(2)}$.
If we clarify this by separating the diagrams of $K^{(2)}$ and $\K^{(2)}$ we find 
that they are related by a single path removal as shown below in (\ref{Eq-KhatK2}). 

\begin{equation}\label{Eq-KhatK2}
\vcenter{\hbox{

}}
\end{equation}
It follows from the above argument that $\phi_n\, K^{(2)}=\K^{(2)}$ for all 
Case (i) examples, thereby establishing for Case (i) %these examples 
the validity of (\ref{Eq-KhatKr}) in the case $r=2$. 
\medskip
%\end{document}

\noindent{\bf Case (ii)} 
By hypothesis, in this case there exists at least one type (ii) path removed from $\H^{(n-1)}$
that enters the $(n-2)$th diagonal at level $k$ or below and becomes coincident with the {\bf\blue blue} path
that differentiates between $\H^{(n-1)}$ and $H^{(n-1)}$. This cannot occur, of course, if this
blue path is confined, as in the $Y=0$ situation, to the $(n-1)$th diagonal. Thus in Case (ii) 
we must be in the $Y>0$ situation.

We then proceed as before by considering pairs of path removals generated by the action of 
$\theta_{n-1}$ on $\H^{(n-1)}$ and $H^{(n-1)}$. Let the sequence of
pairs of type (ii) paths generated in this way be denoted by $\hat{P}_a$ and $P_a$
for $\H^{(n-1)}$ and $H^{(n-1)}$, respectively, with $a=1,2,\ldots,m$. Here $m$ is 
the sum of upright rhombus gradients in the $(n-1)$th diagonal, which is the same for both
$\H^{(n-1)}$ and $H^{(n-1)}$.
As we work up the diagram, 
the two paths $\hat P_a$ and $P_a$, in $\hat H^{(n-1)}$ and $H^{(n-1)}$ respectively, are identical 
so long as the path $\hat P_a$ stays below the route of the {\bf\blue blue} path, due to the same reason as in Case (i).
However, by virtue of our Case (ii) hypothesis, at least one of the paths $\hat P_a$ that enter 
the ($n-2$)th diagonal at or below level $k$, that is to say below the {\bf\blue blue} path, becomes 
coincident with the {\bf\blue blue} path somewhere to its west.
Let $\hat P_c$ be the first such path.
We shall now see that this is a case of bifurcation between the paths $\hat P_c$ and $P_c$ upon encountering a critical rhombus.

Due to the form taken by type (ii) paths passing continuously from one diagonal to another after climbing a ladder in 
each diagonal, the coincidence of $\hat P_c$ with the {\bf\blue blue} path necessarily starts with the path $\hat P_c$ 
entering the {\bf\blue blue} ladder in, say, the $p$th diagonal, traversing its foot rhombus at, say, level $f$, from 
its lower right edge, via its horizontal edge to its upper left edge.
In the picture (\ref{Eq-ii-paths01})%\thetag{108} 
below, the part of $\hat P_c$ up to this foot rhombus is shown in dotted {\bf\magenta magenta}, and the horizontal edge of this rhombus in {\bf black}.
In order for this to happen, this rhombus must have gradient $0$ in this situation.
On the other hand, the path $P_c$, initially coincident with $\hat P_c$ along the dotted {\bf\magenta magenta} path, 
finds that the gradient of the same foot rhombus equal to $1$ instead of $0$, a difference arising from the {\bf\blue blue} 
path distinguishing $H^{(n-1)}$ from $\hat H^{(n-1)}$, and the path $P_c$ therefore passes west below this rhombus.
The situation is marked in (\ref{Eq-ii-paths01}),%\thetag{108}, 
as in Case (i), by placing the symbols {\blue$\mathbf{0}$} above {\red$\mathbf{1}$} in this rhombus, showing its gradients 
before the removals of $\hat P_c$ and $P_c$ respectively.
When reached by the common dotted {\bf\magenta magenta} part of the paths $\hat P_c$ and $P_c$, it causes a bifurcation, 
so this rhombus should duly be called {\em critical} at this point.
Note that this rhombus will carry a common gradient $0$ after this pair of removal, and we shall then call it {\em post-critical} as in Case (i).

To follow the bifurcated paths beyond the critical rhombus, the path $\hat P_c$ necessarily follows the route of the 
original {\bf\blue blue} path to its end at level $j$, since the upright rhombus gradients guiding this route, namely 
those in the original {\bf\blue blue} ladders except for their foot rhombi, have the same values as they did immediately 
before the removal of the original {\bf\blue blue} path during the application of the operator $\theta_n$ to $H^{(n)}$, 
which was missing in the application of $\theta_n$ to $\hat H^{(n)}$.
This part of the original {\bf\blue blue} path is shown in dotted {\bf\blue blue} in (\ref{Eq-ii-paths01}). %\thetag{108}.
The path $P_c$, on the other hand, after decreasing the gradient of the critical rhombus from $1$ to $0$, proceeds along 
what is exemplified as the {\bf\red red} path in \thetag{108}, to end on the left-hand boundary at some level $i$ with $i<j$, 
for the same reason as in Case (i) regarding an inpenetrable barrier below the original {\bf\blue blue} path.

%%%%%%%%%%%%%%%%%%%%
\begin{equation}\label{Eq-ii-paths01}
\vcenter{\hbox{

}}
\end{equation}

%%%%%%%%%%%%%%%%%%%%%%%%%%%%%%%

Let $\H$ and $H$ be the hives obtained from $\H^{(n-1)}$ and $H^{(n-1)}$, respectively, at this stage.
The only edges of $\H$ and $H$  that have different labels constitute what we call a {\em line of difference}. 
Working from right to left, this consists of the 
undotted {\bf\blue  blue} edges, the single {\black} edge and the {\bf\red red} edges, each of whose edge 
labels differs by $\pm1$ as between $\H$ and $H$. For future reference the {\black} edge
label in $\H$ is one greater than that in $H$.
In general there remain further path removals from both $\H$ and $H$ 
under the action of $\theta_{n-1}$.

With an additional assumption 
on the subsequent pairs of paths generated by the action of $\theta_{n-1}$ that none of them 
encounter any critical rhombi except in the starting ($n-1$)th diagonal, such a pair
of path removals from both $\H$ and $H$ emanating from levels at or below
$k$ must proceed together %west or northwest 
weakly above the dotted {\bf\magenta magenta} path and strictly below the {\bf\blue blue} path until they 
enter the $p$th diagonal at a level below (that is geometrically below) level $f$.
Thereafter they follow the dotted {\bf\magenta magenta} path up the $p$th
diagonal, passing inexorably through the post-critical rhombus 
since its gradient is $0$ in both $\H$ and $H$.
They then proceed west or northwest weakly above the dotted {\bf\blue blue} path,
eventually terminating on the left-hand boundary at level $j$ or above. 
Since, by crossing the line of difference at the post-critical rhombus they do not disturb it, these path removals
from $\H$ and $H$ necessarily coincide, and they still leave the {\black} edge having labels 
that differ by $1$ for $\H$ and $H$.

However, because of the distinction between the upright rhombus gradients $X+1$ and $Y-1$ 
in the $(n-1)$th diagonal of $H^{(n-1)}$ at levels $k$ and $l$, 
and the corresponding gradients $X$ and $Y$ in $\H^{(n-1)}$, the next pair of paths
$P_d$ and $\hat{P}_d$ removed from $H^{(n-1)}$ and $\H^{(n-1)}$, respectively,
are no longer coincident. As illustrated below in (\ref{Eq-ii-paths04}), they are
initially coincident along a {\bf\magenta magenta} section in the $(n-1)$th diagonal,
until they both meet the pre-critical rhombus at level $k$. 
Here they diverge, with $P_d$ following a {\bf\red red} path strictly below the %%this was cyan
original {\bf\blue blue path} as far as the post-critical rhombus in the $p$th diagonal at level $q$,
with $\hat{P}_d$ crossing the critical rhombus in its $(n-1)$th diagonal
by means of its {\black} horizontal edge,
and then following the dotted {\bf\blue blue} path, coincident with a 
portion of the original {\bf\blue blue} path, as far once again as the 
post-critical rhombus in the $p$th diagonal at level $q$. At that point the path $P_d$
crosses this post-critical rhombus along what had been in (\ref{Eq-ii-paths01}) its 
{\black} horizontal edge and merges with the path $\hat{P}_d$,
with both then proceeding along a {\bf\magenta magenta} path to the left-hand boundary
at some level $j'\geq j$.

%%%%%%%%%%%%%%%%%%%%Path03
\begin{equation}\label{Eq-ii-paths04}
\vcenter{\hbox{

}}
\end{equation}

Now let $\H'$ and $H'$ be the hives obtained from $\H^{(n-1)}$ and $H^{(n-1)}$, respectively,
at this stage. As far as the two critical rhombi are concerned in 
diagonals $n-1$ and $p$, the gradients are now all $0$ in both $\H'$ and $H'$.
The fact that $\hat{P}_d$, but not $P_d$, passes along the horizontal edge 
of the critical rhombus in the $(n-1)$th diagonal changes the edge 
label in $\H'$ so that it is one lower than that in $H$. 
It is therefore coloured {\black} in (\ref{Eq-ii-paths04}) to signify that it is now to be incorporated 
in the line of difference between $\H'$ and $H'$. In contrast to this, the fact that $P_d$, but not $\hat{P}_d$, 
passes along the horizontal edge of the critical rhombus in the $p$th diagonal of $H$ lowers its edge
label by $1$, bringing it into coincidence with that of $\H'$ and thereby removing this edge 
from the line of difference between $\H'$ and $H'$. This new line of difference thus consists
of the single {\bf\blue blue} and {\black} edges in (\ref{Eq-ii-paths04}), together with the %{\bf\cyan cyan} and 
{\bf\red red} edges.

Continuing to remove the remaining $Y-1$ paths neceassary to reduce to $0$ the upright rhombus gradient at level $l$ in diagonal $n-1$, has
no effect on this line of difference since the paths are coincident in $\H'$ and $H'$, as indeed are all subsequent path removals
generated by the action of $\theta_{n-1}$, including those of type (iii). 
All such paths pass up the $(n-1)$th diagonal to level $l$, crossing the line of difference 
without disturbing it in any way at level $k$, before proceeding either west and northwest or just northwest to the left-hand hive boundary
at or above level $j'$.

As in Case (i) this leave $H^{(n-2)}$ and $\H^{(n-2)}$ differing by a single path removal, namely that
specified this time by the {\bf\red red} part of the line of difference that lies within 
these two $(n-2)$-hives, with $K^{(2)}$ and $\K^{(2)}$
differing exactly as before by the single {\bf\brown} path removal illustrated in (\ref{Eq-KhatK2}). This 
confirms the validity of  $\K^{(2)}=\psi_n\,K^{(2)}$ in all Case (ii) examples 
for which there is just one interior critical rhombus in addition to the critical rhombus on
the right-hand boundary. 

However, there exist other Case (ii) examples involving more than one internal critical rhombus
as will be exemplified next. 
The elucidation of paths proceeds as before. There may exist an initial sequence 
of pairs of paths common to $\H^{(n-1)}$ and $H^{(n-1)}$ that pass to the left-hand boundary 
without meeting the original {\bf\blue blue} path responsible for the distinction
between these two hives. Then by hypothesis there exists a pair of paths $P_c$ and $\hat{P}_c$
that follow a dotted {\bf\magenta magenta} path
along a common route in $\H^{(n-1)}$ and $H^{(n-1)}$ until this reaches the first 
critical rhombus in the $p$th diagonal. Here it bifurcates, with $\hat{P}_c$ following the dotted {\bf\blue blue} path
in $\H^{(n-1)}$ to the left-hand boundary at level $j$ and $P_c$ following the {\bf\red red} path in $H^{(n-1)}$
to the same boundary at level $i$. 
Thereby leaving the following post-critical situation in which the pair of
upright rhombus gradients have both been shown to be $0$.
%%%%%%%%%%%%%%%%%%%%Path05
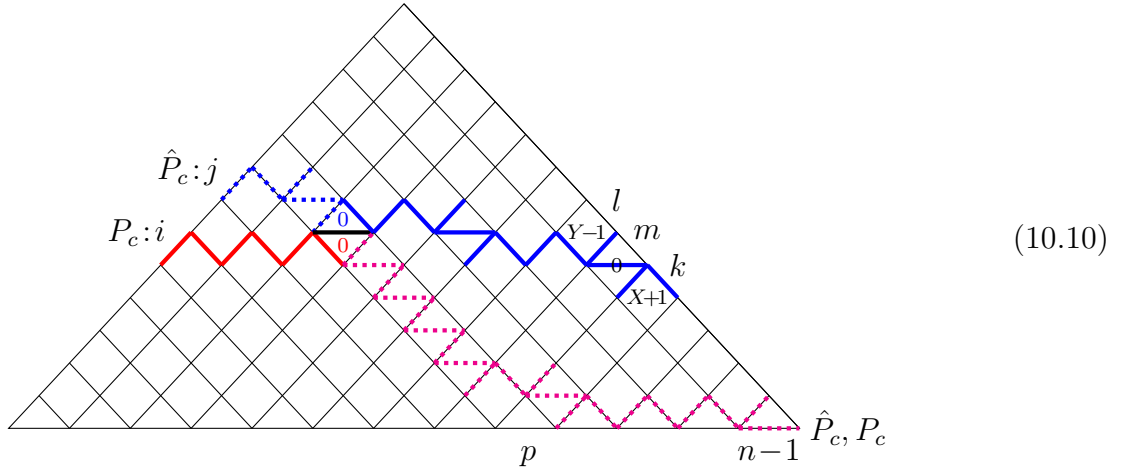
\begin{equation}\label{Eq-ii-paths05}
\vcenter{\hbox{

}}
\end{equation}
%%%%%%%%%%%%%%%%%%%%%%%%%%%%%%%

After this there may 
be a sequence of pairs of path removals common to both $\H^{(n-1)}$ and $H^{(n-1)}$ that proceed
west and northwest strictly below the {\bf\blue blue} path and weakly above the dotted {\bf\magenta magenta} 
path until they intersect it in the $p$th diagonal. They then follow it up this diagonal
passing through the post-critical rhombus and then proceeding west and northwest weakly above the  
dotted {\bf\blue blue} path terminating on the left-hand boundary at level $j$ or above. 

Since there are now to be at least two interior critical rhombi, there 
exists a new pair of path removals, say $P_d$ and $\hat{P}_d$
that follow the dotted {\bf\magenta magenta} path along a common route in $\H^{(n-1)}$ and 
$H^{(n-1)}$ until they reach the critical rhombus in the $q$th diagonal. Here they
diverge making the rhombus post-critical, this time with $P_d$ following the {\bf\red red} %{\bf\cyan cyan} 
path in $H^{(n-1)}$ and 
$\hat{P}_d$ following the dotted {\bf\blue blue} path in 
$\H^{(n-1)}$ that is coincident with the original {\bf\blue blue} path. 
Beyond the $q$th diagonal the %{\bf\cyan cyan} 
{\bf\red red} part of the path $P_d$ eventually reaches the 
$p$th diagonal which it necessarily ascends as far as the post-critical rhombus which it 
then crosses,
only to merge with the dotted {\bf\blue blue} part of the path $\hat{P}_d$ to follow a second dotted 
{\bf\magenta magenta} path to the left-hand boundary weakly above all previous paths to
terminate at some level $j'$. This is illustrated below, with the omission
of the dotted paths from (\ref{Eq-ii-paths05}).

%%%%%%%%%%%%%%%%%%%%Path05
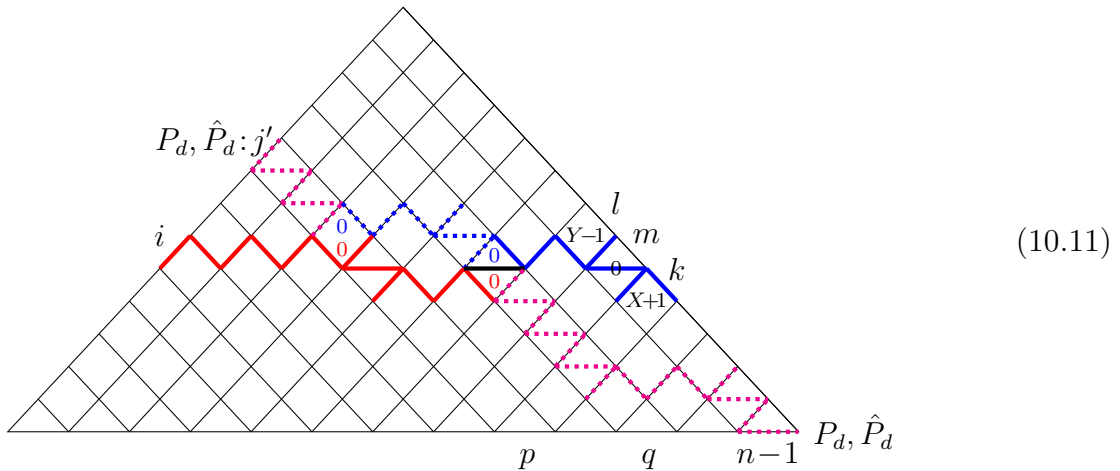
\begin{equation}\label{Eq-ii-paths06}
\vcenter{\hbox{

}}
\end{equation}
%%%%%%%%%%%%%%%%%%%%%%%%%%%%%%%

Now we are in a position analogous to that illustrated in (\ref{Eq-ii-paths01}).
In general there will be a further sequence of pairs of path removals common to both $\H^{(n-1)}$ and $H^{(n-1)}$ that proceed
from the foot of the $(n-1)$th diagonal west and northwest, weakly above the dotted {\bf\magenta magenta} path 
and strictly below the {\bf\blue  blue} path in (\ref{Eq-ii-paths06}). At some point they must reach the $q$th diagonal, 
and then follow it up as far as the second post-critical rhombus. They pass through this, without disturbing the line
of differences, and proceed west and northwest weakly above both the {\bf\blue blue} and {\bf\magenta magenta} 
dotted paths terminating on the left-hand boundary. Provided no more than two internal critical rhombi occur
this sequence of common path removals may be continued
until the upright rhombus at level $k$ in the $(n-1)$th diagonal has become critical with its gradient
reduced to $1$ and $0$ in the cases $\H^{(n-1)}$ and $H^{(n-1)}$, respectively. At this stage,
illustrated below in (\ref{Eq-ii-paths07}), the line of difference consists of the single {\bf\blue blue} and {\black} 
edges, together with the continuous {\bf\red red} path. 

The next pair of path removals are the last of the $X+1$ path removals from $H^{(n-1)}$ to enter diagonal $(n-2)$ at 
level $k$ and the first of the $Y$ path removals from $\H^{(n-1)}$ to enter diagonal $(n-2)$ at level $l$,
These are illustrated below in (\ref{Eq-ii-paths07}) in which the dotted sections of paths in (\ref{Eq-ii-paths06})
have been omitted.

%%%%%%%%%%%%%%%%%%%%Path05
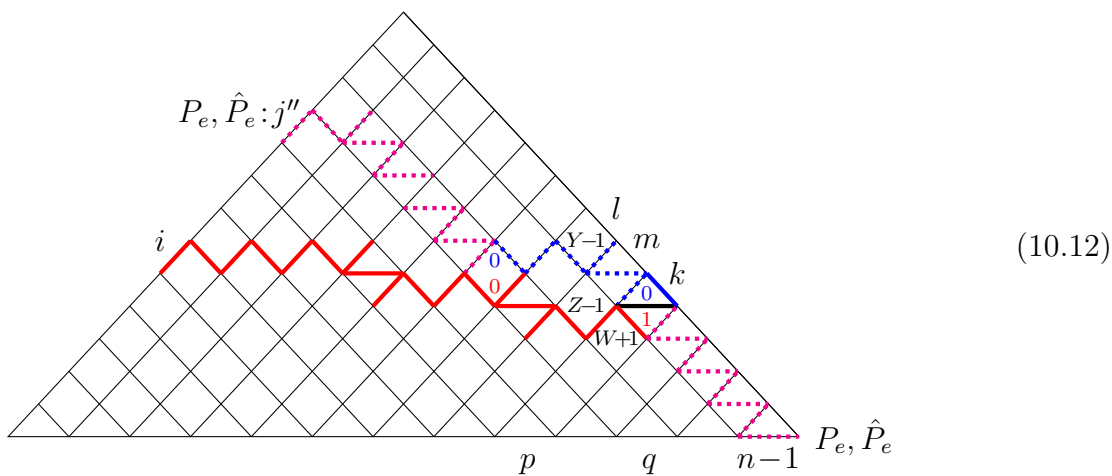
\begin{equation}\label{Eq-ii-paths07}
\vcenter{\hbox{

}}
\end{equation}

%%%%%%%%%%%%%%%%%%%%

%\bigskip
%%{\bf THe following text needs editing, in part because all the colours have changed.}
\bigskip

Let the corresponding pair of paths be $P_e$ and $\hat{P}_e$.
Then $P_e$ traces the same type of path as its predecessors, west and northwest along a {\bf\red red} 
path to the $q$th diagonal, across the critical rhombus decreasing its {\black} edge label in
(\ref{Eq-ii-paths06}) by $1$, thereby removing it from the line of difference, 
and then proceeding west and northwest to the left-hand boundary. 
Path $\hat{P}_e$ follows the undotted {\bf\blue blue} path in (\ref{Eq-ii-paths06}), removing it from the line of difference
and converting it to the dotted {\bf\blue blue} path. In doing so it meets the $q$th diagonal and
thereafter it necessarily merges with $P_e$, as represented in (\ref{Eq-ii-paths07}) by another
dotted {\bf\magenta magenta} portion.

If we remove from (\ref{Eq-ii-paths07}) all dotted edges this exposes more clearly the 
final line of difference, which appears as a {\bf\red red} path extending from
the $(n-1)$th diagonal at level $k$ to the left-hand boundary at level $i$.

%%%%%%%%%%%%%%%%%%%%Path05
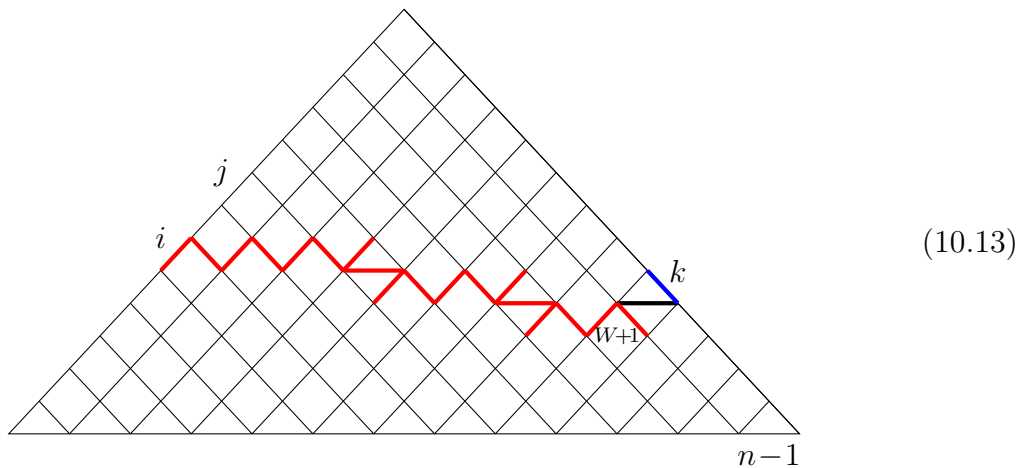
\begin{equation}\label{Eq-ii-paths08}
\vcenter{\hbox{

}}
\end{equation}

As in Case (i) this leaves $H^{(n-2)}$ and $\H^{(n-2)}$ differing by a single path removal,
the {\bf\red red} path, with $K^{(2)}$ and $\K^{(2)}$
differing exactly as before by the single {\bf\brown} path removal illustrated in (\ref{Eq-KhatK2}). 
Thus once again we have $\K^{(2)}=\phi_n\,K^{(2)}$. 

It is clear that this argument may be extended
to the case of an arbitrary number of internal critical rhombi, through the insertion of 
additional pairs of paths that initially coincide, then bifurcate at some critical point
before meeting at the next interior point to the west or northwest, crossing the line of 
difference and thereafter merging and passing to the left hand boundary, precisely as in
(\ref{Eq-ii-paths06}). This still leaves $\K^{(2)}=\phi_n\,K^{(2)}$ in all cases, 
confirming the validity of (\ref{Eq-KhatKr}) in the case $r=2$. 

To extend this argument  to show that $\K^{(r)}=\phi_n\,K^{(r)}$ for $r>2$ 
it should be noted that, just as $H^{(n-1)}$ differs from $\H^{(n-1)}$ only by 
the removal of a single {\bf\blue blue} path entering the $H^{(n-1)}$ hive at level $k$ and leaving at some
level $j$ that we may denote by $h(n-1)$ as shown below in (\ref{Eq-path05}), so $H^{(n-2)}$ differs from $\H^{(n-2)}$ only by 
the removal of a single {\bf\red red} path, 
again entering the $H^{(n-2)}$ hive at level $k$ and leaving at some
level $i<j$ that we may denote by $h(n-2)$.

%%%%%%%%%%%%%%%%%%%%Path05
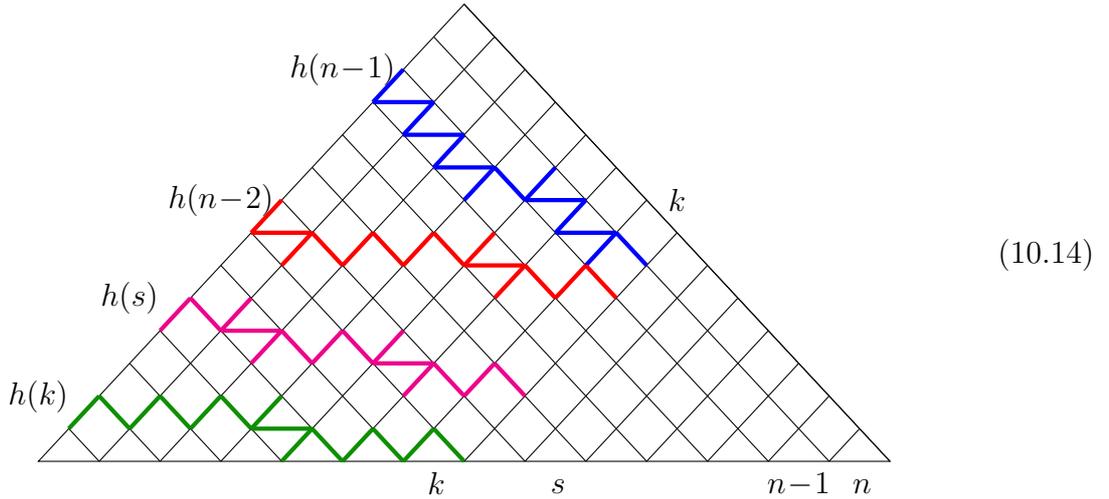
\begin{equation}\label{Eq-path05}
\vcenter{\hbox{

}}
\end{equation}
%%%%%%%%%%%%%%%%%%%%%%%%%%%%%%%

Iterating this process allows one to see that $H^{(s)}$ differs from $\H^{(s)}$ only by the removal of a single 
{\bf\magenta magenta} path entering the $s$-hive region at level $k$ and leaving at 
some level $h(s)$, again as shown above in (\ref{Eq-path05}). This continues until the case $s=k$, exemplified by 
the removal of the %{\bf\oldgreen green} 
%{\bf\darkgreen} 
{\bf\green green}
path that distinguishes $H^{(k)}$ from $\H^{(k)}$.

Now we claim that $H^{(k-1)}=\hat H^{(k-1)}$.
Upon application of $\theta_k$, the hive $H^{(k)}$ affords one extra type (i) path removal compared with $\hat H^{(k)}$, 
whereas $\hat H^{(k)}$ affords one extra type (ii) path removal compared with $H^{(k)}$, whose route coincides with the 
%{\bf\oldgreen green} %
{\bf\green green} 
path, with its rightmost $\beta$ edge, that is to say the edge 
that would initially have been labelled $\beta_{kk}$, 
replaced by the bottom $\gamma$ edge labelled $\lambda_k$.
Thus the pairs of path removals via $\theta_k$, applied to $H^{(k)}$ and $\hat H^{(k)}$, 
eliminate the difference shown by the %{\bf\oldgreen green} 
{\bf\green green}
 path, yielding $H^{(k-1)}=\hat H^{(k-1)}$.
Hereafter it is obvious that the pairs of path removals coincide for both hives all the way %up
to the empty hive $H^{(0)}=\hat H^{(0)}$.

All this is reflected in the partner hive arena.  
The algorithm leading from $K^{(0)}$ and $\K^{(0)}$ to $K^{(n)}$ and $\K^{(n)}$ is such that the 
diagonals of these hives are built up one-by-one by counting the path removals that reach the left-hand
boundaries of $H^{(s)}$ and $\H^{(s)}$ at each level for all $s=n,n-1,\ldots,k$. 
The first step involves the path removals from $H^{(n)}$ and $\H^{(n)}$. These only differ 
by the removal of the single {\bf\blue blue} path illustrated in (\ref{Eq-path01}) 
leaving $H^{(n)}$ at level $j=h(n-1)$ and thereby distinguising the entries $V+1$ and $V$ in
$K^{(1)}$ and $\K^{(1)}$, respectively, where these entries constitute the $n$th diagonal of 
$K^{(n)}$ and $\K^{(n)}$, as shown below in (\ref{Eq-path06}). 

The second step involves the path removals from $H^{(n-1)}$ and $\H^{(n-1)}$. These only differ 
by the removal of a single {\bf\red red} path removed from $H^{(n-1)}$ and leaving at level $i=h(n-2)$ and 
a single {\bf\blue blue} path removed from $\H^{(n-1)}$ and leaving at level $j=h(n-1)$ as shown in (\ref{Eq-path02}). 
This latter path is coincident, as we have explained, with the {\bf\blue blue} path previously removed from $H^{(n)}$.
The corresponding gradients of rhombi in the $(n-1)$th diagonal of $K^{(n)}$ and $\K^{(n)}$ must then differ by $\blue{\pm1}$. 
This, as we have seen in (\ref{Eq-KhatK2}) is just the condition necessary 
for the {\bf\brown} path removal from $K^{(n)}$ to pass continuously from the $n$th diagonal to the
$(n-1)$th diagonal at level $j=h(n-1)$ between the two rhombi whose gradients are changed by $\blue{-1}$ and $\blue{+1}$ in creating $\K^{(n)}$ 
in accordance with the result $\phi_n\ K^{(2)} = \K^{(2)}$. 

At the next stage with the {\bf\red red} path is now interpreted
as the extra path removed from $\H^{(n-2)}$ and there will be an extra path, not shown,
removed from $H^{(n-2)}$. The fact that the {\bf\red red} path serves a dual role
means that the {\bf\brown} path removal from $K^{(n)}$ to give $\K^{(n)}$ extends in a 
continuous manner between the $\red{-1}$ and $\red{+1}$ appearing on the $(n-1)$th 
and $(n-2)$th diagonals at level $i=h(n-2)$ as shown below in (\ref{Eq-path06}).

%%%%%%%%%%%%%%%%%%%%Path06
\begin{equation}\label{Eq-path06}
\vcenter{\hbox{
\begin{tikzpicture}[x={(1cm*0.4,-1.732cm*0.25)},y={(1cm*0.4,1.732cm*0.25)}] % hive Uij coordinates
%boundary edges
\draw(0,0)--(0,14)--(14,14)--cycle;
%boundary edge labels

\path(-1,6)--node{$k$}(0,7); 

\path(1,14)--node{$h(k)$}(3,16);
\path(4,14)--node{$h(s)$}(6,16);
\path(7.5,14.5)--node{$h(n\!-\!2)=i$}(9.5,16.5);
\path(11.5,14.5)--node{$h(n\!-\!1)=j$}(13.5,16.5);

\path(6,6)--node[below]{$~~k{\phantom{1}}$}(7,7); 
\path(8,8)--node[below]{$~~s{\phantom{1}}$}(9,9); 
\path(12,12)--node[below]{$n\!-\!1$}(13,13);
\path(13,13)--node[below]{$~~n{\phantom{1}}$}(14,14); %{$\gamma_{\i}$}
%hive edges
\foreach\i in{1,...,13}\draw(\i,\i)--(\i,14); % ascending alpha edges
\foreach\i in{1,...,14}\draw(0,\i)--(\i,\i); % descending beta edges
%path edges
\foreach \p/\q in { {(0,6)}/{(0,7)},{(0,6)}/{(1,7)},{(1,6)}/{(1,7)},{(1,7)}/{(2,7)},{(2,7)}/{(2,8)},
                    {(2,8)}/{(3,8)},{(3,8)}/{(3,9)},{(3,8)}/{(4,9)},{(4,8)}/{(4,9)},
										{(4,9)}/{(5,9)},{(5,9)}/{(5,10)},{(5,10)}/{(6,10)},{(6,10)}/{(6,11)},     
										{(6,11)}/{(7,11)},{(7,11)}/{(7,12)},{(7,12)}/{(8,12)},{(8,12)}/{(8,13)},
										{(8,12)}/{(9,13)},{(9,12)}/{(9,13)},{(9,12)}/{(10,13)},{(10,12)}/{(10,13)},
										{(10,12)}/{(11,13)},{(11,12)}/{(11,13)},{(11,13)}/{(12,13)},
										{(12,13)}/{(12,14)},{(12,13)}/{(13,14)},{(13,13)}/{(13,14)},{(13,13)}/{(14,14)}}
\draw[brown,ultra thick]\p--\q;

%upright rhombus entries 
%%\foreach \i in {0,...,8}\path(\i,13)--node{$\sc{0}$}(\i+1,14); 
\path(1,6)--node{${\green{\sc{+\!1}}}$}(2,7);
\path(1,7)--node{${\green{\sc{-\!1}}}$}(2,8);
\path(4,8)--node{$\magenta\sc{+\!1}$}(5,9);
\path(4,9)--node{$\magenta\sc{-\!1}$}(5,10);
\path(7,11)--node{$\red\sc{+\!1}$}(8,12);
\path(7,12)--node{$\red\sc{-\!1}$}(8,13);
\path(11,12)--node{$\blue\sc{+\!1}$}(12,13);
\path(11,13)--node{$\blue\sc{-\!1}$}(12,14);
%%%%
\end{tikzpicture}
}}
\end{equation}
%%%%%%%%%%%%%%%%%%%%%%%%%%%%%%%

A similar match occurs in connection with the {\bf\magenta magenta} path removal
from $H^{(s)}$ and $\H^{(s-1)}$. Proceeding in this way for $s=n-2,n-3,\ldots,k$
as far as the {\bf\darkgreen} path,
and taking into account the coincidence of all upright rhombus gradients to be written in 
the $(k-1)$th, $(k-2)$th, $\dots$\ and 2nd diagonals of $K$ and $\hat K$,

one arrives at the required result
$\phi_n\, K^{(n)} = \K^{(n)}$, necessary to complete the proof of Lemma~\ref{Lem-KhatKr}.
\qed      

We may now exploit Lemma~\ref{Lem-KthetaK} in the proof of the following:

%%%%%%%%%%%%%%%%%%%%%
%% taken from terada-hive-conclusion-akt11apr15
%%%%%%%%%%%%%%%%%%%%%

\begin{Lemma}\label{Lem-sigma2H}
Let $H\in{\cal H}^{(n)}(\lambda,\mu,\nu)$ be such that $\ell(\lambda)=n$ then
\begin{equation}\label{Eq-sigma2H}
(\sigma^{(n)})^2\,\theta_n\,H=\theta_n\,(\sigma^{(n)})^2\,H\,. 
\end{equation}
\end{Lemma}

\noindent{\bf Proof}:~~
First let $K=\sigma^{(n)}\,H$ and $L=\sigma^{(n)}\,K$ so that by construction
$K\in{\cal H}^{(n)}(\lambda,\nu,\mu)$ and $L\in{\cal H}^{(n)}(\lambda,\mu,\nu)$. Then 
consider the following diagram:

\begin{equation}\label{Eq-HKL}
\vcenter{\hbox{
\begin{tikzpicture}
\node(h1) at (0,8){$H$};
\node(k1) at (4,8){$K$};
\node(l1) at (8,8){$L$};
%\node(h2) at (0,6){$H$};
\node(k2) at (4,6){$K'$};
\node(l2) at (8,6){$L'$};
%\node(h3) at (0,4){$H$};
\node(k3) at (4,4){$K''$};
\node(l3) at (8,4){$L''$};
\node(h4) at (0,2){$\tilde{H}$};
\node(k4) at (4,2){$K^\dagger$};
\node(l4) at (8,2){$\tilde{L}$};
%\node(h5) at (0,0){$H^\ast$};
%\node(k5) at (4,0){$K^\ast$};
%\node(l5) at (8,0){$L^\ast$};
%Lemmas indicating  commutatitivty
\node(lem) at (6,5){Lemma~\ref{Lem-KthetaK}};
%horizontal maps
\draw[->](h1)--node[above]{$\sigma^{(n)}$}(k1);
\draw[->](k1)--node[above]{$\sigma^{(n)}$}(l1);
\draw[->](k2)--node[above]{$\sigma^{(n)}$}(l2);
\draw[->](k3)--node[above]{$\sigma^{(n)}$}(l3);
\draw[->](h4)--node[above]{$\sigma^{(n)}$}(k4);
\draw[->](k4)--node[above]{$\sigma^{(n)}$}(l4);	
%\draw[->](h5)--node[above]{$\sigma^{(n-1)}$}(k5);
%\draw[->](k5)--node[above]{$\sigma^{(n-1)}$}(l5);	
%vertical maps																		
\draw[->](h1)--node[left]{$\theta_n$}(h4);
%\draw[->](h4)--node[left]{$\delta_n$}(h5);
\draw[->](k1)--node[left]{$\omega_n^{\nu_n}$}(k2);
\draw[->](k2)--node[left]{$\psi_n^{\lambda_n-\mu_n-\nu_n}$}(k3);
\draw[->](k3)--node[left]{$\chi_n^{\mu_n}$}(k4);
%\draw[->](k4)--node[left]{$\delta_n$}(k5);
\draw[->](l1)--node[right]{$\chi_n^{\nu_n}$}(l2);
\draw[->](l2)--node[right]{$\phi_n^{\lambda_n-\mu_n-\nu_n}$}(l3);
\draw[->](l3)--node[right]{$\omega_n^{\mu_n}$}(l4);
%\draw[->](l4)--node[right]{$\delta_n$}(l5);
\end{tikzpicture}
}}
\end{equation}

Just as $\theta_n=\omega_n^{\mu_n}\phi_n^{\lambda_n-\mu_n-\nu_n}\chi_n^{\nu_n}$ maps $H$ and $L$ to $\tilde{H}$ and $\tilde{L}$,
respectively, as in Theorem~\ref{The-Hr-theta-r}, let $\eta_n$ be the operator mapping $K$ to $K^\dagger$, 
that is $\eta_n=\chi_n^{\mu_n}\psi_n^{\lambda_n-\mu_n-\nu_n}\omega_n^{\nu_n}$. 
The type (iii) action of $\omega_n^{\nu_n}$ on $K$ reduces by $\nu_n$ each of the edge labels along a zig-zag down the $n$th diagonal
with the top and bottom edge labels reduced from $\nu_n$ and $\lambda_n$ to $0$ and $\lambda_n-\nu_n$, respectively. 
The type (ii) action of $\psi_n^{\lambda_n-\mu_n-\nu_n}$ then reduces to $0$ all upright rhombus gradients in the $n$th diagonal, 
as well as reducing the bottom edge label from  $\lambda_n-\nu_n$ to $\mu_n$ whilst reducing the right hand boundary edges from $\mu=\mu^{(n)}$
to $(\mu^{(n-1)},\mu_n)$. Finally, the type (i) action of $\chi_n^{\mu_n}$ 
reduces the two boundary edge labels of the triangle at the foot of the $n$th diagonal from $\mu_n$ to $0$. Thus in
the left hand rectangle of the above diagram $\eta_n$ empties the $n$th diagonal of $K$, just as $\theta_n$
empties the $n$th diagonal of $H$.

Since the only difference between $K^\dagger=\sigma^{(n)}\theta_n\,H$ and $K=\sigma^{(n)}\,H$
is a failure to insert the appropriate edge labels and upright rhombus gradients in 
the $n$th diagonal of $K^\dagger$, and the action of $\eta_n$ in mapping $K$ to $K^\dagger$
reduces all these to $0$, it follows that the left hand rectangle of the diagram (\ref{Eq-HKL}) is commutative.

Turning to the upper right hand rectangle, this too is commutative since the
action of $\chi_n^{\nu_n}$ in reducing the edge labels $\lambda_n$ and $\nu_n$ of $K$ 
to $\lambda_n-\nu_n$ and $0$, respectively, 
corresponds exactly to the action of $\omega_n^{\nu_n}$ in reducing the edge labels
$\lambda_n$ and $\nu_n$ of the partner hive $L$ to $\lambda_n-\nu_n$ and $0$, respectively.
A similar result applies to the lower right hand rectangle, and crucially, the central 
rectangles on the right is also commutative by virtue of 
the application of Lemma~\ref{Lem-KthetaK} applied $\lambda_n-\mu_n-\nu_n$ times. 

Thus the following diagram is commutative:
\begin{equation}\label{Eq-HKLHKL}
\vcenter{\hbox{

}}
\end{equation}
This implies the validity of (\ref{Eq-sigma2H}) as required to complete the proof.
\qed

As a consequence of this we have

\begin{Theorem}\label{The-inv}
For all $n\in\N$ and all $n$-hives $H$ we have 
\begin{equation}\label{Eq-inv}
(\sigma^{(n)})^2\,H=H\,. 
\end{equation}
\end{Theorem}

\noindent{\bf Proof}:~~
We proceed by induction with respect to $n$. First it should be noted that in
the case $n=1$ we have
\begin{equation}
\vcenter{\hbox{
%
}}    % gamma
\end{equation}
so that $(\sigma^{(1)})^2\,H=H$ for all $1$-hives $H$.

Now consider the following diagram

\begin{equation}\label{Eq-HKLastdelta}
\vcenter{\hbox{
%
}}
\end{equation}
Here, it should be recalled that $\kappa_n$ is a bijection between $n$-hives whose $n$th diagonals are empty and $(n-1)$-hives.
Thanks to the action of $\theta_n$, $\eta_n$ and $\theta_n$, respectively, 
the $n$th diagonals of $\tilde{H}$, $K^\dagger$ and $\tilde{L}$ are all empty, so that $\kappa_n$ %(see the paragraph before Example~4) 
can act on them legitimately. It simply deletes their $n$th diagonals, with its inverse $\kappa_n^{-1}$ merely restoring them.
It follows once again that this whole diagram is commutative.

By the induction hypothesis, we now assume that $(\sigma^{(n-1)})^2\,H^\ast=H^\ast$.
That is to say $L^\ast=H^\ast$, and we have the following simplified commutative diagram
\begin{equation}\label{Eq-HKLast}
\vcenter{\hbox{
\begin{tikzpicture}
\node(h1) at (0,4){$H$};
\node(k1) at (4,4){$K$};
\node(l1) at (8,4){$L$};
\node(h2) at (0,2){$H^\ast$};
\node(k2) at (4,2){$K^\#$};
\node(l2) at (8,2){$H^\ast$};
%\node(h3) at (0,0){$\delta_n\,H^\ast$};
%\node(k3) at (4,0){$\delta_n\,K^\ast$};
%\node(l3) at (8,0){$\delta_n\,L^\ast$};
%horizontal maps
\draw[->](h1)--node[above]{$\sigma^{(n)}$}(k1);
\draw[->](k1)--node[above]{$\sigma^{(n)}$}(l1);
\draw[->](h2)--node[above]{$\sigma^{(n-1)}$}(k2);
\draw[->](k2)--node[above]{$\sigma^{(n-1)}$}(l2);	
%\draw[->](h3)--node[above]{$\sigma^{(n-1)}$}(k3);
%\draw[->](k3)--node[above]{$\sigma^{(n-1)}$}(l3);	
%vertical maps																		
\draw[->](h1)--node[left]{$\kappa_n\theta_n$}(h2);
%\draw[<->](h2)--node[left]{$\delta_n$}(h3);
\draw[->](k1)--node[left]{$\kappa_n\eta_n$}(k2);
%\draw[<->](k2)--node[left]{$\delta_n$}(k3);
\draw[->](l1)--node[right]{$\kappa_n\theta_n$}(l2);
%\draw[<->](l2)--node[right]{$\delta_n$}(l3);
\end{tikzpicture}
}}
\end{equation}

Thus the hives $H$ and $L$, which are both in ${\cal H}^{(n)}(\lambda,\mu,\nu)$ are
transformed under the action of $\kappa_n\theta_n$ into the same hive $H^\ast\in{\cal H}^{(n-1)}(\lambda^\ast,\mu^\#,\nu^\ast)$ 
for some $\mu^\#$, with $\lambda^\ast=\lambda-\lambda_n\epsilon_n$, $\nu^\ast=\nu-\nu_n\epsilon_n$.

Now we consider the action of $\Theta_n$ on $(H,K^{(0)})$ and $(L,K^{(0)})$ to give $(H^\ast,K_H^{(1)})$ and $(L^\ast,K_L^{(1)})$,
respectively. By construction both $K_H^{(1)}$ and $K_L^{(1)}$ must have lower and 
upper edge labels $\lambda_n$ and $\nu_n$, outer right-hand boundary edge labels $\mu$, as
determined by the $n$-truncated $n$-hive $K^{(0)}$, and inner left-hand boundary edge labels $\mu^\#$,
as determined by the left-hand boundary edge labels of $H^\ast=L^\ast$. These boundary edge labels 
are sufficient to determine an $(n-1)$-truncated $n$-hive completely. It follows that $K_H^{(1)}=K_L^{(1)}$.
Thus both components of $\Theta_n(H,K^{(0)})$ and $\Theta_n(L,K^{(0)})$ coincide.

We know that $H$ and $L$ can be recovered from $\Theta_n(H,K^{(0)})$ and $\Theta_n(L,K^{(0)})$, respectively, 
through applications of the path addition operator $\ov\theta_n$ to the ($n-1$)-hives constituting  
their left-hand components, making use of the upright rhombus gradients in the ($n-1$)-truncated $n$-hives  
constituting their right-hand components.
Hence the equality $\Theta_n(H,K^{(0)})=\Theta_n(L,K^{(0)})$ implies that $H=L$.
That is to say $H=L=(\sigma^{(n)})^2 H$, thereby completing the induction argument
and ensuring the validity of (\ref{Eq-inv}) for all $n$-hives. 
\qed

%%%%%%%%%%%%%%%%%%%%%%%%%%%%%%%%%%%%%%%%%%%

\section{Observations on the centrality of upright rhombus gradients}\label{Sec-obs-U}

Here we briefly mention some ideas that come to mind in reviewing the structure of the hive path removal procedure.
In defining the operator $\theta_r=\omega_r^{\mu_r}\phi_r^{\lambda_r-\mu_r-\nu_r}\chi_r^{\nu_r}$ as in (\ref{Eq-theta-omega-phi-chi}),
its action on $H^{(r)}$ is such that the operators $\chi_r$ and $\omega_r$ not only mutually commute, but also commute with each operator $\phi_r$.
In fact the actions of $\chi_r$ or $\omega_r$ do not change any upright rhombus gradients, causing no change in the way each operator $\phi_r$, 
driven solely by the upright rhombus gradients, finds its path, modifies such gradients and arrives at  
its terminating row number.

This leads to the idea of considering what we call a $U$-system ${\boldsymbol U}^{(r)}$ consisting only of a triangular array of non-negative 
upright rhombus gradients $\{U_{ij}\}_{1\leq i<j\leq r}$, 
representing all $r$-hives having the same upright rhombus gradients $\{U_{ij}\}_{1\leq i<j\leq r}$ with all possible boundary edge labels $\mu$, $\nu$ and $\lambda$, 
on which the action of the operator $\phi_r$ is well defined, yielding a new $U$-system 
${\boldsymbol U'}^{(r)}=\{U'_{ij}\}_{1\leq i<j\leq r}$ and a terminating row number $k$.

Before looking at this action, given any triangular array $\{U_{ij}\}_{1\leq i<j\leq n}$ of non-negative integers, 
it is of interest to know whether a corresponding $U$-system ${\boldsymbol U}^{(n)}$ exists,
that is to say if there exist boundary edge labels $\mu$, $\nu$ and $\lambda$ that satisfy the
criteria laid down in (\ref{Eq-aUb-inequalities}).%% and (\ref{Eq-gabU})}. 
These constraints take the form
\begin{multline}\label{Eq-boundary-region-in-U}
\mu_l-\mu_{l+1}\geq\max\limits_{1\leq i\leq l}\Big(-\textstyle\sum\limits_{k=1}^{i-1}U_{kl}+\textstyle\sum\limits_{k=1}^iU_{k,l+1}\Big),\quad\nu_l-\nu_{l+1}\geq\max\limits_{l<j\leq n}\Big(\textstyle\sum\limits_{k=j}^n U_{lk}-\textstyle\sum\limits_{k=j+1}^nU_{l+1,k}\Big) %\\ (1\leq l\leq n-1)
\end{multline}
for $1\leq  l<n$, with
\begin{equation}\label{Eq-boundary-lambda}
  \lambda_l=\mu_l+\nu_l +\sum_{j=1}^{l-1} U_{jl} - \sum_{k=l+1}^n U_{lk} \qquad\hbox{for $1\leq l\leq n$.}
\end{equation}

One such (canonical, involving neither $\max$ nor $\min$,) solution to determining partitions $\mu$, $\nu$ and $\lambda$ 
satisfying these relations is provided by
\begin{align}
   \mu_i&= \begin{cases}\ds \sum_{1\leq i,j<k\leq n} U_{jk}&\hbox{if $1\leq i<n$};\cr
	                           ~~~~~0&\hbox{if $i=n$}, \end{cases}
	 ~~~~~~~~~~~\nu_i= \begin{cases}\ds \sum_{1\leq i\leq j<k\leq n} U_{jk}&\hbox{if $1\leq i<n$};\cr
	                           ~~~~~0&\hbox{if $i=n$}, \end{cases}\\
	\lambda_i&=  \begin{cases}\ds \sum_{1\leq j\leq i\leq k\leq n:j<k} U_{jk} +2\!\!\!\!\sum_{1\leq i<j<k\leq n} U_{jk}&\hbox{if $1\leq i<n$};\cr
	                           \ds ~~~~~\sum_{j=1}^{n-1} U_{jn}&\hbox{if $i=n$}, \end{cases} 
\end{align}
in which case 
\begin{align}
  &\mu_i-\mu_{i+1}= \sum_{1\leq j\leq i} U_{j,i+1} \qquad \nu_i-\nu_{i+1}= \sum_{i<k\leq n} U_{ik};\\
	&\lambda_i-\lambda_{i+1} = \sum_{j=1}^{i-1} U_{ji}  + \sum_{k=i+2}^{n}  U_{i+1,k}. 
\end{align}
Along with the fact that $U_{jk}>0$ for $1\leq j<k\leq n$ this confirms that in this case all 
three sets of hive rhombus conditions (\ref{Eq-aUb-inequalities}) are satisfied.

It is straightforward to map the corresponding hive $H\in{\cal H}^{(n)}(\lambda,\mu,\nu)$ to
an LR-tableau $T\in{\cal LR}(\lambda/\mu,\nu)$ under the bijection described in Section~\ref{Subsec-hives-tableaux}.
In the case $n=5$ this yields schematically:
\begin{equation}\label{Eq-U-H}
\vcenter{\hbox{

}}  
\end{equation}
It is clear that the semistandardness and lattice permutation requirements of Definition~\ref{Def-ss-lp} are 
rather trivially satisfied by virtue of the distribution of the various blocks of identical entries.
In particular the widths of the {\color{cyan}\textbf{cyan}} blocks are just what are required to ensure the
satisfaction of the lattice permutation property without introducing any max or min operators.

The action of $\theta_r$ on ${\boldsymbol U}^{(r)}$ can be defined to be that of $\phi_r^{U_{r-1,r}+\cdots+U_{1r}}$
which reduces all $U_{ir}$ with $1\leq i<r$ to zero.
Setting ${\boldsymbol U''}^{(r)}=\{U''_{ij}\}_{1\leq i<j\leq r}=\theta_r{\boldsymbol U}^{(r)}$, and further defining the action of $\kappa_r$ 
on ${\boldsymbol U''}^{(r)}$ to be to shave off the components $U''_{ir}$ with $1\leq i<r$ which are already zero, 
yields a smaller $U$-system consisting of the triangular array ${\boldsymbol U''}^{(r-1)}=\{U''_{ij}\}_{1\leq i<j\leq r-1}$.

Then applying $\kappa_2\theta_2\kappa_3\theta_3\cdots\kappa_n\theta_n$ to ${\boldsymbol U}^{(n)}=\{U_{ij}\}_{1\leq i<j\leq n}$
and, while applying each $\theta_r$, recording the numbers $V_{kr}$ for $1\leq k<r\leq n$, of occurrences of $k$ as terminating row numbers 
gives an involutive map $\sigma^{(n)}$ from the $U$-system ${\boldsymbol U}^{(n)}=\{U_{ij}\}_{1\leq i<j\leq n}$ 
to another $U$-system ${\boldsymbol V}^{(n)}=\{V_{kr}\}_{1\leq k<r\leq n}$.

This is illustrated by eliminating the boundary edge labels throughout our Example~\ref{Ex-hive} and restricting the 
operations to be those of $\zeta_r$, $\phi_r$ and $\kappa_r$ alone, for $r=n,n-1,\ldots,1$. This gives
\begin{equation}\label{Eq-4U-HtoK}
\begin{array}{ccccccc} %ccc}
%H4K0
\vcenter{\hbox{
\begin{tikzpicture}[x={(1cm*0.25,-\rootthree cm*0.25)},
                    y={(1cm*0.25,\rootthree cm*0.25)}]
\hiveH{4}{0};  %\node at (4.5,4.5) {\,,\,};
%boundary edge labels
%upright rhombus entries from left to right and top to bottom
\path(0+0,3+0)--node{$\sc1$}(1+0,4+0);
\path(0+0,2+0)--node{$\sc1$}(1+0,3+0);\path(1+0,3+0)--node{$\sc2$}(2+0,4+0);
\path(0+0,1+0)--node{$\sc1$}(1+0,2+0);\path(1+0,2+0)--node{$\sc2$}(2+0,3+0);\path(2+0,3+0)--node{$\sc1$}(3+0,4+0);
%path edges
%K^{(0)} shifted to right by 1 unit
\hiveK{4}{4}{1};
\end{tikzpicture}
}}
&\!\!\!\overset{\kappa_4\phi_4^4\zeta_4}{\longmapsto}\!\!\!
&\vcenter{\hbox{
\begin{tikzpicture}[x={(1cm*0.25,-\rootthree cm*0.25)},
                    y={(1cm*0.25,\rootthree cm*0.25)}]
\hiveH{3}{0};
%upright rhombus entries from left to right and top to bottom
\path(0+0,2+0)--node{$\sc1$}(1+0,3+0);
\path(0+0,1+0)--node{$\sc1$}(1+0,2+0);
\path(1+0,2+0)--node{$\sc3$}(2+0,3+0);
%K^{(1)} shifted to right by 1 unit
\hiveK{3}{4}{1};
%upright rhombus entries from left to right and top to bottom
\path(0+1,3+1)--node{$\sc1$}(1+1,4+1);
\path(1+1,3+1)--node{$\sc1$}(2+1,4+1);
\path(2+1,3+1)--node{$\sc2$}(3+1,4+1);
\end{tikzpicture}
}}
&\!\!\!\overset{\kappa_3\phi_3^4\zeta_3}{\longmapsto}\!\!\!  %(H^{(3)},\zeta_3K^{(1)})=
&\vcenter{\hbox{
\begin{tikzpicture}[x={(1cm*0.25,-\rootthree cm*0.25)},
                    y={(1cm*0.25,\rootthree cm*0.25)}]
\hiveH{2}{0};
%upright rhombus entries from left to right and top to bottom
%\path(0+0,2+0)--node{$\sc0$}(1+0,3+0);
\path(0+0,1+0)--node{$\sc1$}(1+0,2+0);%\path(1+0,2+0)--node{$\sc0$}(2+0,3+0);
%K^{(2)} shifted to right by 1 units
\hiveK{2}{4}{1};
%upright rhombus entries from left to right and top to bottom
\path(0+1,2+1)--node{$\sc1$}(1+1,3+1);\path(0+1,3+1)--node{$\sc1$}(1+1,4+1);
\path(1+1,2+1)--node{$\sc3$}(2+1,3+1);\path(1+1,3+1)--node{$\sc1$}(2+1,4+1);
                                      \path(2+1,3+1)--node{$\sc2$}(3+1,4+1);
\end{tikzpicture}
}}
&\!\!\!\overset{\kappa_1\zeta_1\kappa_2\phi_2\zeta_2}{\longmapsto}\!\!\!   %(H^{(2)},\zeta_2K^{(2)})=
&\vcenter{\hbox{
\begin{tikzpicture}[x={(1cm*0.25,-\rootthree cm*0.25)},
                    y={(1cm*0.25,\rootthree cm*0.25)}]
%%%\hiveH{0}{0}
\node at (0,0)	{$\cdot$};
%K4 shifted to the right by 1 unit								
\hiveK{0}{4}{1};
%upright rhombus entries from left to right and top to bottom
\path(0+1,3+1)--node{$\sc1$}(1+1,4+1);
\path(0+1,2+1)--node{$\sc1$}(1+1,3+1);\path(1+1,3+1)--node{$\sc1$}(2+1,4+1);
\path(0+1,1+1)--node{$\sc1$}(1+1,2+1);\path(1+1,2+1)--node{$\sc3$}(2+1,3+1);\path(2+1,3+1)--node{$\sc2$}(3+1,4+1);
\end{tikzpicture}
}}
\end{array}  %\nonumber
\end{equation}
Thus we have the bijective and indeed involutive map
\begin{equation}\label{Eq-H4-K4}
\begin{array}{ccc} 
%ccc}
%H4K0
\vcenter{\hbox{
\begin{tikzpicture}[x={(1cm*0.25,-\rootthree cm*0.25)},
                    y={(1cm*0.25,\rootthree cm*0.25)}]
\hiveH{4}{0};  %\node at (4.5,4.5) {\,,\,};
%boundary edge labels
%upright rhombus entries from left to right and top to bottom
\path(0+0,3+0)--node{$\sc1$}(1+0,4+0);
\path(0+0,2+0)--node{$\sc1$}(1+0,3+0);\path(1+0,3+0)--node{$\sc2$}(2+0,4+0);
\path(0+0,1+0)--node{$\sc1$}(1+0,2+0);\path(1+0,2+0)--node{$\sc2$}(2+0,3+0);\path(2+0,3+0)--node{$\sc1$}(3+0,4+0);
%path edges
%K^{(0)} shifted to right by 1 unit
%\hiveK{4}{4}{1};
\end{tikzpicture}
}}
&\!\!\!\overset{\sigma^{(4)}}{\longmapsto}\!\!\!
&\vcenter{\hbox{
\begin{tikzpicture}[x={(1cm*0.25,-\rootthree cm*0.25)},
                    y={(1cm*0.25,\rootthree cm*0.25)}]
%%%\hiveH{0}{0}
%\node at (0,0)	{$\cdot$};
%K4 shifted to the right by 1 unit								
\hiveK{0}{4}{0};
%upright rhombus entries from left to right and top to bottom
\path(0+0,3+0)--node{$\sc1$}(1+0,4+0);
\path(0+0,2+0)--node{$\sc1$}(1+0,3+0);\path(1+0,3+0)--node{$\sc1$}(2+0,4+0);
\path(0+0,1+0)--node{$\sc1$}(1+0,2+0);\path(1+0,2+0)--node{$\sc3$}(2+0,3+0);\path(2+0,3+0)--node{$\sc2$}(3+0,4+0);
\end{tikzpicture}
}}
\end{array}  %\nonumber
\end{equation}

The two $U$-systems appearing here may be dressed with boundary edge labels in
many ways to yield pairs of bijectively related LR-hives.
For example, (\ref{Eq-U-H}) in which all $U_{i5}$ are set equal to $0$ 
would yield the example shown
on the left below. More generally, the only requirements are those following from (\ref{Eq-boundary-region-in-U}).
Here these amoumt to the conditions $\mu_1-\mu_2\geq1$, $\mu_2-\mu_3\geq 2$, $\mu_3-\mu_4\geq 2$, $\mu_4\geq0$ and 
$\nu_1-\nu_2\geq1$, $\nu_2-\nu_3\geq 3$, $\nu_3-\nu_4\geq 1$, $\nu_4\geq0$. The hive triangle conditions then imply that
$\lambda_1=\mu_1+\nu_1-3$, $\lambda_2=\mu_2+\nu_2-3$, $\lambda_3=\mu_3+\nu_3+2$ and $\lambda_4=\mu_4+\nu_4+4$.
A more or less random example satisfying these constraints appears below on the right.
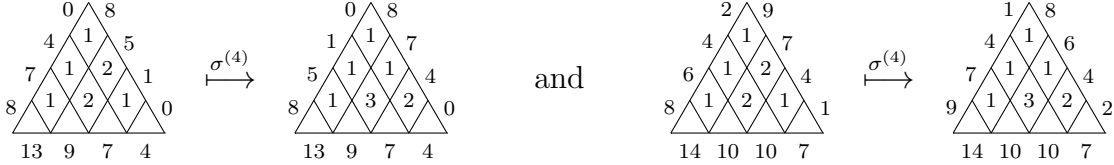
\begin{equation*}
  
\end{equation*}
In fact, in our original Example~\ref{Ex-hive} we had $\lambda=(8,6,5,4)$, $\mu=(6,5,2,0)$ and $\nu=(5,4,1,0)$
as shown below on the left, while we could equally well take
$\lambda=(14,12,11,10)$, $\mu=(8,7,4,2)$ and $\nu=(9,8,5,4)$
as shown below on the right:
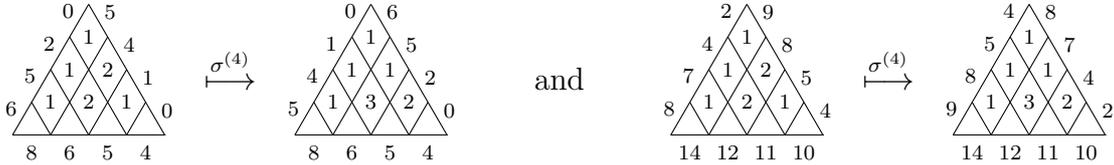
\begin{equation*}
%
  
\end{equation*}
In this case the result on the right follows immediately from that on the left, since
quite generally, we can add $p$, $q$ and $r$ to all $\alpha$-, $\beta$- and $\gamma$-edge
labels of $H\in{\cal H}^{(n)}(\lambda,\mu,\nu)$ without altering the values of $U_{ij}$ for any $i,j$, provided that $r=p+q$, and 
without violating any LR-hive conditions, provided in addition that $p,q,r\in\Z_{\geq0}$.
This amounts to replacing the boundary edge labels $\lambda$, $\mu$ and $\nu$
by $\lambda+(r^n)$, $\mu+(p^n)$ and $\nu+(q^n)$, respectively.
Thereafter, it is only necessary to insert in the appropriate place in examples of the type~\ref{Ex-hive} 
the action of additional operations $\chi_k^q$ and $\omega_k^p$ for $k=1,2,\ldots,n$. This only has the effect
of adding $q$, $p$ and $r=p+q$ to all $\alpha$-, $\beta$- and $\gamma$-edge
labels of $K\in{\cal H}^{(n)}(\lambda,\nu,\mu)$, including of course the boundary edge labels,
without altering the values of $V_{ij}$ for any $i,j$.

In the above $n=4$ example we simply took $p=2$ and $q=4$ so that $r=6$. 
The bijective nature of $\sigma^{(4)}$ then contributes to the proof that $c_{652,541}^{8654}=c_{541,652}^{8654}$
and $c_{8742,9854}^{14\,12\,11\,10}=c_{9854,8742}^{14\,12\,11\,10}$. 
Four other pairs of similar maps involving the same boundary edge labels but different 
pairs of $U$-systems, ${\boldsymbol U}^{(4)}$ and ${\boldsymbol V}^{(4)}$, complete the proof, since
$c_{652,541}^{8654}=c_{8742,9854}^{14\,12\,11\,10}=5$. 

Interestingly, setting $p=-2$ and $q=-4$ leads to
\begin{equation}\label{Eq-UHtoK-pqneg}

\end{equation}
where, as usual $\ov{m}=-m$ for any integer $m$.

Here our hives are no longer Littlewood--Richardson hives since their boundary edge labels, and indeed their interior edge labels as well,
are no longer non-negative.
However, they have a character theoretic significance since the enumeration of hives of this type with given boundary edge labels
serves to give the multiplicities of irreducible constituents in the decomposition of a product of irreducible 
rational characters of $GL_n(\C)$. These irreducible rational characters, $s_{(\lambda;\ov{\rho})}(x_1,\ldots,x_n)$, 
are specified by pairs of partitions~\cite{AK,Ki,St,Ko}
\begin{equation}
(\lambda;\ov{\rho})=(\lambda_1,\lambda_2,\ldots,\lambda_{\ell(\lambda)},0,\ldots,0,\ov{\rho_{\ell(\rho)}}\ldots,\ov{\rho_2},\ov{\rho_1})
\quad\hbox{with}\quad
\ell(\lambda)+\ell(\rho)\leq n.
\end{equation} 
The product rule can then be written in the form
\begin{equation}
   s_{(\mu;\ov{\sigma})}(x_1,\ldots,x_n)\ s_{(\nu;\ov{\tau})}(x_1,\ldots,x_n) 
	       = \sum_{\lambda,\rho}\ c_{(\mu;\ov{\sigma}),(\nu;\ov{\tau})}^{(\lambda;\ov{\rho})} \ s_{(\lambda;\ov{\rho})}(x_1,\ldots,x_n)\,.
\end{equation}

In the case of our $n=4$ example, we have 
\begin{equation}
   c_{(43;\ov{2}),(1;\ov{3}\ov{4})}^{(2;\ov1\ov2)}=c_{(1;\ov{3}\ov{4}),(43;\ov{2})}^{(2;\ov1\ov2)}=5\,,
\end{equation}
where redundant parts equal to $0$ have been dropped, 
and the existence of the bijectve map $\sigma^{(4)}$ on $U$-systems is responsible for the first equality,
and the second follows from the fact that for all $n\in\N$
\begin{equation}
   c_{(\mu;\ov{\sigma})+(p^n),(\nu;\ov{\tau})+(q^n)}^{(\lambda;\ov{\rho})+((p+q)^n)}=c_{(\mu;\ov{\sigma}),(\nu;\ov{\tau})}^{(\lambda;\ov{\rho})}
\end{equation}
for all $p,q\in\Z$. This identity is a consequence of the fact that 
\begin{equation}
 s_{(\mu;\ov{\sigma})+(p^n)}(x_1,\ldots,x_n)=(x_1 \cdots x_n)^p\, s_{(\mu;\ov{\sigma})}(x_1,\ldots,x_n)
\end{equation}
for all $\mu$, $\sigma$ and $p$.

It follows from this that we have a combinatorial hive-based interpretation of the coefficients
\begin{equation}
   c_{(\mu;\ov{\sigma}),(\nu;\ov{\tau})}^{(\lambda;\ov{\rho})} 
\end{equation}
and that the existence of our bijective, involutive  map $\sigma^{(n)}$ provides a combinatorial proof of the
symmetry property
\begin{equation}
   c_{(\mu;\ov{\sigma}),(\nu;\ov{\tau})}^{(\lambda;\ov{\rho})} =c_{(\nu;\ov{\tau}),(\mu;\ov{\sigma})}^{(\lambda;\ov{\rho})} \,.
\end{equation}

It might finally be noted that the above result on multiplicities of irreducible rational characters of $GL_n(\C)$ 
may be extended to the case of irrational characters of certain covering groups of $GL_n(\C)$
by extending the argument to the case of all $p,q\in\R$ or indeed all $p,q\in\C$.

\section{Discussion of the coincidence of commutativity bijections}\label{Sec-coincidence}

We review here the coincidence of the fundamental symmetry maps $\rho_1$, $\rho_2$, $\rho'_2$ in~\cite{PV2}, 
the Henriques--Kamnitzer crystal commutor~\cite{HK2} and 
hive commutor~\cite{HenKam}, as well as the commutativity bijection of Danilov and Koshevoy~\cite{DK05,DK08} on arrays.
Our own commutativity bijections on tableaux and hives both correspond to the
fundamental symmetry map referred to as $\rho_3$ in~\cite{PV2}.
As we have pointed out in Section~\ref{Sec-notation}, 
Littlewood--Richardson coefficients may be evaluated by means of Gelfand--Tsetlin patterns as in~\cite{GZpolyed} and~\cite{BZphy}. 
The first method amounts to exploiting a  bijection, $\tau$, between LR tableaux of shape $\lambda/\mu$ and 
weight $\nu$, and GT patterns of type $\nu$ and weight $\lambda-\mu$ satisfying a $\mu$-dominance condition, as in (\ref{Eq-LRcoeff-GZ}), 
which is equivalent to the non-negativity of the right leaning rhombus gradients $R_{ij}$, as on the left hand side of \eqref{Eq-aUb-inequalities}, 
and the second method amounts to exploiting another
bijection, $\gamma$, between LR tableaux of shape $\lambda/\mu$ and  weight $\nu$, 
and  GT patterns of type $\mu$ and weight the reverse of  $\lambda-\nu$ satisfying a condition, as in (\ref{Eq-LRcoeff-BZ}), 
which is equivalent to the non-negativity of the left leaning rhombus gradients $L_{ij}$, 
as on the right hand side of \eqref{Eq-aUb-inequalities}. 
See also \cite{KiBe,HenKam,PV2}. In~\cite{PV2}, Section~4.5,
the definitions of  bijections, $\tau$ and $\gamma$ are expressed in terms of 
semistandard Young tableaux rather than Gelfand--Tsetlin patterns.
However, these different formulations can be readily translated into the language of hives since the
relevant GT patterns are precisely those exemplified by $G^{(\alpha)}_\mu$ and $G^{(\beta)}_\nu$ in~(\ref{Eq-GTx3}) augmented by appropriate 
missing right and left hand boundary edge labels, as will be illustrated below.

\subsection{Littlewood-Richardson rule, $\mathfrak{gl}_n$-crystal bases and crystal commutor}\label{2subSec-intro} 

Following  \cite{KTWoct} and using $\mathfrak{gl}_n$-crystal bases, Henriques and Kamnitzer studied, in \cite{HK2}, 
the hive model for the general linear Lie algebra, $\mathfrak{gl}_n$, tensor products, and established relations between hives, 
tableaux, Gelfand--Tsetlin patterns, and crystal bases. 

Crystal bases of integrable representations of the quantum group $U_q(\mathfrak{gl}_n)$~\cite{kashiwara} enjoy nice properties. 
In the remainder of this subsection, Greek letters $\lambda$, $\mu$ or $\nu$ will be used to represent partitions of length $\le n$.
For each such $\lambda$, let $B_\lambda$ denote the crystal basis of the irreducible representation $V_\lambda$ of $U_q(\mathfrak{gl}_n)$.
The crystal tensor product $B_\mu\otimes B_\nu$ gives a crystal basis of the tensor product of the representations $V_\mu\otimes V_\nu$ 
of $U_q(\mathfrak{gl}_n)$, and its decomposition into connected components comes from a direct sum decomposition of $V_\mu\otimes V_\nu$ 
into irreducible $U_q(\mathfrak{gl}_n)$-submodules.
Each connected component of $B_\mu\otimes B_\nu$ is isomorphic to some $B_\lambda$ as a crystal, and each $B_\lambda$ has
a unique highest weight element of weight $\lambda$ and a unique lowest weight element of weight $\rev(\lambda)=w_0\lambda$, 
where $w_0$ is the longest element of the Weyl group. 
Thus we may evaluate $c_{\mu\nu}^{\lambda}$ by counting the connected components in $B_\mu\otimes B_\nu$ isomorphic to $B_\lambda$, 
or equivalently, the highest (resp.\ lowest) weight elements of weight $\lambda$ (resp.\ $\rev(\lambda)$) in $B_\mu\otimes B_\nu$.

There is a \qq{reversing} map $\xi$ \cite{HK2} acting on any disjoint union of the crystals of the form $B_\lambda$, including $B_\mu\otimes B_\nu$, 
preserving each connected component but exchanging its highest and lowest weight elements.
In~\cite{HK2} it was shown that it is an involution using a certain uniqueness theorem on crystals.

In the case of $\mathfrak{gl}_n$, we can take $B_\lambda$ to be the set of all semistandard tableaux of shape $\lambda$, in the alphabet $\{1,\dots,n\}$,
suitably equipped with crystal operators~\cite{Nak}.
Then the highest weight element is the Yamanouchi tableau $Y_\lambda$, exemplified on the left below, 
and the lowest weight element $\xi(Y_\lambda)$, exemplified on the right below.
The map $\xi$ coincides with the Sch\"{u}tzenberger involution acting on semistandard tableaux.
\begin{equation}
Y_{753}\ =\  
\YT{0.15in}{}{
 {{1},{1},{1},{1},{1},{1},{1}},
 {{2},{2},{2},{2},{2}},
 {{3},{3},{3}},
} 
\qquad\qquad\qquad
\xi(Y_{753}) \ = \
\YT{0.15in}{}{
 {{1},{1},{1},{2},{2},{3},{3}},
 {{2},{2},{2},{3},{3}},
 {{3},{3},{3}},
} 
\end{equation}

The crystal version of the Littlewood--Richardson rule \cite{T2,Nak,hoon} 
clarifies the equality $c_{\mu\nu}\sp\lambda=|\mathcal{LR}(\lambda/\mu,\nu)|$ by identifying all highest weight elements 
of weight $\lambda$ in $B_\mu\otimes B_\nu$ to be exactly the ones of the form $Y_\mu\otimes T^{(\beta)}_{\nu}$, where $T^{(\beta)}_{\nu}$, 
as exemplified in (\ref{Eq-GTx3}), is the semistandard tableau corresponding to the GT pattern $G^{(\beta)}_\nu$ obtained from $T\in\mathcal{LR}(\lambda/\mu,\nu)$ in the manner described at the end of Section~2.3.
Let $B_T$, for $T\in\mathcal{LR}(\lambda/\mu,\nu)$, 
denote the connected component of $B_\mu\otimes B_\nu$, isomorphic to $B_\lambda$ with 
highest weight element $Y_\mu\otimes T^{(\beta)}_{\nu}$.
In~\cite{HenKam} it was also shown that the lowest weight element of $B_T$ 
is $T^{(\alpha)}_{\mu}\otimes\xi(Y_\nu)$, where $T^{(\alpha)}_{\mu}$, again
as exemplified in (\ref{Eq-GTx3}), is the semistandard tableau corresponding to the GT pattern $G^{(\alpha)}_{\mu}$ obtained from the same $T$.

For the sake of comparison, we note that in \cite{HenKam} the order of tensor factors is reverse to, say, the one used in~\cite{kashiwara,Nak} which we adopt 
here; compare~\cite[p.\ 15, 1st paragraph of Section 5.2]{HenKam} and~\cite[Theorem 1.1.5]{Nak}.
Note also that the drawings of hives in~\cite{HenKam} employ an orientation different from that adopted here.

The tensor product of crystals  is not symmetric in the sense that the map 
from $A\otimes B$ to $B\otimes A$ sending each $a\otimes b$ to $b\otimes a$ does not commute with the crystal operators.
Instead, Henriques and Kamnitzer \cite{HK2} show that the map $\sigma_{A,B}\colon a\otimes b\mapsto\xi(\xi(b)\otimes\xi(a))$ 
does give a crystal isomorphism from $A\otimes B$ to $B\otimes A$,
and that it has the involutive nature expressed by $\sigma_{B,A}\sigma_{A,B}=1$.
Through this isomorphism applied to $B_\mu\otimes B_\nu$, they define an involution $Com_{HK}\colon\mathcal{LR}(\lambda/\mu,\nu)\to\mathcal{LR}(\lambda/\nu,\mu)$, 
sending each $T$ to $T\sp*$ say, in such a way that 
$\sigma_{B_\mu,B_\nu}(B_T)=B_{T\sp*}$ (note that $B_T\subset B_\mu\otimes B_\nu$ while $B_{T\sp*}\subset\linebreak[2] B_\nu\otimes B_\mu$).
Since the map $a\otimes b\mapsto\xi(b)\otimes\xi(a)$ sends the highest weight element $Y_\mu\otimes T^{(\beta)}_{\nu}\in B_T$ to the 
lowest weight element $(T\sp*)^{(\alpha)}_{\nu}\otimes\xi(Y_\mu)\in B_{T\sp*}$, we have $\xi(T^{(\beta)}_{\nu})=(T\sp*)^{(\alpha)}_{\nu}$.
Similarly we have $\xi(T^{(\alpha)}_{\mu})=(T\sp*)^{(\beta)}_{\mu}$.
In terms of the definitions of $\tau$ and $\gamma$ given in~\cite{PV2}, these two formulae take the form $\xi(\tau(T))=\gamma(T\sp*)$ and 
$\xi(\gamma(T))=\tau(T\sp*)$. This means that both $\rho_2=\tau^{-1}\xi\gamma$ and $\rho_2'=\gamma^{-1}\xi\tau$ defined in \cite{PV2} 
coincide with $Com_{HK}$ and are involutions.

More generally, for a map $\rho\colon\mathcal{LR}(\lambda/\mu,\nu)\to\mathcal{LR}(\lambda/\nu,\mu)$, such that $\rho\colon T\mapsto \rho(T)$, 
showing that $(\rho(T))^{(\alpha)}_\nu=\xi(T^{(\beta)}_\nu)$ and $(\rho(T))^{(\beta)}_\mu=\xi(T^{(\alpha)}_\mu)$ would be one way to conclude that 
$\rho$ coincides with $\rho_2$, $\rho'_2$ and $Com_{HK}$, and hence that it is an involution.

 \subsection{LR commutativity bijections and their coincidence}\label{3subSec-intro}

In \cite{PV2}, it is conjectured that $\rho_1$, $\rho_2$, and $\rho_2'$, the inverse of $\rho_2$, and the fundamental symmetry map $\rho_3$, 
described in \cite{PV2}, defined originally in \cite{Az1,Az2} and fully detailed here as $\rho^{(n)}$, coincide. 
From the $\mathfrak{gl}_n$-crystal commutor arguments above we know that $\rho_2=\rho'_2=\rho_2^{-1}=Com_{HK}$.
Using a special case of octahedron recurrence in \cite{KTWoct}, Henriques and Kamnitzer construct a bijective commutor, $Com^h_{HK}$, from 
hives of boundary $(\lambda,\mu,\nu)$ to those of boundary $(\lambda,\nu,\mu)$~\cite{HK2}. The bijection starts with a hive $H$ and consists 
of a series of intermediate stages whose intermediate objects are not hives. It becomes a hive $H^*$ again in the last step. The involution property 
is not proved directly. It is shown that in the equivalent language of GT patterns or semistandard tableaux, intermediate stages correspond to a 
sequence of Bender--Knuth moves \cite{BE} applied to $T^{(\alpha)}_\mu$ whose result is the same as applying the Sch\"utzenberger involution. 
This implies that $Com_{HK}^h=Com_{HK}$ (via the identification of LR tableaux and LR hives by way of $f^{(n)}$), 
from which it follows that the Henriques--Kamnitzer hive commutor $Com^h_{HK}$ is also an involution. 
Therefore $Com_{HK}^h=Com_{HK}=\rho_2=\rho_2'=\rho_2^{-1}$.

In \cite{DK05} Danilov and Koshevoy  introduce a theory of arrays suited to discrete concave functions (hives) on the lattice of integers 
as well as to Young tableaux. The set of arrays is then equipped with the crystal structure  and the tensor product of crystal arrays is 
considered. Using  the Sch\"utzenberger involution, a commutor for the tensor product of arrays is  introduced. Then it is shown that it 
coincides with Henriques--Kamnitzer commutor as well as with the bijections $\rho_1$, $\rho_2$ and $\rho'_2$ defined in terms of Young tableaux. 
This proves the Pak--Vallejo aforementioned conjecture leaving $\rho_3 $ out. The main philosophy in \cite{fpsac08} is to claim  that $\rho_3$ 
can be seen as providing a short cut of tableau-switching when this procedure is applied to LR tableaux, and thereby it coincides with $\rho_1$.

\subsection{Illustration of the maps $\gamma$, $\tau$ and $\xi$}\label{Subsec-maps}

Consider the following example, based on the parameters of (\ref{Ex-tab}) and (\ref{Ex-hive}).

\begin{Example}\label{Ex-tab-hive}
In the case $n=4$ and $\lambda=(8,6,5,4)$, $\mu=(6,5,2,0)$ and $\nu=(5,4,1,0)$, let a typical LR tableau $T\in\mathcal{LR}(\lambda/\mu,\nu)$
and the corresponding LR hive $H\in\mathcal{H}^{(n)}(\lambda/\mu,\nu)$ be as shown below.
\begin{equation}\label{Eq-tab-hive}
T\ = \ \vYTd{0.2in}{}{% T^{(4)}
 {{},{},{},{},{},{},1,1},
 {{},{},{},{},{},1},
 {{},{},1,2,2},
 {1,2,2,3}}
\qquad\qquad
H\ = \ 
\vcenter{\hbox{

}}
\end{equation}
The corresponding left and right leaning GT patterns, when reoriented take the form:
\begin{equation}\label{Eq-GTx2}
G^{(\alpha)}_\mu=
\vcenter{\hbox{%
}}
\end{equation}
It might be noted that $G^{(\alpha)}_\mu$ is of type $\mu=(6,5,2,0)$ and weight $\lambda-\nu$ reversed $=(4,4,2,3)$, and $G^{(\beta)}_\nu$ 
is of type $\nu=(5,4,1,0)$ and weight $\lambda-\mu=(2,1,3,4)$. The corresponding semistandard tableaux $T^{(\alpha)}_\mu$ and $T^{(\beta)}_\nu$
are obtained in the usual way from $G^{(\alpha)}_\mu$ and $G^{(\beta)}_\nu$, respectively, whereby successive rows of each GT pattern determine the shape of 
successive subtableaux restricted to alphabets $\{1,2,\ldots,r\}$ for $r=4,3,2,1$. This yields
\begin{equation}\label{Eq-Tabx2}
\gamma(T)=T^{(\alpha)}_\mu\ = \ \YT{0.2in}{}{% T^{(4)} 
 {1,1,1,1,2,3},
 {2,2,2,4,4},
 {3,4}}
\qquad\qquad
\tau(T)=T^{(\beta)}_\nu\ = \ \YT{0.2in}{}{% T^{(4)} 
 {1,1,2,3,4},
 {3,3,4,4},
 {4}}
\end{equation}

Continuing with this example, we are now in a position to compare the output of Examples~\ref{Ex-tab} and ~\ref{Ex-hive}
with the corresponding result obtained through the action of the Henriques--Kamnitzer hives commutor $Com^h_{HK}$ or octahedral map in \cite{HenKam}. 

Applying the Sch\"utzenberger involution $\xi$ to $ T^{(\alpha)}_\mu$ in the form of Bender--Knuth moves
gives a succession of semistandard tableaux whose Gelfand--Tsetlin patterns 
correspond to the sequence of quasi-hives generated by the action of the octahedral map $Com^h_{HK}$. 
At the end a semistandard tableau $\xi(T^{(\alpha)}_\mu)$ is arrived at
whose GT pattern, denoted by $\xi(G^{(\alpha)}_\mu)$, of type $\mu$ and weight $\lambda-\nu$ corresponds 
to a hive $K\in{\cal H}^{(n)}(\lambda,\nu,\mu)$, as illustrated below:

\begin{equation}\label{Eq-xiT-K}
\xi(T^{(\alpha)}_\mu)\ =\ \YT{0.2in}{}{% T^{(4)} %\vYTd
 {1,1,1,2,3,4},
 {2,3,3,3,4},
 {4,4}}
\qquad\qquad
\xi(G^{(\alpha)}_\mu)\ = \ 
\vcenter{\hbox{
}}
\end{equation}
It is easy to see that this GT pattern is the reoriented form of the constituent GT pattern of the hive $K$
appearing at the end of Example~\ref{Ex-hive}. It has type $\mu=(6,5,2,0)$ and weight $\lambda-\nu=(3,2,4,4)$, 
and the bijection between LR hives and LR tableaux allows us to construct the correponding
LR tableau $S$ that appears at the end of Example~\ref{Ex-tab}, as shown below:
\begin{equation}
K\ = \ 
\vcenter{\hbox{
%
}}
\qquad\qquad
S\ =\ \ \YT{0.2in}{}{% T^{(4)} %\vYTd
 {{},{},{},{},{},1,1,1},
 {{},{},{},{},1,2},
 {{},1,2,2,2},
 {1,2,3,3}}
\end{equation}

This coincidence with the output of our previous Examples~\ref{Ex-hive} and~\ref{Ex-tab}
supports the case for the coincidence of $\rho_2$ and $\rho_3$, that was exemplified by means
of a refinement of the path deletion operations in~\cite{fpsac08}.
\end{Example}

\section{Concluding remarks}\label{Sec-conclusion}

We have given direct combinatorial proofs, using two distinct models, of the bijective and involutive nature of
a procedure first introduced by Azenhas~\cite{Az1,Az2} as a means of establishing combinatorially the symmetry of Littlewood--Richardson coefficients.
The first model was based on the use of Littlewood--Richardson tableaux and followed Azenhas's procedure in
providing a commutativity operator, denoted here by $\rho^{(n)}$, mapping a given LR tableau $T\in\mathcal{LR}(\lambda/\mu,\nu)$ to  
a new LR tableau $S\in\mathcal{LR}(\lambda/\nu,\mu)$, thereby preserving the outer shape while interchanging the inner shape and the weight.
The transformation from $T$ to $S$ consisted of a sequence of \textit{deletions} from $T$,
starting from right to left in each row taken in turn from bottom to top, with the new tableau $S$ 
constructed by arranging the \textit{terminating row numbers} obtained from the deletions in the form of a tableau, 
starting from its bottom row and proceeding upwards, as exemplified in Example~\ref{Ex-tab}.
The second model was based on the use of Littlewood--Richardson \textit{hives}, on which we defined a corresponding commutativity operator
which we denoted by $\sigma^{(n)}$. It transforms a given LR hive $H\in{\cal H}^{(n)}(\lambda,\mu,\nu)$
to a new LR hive $K\in{\cal H}^{(n)}(\lambda,\nu,\mu)$ by the application of what we called \textit{path removals}
from $H$, working again from right to left, with each path starting from the base of the hive, and recording within $K$ 
the level reached by each path, as exemplifed this time in Example~\ref{Ex-hive}.

Although it was not stated explicitly in the text, it should be understood that the two procedures, one in the realm of tableaux 
and the other in that of hives, are intimately related at every stage by the bijective correspondence $f^{(n)}$ between 
LR tableaux and LR hives, as described in Section~\ref{Subsec-hives-tableaux}.
Nevertheless, we have been at pains to keep the two procedures self-contained, without interrupting them in
Sections~\ref{Sec-tab-del-ops}-\ref{Sec-prf-delaying} and Sections~\ref{Sec-hive-path-removal}-\ref{Sec-hive-inv}
with repetitious accounts of the connection between them. In fact each of the two formulations has its own particular 
merits and the development of the argument in the tableaux setting has benefitted from insights
obtained in the hive setting and {\it vice versa}. 
\bigskip

\noindent {\bf Acknowledgements} 
This collaboration has been characterised by frequent gatherings of two or more of the authors not only at the three Universities of Coimbra, Southampton and Tokyo,
but also at Witney in the UK as the guests of Dr Elizabeth King. All the authors wish to acknowledge the hospitality extended to each of them during these visits, as well 
as the generous financial support from the Universities of Coimbra and Tokyo that made each visit possible.
In particular, this work was partially supported by the Centre for Mathematics of the University of Coimbra -- UID/MAT/00324/2013, funded by the Portuguese
Government through FCT/MEC and co-funded by the European Regional Development Fund through the Partnership Agreement PT2020, and by the FCT sabbatical 
grant SFRH/BSAB/113584/2015. The first author (OA) also wishes to acknowledge the hospitality of the University of Vienna where her sabbatical leaving took place.
The third author (IT) was partially supported by JSPS KAKENHI Grant Number 23540008.

%%%%%%%%%%%%%%%%%%%%%%%%%%%%%%%%%%%%%%%%
\label{Sec-biblio}

\end{document}